\documentclass[10pt,francais,leqno]{smfart}

\usepackage[T1]{fontenc}
\usepackage[francais]{babel}

\usepackage{amssymb,url,xspace,smfthm}
\usepackage{mathrsfs,euscript,color}

\usepackage[all]{xy}

\newcommand{\BibTeX}{{\scshape Bib}\kern-.08em\TeX}
\newcommand{\T}{\S\kern .15em\relax }
\newcommand{\AMS}{$\mathcal{A}$\kern-.1667em\lower.5ex\hbox
        {$\mathcal{M}$}\kern-.125em$\mathcal{S}$}

\theoremstyle{plain}

\newtheorem*{montheo}{\textsc{Théorème}}
\newtheorem*{montheo1}{\textsc{Théorème 1}}
\newtheorem*{montheo2}{\textsc{Théorème 2}}
\newtheorem*{mapropo}{\textsc{Proposition}}

\newtheorem*{meshypos}{\textsc{Hypothèses}}

\newtheorem*{monlem}{\textsc{Lemme}}
\newtheorem*{monlem1}{\textsc{Lemme 1}}
\newtheorem*{monlem2}{\textsc{Lemme 2}}
\newtheorem*{monlem3}{\textsc{Lemme 3}}
\newtheorem*{monlem4}{\textsc{Lemme 4}}

\newtheorem*{marema}{\textsc{Remarque}}
\newtheorem*{marema1}{\textsc{Remarque 1}}
\newtheorem*{marema2}{\textsc{Remarque 2}}
\newtheorem*{marema3}{\textsc{Remarque 3}}
\newtheorem*{marema4}{\textsc{Remarque 4}}

\newtheorem*{notation}{\textsc{Notation}}

\catcode`\@=11
\catcode`\Ž=13
\catcode`\=13
\catcode`\ˆ=13
\catcode`\=13
\catcode`\=13
\catcode`\‰=13
\catcode`\™=13
\catcode`\"=13
\catcode`\ž=13
\catcode`\=13
\catcode`\'=13
\def Ž{\'e}
\def {\`e}
\def ˆ{\`a}
\def {\`u}
\def {\^e}
\def ‰{\^a}
\def ™{\^o}
\def "{\^{\i}}
\def ž{\^u}
\def {\c c}
\def '{\"e}
\catcode`\:=13
\def :{~\string:}
\catcode`\;=13
\def ;{~\string;}
\catcode`\!=13
\def !{~\string!}

\def\ES#1{\EuScript{#1}}

\def\wt#1{\widetilde{#1}}
\def\bs#1{\boldsymbol{#1}}
\def\cad{c'est--\`a--dire\ }

\title[LE LEMME FONDAMENTAL POUR L'ENDOSCOPIE TORDUE]
{LE LEMME FONDAMENTAL POUR L'ENDOSCOPIE TORDUE: R\'EDUCTION AUX \'EL\'EMENTS UNIT\'ES}

\author{Bertrand Lemaire}
\address{Institut de Math\'ematiques de Luminy et
UMR 6206 du CNRS\\
Universit\'e Aix--Marseille II, Case Postale 907\\
F--13288 Marseille Cedex 9}
\email{Bertrand.Lemaire@univ-amu.fr}

\author{Colette M\oe glin}
\address{CNRS, Institut de Mathématiques de Jussieu\\ 
2 place Jussieu 75005 Paris}
\email{colette.moeglin@imj-prg.fr}

\author{Jean--Loup Waldspurger}
\address{CNRS, Institut de Mathématiques de Jussieu\\
2 place Jussieu 75005 Paris}
\email{jean-loup.waldspurger@imj-prg.fr}

\begin{document}
\def\smfbyname{}
\small

\begin{abstract}On prouve ici que le lemme fondamental pour l'endoscopie tordue, désormais établi pour les unités des algèbres de Hecke sphériques, entra\^{\i}ne le lemme fondamental pour tous les éléments de ces algèbres de Hecke. La démonstration, dont l'idée est due à Arthur, utilise le transfert, déjà établi lui aussi comme conséquence du lemme fondamental pour les unités.
\end{abstract}

\begin{altabstract}We show here that the fundamental lemma for twisted endoscopy, now proved for the unit elements in the spherical Hecke algebras, implies the fundamental lemma for all elements of these Hecke algebras. The proof, whose idea is due to Arthur, uses the transfer, which is known as a consequence of the fundamental lemma for the units.
\end{altabstract}

\subjclass{22E50}

\keywords{endoscopie tordue, lemme fondamental, transfert géométrique, transfert spectral, représentation elliptique, $R$--groupe}

\altkeywords{twisted endoscopy, fundamental lemma, geometric transfer, spectral transfer, elliptic representation, $R$--groupe}
\maketitle
\setcounter{tocdepth}{3}
\eject

\tableofcontents

\section{Introduction}
\noindent {\bf 1.1.} --- Le lemme fondamental est un résultat essentiel à la stabilisation de la formule des traces. Il est maintenant démontré pour les unités des algèbres de Hecke sphériques, pour l'endoscopie ordinaire comme pour l'endoscopie tordue, en toute caractéristique pourvu que la caractéristique résiduelle soit suffisamment grande. Pour l'endoscopie ordinaire, en caractéristique nulle, Hales a prouvé que le lemme fondamental pour les unités des algèbres de Hecke sphériques en presque toute caractéristique résiduelle implique le lemme fondamental pour tous les autres éléments de ces algèbres de Hecke en toute caractéristique résiduelle. Sa méthode est globale, et généralise celles de Clozel et Labesse pour le changement de base. Cette méthode utilise de manière cruciale une propriété bien connue des groupes non ramifiés de type adjoint, à savoir que les séries principales unitaires non ramifiées (\cad induites à partir d'un caractère unitaire non ramifié du Levi minimal) sont irréductibles. Dans le cas tordu, le groupe adjoint $G_{\rm AD}$ est naturellement remplacé par le groupe $G_\sharp=G/Z(G)^\theta$, ce qui rend difficile l'application de cette méthode.  

\vskip2mm
\noindent{\bf 1.2.} --- 
Depuis les travaux de Hales, un autre résultat a été démontré, pour l'endoscopie tordue en caractéristique nulle: le lemme fondamental pour les unités des algèbres de Hecke sphériques implique le transfert. La méthode est globale, et le résultat obtenu est semblable à celui de Hales: 
si le lemme fondamental pour les unités des algèbres de Hecke sphériques est vrai en presque toute caractéristique résiduelle, alors la conjecture de transfert est vraie en toute caractéristique résiduelle. On dispose donc aujourd'hui du lemme fondamental pour les unités en presque toute caractéristique résiduelle, et du transfert en toute caractéristique résiduelle. On démontre ici (pour l'endoscopie tordue en caractéristique nulle), de manière purement locale grâce au transfert --- qui, comme on vient de le dire, a été obtenu par voie globale ---, que le lemme fondamental pour les unités des algèbres de Hecke sphériques implique le lemme fondamental pour tous les autres éléments de ces algèbres de Hecke. 

\vskip2mm
\noindent{\bf 1.3.} --- Le résultat démontré ici (pour l'endoscopie tordue) est donc moins fort que celui de Hales (pour l'endoscopie ordinaire), puisque notre méthode ne permet pas d'ôter la restriction sur la caractéristique résiduelle. On doit très probablement pouvoir Žliminer cette restriction en appliquant la méthode de \cite[5]{HL}. Rappelons que cette méthode, globale, consiste à utiliser la formule des traces (simple) en sens inverse par rapport au sens habituel, \cad à déduire d'informations sur les traces une information sur les intégrales orbitales. Une autre possibilité serait d'établir le lemme fondamental pour les unités des algèbres de Hecke sphériques dans le seul cas où on l'utilise ici: celui où la donnée endoscopique elliptique non ramifiée a pour groupe sous--jacent un tore (voir 1.7 et 1.8). C'est ce que l'on fait ici dans le cas du groupe $G=GL(n)$ tordu par l'automorphisme extérieur $g\mapsto {^{\rm t}g^{-1}}$. Dans ce cas, on démontre que le lemme fondamental pour tous les éléments des algèbres de Hecke sphériques est vrai, sans restriction sur la caractéristique résiduelle --- cf. \ref{le cas de GL(n) tordu}. 

Une fois le lemme fondamental établi pour tous les éléments des algèbres de Hecke sphériques en toute caractéristique résiduelle, on pourra ensuite le ramener à la caractéristique non nulle par la méthode des corps proches. En caractéristique non nulle, signalons aussi le travail récent d'Alexis Bouthier, qui prouve sous certaines hypothèses (groupe dérivé simplement connexe, et caractéristique première à l'ordre du groupe de Weyl) le lemme fondamental pour tous les éléments des algèbres de Hecke sphériques. 

Ces considérations faites, revenons à la caractéristique nulle. 
Le lemme fondamental pour tous les éléments des algèbres de Hecke sphériques en presque toute caractéristique résiduelle est utilisé dans \cite{Stab X} pour achever la stabilisation de la formule des traces tordues. 

\vskip2mm
\noindent{\bf 1.4.} --- Pour l'endoscopie tordue, les travaux fondamentaux de Kottwitz--Shelstad et Labesse ont été repris (et modifiés sur certains points) dans \cite{Stab I}. Ce sont les notions et les notations de cet article que l'on utilise ici. 
Soit $F$ une extension finie d'un corps $p$--adique ${\Bbb Q}_p$. Soit $G$ un groupe algébrique réductif connexe défini et quasi--déployé sur $F$, et déployé sur une extension non ramifiée de $F$. 
L'action galoisienne sur le groupe dual $\hat{G}$ est donc donnée par celle d'un élément de Frobenius $\phi\in W_F$, où $W_F$ est le groupe de Weil de $F$. Soient aussi $\wt{G}$ un $G$--espace tordu défini sur $F$, et $\omega$ un caractère\footnote{Dans tout ce papier, on appelle {\it caractère} un homomorphisme continu dans ${\Bbb C}^\times$.} de $G(F)$. On suppose vérifiées les hypothèses suivantes:
\begin{itemize}
\item l'ensemble $\wt{G}(F)$ des points $F$--rationnels de $\wt{G}$ n'est pas vide;
\item le $F$--automorphisme $\theta$ de $Z(G)$ défini par $\wt{G}$ est d'ordre fini;
\item le $G(F)$--espace tordu $\wt{G}(F)$ possède un sous--espace hyperspécial $(K,\wt{K})$.   
\item le caractère $\omega$ est trivial sur $Z(G;F)^\theta$;
\item la classe de cohomologie $\bs{a}\in {\rm H}^1(W_F, Z(\hat{G}))$ associée à $\omega$ est non ramifiée;  
\end{itemize}

Soit $\bs{G}'=(G',\ES{G}', \tilde{s})$ une donnée endoscopique pour $(\wt{G},\bs{a})$ --- cf. \ref{données endoscopiques}. On suppose que cette donnée est non ramifiée --- cf. \ref{données endoscopiques nr}. Cela entra\^{\i}ne que le groupe $G'$ (qui est défini et quasi--déployé sur $F$) se déploie sur une extension non ramifiée de $F$. Cela entra\^{\i}ne aussi que le $L$--groupe ${^LG'}$ est isomorphe à $\ES{G}'$. On peut donc, dans cette situation non ramifiée, faire l'économie des données auxiliaires. \`A la donnée $\bs{G}'$ est associé un $G'$--espace tordu $\wt{G}'$, défini sur $F$ et à torsion intérieure --- cf. \ref{espaces endoscopiques}. Au sous--espace hyperspécial $(K,\wt{K})$ de $\wt{G}(F)$ sont associés un sous--espace hyperspécial $(K'\!,\wt{K}')$ de $\wt{G}'(F)$, bien défini à conjugaison près par $G'_{\rm AD}(F)$, et un facteur de transfert normalisé
$$
\Delta: \ES{D}(\bs{G}') \rightarrow {\Bbb C}^\times.
$$
Ici, $\ES{D}(\bs{G}')$ est l'ensemble des couples $(\delta,\gamma)\in \wt{G}'(F)\times \wt{G}(F)$ formés d'éléments semisimples dont les classes de conjugaison (sur $\overline{F}$) se correspondent, et tels que $\gamma$ est fortement régulier --- 
cf. \ref{classes de conjugaison}. Pour $\gamma\in \wt{G}(F)$ fortement régulier (semi--simple), 
on définit l'intégrale orbitale ordinaire $I^{\wt{G}}(\gamma,f,\omega)$ d'une fonction 
$f\in C^\infty_{\rm c}(\wt{G}(F))$. Pour $\delta\in \wt{G}'(F)$ fortement $\wt{G}$--régulier, on définit l'intégrale orbitale stable $S^{\wt{G}'\!}(\delta,f')$ d'une fonction $f'\in C_{\rm c}^\infty(\wt{G}'(F))$, et l'intégrale orbitale endoscopique
$$
I^{\wt{G},\omega}(\delta,f)= d_\theta^{1/2}\sum_{\gamma}d_{\gamma}^{-1}\Delta(\delta,\gamma) I^{\wt{G}}(\gamma,f,\omega)
$$
d'une fonction $f\in C^\infty_{\rm c}(\wt{G}(F))$; où $\gamma$ parcourt les éléments de $\wt{G}(F)$ tels que $(\delta,\gamma)\in \ES{D}(\bs{G}')$, modulo conjugaison par $G(F)$. On renvoie à \ref{transfert nr} pour des définitions précises. On dit qu'une fonction $f'\in C^\infty_{\rm c}(\wt{G}'(F))$ est un transfert de $f\in C^\infty_{\rm c}(G(F))$ si pour tout $\delta\in \wt{G}'(F)$ fortement $\wt{G}$--régulier, on a l'égalité
$$
S^{\wt{G}'\!}(\delta,f')= I^{\wt{G},\omega}(\delta,f).
$$

Soit $\bs{1}_{\wt{K}}$ la fonction caractéristique de $\wt{K}$, et soit $\bs{1}_{\wt{K}'}$ la fonction caractéristique de $\wt{K}'$. 
Le lemme fondamental pour les unités des algèbres de Hecke sphériques (\ref{lemme fondamental}, théorème 1) dit que $\bs{1}_{\wt{K}'}$ est un transfert de $\bs{1}_{\wt{K}}$. Soit $\mathcal{H}_K$ l'algèbre de Hecke formée des fonctions sur $G(F)$ qui sont bi--invariantes par $K$ et à support compact, et soit $\mathcal{H}_{K'}$ l'algèbre de Hecke formée des fonctions $f'$ sur $G'(F)$ qui sont bi--invariantes par $K'$ et à support compact. Via les isomorphismes de Satake, on a un homomorphisme naturel $b: \mathcal{H}_K\rightarrow \mathcal{H}_{K'}$. Le lemme fondamental pour tous les éléments des algèbres de Hecke sphériques (\ref{lemme fondamental}, théorème 2) dit que pour toute fonction $f\in \mathcal{H}_K$, la fonction $b(f)*\bs{1}_{\wt{K}'}= \bs{1}_{\wt{K}'}* b(f)$ sur $\wt{G}'(F)$ est un transfert de $f*\bs{1}_{\wt{K}}$ (resp. de $\bs{1}_{\wt{K}}* \omega^{-1}f$). 

On démontre dans cet article que le théorème 1 implique le théorème 2.   

\vskip2mm
{\bf 1.5.} --- Décrivons brièvement la démonstration proposée ici. 
On commence par se ramener au cas où la donnée $\bs{G}'$ est elliptique (par descente parabolique). Ensuite, suivant l'idée d'Arthur, on traduit le théorème 2 en termes spectraux, grâce au transfert géométrique. Pour cela, on introduit l'espace $\bs{D}(\wt{G}(F),\omega)$ des distributions sur $\wt{G}(F)$ qui sont des combinaisons linéaires (finies, à coefficients complexes) de caractères de représentations $G(F)$--irréductibles tempérées de $(\wt{G}(F),\omega)$ --- cf. \ref{caractères tempérés}. Pour $\bs{G}'$, on définit de la même manière l'espace $\bs{D}(\wt{G}'(F))$, en prenant pour $\omega$ le caractère trivial de $G'(F)$, et l'on note $\bs{SD}(\wt{G}'(F))$ le sous--espace de $\bs{D}(\wt{G}'(F))$ formé des distributions qui sont stables. Soit $\bs{D}^K(\wt{G}(F),\omega)$ le sous--espace de $\bs{D}(\wt{G}'(F),\omega)$ engendré par les caractères des représentations $G(F)$--irréductibles tempérées de $(\wt{G}(F),\omega)$ qui sont $K$--sphériques, et soit $\bs{D}^{K'}(\wt{G}'(F))$ le sous--espace  de $\bs{D}(\wt{G}'(F))$ défini de la même manière. \`A un paramètre  tempéré et non ramifié $\varphi': W_F \rightarrow {^LG'}$ est associé un paramètre tempéré et non ramifié $
\varphi : W_F \buildrel \varphi'\over{\longrightarrow} {^LG'}\simeq \ES{G}'\hookrightarrow {^LG}.
$
Via les isomorphismes de Satake, on en déduit un homomorphisme de transfert spectral (sphérique)
$$
\bs{\rm t}=\bs{\rm t}_{\bs{G}'}: \bs{D}^{K'}(\wt{G}'(F))\rightarrow \bs{D}^K(\wt{G}(F),\omega).
$$
D'autre part, d'après les résultats de \cite{Moe} qui généralisent au cas tordu ceux de \cite{A}, le transfert géométrique $f\mapsto f'$ définit dualement un homomorphisme de transfert spectral (usuel)
$$
\bs{\rm T}=\bs{\rm T}_{\bs{G}'}: \bs{SD}(\wt{G}'(F))\rightarrow \bs{D}(\wt{G}(F),\omega).
$$
Pour $\Theta'\in \bs{D}(\wt{G}'(F))$, la distribution $\bs{\rm T}(\Theta')$ sur $\wt{G}(F)$ est donnée par $\bs{\rm T}(\Theta')(f)= \Theta'(f')$ pour toute paire de fonctions $(f,f')\in C^\infty_{\rm c}(\wt{G}(F))\times C^\infty_{\rm c}(\wt{G}'(F))$ telle que $f'$ est un transfert de $f$. 
Le lemme fondamental pour tous les éléments des algèbres de Hecke sphériques équivaut à la commutativité du diagramme suivant 
$$
\xymatrix{
\bs{SD}(\wt{G}'(F)) \ar@{^{(}->}[r] \ar[d]_{\bs{\rm T}} & \bs{D}(\wt{G}'(F)) \ar[r]^{p'^{K'}} & \bs{D}^{K'}(\wt{G}'(F))\ar[d]^{\bs{\rm t}}\\
\bs{D}(\wt{G}(F),\omega) \ar[rr]^{p^K} && \bs{D}^K(\wt{G}(F),\omega)
}\leqno{(1)}
$$
où les homomorphismes $p^K$ et $p^{K'}$ sont les projections naturelles. 

On se ramène ensuite à la partie elliptique de la théorie. Soit $\bs{D}_{\rm ell}(\wt{G}(F),\omega)$ le sous--espace de $\bs{D}(\wt{G}(F),\omega)$ engendré par les représentations elliptiques au sens d'Arthur --- précisément, leur variante tordue \cite{W}. On définit de la même manière le sous--espace $\bs{D}_{\rm ell}(\wt{G}'(F))$ de $\bs{D}(\wt{G}'(F))$, et l'on pose $\bs{SD}_{\rm ell}(\wt{G}'(F))= \bs{SD}(\wt{G}'(F))\cap \bs{D}_{\rm ell}(\wt{G}'(F))$. L'homomorphisme $\bs{\rm T}$ envoie $\bs{SD}_{\rm ell}(\wt{G}'(F))$ dans $\bs{D}_{\rm ell}(\wt{G}(F),\omega)$, et pour prouver la commutativité du diagramme (1), on est ramené à prouver celle du diagramme suivant
$$
\xymatrix{
\bs{SD}_{\rm ell}(\wt{G}'(F)) \ar@{^{(}->}[r] \ar[d]_{\bs{\rm T}} & \bs{D}(\wt{G}'(F)) \ar[r]^{p'^{K'}} & \bs{D}^{K'}(\wt{G}'(F))\ar[d]^{\bs{\rm t}}\\
\bs{D}_{\rm ell}(\wt{G}(F),\omega) \ar[rr]^{p^K} && \bs{D}^K(\wt{G}(F),\omega)
}.\leqno{(2)}
$$

\vskip2mm
\noindent{\bf 1.6.} --- 
Une représentation $G(F)$--irréductible tempérée de $(\wt{G}(F),\omega)$ est dite {\it de support non ramifié} si 
la représentation de $G(F)$ sous--jacente est une sous--représentation d'une série principale induite à partir d'un caractère unitaire et non ramifié du Levi minimal, et elle est dite {\it de support ramifié} sinon. On note $\bs{D}^{\rm nr}(\wt{G}(F),\omega)$ le sous--espace de $\bs{D}(\wt{G}(F),\omega)$ engendré par les caractères $G(F)$--irréductibles tempérés de $(\wt{G}(F),\omega)$ qui sont de support non ramifié, et $\bs{D}^{\rm ram}(\wt{G}(F),\omega)$ celui engendré par ceux qui sont de support ramifié. Pour $*={\rm nr},\, {\rm ram}$, on pose 
$$
\bs{D}_{\rm ell}^{*}(\wt{G}(F),\omega)=\bs{D}_{\rm ell}(\wt{G}(F),\omega)\cap 
\bs{D}^{*}(\wt{G}(F),\omega).
$$
Alors on a la décomposition
$$
\bs{D}_{\rm ell}(\wt{G}(F),\omega)=
\bs{D}_{\rm ell}^{\rm nr}(\wt{G}(F),\omega)\oplus \bs{D}_{\rm ell}^{\rm ram}(\wt{G}(F),\omega).
\leqno{(3)}
$$
On définit de la même manière les sous--espaces $\bs{D}^{\rm nr}(\wt{G}'(F))$ et $\bs{D}^{\rm ram}(\wt{G}'(F))$ de $\bs{D}(\wt{G}'(F))$, et 
pour $*={\rm nr},\, {\rm ram}$, on pose 
$$
\bs{SD}^*_{\rm ell}(\wt{G}'(F))= \bs{D}^*(\wt{G}'(F))\cap \bs{SD}_{\rm ell}(\wt{G}'(F)).
$$
Alors on a la décomposition
$$
\bs{SD}_{\rm ell}(\wt{G}'(F))=\bs{SD}^{\rm nr}_{\rm ell}(\wt{G}'(F))\oplus 
\bs{SD}^{\rm ram}_{\rm ell}(\wt{G}'(F)).\leqno{(4)}
$$
L'astuce, déjà utilisée par Arthur pour les formes tordues du groupe linéaire, consiste à remarquer que les sous--représentations irréductibles d'une série principale tempérée non ramifiée de $G(F)$ sont permutées transitivement par $G_{\rm AD}(F)$. Puisque $\wt{G}'$ est à torsion intérieure, cela a comme conséquence que l'espace $\bs{SD}_{\rm ell}^{\rm nr}(\wt{G}'(F))$ est nul si $G'$ n'est pas un tore. Pour prouver la commutativité du diagramme (2), on est donc ramené à le faire dans le cas où le groupe $G'$ est un tore, et à prouver aussi (dans tous les cas) que $\bs{\rm T}(\bs{SD}_{\rm ell}^{\rm ram}(\wt{G}'(F))$ est contenu dans $\bs{D}_{\rm ell}^{\rm ram}(\wt{G}(F),\omega)$. Pour ce faire, on dispose de deux décompositions de l'espace $\bs{D}_{\rm ell}(\wt{G}(F),\omega)$: celle donnŽe par les caractères elliptiques d'Arthur --- précisément, leur variante tordue \cite{W} --- , et celle donnée par les transferts des éléments de $\bs{SD}_{\rm ell}(\wt{G}'(F))$ pour $\bs{G}'= (G',\ES{G}',\tilde{s})$ parcourant {\it toutes} les données endoscopiques elliptiques pour $(\wt{G},\bs{a})$. On prouve que sur l'espace $\bs{D}_{\rm ell}^{\rm nr}(\wt{G}(F),\omega)$, ces deux dŽcompositions s'expriment l'une en fonction de l'autre. En d'autres termes, on prouve la paramŽtrisation endoscopique des caractres elliptiques non ramifiés de $(\wt{G}(F),\omega)$.

\vskip2mm
\noindent{\bf 1.7.} --- L'espace $\bs{D}_{\rm ell}(\wt{G}(F),\omega)$ se décompose en somme directe de droites $D_\tau$, où $\tau$ parcourt les classes de conjugaison de triplets elliptiques essentiels $(M,\sigma,\tilde{r})$ pour $(\wt{G}(F),\omega)$ --- cf. \ref{triplets elliptiques essentiels}. Ici $M$ est un $F$--Levi de $G$, $\sigma$ est une représentation irréductible et de la série discrète de $M(F)$, et $\tilde{r}$ est un élément régulier du $R$--groupe tordu $R^{\smash{\wt{G},\omega}}(\sigma)$. Un tel triplet $\tau$ se relève en un triplet $\bs{\tau}=(M,\sigma,\tilde{\bs{r}})$, où $\tilde{\bs{r}}$ est un relèvement de $\tilde{r}$ dans le $R$--groupe tordu $\ES{R}^{\wt{G},\omega}(\lambda)$, lequel définit une distribution $\Theta_{\bs{\tau}}\in \bs{D}_{\rm ell}(\wt{G}(F),\omega)$ qui engendre l'espace $D_\tau$. L'élément $\Theta_{\bs{\tau}}$ appartient à $\bs{D}_{\rm ell}^{\rm nr}(\wt{G}(F),\omega)$ si le triplet $\bs{\tau}$ est {\it non ramifié}, \cad de la forme $(T,\lambda,\tilde{\bs{r}})$ pour un caractère unitaire et non ramifié $\lambda$ de $T(F)$ où $T$ est un $F$--Levi minimal de $G$, et il appartient à $\bs{D}_{\rm ell}^{\rm ram}(\wt{G}(F),\omega)$ sinon. On note $\ES{E}_{\rm ell}^{\rm nr}$ l'ensemble de ces triplets $\bs{\tau}$ qui sont non ramifiés. 

D'autre part, toute distribution $\Theta\in \bs{D}_{\rm ell}(\wt{G}(F),\omega)$ se décompose (de manière unique) en
$$
\Theta= \sum_{\bs{G}'\in \mathfrak{E}} \bs{\rm T}_{\bs{G}'}( \Theta^{\bs{G}'}), \quad \Theta^{\bs{G}'}\in \bs{SD}_{\rm ell}(\wt{G}'(F)),\leqno{(5)}
$$
où $\mathfrak{E}$ est un ensemble de représentants des classes d'isomorphisme de données endoscopiques elliptiques pour $(\wt{G},\bs{a})$ --- cf. \ref{décomposition endoscopique}. Pour les données $\bs{G}'$ qui sont ramifiées, il faut bien sûr fixer des données auxiliaires, mais nous ignorons cela dans cette introduction. Notons que l'espace $\bs{D}_{\rm ell}(\wt{G}(F),\omega)$, et 
chacun des espaces $\bs{SD}_{\rm ell}(\wt{G}'(F))$ pour $\bs{G}'\in \mathfrak{E}$, sont munis de produits scalaires hermitiens définis positifs, et la décomposition (5) préserve ces produits scalaires --- cf. \ref{produits scalaires elliptiques}. Soit $\mathfrak{E}_{\rm nr}\subset \mathfrak{E}$ le sous--ensemble formé des données non ramifiées, et soit $\mathfrak{E}_{\rm t-nr}\subset \mathfrak{E}_{\rm nr}$ le sous--ensemble formé des données non ramifiées telles que le groupe sous--jacent est un tore. 

Remarquons tout d'abord que si $\bs{D}_{\rm ell}^{\rm nr}(\wt{G}(F),\omega)=\{0\}$, le diagramme (2) est trivialement commutatif. 
On suppose donc que l'espace $\bs{D}_{\rm ell}^{\rm nr}(\wt{G}(F),\omega)$ n'est pas nul, ce qui implique que l'ensemble $\mathfrak{E}_{\rm t-nr}$ n'est pas vide (en fait, ces deux conditions sont Žquivalentes). Pour séparer les données $\bs{T}'\in \mathfrak{E}_{\rm t-nr}$ et les éléments de l'espace $\bs{SD}_{\rm ell}^{\rm nr}(\wt{T}'(F))= \bs{D}^{\rm nr}(\wt{T}'(F))$, on dispose essentiellement de trois ingrédients: le lemme fondamental pour les unités des algèbres de Hecke sphériques, l'homomorphisme naturel $\xi_{Z}:Z(G)\rightarrow T'$ (qui permet de comparer les caractres centraux), et le caractère $\omega_{\bs{T}'}$ de $G_\sharp(F)$ associé à $\bs{T}'$. On sait en effet que, pour $\bs{G}'\in \mathfrak{E}$, le transfert $\bs{\rm T}_{\bs{G}'}(\Theta')$ d'une distribution $\Theta'\in \bs{SD}(\wt{G}'(F))$ est un vecteur propre pour l'action (par conjugaison) de $G_\sharp(F)$ relativement à un caractère $\omega_{\bs{G}'}$ de ce groupe. Supposons pour commencer que le groupe adjoint $\hat{G}_{\rm AD}$ est simple. Alors on démontre que:
\begin{itemize}
\item le groupe $\hat{G}_{\rm AD}$ est isomorphe à $PGL(n,{\Bbb C})$ avec actions triviales de $\hat{\theta}$ et $\phi$;
\item pour $\bs{T}'_1,\, \bs{T}'_2\in \mathfrak{E}_{\rm t-nr}$, si $\omega_{\bs{T}'_1}= \omega_{\bs{T}'_2}$ alors $\bs{T}'_1=\bs{T}'_2$; 
\item pour $\bs{T}'=(T',\ES{T}',\tilde{s})\in \mathfrak{E}_{\rm t-nr}$, si $(\chi'_1,\tilde{\chi}'_1)$, $(\chi'_2,\tilde{\chi}'_2)$ sont deux caractères affines de $\wt{T}'(F)$ tels que les caractères $\chi'_1$, $\chi'_2$ de $T'(F)$ soient unitaires et non ramifiés, et vérifient $\chi'_1\circ \xi_Z = \chi'_2\circ \xi_Z$, alors à homothétie près, $\tilde{\chi}'_1$ et $\tilde{\chi}'_2$ se déduisent l'un de l'autre par un automorphisme de la donnée $\bs{T}'$. 
\end{itemize}
Ici on appelle {\it caractère affine} une représentation de dimension $1$. Un caractère affine 
$(\chi,\tilde{\chi}')$ de $\wt{T}'(F)$ est donc la donnée d'un caractère $\chi'$ de $T'(F)$ et d'un prolongement $\tilde{\chi}'$ de $\chi'$ à l'espace $\wt{T}'(F)$. Si de plus $\tilde{\chi}'$ est à valeurs dans ${\Bbb U}$ (ce qui implique que $\chi$ est unitaire), on dit qu'il est unitaire. \`A toute paire $(\bs{T}',\tilde{\chi}')$ formé d'une donnée $\bs{T}'\in \mathfrak{E}_{\rm t-nr}$ et d'un caractère affine unitaire et non ramifié $\tilde{\chi}'$ de $\wt{T}'(F)$, est associé un triplet $\bs{\tau}_{\tilde{\chi}'}\in \ES{E}_{\rm ell}^{\rm nr}$, bien défini à conjugaison près, et donc une distribution $\Theta_{\bs{\tau}_{\tilde{\chi}'}}\in \bs{D}_{\rm ell}^{\rm nr}(\wt{G}(F),\omega)$ --- cf. \ref{l'application tau}. Cette distribution est elle aussi un vecteur propre pour l'action de $G_\sharp(F)$ relativement à un caractère de ce groupe, et l'on vérifie que ce caractère est précisément $\omega_{\chi'}$. Comme à conjugaison près, tout triplet $\bs{\tau}=(T,\lambda,\tilde{\bs{r}})\in \ES{E}_{\rm ell}^{\rm nr}$ est obtenu de cette manière, on obtient que pour une paire $(\bs{T}',\tilde{\chi}')$ comme ci--dessus, le transfert usuel $\bs{\rm T}_{\bs{T}'}(\tilde{\chi}')$ de $\tilde{\chi}'$ s'écrit
$$
\bs{\rm T}_{\bs{T}'}(\tilde{\chi}')= c\,\Theta_{\bs{\tau}_{\tilde{\chi}'}} + \Theta^{\rm ram}\leqno{(6)}
$$
pour une constante $c\in {\Bbb C}$ et une distribution $\Theta^{\rm ram}\in \bs{D}_{\rm ell}^{\rm ram}(\wt{G}'(F),\omega)$. 
Le lemme fondamental pour les unités des algèbres de Hecke sphériques appliqué à la donnée $\bs{T}'$ entra\^{\i}ne que cette constante $c$ n'est pas nulle. On a même $c=1$ (du fait de la normalisation de  
l'application $(\bs{T}',\tilde{\chi}')\mapsto \bs{\tau}_{\tilde{\chi}'}$). En calculant les produit scalaires $(\bs{\rm T}_{\bs{T}'}(\tilde{\chi}'), \bs{\rm T}_{\bs{T}'}(\tilde{\chi}' ))_{\rm ell}$ et $(\Theta_{\bs{\tau}_{\tilde{\chi}'}},\Theta_{\bs{\tau}_{\tilde{\chi}'}})_{\rm ell}$, on montre qu'ils sont égaux. On a donc $\Theta^{\rm ram}=0$ et
$$
\bs{\rm T}_{\bs{T}'}(\tilde{\chi}')= \Theta_{\bs{\tau}_{\tilde{\chi}'}}.\leqno{(7)}
$$
On en déduit alors facilement l'égalité 
$p^K(\Theta_{\bs{\tau}})= \bs{\rm t}_{\bs{T}'}(\tilde{\chi}')$ et la commutativité du diagramme (2) pour la distribution $\tilde{\chi}'\in \bs{D}^{\rm nr}(\wt{T}'(F))=\bs{SD}_{\rm ell}^{\rm nr}(\wt{T}'(F))$. 

Comme conséquence de l'égalité (7), on obtient aussi que l'espace $\bs{D}_{\rm ell}^{\rm ram}(\wt{G}(F),\omega)$ est somme (directe) des sous--espaces $\bs{\rm T}_{\bs{T}'}(\bs{D}^{\rm ram}(\wt{T}'(F)))$ pour $\bs{T}'\in \mathfrak{E}_{\rm t-nr}$ et des sous--espaces $\bs{\rm T}_{\bs{G}'}(\bs{SD}_{\rm ell}(\wt{G}'(F))$ pour $\bs{G}'\in \mathfrak{E}\smallsetminus \mathfrak{E}_{\rm t-nr}$. En particulier pour $\bs{G}'\in \mathfrak{E}_{\rm nr}$, on a l'inclusion
$$
\bs{\rm T}_{\bs{G}'}(\bs{SD}_{\rm ell}^{\rm ram}(\wt{G}'(F))\subset \bs{D}_{\rm ell}^{\rm ram}(\wt{G}(F),\omega).
$$Cela démontre la commutativité du diagramme (2) en général (sous l'hypothèse que $\hat{G}_{\rm AD}$ est simple). 

\vskip2mm
\noindent{\bf 1.8.} --- Reste à supprimer l'hypothèse que $\hat{G}_{\rm AD}$ est simple. On suppose toujours $\mathfrak{E}_{\rm t-nr}\neq \emptyset$. Alors $\hat{G}_{\rm AD}$ est isomorphe ˆ un produit de groupes de type $A_{n-1}$, avec actions de $\hat{\theta}$ et $\phi$ par permutations des facteurs de ce produit. Comme dans le cas o $\hat{G}_{\rm AD}$ est simple, les caractres $\omega_{\bs{T}'}$ de $G_\sharp(F)$ permettent de sŽparer les donnŽes $\bs{T'}\in \mathfrak{E}_{\rm t-nr}$ dans la dŽcomposition (5). En revanche, pour $\bs{T}'\in \mathfrak{E}_{\rm t-nr}$ fixŽ, la restriction à $Z(G;F)$ ne permet pas de sŽparer les caractres unitaires et non ramifiés de $T'(F)$. On introduit pour cela le sous--groupe $C$ de $Z(G)\times G$ formŽ des $(z,g)$ tels que pour tout 
$\sigma\in \Gamma_F$, on ait
$$
g\sigma(g)^{-1}\in Z(G),\quad (1-\theta)(g\sigma(g)^{-1})= z\sigma(z)^{-1}.
$$
Il contient le sous--groupe $Z(G)^*=\{(1-\theta)(z),z):z\in Z(G)\}$, 
et on pose $\ES{C}=C/Z(G)^*$. Un ŽlŽment $(z,g)\in C$ dŽfinit un $F$--automorphisme 
$
(g'\mapsto {\rm Int}_{g^{-1}}(g'),\gamma\mapsto z {\rm Int}_{g^{-1}}(\gamma))
$
de $(G,\wt{G})$, qui ne dŽpend que de l'image 
de $(z,g)$ dans $\ES{C}$. Cela dŽfinit un homomorphisme injectif de $\ES{C}$ dans l'opposŽ du groupe des $F$--automorphismes de $(G,\wt{G})$. On a des plongements naturels $Z(G;F)\hookrightarrow \ES{C}$ et $G_\sharp(F)\hookrightarrow \ES{C}$, mais $\ES{C}$ est en général plus gros que le sous--groupe engendré par $Z(G;F)$ et $G_\sharp(F)$. Pour $\bs{T}'\in \mathfrak{E}_{\rm t-nr}$, on démontre que:
\begin{itemize}
\item si $(\chi',\tilde{\chi}')$ un caractre affine de $\wt{T}(F)$ tel que le caractère $\chi'$ de $T'(F)$ soit unitaire et non ramifié, alors le transfert $\bs{\rm T}_{\bs{T}'}(\tilde{\chi}')$ est un vecteur propre pour l'action du groupe $\ES{C}$ relativement à un caractre $\omega_{\chi'}$ de ce groupe;
\item si $(\chi'_1,\tilde{\chi}'_1)$, $(\chi'_2,\tilde{\chi}'_2)$ sont deux caractres affines de $\wt{T}'(F)$ tels que les caractres $\chi'_1$, $\chi'_2$ de $T'(F)$ soient unitaires et non ramifiŽs, et vŽrifient $\omega_{\chi'_1}=\omega_{\chi'_2}$, alors ˆ homothŽtie prs, $\tilde{\chi}'_1$ et $\tilde{\chi}'_2$ se dŽduisent l'un de l'autre par un automorphisme de la donnŽe $\bs{T}'$.
\end{itemize}
Pour une paire $(\bs{T}',\tilde{\chi}')$ formé d'une donnée $\bs{T}'\in \mathfrak{E}_{\rm t-nr}$ et d'un caractère affine unitaire et non ramifié $\tilde{\chi}'$, l'élément $\Theta_{\bs{\tau}_{\tilde{\chi}'}}\in \bs{D}_{\rm ell}^{\rm nr}(\wt{G}'(F),\omega)$ est encore un vecteur propre pour l'action du groupe $\ES{C}$ relativement au caractère $\omega_{\chi'}$. Grâce aux propriétés ci--dessus, on peut, de la même manière que dans le cas où $\hat{G}_{\rm AD}$ est simple, obtenir la décomposition (6), puis l'égalité (7), et la commutativité du diagramme (2). 

Cela démontre, en général, que si le lemme fondamental pour les unités des algèbres de Hecke sphériques est vrai (théorème 1 de 2.8) --- plus précisément, s'il est vrai pour les données $\bs{T}'\in \mathfrak{E}_{\rm t-nr}$ ---, alors le lemme fondamental pour tous les éléments de ces  algèbres de Hecke l'est aussi (théorème 2 de 2.8). 

\section{\'Enoncé du théorème}

\subsection{Les objets; hypothèses et conventions}\label{objets}
Soit $F$ un corps commutatif localement compact non archimédien de caractéristique nulle. On fixe une cl\^oture algébrique $\overline{F}$ de $F$, et l'on note $\Gamma_F$ le groupe de Galois de l'extension $\overline{F}/F$, $I_F\subset \Gamma_F$ son groupe d'inertie, et $W_F\subset \Gamma_F$ son groupe de Weil. Soit $F^{\rm nr}/F$ la sous--extension non ramifiée maximale de $\overline{F}/F$ (ainsi $I_F$ est le groupe de Galois de $\overline{F}/F^{\rm nr}$). Soit $v_F$ la valuation sur $\overline{F}$ normalisée par $v_F(F^\times)={\Bbb Z}$. On fixe un élément de Frobenius $\phi\in W_F$, \cad tel que la restriction de $\phi$ à $F^{\rm nr}$ soit l'automorphisme de Frobenius de l'extension $F^{\rm nr}/F$. 

Pour une variété algébrique $V$ définie sur $F$, on identifie $V$ à l'ensemble $V(\overline{F})$ de ses points $\overline{F}$--rationnels, et l'on munit l'ensemble $V(F)$ de ses points $F$--rationnels de la topologie définie par $F$.

Soit $G$ un groupe algébrique réductif connexe défini sur $F$, et soit $\tilde{G}$ un $G$--espace (algébrique) tordu. On rappelle que $\widetilde{G}$ est une variété algébrique affine définie sur $F$, munie de deux actions de $G$ définies sur $F$, l'une à gauche et l'autre à droite,
$$
G\times \tilde{G},\, (g,\delta)\mapsto g\delta,\quad 
\widetilde{G}\times G\rightarrow \widetilde{G},\, (\delta,g)\mapsto \delta g,
$$
commutant entre elles, et telles que pour $\delta\in \widetilde{G}$, les applications
$$
G\rightarrow \widetilde{G},\, g\mapsto g\delta,\quad G\rightarrow \widetilde{G},\, g\mapsto \delta g,
$$
soient des isomorphismes de variétés algébriques. On note $\delta \mapsto {\rm Int}_\delta$ l'application de $\widetilde{G}$ dans le groupe des automorphismes de $G$ définie par
$$
{\rm Int}_\delta(g)\delta = \delta g,\quad g\in G.
$$
Pour un élément $\delta\in \wt{G}$, on note $G^\delta$ le commutant de $\delta$ dans $G$ (\cad l'ensemble des points fixes de ${\rm Int}_\delta$) et $G_\delta=G^{\delta,\circ}$ sa composante neutre. 

Pour $\delta\in \wt{G}$, l'automorphisme ${\rm Int}_\delta$ de $G$ définit par fonctorialité des automorphismes de divers objets attachés à $G$; par exemple du revêtement simplement connexe $G_{\rm SC}$ du groupe dérivé de $G$, ou de son groupe adjoint $G_{\rm AD}$. Quand cet automorphisme ne dépend pas de $\delta$, on le note $\theta$ comme dans \cite{Stab I}. Ainsi on a un automorphisme $\theta$ du centre $Z(G)$ de $G$. On note $Z(G;F)^\theta$ le groupe des points $F$--rationnels de $Z(G)$ qui sont fixés par $\theta$.
 
\begin{meshypos}
{\rm 
On suppose que:
\begin{itemize}
\item Le groupe $G$ est quasi--déployé sur $F$ et déployé sur $F^{\rm nr}$.
\item L'ensemble $\widetilde{G}(F)$ des points $F$--rationnels de $\widetilde{G}$ n'est pas vide.
\item L'automorphisme $\theta$ de $Z(G)$ est d'ordre fini.
\end{itemize}
}
\end{meshypos}

\subsection{Le $L$--groupe}\label{L-groupe} 
Puisque le groupe $G$ est quasi--déployé sur $F$, on peut fixer une paire de Borel épinglée 
$\ES{E}=(B,T,\{E_\alpha\}_{\alpha\in \Delta})$ de $G$ définie sur $F$:
\begin{itemize}
\item $(B,T)$ est une paire de Borel de $G$ définie sur $F$;
\item $\Delta$ est la base relativement à $B$ du système de racines $\Sigma=\Sigma(G,T)$ de $G$; 
\item $E_\alpha$ est un élément non nul de l'algèbre de Lie $\mathfrak{u}_\alpha$ du sous--groupe radiciel $U_\alpha$ de $G$ correspondant à la racine $\alpha$, les $E_\alpha$ étant choisis de telle manière que l'ensemble $\{E_\alpha\}_{\alpha\in \Delta}$ soit $\Gamma_F$--stable.
\end{itemize}
\vskip1mm
Soit $Z(\wt{G},\ES{E})$ l'ensemble des éléments $\delta\in \wt{G}$ tels que l'automorphisme ${\rm Int}_\delta$ de $G$ stabilise $\ES{E}$. Cet ensemble n'est pas vide, et c'est un espace tordu sous $Z(G)$, défini sur $F$. On note 
$\theta_{\ES{E}}$, ou encore $\theta$, le $F$--automorphisme de $G$ défini par 
$\theta= {\rm Int}_\epsilon$ pour un (i.e. pour tout) $\epsilon\in Z(\wt{G},\ES{E})$. Cet automorphisme est quasi--semisimple (puisqu'il stabilise $(B,T)$), et même semisimple puisqu'il est d'ordre fini (d'après l'hypothèse de finitude sur l'automorphisme $\theta$ de $Z(G)$) et que $F$ est de caractéristique nulle. 
Notons que l'ensemble $Z(\wt{G},\ES{E};F)$ des points $F$--rationnels de $Z(\wt{G},\ES{E})$ peut être vide.

On note $X(T)$ le groupe des caractères algébriques de $T$, $\check{X}(T)$ le groupe des cocaractères algébriques de $T$, $\check{\Sigma}=\check{\Sigma}(G,T)$ l'ensemble des coracines de $G$, $\check{\Delta}$ la base de $\check{\Sigma}$ relativement à $B$, et $W=W^G(T)$ le groupe de Weyl $N_G(T)/T$; où $N_G(T)$ est le normalisateur de $T$ dans $G$. On a donc les inclusions
$$
\Delta\subset \Sigma \subset X(T),\quad \check{\Delta}\subset \check{\Sigma}\subset \check{X}(T).
$$
L'action de $\Sigma_F$ sur $G$ induit une action sur $X(T)$, $\check{X}(T)$, $\Sigma, \check{\Sigma}$, $\Delta$, $\check{\Delta}$ et $W$, et $\theta=\theta_{\ES{E}}$ induit un automorphisme de chacun de ces objets. Rappelons que le sous--groupe $W^\theta$ de $W$ formé des points fixes sous $\theta$, vérifie les propriétés: 
\begin{itemize}
\item un élément $w\in W$ appartient à $W^\theta$ si et seulement s'il stabilise 
$T^\theta$ ou $T^{\theta,\circ}$;
\item pour $\epsilon\in Z(\wt{G},\ES{E})$, $W^\theta$ s'identifie au groupe de Weyl $W^{G_\epsilon}(T_\epsilon)$ de $G_\epsilon$, où $T_\epsilon$ est le tore maximal $T^{\theta,\circ}$ de $G_\epsilon$. 
\end{itemize}

\vskip1mm
Soit $\hat{G}$ le groupe dual de $G$. Il est muni d'une action algébrique de 
$\Gamma_F$, notée $\sigma\mapsto \sigma_G$, et d'une paire de Borel épinglée 
$\hat{\ES{E}}=(\hat{B},\hat{T},\{\hat{E}_\alpha\}_{\alpha\in \hat{\Delta}})$ qui est définie sur $F$, \cad $\Gamma_F$--stable. Puisque $G$ est déployé sur $F^{\rm nr}$, l'action de $\Gamma_F$ sur $\hat{G}$ se factorise à travers $\Gamma_F/I_F$. Elle est donc entièrement déterminée par la donnée de l'automorphisme $\phi_G$ de $\hat{G}$. 

On pose ${^L{G}}=\hat{G}\rtimes W_F$. 
On a des isomorphismes en dualité
$$
\jmath: \check{X}(T)\rightarrow X(\hat{T}),\quad \hat{\jmath}: \check{X}(\hat{T})\rightarrow X(T),
$$
qui sont $\Gamma_F$--équivariants. Il existe un unique automorphisme $\hat{\theta}$ de $\hat{G}$ qui stabilise $\hat{\ES{E}}$ et vérifie, pour tout $x\in \check{X}(T)$ et tout 
$y\in \check{X}(\hat{T})$,
$$
\jmath(\theta\circ x)= \jmath(x)\circ \hat{\theta},\quad \hat{\jmath}(\hat{\theta}\circ y)=  \hat{\jmath}(y)\circ \theta.
$$
Le groupe de Weyl $W^{\smash{\hat{G}}}(\hat{T})=N_{\smash{\hat{G}}}(\hat{T})$ de $\hat{G}$ s'identifie par dualité au groupe de Weyl $W$ de $G$, et on a l'égalité $W^{\hat{\theta}}=W^{\theta}$. 
Précisons cette identification. Un automorphisme $u$ du tore $T$ définit fonctoriellement un automorphisme $\hat{u}$ de $\hat{T}$ (pour $u=\theta\vert_T$, on a $\hat{u}= \hat{\theta}\vert_{\hat{T}}$). L'application $u\mapsto \hat{u}$ est contravariante: on a $(u_1\circ u_2)\,\hat{}= \hat{u}_2 \circ \hat{u}_1$. En particulier elle induit un isomorphisme de $W$ sur le groupe opposé de  $W^{\hat{G}}(\hat{T})$. On identifie $W$ à $W^{\hat{G}}(\hat{T})$ par l'application $w\mapsto \hat{w}^{-1}$. 
L'automorphisme $\hat{\theta}$ commute à l'action de $\Gamma_F$ sur $\hat{G}$, donc définit un automorphisme 
${^L\theta}$ de ${^LG}$: pour $(g,w)\in \hat{G}\rtimes W_F$, on a
$$
{^L\theta}(g\rtimes w)= \hat{\theta}(g)\rtimes w.
$$
On pose ${^L\wt{G}}={^LG}{^L\theta}$. C'est un espace tordu sous ${^LG}$.

\begin{marema}
{\rm Soit $\hat{\ES{E}}_1$  une autre paire de Borel épinglée de $\hat{G}$, pas forcément définie sur $F$. Choisissons un élément $y\in \hat{G}_{\rm SC}$ tel que ${\rm Int}_{y^{-1}}(\hat{\ES{E}}_1)= \hat{\ES{E}}$, et pour $\sigma\in \Gamma_F$, posons
$$
\sigma_{1,G}= {\rm Int}_y \circ \sigma_G \circ {\rm Int}_{y^{-1}}.
$$
La paire de Borel épinglée $\hat{\ES{E}}_1$ est $\Gamma_F$--stable pour l'action $\sigma \mapsto \sigma_{1,G}$ de $\Sigma_F$ sur $\hat{G}$, et ${^LG}$ est encore le produit semidirect 
$\hat{G}\rtimes W_F$ pour cette nouvelle action: l'application
$$
g\rtimes w\mapsto 
g w_G(y)y^{-1}\rtimes y$$
envoie un produit semidirect sur l'autre. Posons $\hat{\theta}_1= y\hat{\theta}(y)^{-1}\hat{\theta}\in \hat{G}\hat{\theta}$. On note encore $\hat{\theta}_1$ l'automorphisme 
${\rm Int}_{y\hat{\theta}(y)^{-1}}\circ \hat{\theta}$ de $\hat{G}$. Alors $\hat{\theta}_1$ stabilise $\ES{E}_1$, commute à l'action galoisienne $\sigma\mapsto \sigma_{1,G}$, et on a l'égalité $\hat{G}\hat{\theta}_1= \hat{G}\hat{\theta}$. Notons que $y$ n'est défini qu'à multiplication près par un élément de 
$Z(\hat{G}_{\rm SC})$. 
}
\end{marema} 

\subsection{Données endoscopiques}\label{données endoscopiques} On fixe 
aussi un caractère $\omega$ de $G(F)$, \cad un homomorphisme continu de $G(F)$ dans ${\Bbb C}^\times$. D'après un théorème de Langlands, ce caractère correspond à une classe de cohomologie $\boldsymbol{a}\in {\rm H}^1(W_F,Z(\hat{G}))$ 
--- cf. \cite[1.13]{Stab I}. Notons que pour que la théorie ne soit pas vide, il faut que le caractère $\omega$ soit trivial sur $Z(G;F)^\theta$, ce que l'on supposera à partir de \ref{données endoscopiques nr}. 

Une {\it donnée endoscopique pour $(\wt{G},\boldsymbol{a})$} est un triplet $\boldsymbol{G}'=(G',\ES{G}',\tilde{s})$ vérifiant 
les conditions:
\begin{itemize}
\item $G'$ est un groupe réductif connexe défini et quasi--déployé sur $F$.
\item $\tilde{s}=s\hat{\theta}$ est un élément semisimple de $\hat{G}\hat{\theta}$.
\item $\ES{G}'$ est un sous--groupe fermé de ${^LG}$ tel que $\ES{G}'\cap \hat{G}= \hat{G}_{\tilde{s}}$ (la composante neutre du commutant $\hat{G}^{\tilde{s}}$ de $\tilde{s}$ dans $\hat{G}$).  
\item La suite
$$
1\rightarrow \hat{G}_{\tilde{s}} \rightarrow \ES{G}'\rightarrow W_F \rightarrow 1
$$
est exacte et scindée (\cad qu'il existe une section continue $W_F\rightarrow \ES{G}'$), où la flèche $\ES{G}'\rightarrow W_F$ est la restriction à $\ES{G}'$ de la projection naturelle ${^LG}\rightarrow W_F$.
\item Fixée une paire de Borel épinglée $\hat{\ES{E}}'$ de $\hat{G}_{\tilde{s}}$, pour chaque $w\in W_F$, on peut choisir un élément $\tilde{g}'_w= (g'_w, w)\in \ES{G}'$ tel que l'automorphisme ${\rm Int}_{\tilde{g}'_w}$ de $\hat{G}_{\tilde{s}}$ stabilise $\hat{\ES{E}}'$. L'application 
$w\mapsto {\rm Int}_{\tilde{g}'_w}$ s'étend en une action de $\Gamma_F$ sur $\hat{G}_{\tilde{s}}$, 
et l'on suppose que $\hat{G}_{\tilde{s}}$ muni de cette action est un groupe dual de $G'$. On peut 
donc poser $\hat{G}'= \hat{G}_{\tilde{s}}$ et noter $\sigma \mapsto \sigma_{G'}$ cette action.
\item Il existe un cocycle $a:W_F\rightarrow Z(\hat{G})$ dans la classe de cohomologie $\boldsymbol{a}$ tel que pour tout $(g',w)\in \ES{G'}$, on ait l'égalité
$$
{\rm Int}_{\tilde{s}}(g',w)=(g' a(w),w).
$$ 
\end{itemize}

\vskip1mm Un isomorphisme entre deux données endoscopiques $(G'_1,\ES{G}'_1,\tilde{s}_1)$ et 
$(G'_2,\ES{G}'_2,\tilde{s}_2)$ pour $(\wt{G},\boldsymbol{a})$ est par définition un élément $x\in \hat{G}$ tel que
$$
x\ES{H}'_1x^{-1}= \ES{H}'_2,\quad x\tilde{s}_1x^{-1}\in Z(\hat{G})\tilde{s}_2.
$$ De l'isomor\-phisme 
${\rm Int}_{x^{-1}}: \hat{G}'_2\rightarrow\hat{G}'_1$ se déduit par dualité un isomorphisme
$$
\alpha_x: G'_1\rightarrow G'_2,
$$
défini sur $F$; en fait une classe de tels isomorphismes modulo l'action de $G'_{1,{\rm AD}}(F)$ ou de $G'_{2,{\rm AD}}(F)$. En particulier, pour une seule donnée endoscopique $\boldsymbol{G}'$ pour 
$(\wt{G},\boldsymbol{a})$, le groupe ${\rm Aut}(\boldsymbol{G}')$ des automorphismes de $\boldsymbol{G}'$ 
contient $\hat{G}'$. On note ${\rm Out}(\boldsymbol{G}')$ le sous--groupe de ${\rm Out}_F(G')= 
{\rm Aut}_F(G')/G'_{\rm AD}(F)$ formé des images des $\alpha_x$ dans ${\rm Out}_F(G')$ pour $x\in {\rm Aut}(\boldsymbol{G}')$. D'après \cite[2.1]{KS1}, on a une suite exacte courte
$$
1 \rightarrow [Z(\hat{G})/(Z(\hat{G})\cap \hat{G}')]^{\Gamma_F}\rightarrow {\rm Aut}(\boldsymbol{G}')/\hat{G}'\rightarrow {\rm Out}(\boldsymbol{G}')\rightarrow 1.
$$ 

\vskip1mm
Soit $\boldsymbol{G}'=(G',\ES{G}',\tilde{s})$ une donnée endoscopique pour $(\wt{G},\boldsymbol{a})$. Fixons 
une paire de Borel épinglée $\hat{\ES{E}}=(\hat{B},\hat{T}, \{\hat{E}_\alpha\}_{\alpha \in \hat{\Delta}})$  de $\hat{G}$ telle que l'automorphisme ${\rm Int}_{\tilde{s}}$ de $\hat{G}$ stabilise $(\hat{B},\hat{T})$. Posons $
\hat{B}'=\hat{B}\cap \hat{G}'$, $\hat{T}'=\hat{T}\cap \hat{G}'$, 
et complétons $(\hat{B}',\hat{T}')$ en une paire de Borel épinglée $\hat{\ES{E}}'=(\hat{B}',\hat{T}',\{\hat{E}'_{\alpha'}\}_{\alpha' \in \hat{\Delta}'})$ de $\hat{G}'=\hat{G}_{\tilde{s}}$. Deux telles paires $\hat{\ES{E}}$ et $\hat{\ES{E}}'$ --- i.e. telles que ${\rm Int}_{\tilde{s}}$ stabilise $(\hat{B},\hat{T})$ et 
$(\hat{B}',\hat{T}')=(\hat{B}\cap \hat{G}',\hat{T}\cap\hat{G}')$ --- sont dites {\it compatibles}. Quitte à modifier l'action $\sigma\mapsto \sigma_G$ de $\Gamma_F$ sur $\hat{G}$ et l'isomorphisme ${^LG}\simeq \hat{G}\rtimes W_F$ comme dans la remarque de \ref{L-groupe}, on peut supposer que:
\begin{itemize}
\item la paire de Borel épinglée $\hat{\ES{E}}$ est définie sur $F$;
\item l'automorphisme $\hat{\theta}$ de $\hat{G}$ stabilise $\hat{\ES{E}}$ et commute à l'action galoisienne $\sigma\mapsto \sigma_G$.
\end{itemize}
Puisque l'élément $\tilde{s}$ stabilise $(\hat{B},\hat{T})$, il s'écrit $\tilde{s}=s\hat{\theta}$ pour un élément $s\in \hat{T}$. On construit comme ci--dessus l'action $\sigma\mapsto \sigma_{G'}$ de $\Gamma_F$ sur $\hat{G}'$ qui stabilise $\hat{\ES{E}}'$. L'égalité $\hat{T}'= \hat{T}^{\hat{\theta},\circ}$ identifie le groupe de Weyl $W'$ de $\hat{G}'$ (ou de $G'$) à un sous--groupe du groupe des invariants $W^{\hat{\theta}}=W^\theta$ du groupe de Weyl $W$ de $\hat{G}$ (ou de $G$). Le plongement $\hat{\xi}: \hat{T}'\hookrightarrow \hat{T}$ n'est en général pas équivariant pour les actions galoisiennes: il existe un cocycle $\sigma\mapsto \alpha_{G'}(\sigma)$ de $\Gamma_F$ dans $W^\theta$ tel que
$$
\alpha_{G'}(\sigma) \circ \sigma(\hat{\xi})= \hat{\xi},\quad \sigma(\hat{\xi})= \sigma_G\circ \hat{\xi}\circ \sigma_{G'}^{-1};
$$
où l'on a identifié $\alpha_{G'}(\sigma)$ à l'automorphisme ${\rm Int}_{\alpha_{G'}(\sigma)}$ de 
$\hat{T}$.

\subsection{Espaces endoscopiques tordus}\label{espaces endoscopiques} Continuons avec les notations de \ref{données endoscopiques}. Soit $(B',T')$ une paire de Borel de $G'$ définie sur $F$. Le tore $\hat{T}^{\hat{\theta},\circ}$ est dual de $T/(1-\theta_{\ES{E}})(T)$. Du plongement $\hat{\xi}: \hat{T}'\rightarrow \hat{T}$ se déduit par dualité un homomorphisme
$$
\xi: T \rightarrow T/(1-\theta_{\ES{E}})(T)\simeq T'
$$
vérifiant
$$
\sigma(\xi)= \xi \circ {\rm Int}_{\alpha_{G'}}(\sigma),\quad \sigma \in \Gamma_F;
$$
où l'on a identifié $\alpha_{G'}(\sigma)$ à l'automorphisme ${\rm Int}_{\alpha_{G'}(\sigma)}$ de 
$T$. D'après \cite[1.7]{Stab I}, on a l'inclusion
$$
\xi(Z(G))\subset Z(G').
$$
Notons $\ES{Z}(\wt{G},\ES{E})$ le quotient de $Z(\wt{G},\ES{E})$ par l'action de $Z(G)$ par conjugaison. C'est un espace principal homogène (à gauche et à droite) sous
$$
\ES{Z}(G)=Z(G)/(1-\theta)(Z(G)).
$$
La restriction de $\xi$ à $Z(G)$, notée
$$
\xi_Z:Z(G)\rightarrow Z(G'),
$$ se quotiente en un homomorphisme
$$
\xi_{\ES{Z}}: \ES{Z}(G)\rightarrow Z(G')
$$
qui est $\Gamma_F$--équivariant. On pose
$$
\wt{G}'= G'\times_{\ES{Z}(G)}\ES{Z}(\wt{G},\ES{E}),
$$
\cad que $\wt{G}'$ est le quotient de $G'\times \ES{Z}(\wt{G},\ES{E})$ par la relation 
d'équivalence
$$
(g'\xi_{\ES{Z}}(z),\tilde{z})= (g', z\tilde{z}),\quad z\in \ES{Z}(G).
$$
Les actions à gauche et à droite de $G'$ sur $G'\times \ES{Z}(\wt{G},\ES{E})$ se descendent en des actions sur $\wt{G}'$, et l'action de $\Gamma_F$ sur $G'\times \ES{Z}(\wt{G},\ES{E})$ se descend en une action sur $\wt{G}'$. Cela fait de $\wt{G}'$ un espace tordu sous $G'$, défini sur $F$ et à torsion intérieure. Notons que l'ensemble $\wt{G}'(F)$ des points $F$--rationnels de $\wt{G}'$ peut être vide 
\cite[1.7, rem. (2)]{Stab I}.

\begin{marema}
{\rm 
Pour une paire de Borel épinglée $\ES{E}_1$ de $G$, pas forcément définie sur $F$, on définit comme ci--dessus le $\ES{Z}(G)$--espace tordu $\ES{Z}(\wt{G},\ES{E}_1)$. Comme 
dans \cite[1.2]{Stab I}, on peut introduire la limite inductive $\ES{Z}(\wt{G})$ des $\ES{Z}(\wt{G},\ES{E}_1)$ sur toutes les paires de Borel épinglées $\ES{E}_1$ de $G$. C'est un espace tordu sous $\ES{Z}(G)$, et il est défini sur $F$ pour l'action de $\Gamma_F$ définie dans loc.~cit. On peut définir l'espace tordu $\wt{G}'$ indépendamment  du choix de $\ES{E}$ (cf. loc.~cit.), mais puisque $G$ est quasi--déployé sur $F$, ce n'est pas vraiment nécessaire ici: on peut identifier $\ES{Z}(\wt{G})$ à $\ES{Z}(\wt{G},\ES{E})$ et l'action de $\Gamma_F$ sur $\ES{Z}(\wt{G})$ à l'action naturelle sur $\ES{Z}(\wt{G},\ES{E})$. 
}
\end{marema}

\subsection{Classes de conjugaison semisimples; correspondance}\label{classes de conjugaison} D'après l'hypo\-thèse de finitude sur l'automorphisme $\theta$ de $Z(G)$, les éléments quasi--semisimples de $\wt{G}$ sont semisimples. Rappelons qu'un élément semisimple $\gamma\in \wt{G}$ est dit 
{\it fortement régulier} si le centralisateur $G^\gamma$ est abélien et si le centralisateur connexe $G_\gamma$ est un tore.

Soit $\gamma_1\in \wt{G}$ un élement semisimple. L'automorphisme ${\rm Int}_{\gamma_1}$ de $G$ stabilise une paire de Borel $(B_1,T_1)$ de $G$. Choisissons un $g\in G_{\rm SC}$ tel que ${\rm Int}_g(B,T)= (B_1,T_1)$. L'élément $\gamma= {\rm Int}_{g^{-1}}(\gamma_1)$ stabilise $(B,T)$. D'après  \cite[1.3]{Stab I}, il s'écrit $\gamma=t \epsilon$ pour un $t\in T$ et un $\epsilon\in Z(\wt{G},\ES{E})$. Soient:
\begin{itemize}
\item $\bar{t}$ l'image de $t$ dans 
$[T/(1-\theta_{\ES{E}})(T)]/W^{\theta_{\ES{E}}}$;
\item $\bar{\epsilon}$ l'image de $\epsilon$ dans $\ES{Z}(\wt{G},\ES{E})$;
\item $\bar{\gamma}$ l'image de $(\bar{t},\bar{\epsilon})$ dans 
$[T/(1-\theta_{\ES{E}})(T)]/W^{\theta_{\ES{E}}}\times_{\ES{Z}(G)}\ES{Z}(\wt{G},\ES{E})$.
\end{itemize}
D'après \cite[1.2, 1.8]{Stab I}, l'élément $\bar{\gamma}$ est bien défini (en particulier il ne dépend 
pas du choix de $g$), et l'application $\gamma_1\mapsto \bar{\gamma}$ se quotiente en une bijection de l'ensemble des classes de $G$--conjugaison d'éléments semisimples de $\wt{G}$ sur 
$$[T/(1-\theta_{\ES{E}})(T)]/W^{\theta_{\ES{E}}}\times_{\ES{Z}(G)}\ES{Z}(\wt{G},\ES{E}).
$$
De plus cette bijection est définie sur $F$. Notons qu'une classe de conjugaison  semisimple dans $\wt{G}$ peut être définie sur $F$ mais ne contenir aucun élément $F$--rationnel. 

\vskip1mm
Soit $\boldsymbol{G}'=(G',\ES{G}',\tilde{s})$ une donnée endoscopique pour $(\wt{G},\bs{a})$. Reprenons les constructions de \ref{données endoscopiques} et \ref{espaces endoscopiques}. On a une paire de Borel $(B',T')$ de $G'$ définie sur $F$, et un isomorphisme $T'\simeq T/(1-\theta_{\ES{E}})(T)$ via lequel le groupe de Weyl $W'=W^{G'}(T')$ de $G'$ s'identifie à un sous--groupe de $W^{\theta_{\ES{E}}}$. Puisque $\wt{G}'$ est à torsion intérieure, on a $\ES{Z}(G')= Z(G')$, et l'on pose
$$
\ES{Z}(\wt{G}',\ES{E})= Z(G')\times_{\ES{Z}(G)}\ES{Z}(\wt{G},\ES{E}).
$$
On en déduit une application surjective \cite[1.8]{Stab I}
$$
(T'/W') \times_{Z(G')} \ES{Z}(\wt{G}',\ES{E}) \rightarrow [T/(1-\theta_{\ES{E}})(T)]/W^{\theta_{\ES{E}}}\times_{\ES{Z}(G)}\ES{Z}(\wt{G},\ES{E}),
$$
\cad une surjection de l'ensemble des classes de conjugaison semisimples dans $\wt{G}'$ sur l'ensemble des classes de conjugaison semisimples dans $\wt{G}$. Cette application est définie sur $F$. Notons que restreinte aux éléments invariants sous $\Gamma_F$, \cad aux classes de conjugaison définies sur $F$, cette application n'est en générale plus surjective. 

On note $\ES{D}(\bs{G}')$ l'ensemble des couples $(\delta,\gamma)\in \wt{G}'(F)\times \wt{G}(F)$ formés d'éléments semisimples dont les classes de conjugaison (sur $\overline{F}$) se correspondent via la surjection ci--dessus, et tels que $\gamma$ est fortement régulier dans $\wt{G}$ (on dit alors que $\delta$ est fortement $\wt{G}$--régulier). On sait qu'un élément $\delta\in \wt{G}'$ fortement 
$\wt{G}$--régulier est a fortiori fortement régulier dans $\wt{G}'$ \cite[lemma 3.3.C]{KS1}. 

La donnée $\bs{G}'$ est dite {\it relevante} si l'ensemble $\ES{D}(\bs{G}')$ n'est pas vide.    

\subsection{Données endoscopiques non ramifiées}\label{données endoscopiques nr} On suppose de plus que la classe de cohomologie $\boldsymbol{a}\in {\rm H}^1(W_F,Z(\hat{G}))$ correspondant à $\omega$ est non ramifiée, \cad qu'elle provient par inflation d'un élément de ${\rm H}^1(W_F/I_F,Z(\hat{G}))$. On dira aussi que le caractère $\omega$ est non ramifié. Soit $\boldsymbol{G}'=(G',\ES{G}',\tilde{s})$ une donnée endoscopique pour $(\wt{G},\bs{a})$. Supposons que la donnée $\boldsymbol{G}'$ est {\it non ramifiée}, \cad que l'inclusion suivante est vérifiée:
$$
I_F\subset \ES{G}'.
$$ 
Cela entraîne que le $F$--groupe (quasi--déployé) $G'$ est déployé sur $F^{\rm nr}$ \cite[6.2]{Stab I}. 
Normalisons les actions galoisiennes 
$\sigma\mapsto \sigma_G$ sur $\hat{G}$ et $\sigma\mapsto \sigma_{G'}$ sur $\hat{G}'$ comme en 
\ref{données endoscopiques}, \cad en choisissant des paires de Borel épinglées compatibles $\hat{\ES{E}}$ de $\hat{G}$ et $\hat{\ES{E}}'$ de $\hat{G}'$, et en imposant que les actions galoisiennes préservent ces paires. Ces actions sont déterminées par la donnée des automorphismes $\phi_G$ de $\hat{G}$ et $\phi_{G'}$ de $\hat{G}'$. Soit un élément $h_\phi=(h,\phi)\in \ES{G}'$ tel que 
${\rm Int}_{h_\phi}$ agisse comme $\phi_{G'}$ sur $\hat{G}'$. Alors $\ES{G'}$ est le produit 
semidirect $(\hat{G}'\times I_F)\rtimes \langle  h_\phi \rangle$, et l'application
$$
\ES{G}'\rightarrow {^L{G}'},\, (x,w)h_\phi^n \mapsto (x, w\phi^n),\quad (x,w)\in \hat{G}'\times I_F,\, n\in {\Bbb Z},
$$
est un isomorphisme \cite[6.3]{Stab I}. Cela entraîne que dans cette situation non ramifiée, on peut faire l'économie des données auxiliaires. D'autre part on a un automorphisme $\hat{\theta}$ de $\hat{G}$ qui stabilise $\hat{\ES{E}}$ et commute à l'action galoisienne $\sigma\mapsto \sigma_G$. \'Ecrivons $\tilde{s}=s\hat{\theta}$, et posons $\bs{h}= h \phi$. Alors la condition
$$
{\rm Int}_{\tilde{s}}(g',w)= (a'(w)g',w),\quad (g',w)\in \ES{G}',
$$
équivaut à
$$
\tilde{s}\bs{h} = a(\phi)\bs{h} \tilde{s};\leqno{(1)}
$$
où l'égalité est vue dans $\hat{G}W_F\hat{\theta}= \hat{G}\hat{\theta}W_F$.

\vskip1mm 
Rappelons qu'à une paire de Borel épinglée $\ES{E}_1$ de $G$ définie sur $F$, la théorie de Bruhat--Tits associe un sous--groupe (compact maximal) hyperspécial $K_{\ES{E}_1}$ de $G(F)$. Précisément, à $\ES{E}_1$ est associé un schéma en groupes lisse $\ES{K}_{\ES{E}_1}$ défini sur l'anneau des entiers $\mathfrak{o}$ de $F$, de fibre générique $G$, et le groupe $K_{\ES{E}_1}=\ES{K}_{\ES{E}_1}(\mathfrak{o})$ des points $\mathfrak{o}$--rationnels de $\ES{K}_{\ES{E}_1}$ est un sous--groupe hyperspécial de $G(F)$. De plus, notant $\mathfrak{o}^{\rm nr}$ l'anneau des entiers de $F^{\rm nr}$, le sous--groupe $\ES{K}_{\ES{E}_1}(\mathfrak{o}^{\rm nr})$ de $G(F^{\rm nr})$ ne dépend que de $K= K_{\ES{E}_1}$, et on le note $K^{\rm nr}$. 

Soit $K$ un sous--groupe hyperspécial de $G(F)$. Posons
$$
N_{\smash{\wt{G}(F)}}(K)=\{\delta\in \wt{G}(F): {\rm Int}_\delta(K)=K\}.
$$
Si $N_{\smash{\wt{G}(F)}}(K)$ n'est pas vide, alors c'est un espace principal homogène sous $Z(G;F)K$, et tout élément $\delta\in N_{\smash{\wt{G}(F)}}(K)$ définit un $K$--espace tordu $\wt{K}=K\delta=\delta K$. Si de plus il existe une paire de Borel épinglée $\ES{E}_1$ de $G$ définie sur $F$ telle que $K= K_{\ES{E}_1}$ et $\wt{K} \cap K^{\rm nr}Z(\wt{G},\ES{E}_1; F^{\rm nr})$ ne soit pas vide, alors on dit que $\wt{K}$ est un {\it sous--espace hyperspécial} de $\wt{G}(F)$. 

\begin{marema}
{\rm Les paires de Borel épinglées de $G$ définies sur $F$, et donc aussi les sous--groupes hyperspéciaux de $G(F)$, forment une seule orbite sous l'action de $G_{\rm AD}(F)$ par conjugaison. Notons que cela n'implique pas que si $N_{\smash{\wt{G}(F)}}(K)\neq \emptyset$ pour un sous--groupe hyperspécial $K$ de $G(F)$, alors $N_{\smash{\wt{G}(F)}}(K_1)\neq \emptyset$ pour tout sous--groupe hyperspécial $K_1$ de $G(F)$. On renvoie au lemme 1 de \ref{action du groupe C} pour une variante de ce résultat de conjugaison pour les sous--espaces hyperspéciaux de $\wt{G}(F)$. 
} 
\end{marema}

\begin{meshypos}{\rm On suppose que:
\begin{itemize}
\item Le caractère $\omega$ est trivial sur $Z(G;F)^\theta$;
\item La classe de cohomologie $\bs{a}\in {\rm H}^1(W_F,Z(\hat{G}))$ est non ramifiée.
\item $\wt{G}(F)$ possède un sous--espace hyperspécial.
\end{itemize}}
\end{meshypos}

Fixons un sous--espace hyperspécial $(K,\wt{K})$ de $\wt{G}(F)$. Par définition, il existe une paire de Borel épinglée $\ES{E}_1$ de $G$ définie sur $F$ telle que $K=K_{\ES{E}_1}$, l'ensemble $N_{\smash{\wt{G}(F)}}(K)$ n'est pas vide, et $\wt{K}= K\delta = \delta K$ pour un élément $\delta\in N_{\smash{\wt{G}(F)}}(K)$ de la forme
$$
\delta = k \epsilon, \quad k\in K^{\rm nr},\, \epsilon \in Z(\wt{G},\ES{E}_1; F^{\rm nr}).
$$
Quitte à remplacer la paire $\ES{E}$ fixée en \ref{L-groupe} par une autre paire de Borel épinglée de $G$ définie sur $F$, on peut supposer que $\ES{E}_1=\ES{E}$. 

Soit $(\wt{B},\wt{T})$ la paire de Borel de $\wt{G}$ associée à $(B,T)$, \cad que $\wt{B}$ est le normalisateur de $B$ dans $\wt{G}$, et $\wt{T}$ est le normalisateur de $(B,T)$ dans $\wt{G}$. 
Fixé un élément $\epsilon\in Z(\wt{G},\ES{E})$, on a
$$
\wt{B}=B\epsilon,\quad \wt{T}=T\epsilon.
$$
En particulier les ensembles $\wt{B}$ et $\wt{T}$ ne sont pas vides, $\wt{B}$ est un $B$--espace tordu défini sur $F$, et $\wt{T}$ est un $T$--espace tordu défini sur $F$. 
Notons que, même si $Z(\wt{G},\ES{E};F)$ est vide, les ensembles $\wt{B}(F)$ et $\wt{T}(F)$ ne sont pas vides. De plus on a la décomposition en produit semi--direct
$$
\wt{B}(F)=\wt{T}(F)\rtimes U_B(F),
$$
où $U_B$ est le radical unipotent de $B$. Soit $T(F^{\rm nr})_1$ le plus grand sous--groupe borné de $T(F^{\rm nr})$, \cad l'ensemble des $t\in T(F^{\rm nr})$ tels que $v_F(x(t))=0$ pour tout caractère $x \in X(T)$. On a $T(F^{\rm nr})_1= T(F^{\rm nr})\cap K^{\rm nr}$, et d'après \cite[6.1.(5)]{Stab I}, l'intersection
$$
Z(\wt{G},\ES{E};F^{\rm nr})\cap T(F^{\rm nr})_1\wt{K}
$$
n'est pas vide. L'espace $\wt{K}$ s'écrit donc
$$
\wt{K}=K\delta_\circ=\delta_\circ K
$$
pour un élément $\delta_\circ\in \wt{T}(F)$ de la forme 
$$
\delta_\circ= t\epsilon,\quad t\in T(F^{\rm nr})_1,\,\epsilon\in Z(\wt{G},\ES{E};F^{\rm nr}).
$$
Le $F$--automorphisme $\theta_\circ={\rm Int}_{\delta_\circ}$ de $G$ est semisimple (car il stabilise $(B,T)$) 
et sa restriction à $T$ co\"{\i}ncide avec celle de $\theta_{\ES{E}}$. 

Soit 
$\bs{G}'=(G',\ES{G}',\tilde{s})$ une donnée endoscopique non ramifiée pour $(\wt{G},\bs{a})$. D'après le lemme de \cite[6.2]{Stab I}, la donnée $\bs{G}'$ est relevante. 
D'après \cite[6.2, 6.3]{Stab I}, à $(K,\wt{K})$ sont associés un sous--espace hyperspécial $(K',\wt{K}')$ de 
$\wt{G}'(F)$, bien défini à conjugaison près par un élément de $G'_{\rm AD}(F)$, et un facteur de transfert normalisé
$$
\Delta:\ES{D}(\bs{G}')\rightarrow {\Bbb C}^\times.
$$
D'après \cite[theorem 5.1.D]{KS1}, pour $(\delta,\gamma)\in \ES{D}(\bs{G}')$ et $g\in G(F)$, 
on a
$$
\Delta(\delta ,g^{-1}\gamma g)= \omega(g)\Delta(\delta,\gamma).\leqno{(2)}
$$
Posons $G_\sharp= G/Z(G)^\theta$. C'est un groupe réductif connexe défini sur $F$, et le groupe $G_\sharp(F)$ de ses points $F$--rationnels opère par conjugaison sur $\wt{G}(F)$. D'après \cite[2.7]{Stab I}, à $\bs{G}'$ est associé un caractère $\omega'=\omega_{\bs{G}'}$ de $G_\sharp(F)$ tel que pour 
$(\delta,\gamma)\in \ES{D}(\bs{G}')$ et $g\in G_\sharp(F)$, 
on ait
$$
\Delta(\delta ,{\rm Int}_{\smash{g^{-1}}}(\gamma))= \omega'(g)\Delta(\delta,\gamma).\leqno{(3)}
$$
Rappelons que $\omega$ est trivial sur $Z(G;F)^\theta=Z(G)^\theta(F)$. La projection naturelle $G\rightarrow G_\sharp$ induit un homomorphisme injectif $G(F)/Z(G;F)^\theta\rightarrow G_\sharp(F)$, et la 
restriction de $\omega'$ ˆ l'image de $G(F)$ dans $G_\sharp(F)$ co\"{\i}ncide avec $\omega$.

\subsection{Transfert endoscopique non ramifié}\label{transfert nr}On note $C^\infty_{\rm c}(\wt{G}(F))$ l'espace vectoriel des fonctions complexes sur $\wt{G}(F)$ qui sont localement constantes et à support compact. Fixons une mesure de Haar $dg$ sur $G(F)$. Pour 
$\gamma\in \wt{G}(F)$ fortement régulier, on pose 
$$
D^{\wt{G}}(\gamma)= \vert \det (1- {\rm ad}_\gamma; \mathfrak{g}/\mathfrak{g}_\gamma) \vert_F
$$
où:
\begin{itemize}
\item $\mathfrak{g}$ est l'algèbre de Lie de $G$;
\item $\mathfrak{g}_\gamma $ est l'algèbre de Lie de $G_\gamma$;
\item ${\rm ad}_\gamma$ est l'endomorphisme de $\mathfrak{g}$ déduit de l'automorphisme ${\rm Int}_\gamma$ de $G(F)$;
\item $\vert\;\vert_F$ est la valeur absolue normalisée sur $F$, \cad telle que $\vert x\vert_F =q^{-v_F(x)}$ où $q$ est le cardinal du corps résiduel de $F$.
\end{itemize} 
Fixons une mesure de Haar $dg_\gamma$ sur $G_\gamma(F)$, et notons $d\bar{g}_\gamma$ 
la mesure quotient ${dg\over dg_\gamma}$ sur $G_\gamma(F)\backslash G(F)$. Soit 
$f\in C^\infty_{\rm c}(\wt{G}(F))$. Si le caractère $\omega$ est trivial sur $G_\gamma(F)$, on pose
$$
I^{\wt{G}}(\gamma,\omega, f)= D^{\wt{G}}(\gamma)^{1/2}\int_{G_\gamma(F)\backslash G(F)}
\omega(g)f(g^{-1}\gamma g) d\bar{g}_\gamma.
$$
Sinon, on pose $I^{\wt{G}}(\gamma,\omega, f)=0$. On note $\bs{I}(\wt{G}(F),\omega)$ le quotient de 
$C^\infty_{\rm c}(\wt{G}(F))$ par le sous--espace annulé par toutes les distributions $I^{\wt{G}}(\gamma,\omega,\cdot)$ pour $\gamma\in \wt{G}(F)$ fortement régulier.

Soit $\bs{G}'=(G',\ES{G}',\tilde{s})$ une donnée endoscopique non ramifiée pour $(\wt{G},\bs{a})$. 
On note $C^\infty_{\rm c}(\wt{G}'(F))$ l'espace vectoriel des fonctions complexes sur $\wt{G}'(F)$ qui sont localement constan\-tes et à support compact. Fixons une mesure de Haar $dg'$ sur $G'(F)$. Pour 
$\delta\in \wt{G}'(F)$ fortement $\wt{G}$--régulier, on fixe une mesure de Haar $dg'_\delta$ sur 
$G_\delta(F)$, et l'on note $d\bar{g}'_\delta$ la mesure quotient ${dg'\over dg'_\delta}$ sur 
$G'_\delta(F)\backslash G'(F)$. Soit $f'\in C^\infty_{\rm c}(\wt{G}'(F))$. On pose
$$
I^{\wt{G}'}(\delta,f')= D^{\wt{G}'}(\delta)^{1/2} \int_{G'_\delta(F)\backslash G'(F)}f'(g'^{-1}\delta g')d\bar{g}'_\delta,
$$
et
$$
S^{\wt{G}'\!}(\delta,f')= \sum_{\delta'}I^{\wt{G}}(\delta',f')
$$
où $\delta'$ parcourt les éléments de la classe de conjugaison stable de $\gamma$ modulo conjugaison par $G'(F)$. On note $\bs{SI}(\wt{G}'(F))$ le quotient de $C^\infty_{\rm c}(\wt{G}'(F))$ par le sous--espace annulé par les distributions $S^{\wt{G}'\!}(\delta,\cdot)$ pour $\delta\in \wt{G}'(F)$ fortement régulier (dans $\wt{G}'$).

\vskip1mm
Soit un élément $\delta\in G'(F)$ fortement $\wt{G}$--régulier. Pour $\gamma\in G(F)$ tel que $(\delta,\gamma)\in \ES{D}(\bs{G}')$, on a un homomorphisme naturel
$$
G_\gamma(F)\rightarrow G'_\delta(F)
$$
qui est un homéomorphisme local, et un revêtement sur son image (cf. \cite[2.4]{Stab I}). On fixe des mesures de Haar $dg_\gamma$ sur $G_\gamma(F)$ et $dg'_\delta$ sur $G'_\delta(F)$ qui se correspondent par cet homomorphisme. On pose
$$
d_{\theta}= \vert \det (1-\theta_{\ES{E}}; \mathfrak{t}/\mathfrak{t}_{\theta_{\ES{E}}})\vert_F 
$$
où $\mathfrak{t}$ est l'algèbre de Lie de $T$, et $\mathfrak{t}_{\theta_{\ES{E}}}$ est l'algèbre de Lie de $T_{\theta_{\ES{E}}}=T^{\theta_{\ES{E}},\circ}$ (\cad la sous--algèbre de Lie de $\mathfrak{t}$ formée des éléments fixés par $\theta_{\ES{E}}$). 
Pour $f\in C^\infty_{\rm c}(\wt{G}(F))$, on pose
$$
I^{\wt{G},\omega}(\delta,f)= d_\theta^{1/2}\sum_\gamma d_\gamma^{-1}\Delta(\delta,\gamma)I^{\wt{G}}(\gamma,\omega,f)
$$
où $\gamma$ parcourt les éléments de $\wt{G}(F)$ tels que $(\delta,\gamma)\in \ES{D}(\bs{G}')$, modulo conjugaison par $G(F)$, et $d_\gamma = [G^\gamma(F):G_\gamma(F)]$. Pour 
$f'\in C^\infty_{\rm c}(\wt{G}'(F))$, on dit que $f'$ est un transfert de $f$ si pour tout $\delta\in \wt{G}'(F)$ fortement $\wt{G}$--régulier, on a l'égalité
$$
S^{\wt{G}'\!}(\delta,f')= I^{\wt{G},\omega}(\delta,f).
$$
D'après la conjecture de transfert maintenant démontrée --- pour n'importe quelle donnée endoscopique pour $(\wt{G},\bs{a})$, pas seulement pour celles qui sont non ramifiées! ---, on sait que toute fonction $f\in C^\infty_{\rm c}(\wt{G}(F))$ admet un transfert $f'\in C^\infty_{\rm c}(\wt{G}'(F))$. 

Par passage aux quotients, on peut voir le transfert comme une application linéaire
$$
\bs{I}(\wt{G}(F),\omega)\rightarrow \bs{SI}(\wt{G}'(F)), \,\bs{f}\mapsto \bs{f}'=\bs{f}^{\wt{G}'}.\leqno{(1)}
$$
Rappelons que le facteur de transfert a été normalisé grâce au choix de l'espace hyperspécial $(K,\wt{K})$. Il convient de normaliser aussi les mesures de Haar $dg$ sur $G(F)$ et $dg'$ sur $G'(F)$: on suppose désormais que
$$
{\rm vol}(K,dg)={\rm vol}(K',dg')=1,
$$
où $(K',\wt{K}')$ est un sous--espace hyperspécial de $\wt{G}'(F)$ associé à $(K,\wt{K})$. 

\subsection{Le lemme fondamental}\label{lemme fondamental}
Soit $\bs{G}'=(G',\ES{G}',\tilde{s})$ une donnée endoscopique non ramifiée pour $(\wt{G},\bs{a})$, et soit $(K',\wt{K}')$ un sous--espace hyperspécial de $\wt{G}'(F)$ associé à $(K,\wt{K}')$. On note $\bs{1}_{\wt{K}}\in C^\infty_{\rm c}(\wt{G}(F))$ la fonction caractéristique de $\wt{K}$, et 
$\bs{1}_{\wt{K}'}\in C^\infty_{\rm c}(\wt{G}'(F))$ la fonction caractéristique de $\wt{K}'$. 

Si la caractŽristique rŽsiduelle de $F$ est assez grande, alors le théorème suivant (lemme fondamental pour les unités) est maintenant démontré, 
gr‰ce au rŽsultat de Ngo Bao Chau en caractŽristique non nulle.

\begin{montheo1}
La fonction $\bs{1}_{\wt{K}'}$ est un transfert de $\bs{1}_{\wt{K}}$.
\end{montheo1}

Notons $\mathcal{H}_{K}$ l'algèbre des fonctions complexes sur $G(F)$ qui sont bi--invariantes par $K$ et à support compact, et $\mathcal{H}_{K'}$ l'algèbre des fonctions complexes sur $G'(F)$ qui sont bi--invariantes par $K'$ et à support compact. Rappelons que $\phi\in W_F$ est un élément de Frobenius. Notons 
$\hat{\mathcal{H}}_\phi$ l'algèbre des fonctions polynomiales sur $\hat{G}\rtimes \phi \subset {^LG}$ qui sont invariantes par conjugaison par $\hat{G}$, et $\hat{\mathcal{H}}'_\phi$ l'algèbre des fonctions polynomiales sur $\ES{G}'\cap (\hat{G}\rtimes \phi)\simeq \hat{G}'\rtimes \phi \subset {^L{G}'}$ qui sont invariantes par conjugaison par $\hat{G}'$. Soit $b:\mathcal{H}_K\rightarrow \mathcal{H}_{K'}$ l'homomorphisme d'algèbres qui rend commutatif le diagramme
$$
\xymatrix{
\mathcal{H}_K \ar[r]^{\simeq} \ar[d]_b & \hat{\mathcal{H}}_\phi \ar[d]^{\hat{b}}\\
\mathcal{H}_{K'} \ar[r]^{\simeq} & \hat{\mathcal{H}}'_\phi
},
$$
où les flèches horizontales sont les isomorphismes de Satake, et $\hat{b}$ est l'homomorphisme de restriction. L'algèbre $\mathcal{H}_K$ opère par convolution à gauche et à droite sur $C^\infty_{\rm c}(\wt{G}(F))$, et l'algèbre $\mathcal{H}_{K'}$ opère par convolution à gauche et à droite sur $C^\infty_{\rm c}(\wt{G}'(F))$.

\begin{montheo2}
Soit $f\in \mathcal{H}_K$, et soit $f'=b(f)\in \mathcal{H}_{K'}$. Alors $f'*\bs{1}_{\wt{K}'}= \bs{1}_{\wt{K}'}* f'$ est un transfert de 
$f* {\bf 1}_{\wt{K}}$ (et de $\bs{1}_{\wt{K}}* \omega^{-1}f)$.
\end{montheo2}

On prouve dans ce papier que le théorème 1 implique le théorème 2.

\subsection{Descente parabolique}\label{descente parabolique}Commen\c{c}ons par quelques 
rappels \cite[3.1]{Stab I}. Soit $(P,M)$ une paire parabolique de $G$ (pas forcément définie sur $F$). On note 
$\wt{P}=N_{\smash{\wt{G}}}(P)$ le normalisateur $\{\gamma\in \wt{G}: {\rm Int}_\gamma(P)=P\}$ de $P$ dans $\wt{G}$, et $\wt{M}_P=N_{\smash{\wt{G}}}(P,M)$ le normalisateur de $(P,M)$ dans $\wt{G}$, \cad l'ensemble $\{\gamma \in \wt{G}: {\rm Int}_\gamma(P,M)=(P,M)\}$. Si $\wt{P}$ n'est pas vide, alors $\wt{M}_P$ ne l'est pas non plus, et l'on dit que $(\wt{P},\wt{M}_P)$ est une {\it paire parabolique de $\wt{G}$}. En particulier (toujours si $\wt{P}\neq \emptyset $), on a la décomposition en produit semi--direct
$$
\wt{P}= \wt{M}_P\rtimes U_P,\leqno{(1)}
$$
où $U_P$ est le radical unipotent de $P$. Notons que $\wt{M}_P$ dépend vraiment de la paire $(P,M)$, 
et pas seulement de la composante de Levi $M$ (sauf si $\wt{G}$ est à torsion intérieure, auquel cas $\wt{M}_P$ est l'ensemble des $\gamma\in \wt{G}$ tels que ${\rm Int}_\gamma\in M/Z(G)\subset G_{\rm AD}$). 
Supposons de plus que la paire $(P,M)$ est définie sur $F$. Si $\wt{P}$ n'est pas vide, alors 
la paire $(\wt{P},\wt{M}_P)$ et la décomposition (1) sont définies sur $F$, $\wt{P}(F)$ et 
$\wt{M}_P(F)$ ne sont pas vides, et on a la décomposition en produit semi--direct
$$
\wt{P}(F)=\wt{M}_P(F)\rtimes U_P(F). \leqno{(2)}
$$
En ce cas (si $\wt{P}\neq \emptyset$), on dit que $(\wt{P},\wt{M}_P)$ est une paire parabolique de $\wt{G}$ définie sur $F$. 

Pour alléger l'écriture, on appelle simplement {\it $F$--Levi de $G$} la composante de Levi d'une paire parabolique de $G$ définie sur $F$, et {\it $F$--Levi de $\wt{G}$} la composante de Levi d'une paire parabolique de $\wt{G}$ définie sur $F$. 

Rappelons qu'on a fixé en \ref{L-groupe} une paire de Borel épinglée $\ES{E}=(B,T,\{E_\alpha\}_{\alpha\in \Delta})$ de $G$ définie sur $F$, et en \ref{données endoscopiques nr} un sous--espace hyperspécial $(K,\wt{K})$ de $\wt{G}(F)$ tel que $K= K_{\ES{E}}$ et l'ensemble 
$\wt{K} \cap K^{\rm nr}Z(\wt{G},\ES{E}; F^{\rm nr})$ n'est pas vide, ce qui assure que l'ensemble $\wt{K} \cap T(F^{\rm nr})_1Z(\wt{G},\ES{E}; F^{\rm nr})$ n'est pas vide. 
Une paire parabolique $(P,M)$ de $G$ est dite standard si elle contient $(B,T)$, et semi--standard si $M$ contient $T$. Un sous--groupe parabolique de $P$ de $G$ est dit standard (reps. semi--standard) s'il contient $B$ (reps. $T$), et une composante de Levi de $P$ est dite semi--standard si elle contient $T$. Posons $\wt{T}=N_{\smash{\wt{G}}}(B,T)$. Puisque $\wt{T}$ n'est pas vide, $(\wt{B},\wt{T})$ est une paire parabolique (minimale) de $\wt{G}$, et elle est définie sur $F$. On définit de la même manière, en rempla\c{c}ant la paire $(B,T)$ par $(\wt{B},\wt{T})$, les notions de paire parabolique standard (resp. semi--standard), de 
sous--espace parabolique standard (reps. semi--standard), et de composante de Levi semi--standard, de $\wt{G}$. L'application $(P,M)\mapsto (\wt{P},\wt{M}_P)$ est une bijection entre:
\begin{itemize}
\item les paires paraboliques standards de $G$ qui sont définies sur $F$ et $\theta_{\ES{E}}$--stables;
\item les paires paraboliques standards de $\wt{G}$ qui sont définies sur $F$;
\item les classes de $G$--conjugaison de paires paraboliques $\wt{G}$ qui sont définies sur $F$.
\end{itemize}

\begin{marema}
{\rm Fixons une paire de Borel épinglée $\hat{\ES{E}}=\{\hat{B},\hat{T},\{\hat{E}_\alpha\}_{\alpha\in \hat{\Delta}})$ de $\hat{G}$. Modifions l'action 
$\sigma \mapsto \sigma_G$ de $\Gamma_F$ sur $\hat{G}$ et l'isomorphisme ${^LG}\simeq \hat{G}\rtimes W_F$ de telle manière que la paire $\hat{\ES{E}}$ soit définie sur $F$ --- cf. \ref{données endoscopiques}. On a aussi un automorphisme $\hat{\theta}$ de $\hat{G}$ qui stabilise $\hat{\ES{E}}$ et commute à l'action galoisienne $\sigma\mapsto \sigma_G$. Une paire parabolique de $\hat{G}$ est dite standard si elle contient $(\hat{B},\hat{T})$. On a aussi une bijection $(P,M)\mapsto (\hat{P},\hat{M})$ entre:
\begin{itemize}
\item les paires paraboliques standards de $G$ qui sont définies sur $F$ et $\theta_{\ES{E}}$--stables;
\item les paires paraboliques standards de $\hat{G}$ qui sont définies sur $F$ et $\hat{\theta}$--stables.
\end{itemize}}
\end{marema}

Soit $(P,M)$ une paire parabolique de $G$, semi--standard et définie sur $F$. Posons $\wt{M}=\wt{M}_P$, et supposons que $\wt{M}$ est semi--standard. Puisque l'ensemble $\wt{M}$ contient $\wt{T}$, il n'est pas vide, et $(\wt{P},\wt{M})$ est une paire parabolique de $\wt{G}$, semi--standard et définie sur $F$. La paire $\ES{E}$ donne par restriction une paire de Borel épinglée $\ES{E}_M=(B\cap M,T,\{E_\alpha\}_{\alpha\in \Delta_M})$ de $M$, elle aussi définie sur $F$. On a
$$
Z(\wt{M},\ES{E}_M)=Z(M)Z(\wt{G},\ES{E})
$$
et le $F$--automorphisme de $M$ défini par $\wt{M}$ et $\ES{E}_M$ (cf. \ref{L-groupe}) n'est autre que la restriction de $\theta_{\ES{E}}$ à $M$. On le note $\theta_{\ES{E},M}$. Puisque $\theta_{\ES{E}}$ est d'ordre fini, 
$\theta_{\ES{E},M}$ l'est aussi, et l'automorphisme $\theta_M$ de $Z(M)$ défini par $\wt{M}$ 
(qui est la restriction de $\theta_{\ES{E},M}$ à $Z(M)$) est d'ordre fini. L'espace tordu $(M,\wt{M})$ vérifie donc les trois hypothèses de \ref{objets}. Le sous--groupe hyperspécial $K_M=K_{\ES{E}_M}$ de $M(F)$ associé à $\ES{E}_M$ est donné par
$$
K_M=K\cap M(F).
$$
Puisque le sous--espace hyperspécial $\wt{K}$ de $\wt{G}(F)$ est de la forme $\wt{K}=K\delta_\circ= \delta_\circ K$ pour un $\delta_\circ\in \wt{T}(F)\subset \wt{M}(F)$ de la forme $\delta_0=t\epsilon$ avec $t\in T(F^{\rm nr})_1$ et $\epsilon\in Z(\wt{G},\ES{E};F^{\rm nr})$, 
l'ensemble 
$$
\wt{K}_M =\wt{K}\cap \wt{M}(F)\;(= K_M\delta_\circ = \delta_\circ K_M)
$$
est un sous--espace hyperspécial de $\wt{M}(F)$.

\begin{notation}
{\rm On note $\ES{L}(\wt{T},\omega)$ l'ensemble des $F$--Levi semi--standards $\wt{M}$ de 
$\wt{G}$ tels que le caractère $\omega$ est trivial sur $Z(M; F)^{\theta_M}$.
}
\end{notation}

\vskip1mm Soit $\wt{M}\in \ES{L}(\wt{T},\omega)$. Choisissons une paire de Borel $(P,M)$ de $G$, 
semi--standard et définie sur $F$, telle que $\wt{M}=\wt{M}_P$. Puisque $K$ est le sous--groupe hyperspécial $K_{\ES{E}}$ de $G(F)$ associé à $\ES{E}$, il est en bonne position par rapport à $(P,M)$, \cad qu'on a la décomposition
$$
P(F)\cap K = (M(F)\cap K)(U_P(F)\cap K).
$$
Posons $U=U_P$, et soient $dk$, $dm$ et $du$ les mesures de Haar sur $K$, $M(F)$, $U(F)$, normalisées par $K$, \cad telles que
$$
{\rm vol}(K,dk)= {\rm vol}(M(F)\cap K,dm)={\rm vol}(U(F)\cap K,du)=1.
$$
On définit l'homomorphisme {\it terme constant (suivant $(\wt{P},\omega)$)}
$$
C^\infty_{\rm c}(\wt{G}(F))\rightarrow C^\infty_{\rm c}(\wt{M}(F)),\, f\mapsto f_{\wt{P},\omega}
$$
par la formule
$$
f_{\wt{P},\omega}=\int\!\!\!\int_{U(F)\times K} \omega^{-1}(k)f(k^{-1}u^{-1}\gamma uk)dudk.
$$
Pour $\gamma\in \wt{M}(F)$ fortement régulier dans $\wt{G}$, on a la formule de descente
$$
I^{\wt{G}}(\gamma,\omega,f)= I^{\wt{M}}(\gamma,\omega,f_{\wt{P},\omega}),\quad 
f\in C^\infty_{\rm c}(\wt{G}(F)),\leqno{(3)}
$$
pourvu que l'on choisisse la même mesure de Haar $dg_\gamma = dm_\gamma$ 
sur $G_\gamma(F)=M_\gamma(F)$. Pour définir la mesure quotient sur $M_\gamma(F)\backslash M(F)$, on prend bien sûr $dm$ comme mesure de Haar sur $M(F)$. Choisissons une autre paire de Borel $(Q,M)$ de $G$, semi--standard et définie sur $F$, telle que $\wt{M}_Q=\wt{M}$. On définit de la même manière l'homomorphisme terme constant (suivant $\wt{Q}$)
$$
C^\infty_{\rm c}(\wt{G}(F))\rightarrow C^\infty_{\rm c}(\wt{M}(F)),\, f\mapsto f_{\wt{Q},\omega}.
$$
Pour $\gamma\in \wt{M}(F)$ fortement régulier dans $\wt{G}$, on a aussi la formule de descente
$$
I^{\wt{G}}(\gamma,\omega,f)= I^{\wt{M}}(\gamma,\omega,f_{\wt{Q},\omega}),\quad 
f\in C^\infty_{\rm c}(\wt{G}(F)).
$$
Par passage aux quotients, on obtient un homomorphisme
$$
\bs{I}(\wt{G},\omega)\rightarrow \bs{I}(\wt{M},\omega),\, \bs{f}\mapsto \bs{f}_{\wt{M},\omega},
\leqno{(4)}
$$
qui ne dépend pas du choix de la paire parabolique $(P,M)$ de $G$, semi--standard et définie sur $F$,  telle que $\wt{M}_P= \wt{M}$. On peut aussi vérifier que si l'on remplace $K$ par un sous--groupe (compact maximal) spécial $K_1$ de $G(F)$ en bonne position par rapport à $(P,M)$ dans la définition de l'homomorphisme terme constant (suivant $\wt{P}$), alors par passage aux quotients, on obtient le même homomorphisme (4).

\subsection{Réduction aux données endoscopiques elliptiques}\label{réduction aux données elliptiques}Rappelons qu'une donnée endoscopique $\bs{G}'=(G',\ES{G}',\tilde{s})$ pour $(\wt{G},\bs{a})$ est dite {\it elliptique} si on a 
l'égalité
$$
Z(\hat{G}')^{\Gamma_F,\circ}=[ Z(\hat{G})^{\hat{\theta}}]^{\Gamma_F,\circ}.
$$

Soit $\bs{G}'=(G',\ES{G}',\tilde{s})$ une donnée endoscopique non ramifiée pour $(\wt{G},\bs{a})$. 
On fixe des paires de Borel épinglées compatibles 
$\hat{\ES{E}}=(\hat{B},\hat{T},\{\hat{E}_\alpha\}_{\alpha\in \hat{\Delta}})$ de $\hat{G}$ 
et $\hat{\ES{E}}'=
(\hat{B}',\hat{T}',\{\hat{E}'_\alpha\}_{\alpha\in \hat{\Delta}'})$ de $\hat{G}'=\hat{G}_{\tilde{s}}$, et l'on normalise les actions galoisiennes $\sigma\mapsto \sigma_G$ sur $\hat{G}$ et $\sigma\mapsto \sigma_{G'}$ sur $\hat{G}'$ de manière à ce qu'elles préservent ces paires --- cf. \ref{données endoscopiques} . On a aussi un automorphisme $\hat{\theta}$ de $\hat{G}$ qui préserve $\hat{\ES{E}}$ et commute à l'action galoisienne $\sigma \mapsto \sigma_G$. 

Notons $\hat{M}$, $\ES{M}$, $\wt{\ES{M}}$ les commutants de $Z(\hat{G}')^{\Gamma_F,\circ}$ dans $\hat{G}$, ${^LG}$, ${^L{\wt{G}}}$. Le groupe $\hat{M}$ est un sous--groupe de Levi semi--standard --- \cad contenant $\hat{T}$ --- de $\hat{G}$, qui est dŽfini sur $F$ et $\hat{\theta}$--stable. Fixons un cocaractre $x\in \check{X}(Z(\hat{G}')^{\Gamma_F,\circ})$ en position gŽnŽrale. Il dŽtermine un sous--groupe parabolique $\hat{P}$ de $\hat{G}$, engendrŽ par $\hat{M}$ et les sous--groupes radiciels de $\hat{G}$ associŽs aux racines $\alpha$ de $\hat{T}$ telles que $\langle \alpha, x\rangle >0$. Posons $\ES{P}= \hat{P}\ES{M}$ et $\wt{\ES{P}}= \hat{P}\wt{\ES{M}}$. D'aprs \cite[3.4]{Stab I}, le couple $(\wt{\ES{P}},\wt{\ES{M}})$ est une paire parabolique de ${^L\wt{G}}$: on a
$$
\ES{P}= \hat{P}\rtimes W_F,\quad \ES{M}= \hat{M}\rtimes W_F,
$$
$\wt{\ES{P}}$ est le normalisateur de $\ES{P}$ dans ${^L{\wt{G}}}$, et $\wt{\ES{M}}$ est le normalisateur de $(\ES{P},\ES{M})$ dans ${^L{\wt{G}}}$; de plus l'ensemble 
$\wt{\ES{P}}$ n'est pas vide, car il contient $\tilde{s}$. D'aprs \cite[3.1, (4)]{Stab I}, la paire $(\hat{P},\hat{M})$ est conjuguŽe dans $\hat{G}$ ˆ une paire parabolique standard de $\hat{G}$, dŽfinie sur $F$ et $\hat{\theta}$--stable. Quitte ˆ effectuer une telle conjugaison, on peut supposer que la paire $(\hat{P},\hat{M})$ est elle--mme standard, dŽfinie sur $F$ et $\hat{\theta}$--stable. Alors on a
$$
\wt{\ES{P}}=(\hat{P}\rtimes W_F){^L\theta},\quad \wt{\ES{M}}=(\hat{M}\rtimes W_F){^L\theta}.
$$
À $(\hat{P},\hat{M})$ correspond une paire parabolique standard $(P,M)$ de $G$ qui est définie sur $F$ et $\theta_{\ES{E}}$--stable (cf. la remarque de \ref{descente parabolique}). Soit $(\wt{P},\wt{M})$ la paire parabolique standard de $\wt{G}$ associée à $(P,M)$, \cad telle que $\wt{P}=N_{\smash{\wt{G}}}(P)$ et $\wt{M}=N_{\smash{\wt{G}}}(P,M)$. Elle 
est définie sur $F$, et l'on a
$$
\wt{P}=P\delta_\circ,\quad \wt{M}= M\delta_\circ.
$$
D'après \ref{descente parabolique}, l'espace tordu $(M,\wt{M})$ vérifie les trois hypothèses de \ref{objets}, et
$$
\wt{K}_M = \wt{K}\cap \wt{M}(F)\;(= K_M\delta_\circ = \delta_\circ K_M)
$$
est un sous--espace hyperspécial de $\wt{M}(F)$. 
Le groupe $\ES{M}$ s'identifie au $L$--groupe ${^LM}$, et $\wt{\ES{M}}$ s'identifie au $L$--espace tordu ${^L{\wt{M}}}= {^LM}{^L\theta}$. Notons $\boldsymbol{a}_M$ l'image de $\boldsymbol{a}$ par l'homomorphisme naturel
$$
{\rm H}^1(W_F,Z(\hat{G}))\rightarrow {\rm H}^1(W_F,Z(\hat{M})),
$$
\cad la classe de cohomologie (non ramifiée) correspondant au caractère $\omega$ 
restreint à $M(F)$. 

Puisque ${\rm Int}_{\tilde{s}}$ stabilise $(\hat{B},\hat{T})$, l'élément $\tilde{s}=s\hat{\theta}$ appartient à $\hat{T}\hat{\theta}\subset \hat{M}\hat{\theta}$. De plus on a les égalités
$$
\hat{G}'= \hat{M}_{\tilde{s}},\quad Z(\hat{G}')^{\Gamma_F,\circ}= [Z(\hat{M})^{\hat{\theta}}]^{\Gamma_F,\circ}.
$$
Comme $\ES{G}'\cap \hat{G}= \hat{M}_{\tilde{s}}$ est contenu dans $\hat{M}$, le groupe $\ES{G}'$ est contenu dans $\ES{M}={^LM}$. Par conséquent $\boldsymbol{G}'$ s'identifie à une donnée endoscopique elliptique (non ramifiée) pour $(\wt{M},\bs{a}_M)$. Notons $\ES{D}^{\wt{M}}(\bs{G}')$ 
le sous--ensemble de $\wt{G}'(F)\times \wt{M}(F)$ obtenu en rempla\c{c}ant $\wt{G}$ par $\wt{M}$ dans la définition de $\ES{D}(\bs{G}')$. D'après le lemme de \cite[6.2]{Stab I}, l'ensemble 
$\ES{D}^{\wt{M}}(\bs{G}')$ n'est pas vide. 
D'après \cite[6.2, 6.3]{Stab I}, à $(K_M,\wt{K}_M)$ 
sont associés un sous--espace hyperspécial $(K'_M,\wt{K}'_M)$ de $\wt{G}'(F)$, bien défini à conjugaison près par un élément de $G'_{\rm AD}(F)$, et un facteur de transfert normalisé
$$
\Delta^{\wt{M}}: \ES{D}^{\wt{M}}(\bs{G}')\rightarrow {\Bbb C}^\times.
$$
On peut prendre $(K'_M,\wt{K}'_M)=(K',\wt{K}')$. En rempla\c{c}ant $\wt{G}$ par $\wt{M}$ dans la relation (2) de \ref{données endoscopiques nr}, on voit que le caractère $\omega$ est trivial sur 
$Z(M;F)^{\theta_M}$. La paire $(\wt{M},\bs{a}_M)$ vérifie donc les trois hypothèses de \ref{données endoscopiques nr}. En particulier le $F$--Levi $\wt{M}$ de $\wt{G}$ appartient à $\ES{L}(\wt{T},\omega)$.

Soit $(\delta,\gamma)\in \wt{G}'(F)\times \wt{M}(F)$ tel que $\gamma$ est fortement régulier dans $\wt{M}$. Si la classe de $G'$--conjugaison de $\delta$ correspond à la classe de $M$--conjugaison de $\gamma$ (par la correspondance de \ref{classes de conjugaison} pour 
$(\wt{G}',\wt{M})$), alors elle correspond aussi à la classe de $G$--conjugaison de $\gamma$, et $\gamma$ est fortement régulier dans $\wt{G}$. On a donc l'inclusion $
\ES{D}^{\wt{M}}(\bs{G}')\subset \ES{D}(\bs{G}')$. 
Inversement, pour $(\delta,\gamma)\in \ES{D}(\bs{G}')$ tel que $\gamma\in \wt{M}(F)$, il existe un $n\in N_{G(F)}(\wt{M})$ tel que $(\delta, {\rm Int}_n(\gamma))\in \ES{D}^{\wt{M}}(\bs{G}')$. On vérifie (cf. \cite[6.3]{Stab I}) que le facteur de transfert normalisé $\Delta^{\wt{M}}: \ES{D}^{\wt{M}}(\bs{G}')\rightarrow {\Bbb C}^\times$ co\"{\i}ncide avec la restriction à $\ES{D}^{\wt{M}}(\bs{G}')$ du facteur de transfert normalisé $\Delta: \ES{D}(\bs{G}')\rightarrow {\Bbb C}^\times$. On en déduit que pour toute fonction $f\in C^\infty_{\rm c}(\wt{G}(F))$, si une fonction $f'\in C^\infty_{\rm c}(\wt{G}'(F))$ est un transfert de $f_{\wt{P},\omega}$, alors c'est aussi un transfert de $f$. Autrement dit l'homomorphisme de transfert
$$
\bs{I}(\wt{G}(F),\omega)\rightarrow \bs{SI}(\wt{G}'(F)),\, \bs{f}\mapsto \bs{f}^{\wt{G}'}
$$
co\"{\i}ncide avec la composée de l'homomorphisme terme constant
$$
\bs{I}(\wt{G}(F),\omega)\rightarrow \bs{I}(\wt{M}(F),\omega),\, \bs{f}\mapsto \bs{f}_{\wt{M},\omega}
$$
et de l'homomorphisme de transfert
$$
\bs{I}(\wt{M}(F),\omega)\rightarrow \bs{SI}(\wt{G}'(F)),\, \bs{h}\mapsto \bs{h}^{\wt{G}'}.
$$

\vskip1mm
On définit comme en \ref{descente parabolique} l'homomorphisme terme constant (suivant $(P,\omega)$)
$$
C^\infty_{\rm c}(G(F))\rightarrow C^\infty_{\rm c}(M(F)),\, f\mapsto f_{P,\omega}.
$$
Notons $\mathcal{H}_{K_M}$ l'algèbre des fonctions complexes sur $M(F)$ qui sont bi--invariantes par $K_M$ et à support compact, et $\hat{\mathcal{H}}_\phi^M$ l'algèbre des fonctions polynomiales sur $\hat{M}\rtimes \phi \subset {^LM}$ qui sont invariantes par conjugaison par $\hat{M}$. On a un diagramme commutatif
$$
\xymatrix{
\mathcal{H}_K \ar[r]^{\simeq} \ar[d]_a & \hat{\mathcal{H}}_\phi \ar[d]^{\hat{a}}\\
\mathcal{H}_{K_M} \ar[r]^{\simeq} & \hat{\mathcal{H}}_\phi^M
}
$$
où les flèches horizontales sont les isomorphismes de Satake, la flèche verticale de gauche est donnée par l'homomorphisme terme constant, et la flèche verticale de droite est donnée par la restriction à $\hat{M}\rtimes \phi$ des fonctions sur $\hat{G}\rtimes \phi$. On a aussi un diagramme commutatif
$$
\xymatrix{
\mathcal{H}_{K_M} \ar[r]^{\simeq} \ar[d]_{b_M} & \hat{\mathcal{H}}_\phi^M \ar[d]^{\hat{b}_M}\\
\mathcal{H}_{K'} \ar[r]^{\simeq} & \hat{\mathcal{H}}'_\phi
}
$$
obtenu en rempla\c{c}ant $G$ par $M$ dans le diagramme de \ref{lemme fondamental}. Puisque 
$\hat{b}= \hat{b}_M\circ \hat{a}$, on a $b= b_M\circ a$. 
Soit $f\in \mathcal{H}_K$, et soit $f'=b(f) = b_M(f_{P,\omega})$. On a
$$
(f*{\bf 1}_{\wt{K}})_{\wt{P},\omega}= f_{P,\omega}* {\bf 1}_{\wt{K}_M},\quad 
({\bf 1}_{\wt{K}} *\omega^{-1}f)_{\wt{P},\omega}= {\bf 1}_{\wt{K}_M}*\omega^{-1}f_{P,\omega}.
$$
On en déduit que si $f'*{\bf 1}_{\wt{K}'}={\bf 1}_{\wt{K}'}* f'$ est un transfert de 
$f_{P,\omega}*{\bf 1}_{\wt{K}_M}$ (resp. de ${\bf 1}_{\wt{K}_M}*\omega^{-1}f_{P,\omega}$), alors c'est un 
transfert de $f*{\bf 1}_{\wt{K}}$ (resp. de ${\bf 1}_{\wt{K}}* \omega^{-1}f$). Il suffit donc de prouver le théorème 2 dans le cas où $\bs{G}'$ est une donnée endoscopique non ramifiée {\it elliptique} pour 
$(\wt{G},\bs{a})$.  

\section{Version spectrale du théorème}

\subsection{Caractères tempérés de $(\wt{G}(F),\omega)$}\label{caractères tempérés} 
Toutes les représentations considérées ici sont supposées lisses et à valeurs dans le groupe des automorphismes d'un espace vectoriel sur ${\Bbb C}$. 
On appelle  représentation de $(\wt{G}(F),\omega)$ la donnée d'une représentation $(\pi,V)$ de $G(F)$ et d'une application
$$\tilde{\pi}:\wt{G}(F)\rightarrow {\rm Aut}_{\Bbb C}(V)
$$
vérifiant
$$
\tilde{\pi}(g\delta g')= \omega(g')\pi(g) \tilde{\pi}(\delta)\pi(g'),\quad g,\, g'\in G(F),\, \delta\in \wt{G}(F).
$$
Ainsi pour chaque $\delta\in \wt{G}(F)$, l'opérateur $\pi(\delta)$ entrelace $\omega\pi=\pi\otimes \omega$ et $\pi^\delta = \pi\circ {\rm Int}_\delta$. 

Soit $(\pi,\tilde{\pi})$ --- on dira aussi simplement $\tilde{\pi}$ --- une représentation de $(\wt{G}(F),\omega)$, d'espace $V$. Pour $z\in {\Bbb C}^\times$, on note $(\pi,z\tilde{\pi})$ la représentation de 
$(\wt{G}(F),\omega)$ donnée par
$$(z\tilde{\pi})(\gamma)= 
z \tilde{\pi}(\gamma),\quad \gamma\in \wt{G}(F).
$$
Pour $f\in C^\infty_{\rm c}(\wt{G}(F))$, on note $\tilde{\pi}(fd\gamma)\in {\rm End}_{\Bbb C}(V)$ l'opérateur défini par
$$
\tilde{\pi}(fd\gamma)(v)=\int_{\wt{G}(F)}f(\gamma)\tilde{\pi}(\gamma)d\gamma;
$$
où $d\gamma$ est la mesure positive $G(F)$--invariante (à gauche et à droite) sur $\wt{G}(F)$ normalisée par $\wt{K}$, \cad telle que ${\rm vol}(\wt{K},d\gamma)=1$. Si la représentation $\pi$ de $G(F)$ est admissible, alors cet opérateur est de rang fini, et l'on peut définir sa trace $\Theta_{\tilde{\pi}}= {\rm tr}\,\tilde{\pi}$, qui est une distribution sur $\wt{G}(F)$:
$$
\Theta_{\tilde{\pi}}(f)= {\rm tr}(\tilde{\pi}(fd\gamma)),\quad f\in C^\infty_{\rm c}(\wt{G}(F)).
$$
Pour $z\in {\Bbb C}^\times$, on a
$$
\Theta_{z\tilde{\pi}}=z\Theta_{\tilde{\pi}}.
$$
Si la représentation $\pi$ est de longueur finie (\cad admissible et de type fini), alors on sait \cite{Le} que la distribution $\Theta_{\tilde{\pi}}$ est donnée par une fonction localement constante, disons $\theta_{\tilde{\pi}}$, sur $\wt{G}_{\rm reg}(F)=\{\gamma\in \wt{G}(F): D_{\wt{G}}(\gamma)\neq 0\}$. Puisque $F$ est de caractéristique nulle, on sait aussi \cite{C} que cette fonction $\theta_{\tilde{\pi}}$ est localement intégrable sur $\wt{G}(F)$: pour toute fonction $f\in C^\infty_{\rm c}(\wt{G}(F))$, on a l'égalité
$$
\Theta_{\tilde{\pi}}(f)= \int_{\wt{G}(F)}f(\gamma)\theta_{\tilde{\pi}}(\gamma)d\gamma;
$$
où l'intégrale est absolument convergente. Notons que la distribution $\Theta_{\tilde{\pi}}$ dépend du choix de la mesure $d\gamma$, mais que la fonction $\theta_{\tilde{\pi}}$ n'en dépend pas.

Une représentation $(\pi,\tilde{\pi})$ de $\wt{G}(F)$ est dite unitaire s'il existe un produit scalaire hermitien défini positif sur l'espace de $\pi$ tel que pour tout $\gamma\in \wt{G}(F)$, l'opérateur $\tilde{\pi}(\gamma)$ soit unitaire. Si $(\pi,\tilde{\pi})$ est unitaire, alors pour tout $z\in {\Bbb U}$, la représentation $(\pi,z\tilde{\pi})$ est encore unitaire. Si $\pi$ est irréductible et tempérée, on peut toujours trouver un $z\in{\Bbb C}^\times$ tel que $(\pi,z\tilde{\pi})$ soit unitaire --- on dit alors que $(\pi,z\tilde{\pi})$ est $G(F)$--irréductible et tempérée.

Pour un caractère $\lambda$ de $T(F)$, \cad un homomorphisme continu $\lambda:T(F)\rightarrow {\Bbb C}^\times$, on note $i_B^G(\lambda)$ la série principale associée à $\lambda$, \cad l'induite parabolique normalisée de $\lambda$ à $G(F)$ suivant $B(F)$. Si le caractère $\lambda$ est unitaire, alors $i_B^G(\lambda)$ est une représentation semisimple et tempérée de $G(F)$. 

Soit $K_1$ un sous--groupe hyperspécial de $G(F)$. Une représentation irréductible $(\pi,V)$ de $G(F)$ est dite {\it $K_1$--sphérique} si le sous--espace $V^{K_1}$ de $V$ formé des vecteurs $K_1$--invariants, est non nul. Soit $\pi$ une représentation irréductible $K_1$--sphérique de $G(F)$. Alors $\pi$ est isomorphe à un sous--quotient d'une série principale $i_B^G(\lambda)$ pour un caractère non ramifié $\lambda$ de $T(F)$ uniquement déterminé à conjugaison près par $N_G(T)(F)$, et si de plus $\pi$ est tempérée, alors ce caractère $\lambda$ est unitaire.

\begin{marema1}
{\rm Soit $(\pi,\tilde{\pi})$ une représentation de $(\wt{G}(F),\omega)$, et soit 
$K_1$ un sous--groupe hyperspécial de $G(F)$. Puisque le caractère $\omega$ de $G(F)$ est trivial sur $K_1$, si $\pi$ est irréductible et $K_1$--sphérique, alors pour $\delta_1\in N_{\smash{\wt{G}(F)}}(K_1)$, l'automorphisme $\tilde{\pi}(\delta_1)$ de l'espace $V$ de $\pi$ induit par restriction un automorphisme de $V^{K_1}$.
}
\end{marema1}

\begin{marema2}
{\rm 
Soit $\lambda$ un caractère unitaire de $T(F)$, et soit $\pi$ une sous--représentation irréductible de $i_B^G(\lambda)$. Si $\pi$ se prolonge en une représentation $(\pi,\tilde{\pi})$ de $(\wt{G}(F),\omega)$,  alors l'ensemble
$$
N_{\wt{G}(F),\omega}(\lambda)=\{\delta\in \wt{G}(F): {\rm Int}_\delta(T)=T,\, \lambda^\delta = \omega \lambda\}
$$
n'est pas vide, et c'est un espace principal homogène pour la multiplication à gauche (ou ˆ droite) de
$$
N_{G(F)}(\lambda)=\{g\in G(F): {\rm Int}_g(T)=T,\, \lambda^g = \lambda\}.
$$
Notons qu'il n'est en général pas possible de prolonger $\lambda$ en un caractère $\tilde{\lambda}$ de $\wt{T}(F)$, où (rappel) $\wt{T}$ est le normalisateur de $(B,T)$ dans $\wt{G}$.
}
\end{marema2}

On note:
\begin{itemize}
\item $\bs{D}(\wt{G}(F),\omega)$ l'espace des distributions sur $\wt{G}(F)$ qui sont des combinaisons linéaires (finies, à coefficients complexes) de traces $\Theta_{\tilde{\pi}}$ où $(\pi,\tilde{\pi})$ est une représentation de $(\wt{G}(F),\omega)$ telle que $\pi$ est une représentation irréductible tempérée de $G(F)$;
\item $\bs{D}^{\rm nr}(\wt{G}(F),\omega)$ l'espace des distributions sur $\wt{G}(F)$ qui sont des combinaisons linéaires de traces $\Theta_{\tilde{\pi}}$ où $(\pi,\tilde{\pi})$ est une représentation de $(\wt{G}(F),\omega)$ telle que $\pi$ est une sous--représentation irréductible d'une série principale $i_B^G(\lambda)$ pour un caractère unitaire non ramifié $\lambda$ de $T(F)$;
\item $\bs{D}^K(\wt{G}(F),\omega)$ l'espace des distributions sur $\wt{G}(F)$ qui sont des combinaisons linéaires de traces $\Theta_{\tilde{\pi}}$ où $(\pi,\tilde{\pi})$ est une représentation de $(\wt{G}(F),\omega)$ telle que $\pi$ est une repré\-sentation irréductible tempérée $K$--sphérique de $G(F)$.
\end{itemize}
On a donc les inclusions
$$
\bs{D}^K(\wt{G}(F),\omega)\subset \bs{D}^{\rm nr}(\wt{G}(F),\omega)\subset \bs{D}(\wt{G}(F),\omega).
$$
On note aussi:
\begin{itemize}
\item $\bs{D}^{\rm ram}(\wt{G}(F),\omega)$ l'espace des distributions sur $\wt{G}(F)$ qui sont des combinaisons linéaires de traces $\Theta_{\tilde{\pi}}$ où $(\pi,\tilde{\pi})$ est une représentation de $(\wt{G}(F),\omega)$ telle que $\pi$ est une représentation irréductible tempérée de $G(F)$ qui n'est isomorphe à aucune sous--représentation irréductible d'une série principale $i_B^G(\lambda)$ pour un caractère unitaire non ramifié $\lambda$ de $T(F)$.
\end{itemize}
On a la décomposition
$$
\bs{D}(\wt{G}(F),\omega)= \bs{D}^{\rm nr}(\wt{G}(F),\omega)\oplus \bs{D}^{\rm ram}(\wt{G}(F),\omega).\leqno{(1)}
$$
Notons que $\bs{D}(\wt{G}(F),\omega)$ est aussi l'espace des distribution sur $\wt{G}(F)$ qui sont des combinaisons linéaires de traces $\Theta_{\tilde{\pi}}$ où $\tilde{\pi}$ est une représentation $G(F)$--irréductible tempérée de $(\wt{G}(F),\omega)$. 
Pour $*= {\rm ram},\,{\rm nr},\, K$, On note
$$p^*: \bs{D}(\wt{G}(F),\omega)\rightarrow  \bs{D}^*(\wt{G}(F),\omega)
$$
la projection naturelle (application linéaire) qui envoie $\Theta_{\tilde{\pi}}$ sur $\Theta_{\tilde{\pi}}$ si $\tilde{\pi}$ est une représentation $G(F)$--irréductible tempérée de $(\wt{G}(F),\omega)$ telle que 
$\Theta_{\tilde{\pi}}\in \bs{D}^*(\wt{G}(F),\omega)$, et sur $0$ sinon.

\subsection{Paramètres de Langlands}\label{paramètres}Soit $\lambda$ un caractère non ramifié de $T(F)$, \cad un homomorphisme de $T(F)/T(F)_1$ dans ${\Bbb C}^\times$, où
$$
T(F)_1=T(F)\cap T(F^{\rm nr})_1
$$
est le sous--groupe compact maximal de $T(F)$. 
Pour $t\in T(F^{\rm nr})$, notons $\nu(t)$ l'élément de ${\rm Hom}(X(T),{\Bbb Z})$ donné par
$$
\nu(t)(\chi)= v_F(\chi(t)).
$$
L'application $t\mapsto \nu(t)$ ainsi définie induit un isomorphisme de $T(F)/T(F)_1$ sur
$$
{\rm Hom}(X(T),{\Bbb Z})^\phi \simeq \check{X}(T)^\phi \simeq X(\hat{T})^\phi.
$$
On en déduit isomorphisme
$$
{\rm Hom}(T(F)/T(F)_1,{\Bbb C}^\times)\buildrel \simeq\over{\longrightarrow} {\rm Hom}(X(\hat{T})^\phi,{\Bbb C}^\times)= \hat{T}/(1-\phi)(\hat{T}),\leqno{(1)}
$$
qui n'est autre que la restriction de 
la correspondance de Langlands locale
$${\rm Hom}(T(F),{\Bbb C}^\times)\simeq {\rm H}^1(W_F,\hat{T})
$$ pour le tore $T$ (ici ${\rm Hom}$ désigne le groupe des homomorphismes continus, et ${\rm H}^1$ le groupe des classes de $1$--cocycles continus). \`A un caractère non ramifié $\lambda$ de $T(F)$ est associé via l'isomorphisme (1) un $L$--homomorphisme (continu) $\varphi_\lambda^T:W_F \rightarrow {^LT}$, bien défini à conjugaison près par $\hat{T}$, et un $L$--homomorphisme
$$
\varphi_\lambda: W_F \xrightarrow{\varphi^T} {^LT}\hookrightarrow {^LG}.
$$
Par construction, le paramètre $\varphi_\lambda$ est non ramifié (\cad que son image contient $I_F$), et il est tempéré (\cad que son image est bornée) si et seulement si le caractère $\lambda$ est unitaire. D'autre part, pour un caractère unitaire non ramifié $\lambda$ de $T(F)$, la série principale $i_B^G(\lambda)$ a une unique sous--représentation irréductible $K$--sphérique. On en déduit une bijection, disons $\varphi\mapsto \pi_\varphi$, entre:
\begin{itemize}
\item les classes de $\hat{G}$--conjugaison de paramètres de Langlands $\varphi: W_F\rightarrow {^LG}$ qui sont tempérés et non ramifiés;
\item les classes d'isomorphisme de représentations irréductibles de $G(F)$ qui sont tempérées et $K$--sphériques.
\end{itemize}
Via l'isomorphisme de Satake ${\mathcal H}_K\rightarrow \hat{\mathcal H}_\phi,\, f\mapsto \hat{f}$, cette bijection est donnée par
$$
\Theta_{\pi_\varphi}(f)= \hat{f}(\varphi(\phi)),\quad f\in {\mathcal H}_K;\leqno{(2)}
$$
où $\Theta_{\pi_\varphi}$ est la distribution ${\rm tr}\,\pi_\varphi$ sur $G(F)$ définie par la mesure de Haar $dg$ (normalisée par ${\rm vol}(K,dg)=1$).

\vskip1mm
Soit $\pi$ une représentation irréductible de $G(F)$, tempérée et $K$--sphérique, de paramètre 
$\varphi:W_F\rightarrow {^LG}$. Supposons que $\pi$ se prolonge en une représentation $(\pi,\tilde{\pi})$ de $(\wt{G}(F),\omega)$. \'Ecrivons $\wt{K}=K\delta =\delta K$ pour un $\delta\in N_{\smash{\wt{G}(F)}}(K)$ --- par exemple $\delta=\delta_\circ$. L'opérateur $\tilde{\pi}(\delta)$ entrelace $\omega\pi$ et $\pi^\delta$. Les représentations $\omega\pi$ et $\pi^\delta$ de $\wt{G}(F)$ sont encore irréductibles, tempérées et $K$--sphériques. Or $\omega\pi$ a pour paramètre $a\cdot \varphi$ où $a:W_F\rightarrow Z(\hat{G})$ est un cocycle dans la classe de cohomologie $\bs{a}$, et $\pi^\delta$ a pour paramètre ${^L\theta}\circ\varphi$. Posant
$$
\wt{S}_{\varphi,a}= \{\tilde{g}\in \hat{G}\hat{\theta}: {\rm Int}_{\tilde{g}}\circ \varphi = a\cdot \varphi\},
$$
on obtient une bijection entre:
\begin{itemize}
\item les classes de $\hat{G}$--conjugaison de paramètres de Langlands $\varphi: W_F\rightarrow {^LG}$ qui sont tempérés, non ramifiés, et tels que $\wt{S}_{\varphi,a}$ n'est pas vide;
\item les classes d'isomorphisme de représentations irréductibles de $G(F)$ qui sont tempérées, $K$--sphériques, et se prolongent en une représentation de $(\wt{G}(F),\omega)$.
\end{itemize}

\begin{marema1}
{\rm 
Pour une représentation irréductible tempérée $K$--sphérique $(\pi,V)$ de $G(F)$, on sait que l'espace $V^K$ est de dimension $1$. Si de plus $\pi$ se prolonge en une représentation $\tilde{\pi}$ de $(\wt{G}(F),\omega)$, on peut normaliser $\tilde{\pi}$ en imposant que la restriction de 
$\tilde{\pi}(\delta)$ à $V^K$ soit l'identité, \cad que $\Theta_{\tilde{\pi}}({\bf 1}_{\wt{K}})=1$. 
Alors pour toute classe de $\hat{G}$--conjugaison de paramètres tempérés non ramifiés 
$\varphi:W_F\rightarrow {^LG}$ tels que $\wt{S}_{\varphi,a}$ n'est pas vide, notant $\tilde{\pi}_\varphi$ le prolongement normalisé de $\pi_\varphi$ à 
$(\wt{G}(F),\omega)$, on a l'égalité
$$
\Theta_{\tilde{\pi}_\varphi}(f*{\bf 1}_{\wt{K}})= \hat{f}(\varphi(\phi)),\quad f\in {\mathcal H}_K.
$$
}
\end{marema1}

\begin{marema2}
{\rm 
Supposons que $\wt{G}$ est à torsion intérieure, \cad tel que pour tout (i.e. pour un) $\delta\in \wt{G}$, l'automorphisme ${\rm Int}_\delta$ de $G$ est intérieur. Supposons aussi que $\omega=1$. L'ensemble $Z(\wt{G},\ES{E})$ attaché à la paire de Borel épinglée $\ES{E}$ de $G$ est en fait indépendant de $\ES{E}$ (c'est l'ensemble des $\delta\in \wt{G}$ tels que ${\rm Int}_\delta$ est  l'identité). En particulier $\theta_{\ES{E}}=1$ et $\hat{\theta}=1$. \'Ecrivons $\wt{K}= K\delta_\circ =\delta_\circ K$, $\delta_\circ \in \wt{T}(F)$ --- cf. \ref{données endoscopiques nr}. Puisque la restriction de ${\rm Int}_{\delta_\circ}$ à $T$ co\"{\i}ncide avec celle de $\theta_{\ES{E}}=1$, c'est l'identité. Ainsi tout caractère unitaire $\lambda$ de $T(F)$ se prolonge en un caractère $\wt{\lambda}$ de $\wt{T}(F)$, et toute représentation irréductible tempérée $K$--sphérique de $G(F)$ se prolonge à $\wt{G}(F)$. Cela correspond au fait que pour un paramètre tempéré non ramifié $\varphi:W_F\rightarrow {^LG}$, l'ensemble
$$
S_\varphi=\{g\in \hat{G}: {\rm Int}_g\circ \varphi =\varphi\}
$$
n'est pas vide (d'ailleurs c'est un groupe!). 
}
\end{marema2}

\subsection{Transfert spectral}\label{transfert spectral}
Soit $\bs{G}'=(G',\ES{G}',\tilde{s})$ une donnée endoscopique elliptique non ramifiée pour 
$(\wt{G},\bs{a})$, et soit $(K',\wt{K}')$ un sous--espace hyperspécial de $\wt{G}'(F)$ associé à 
$(K,\wt{K})$. On définit comme en \ref{paramètres} (pour $\omega=1$) les espaces
$$
\bs{D}^{K'}(\wt{G}'(F))\subset \bs{D}^{\rm nr}(\wt{G}'(F))\subset \bs{D}(\wt{G}'(F))
$$
et pour $*={\rm ram},\, {\rm nr},\, K'$, on note
$$
p'^*: \bs{D}(\wt{G}'(F))\rightarrow \bs{D}^*(\wt{G}'(F))
$$
la projection naturelle. 
On fixe comme en \ref{données endoscopiques nr} un isomorphisme $\ES{G}'\simeq {^LG'}$. Soit $\varphi':W_F\rightarrow {^LG'}$ un $L$--homomorphisme tempéré et non ramifié. On lui associe un $L$--homomorphisme
$$
\varphi: W_F\buildrel \varphi'\over{\longrightarrow} {^LG'}
\simeq\ES{G}'\hookrightarrow {^LG}.
$$
Soit $\pi'$ une représentation irréductible tempérée $K'$--sphérique de $G'(F)$ associée à $\varphi'$, que l'on prolonge en une représentation $(\pi',\tilde{\pi}')$ de $\wt{G}'(F)$. C'est possible d'après la remarque 2 de \ref{paramètres}, puisque $\wt{G}'$ est à torsion intérieure. Rappelons que pour $(g',w)\in \ES{G}'$, on a
$${\rm Int}_{\tilde{s}}(g',w)= (g'a(w),w).
$$
Par conséquent l'ensemble $\wt{S}_{\varphi,a}$ n'est pas vide (il contient $\tilde{s}$). Soit $\pi$ une représentation irréductible tempérée $K$--sphérique associée à $\varphi$, que l'on prolonge en une représentation $(\pi,\tilde{\pi})$ de 
$(\wt{G}(F),\omega)$. On normalise ces prolongements de telle manière que
$$
\Theta_{\wt{\pi}}({\bf 1}_{\wt{K}})=\Theta_{\wt{\pi}'}({\bf 1}_{\wt{K}'}).
$$
D'après les remarques 1 et 2 de \ref{paramètres}, on a l'égalité
$$
\Theta_{\wt{\pi}}(f*{\bf 1}_{\wt{K}})= \Theta_{\wt{\pi}'}(b(f)*{\bf 1}_{{\wt K}'}),\quad f\in {\mathcal H}_K.
\leqno{(1)}
$$
L'égalité (1) définit par linéarité une application de transfert spectral (sphérique)
$$
\bs{\rm t}=\bs{\rm t}_{\bs{G}'}^{K,K'}: \bs{D}^{K'}(\wt{G}'(F))\rightarrow \bs{D}^{K}(\wt{G}(F),\omega).
$$

\vskip1mm
D'autre part, notons $\bs{SD}(\wt{G}'(F))$ le sous--espace de $\bs{D}(\wt{G}(F))$ formé des distributions qui sont stables. D'après \cite{Moe}, qui généralise au cas tordu le résultat d'Arthur \cite{A}, l'application de transfert géométrique
$$
\bs{I}(\wt{G}(F),\omega)\rightarrow \bs{SI}(\wt{G}'(F))
$$
définit dualement une application de transfert spectral (usuel)
$$
\bs{\rm T}=\bs{\rm T}_{\bs{G}'}: \bs{SD}(\wt{G}'(F))\rightarrow \bs{D}(\wt{G}(F),\omega).
$$
Pour toute paire $(f,f')\in C^\infty_{\rm c}(\wt{G}(F))\times C^\infty_{\rm c}(\wt{G}'(F))$ tel que $f'$ soit un transfert de $f$ (\ref{transfert nr}), on a l'égalité
$$
\bs{\rm T}(\Theta')(f) = \Theta'(f'),\quad \Theta'\in \bs{SD}(\wt{G}'(F)).\leqno{(2)}
$$

\begin{marema}
{\rm 
Ces deux homomorphismes de transfert spectral
$$\bs{\rm t}: \bs{D}^{K'}(\wt{G}'(F))\rightarrow \bs{D}^K(\wt{G}(F),\omega),\quad 
\bs{\rm T}: \bs{SD}(\wt{G}'(F))\rightarrow \bs{D}(\wt{G}(F),\omega),
$$
dépendent du choix du sous--espace hyperspécial $(K,\wt{K})$ de $\wt{G}(F)$, qui détermine 
la classe de conjugaison dans $G'_{\rm AD}(F)$ du sous--espace hyperspécial $(K',\wt{K}')$ de 
$\wt{G}'(F)$, et le facteur de transfert normalisé $\Delta: \ES{D}(\bs{G}')\rightarrow {\Bbb C}^\times$. 
}
\end{marema}

Le théorème 2 de \ref{lemme fondamental} pour la donnée $\bs{G}'=(G',\ES{G}',\tilde{s})$ équivaut au théorème suivant.

\begin{montheo}
Le diagramme suivant
$$
\xymatrix{
\bs{SD}(\wt{G}'(F)) \ar@{^{(}->}[r] \ar[d]_{\bs{\rm T}} & \bs{D}(\wt{G}'(F)) \ar[r]^{p'^{K'}} & \bs{D}^{K'}(\wt{G}'(F))\ar[d]^{\bs{\rm t}}\\
\bs{D}(\wt{G}(F),\omega) \ar[rr]^{p^K} && \bs{D}^K(\wt{G}(F),\omega)
}
$$
est commutatif.
\end{montheo}

\subsection{Transfert spectral elliptique}\label{transfert spectral elliptique}Reprenons les notations et les définitions de \ref{descente parabolique}. 
Soit $\wt{M}\in \ES{L}(\wt{T},\omega)$. Choisissons une paire parabolique $(P,M)$ de $G$, semi--standard et définie sur $F$, telle $\wt{M}=\wt{M}_P$. Pour une représentation 
$(\sigma,\tilde{\sigma})$ de $(\wt{M}(F),\omega)$, notons $i_{\wt{P}}^{\wt{G}}(\sigma,\wt{\sigma})$, ou simplement $i_{\wt{P}}^{\wt{G}}(\tilde{\sigma})$, l'induite parabolique normalisée de $(\sigma,\tilde{\sigma})$ à $\wt{G}(F)$ suivant $\wt{P}(F)$. C'est une représentation de $(\wt{G}(F),\omega)$, et la représentation de $G(F)$ sous--jacente est l'induite parabolique normalisée $i_P^G(\sigma)$ de $\sigma$ à $G(F)$ suivant $P(F)$. Soit $(Q,M)$ une autre paire parabolique semi--standard de $G$,  définie sur $F$ et telle que $\wt{M}_Q=\wt{M}$. On définit de la même maniere l'induite parabolique normalisée $i_{\wt{Q}}^{\wt{G}}(\sigma,\tilde{\sigma})$ de $(\sigma,\tilde{\sigma})$ à $\wt{G}(F)$ suivant $\wt{Q}(F)$. Si de plus la représentation $\sigma$ de $G(F)$ est admissible, alors on a l'égalité des traces
$$
\Theta_{\smash{i_{\wt{P}}^{\wt{G}}(\tilde{\sigma})}}=\Theta_{\smash{i_{\wt{Q}}^{\wt{G}}(\tilde{\sigma})}}.
$$
D'où une application linéaire
$$
\bs{i}_{\wt{M}}^{\wt{G}}: \bs{D}(\wt{M}(F),\omega)\rightarrow \bs{D}(\wt{G}(F),\omega),
$$
donnée par
$$
\bs{i}_{\wt{M}}^{\wt{G}}(\Theta_{\tilde{\sigma}})= \Theta_{\smash{i_{\wt{P}}^{\wt{G}}(\tilde{\sigma}})}
$$
pour toute représentation $(\sigma,\tilde{\sigma})$ de $(\wt{M}(F),\omega)$ telle que $\sigma$ est une représentation irréductible tempérée de $M(F)$, et toute paire parabolique semi--standard $(P,M)$ de $G$, définie sur $F$ et telle que $\wt{M}_P=\wt{M}$. 

Dans \cite[2.12]{W} est défini un sous--espace $\bs{D}_{\rm ell}(\wt{G},\omega)$ de $\bs{D}(\wt{G}(F),\omega)$, identifié dans \cite[1.7]{Moe} à l'espace des distributions sur $\wt{G}(F)$ qui sont des combinaisons linéaires de traces $\Theta_{\tilde{\pi}}$ où $\tilde{\pi}$ est une représentation 
$G(F)$--irréductible {\it supertempérée} de $(\wt{G}(F),\omega)$. Pour $\wt{M}\in \ES{L}(\wt{T},\omega)$, on a aussi un sous--espace $\bs{D}_{\rm ell}(\wt{M}(F),\omega)$ de $\bs{D}(\wt{M}(F),\omega)$. Le groupe
$$
N_{G(F)}(\wt{M})=\{g\in G(F): {\rm Int}_g(\wt{M})=\wt{M}\}
$$
opère naturellement sur $\bs{D}(\wt{M}(F),\omega)$: pour $n\in N_{G(F)}(\wt{M})$ et $\Theta\in \bs{D}(\wt{M}(F),\omega)$, on note ${^n\Theta}$ la distribution sur $\wt{M}(F)$ donnée par
$$
{^n\Theta}(f)= \omega^{-1}(n)\Theta(f^n),\quad f^n= f\circ {\rm Int}_n,\quad f\in C^\infty_{\rm c}(\wt{M}(F),\omega).
$$
Cette action se factorise en une action de
$$
W^G(\wt{M})= N_{G(F)}(\wt{M})/M(F)
$$
sur $\bs{D}(\wt{M}(F),\omega)$, qui stabilise $\bs{D}_{\rm ell}(\wt{M}(F),\omega)$. On note 
$\bs{D}_{\rm ell}(\wt{M}(F),\omega)^{W^G(\wt{M})}$ le sous--espace de $\bs{D}_{\rm ell}(\wt{M}(F),\omega)$ formé des invariants sous $W^G(\wt{M})$. Posons $\wt{W}=W^G(\wt{T})$. Ce groupe opère naturellement, \cad par conjugaison, sur l'ensemble $\ES{L}(\wt{T},\omega)$. D'après la proposition de \cite[2.12]{W}, on a la décomposition en somme directe
$$
\bs{D}(\wt{G}(F),\omega)= \oplus_{\wt{M}\in \ES{L}(\wt{T},\omega)/\wt{W}}\;
\bs{i}_{\wt{M}}^{\wt{G}}(\bs{D}_{\rm ell}(\wt{M}(F),\omega)^{W^G(\wt{M})}),\leqno{(1)}
$$
où $ \ES{L}(\wt{T},\omega)/\wt{W}$ désigne un système de représentants des 
$\wt{W}$--orbites dans $\ES{L}(\wt{T},\omega)$.

\begin{marema1} 
{\rm Supposons que $\wt{G}$ est à torsion intérieure, et que $\omega=1$. Alors à tout $F$--Levi 
$M$ de $G$ est associé un $F$--Levi $\wt{M}$ de $\wt{G}$: on choisit un sous--groupe parabolique $P$ de $G$, défini sur $F$ et de composante de Levi $M$, et l'on pose 
$\wt{M}=N_{\smash{\wt{G}}}(P,M)$. Cet espace $\wt{M}$ dépend seulement de $M$, et pas 
du choix de $P$. 
Ainsi l'ensemble $\ES{L}(\wt{T})$ des $F$--Levi semi--standards 
de $\wt{G}$ s'identifie à l'ensemble $\ES{L}(T)$ des $F$--Levi semi--standards de $G$, et $\wt{W}$  est le 
groupe de Weyl $W$ de $G$. Pour $M\in \ES{L}(T)$, on pose
$$
\bs{SD}_{\rm ell}(\wt{M}(F))= \bs{SD}(\wt{M}(F))\cap \bs{D}_{\rm ell}(\wt{M}(F)).
$$
Le groupe $W^G(\wt{M})$ stabilise $\bs{SD}_{\rm ell}(\wt{M}(F))$. On note 
$\bs{SD}_{\rm ell}(\wt{M}(F))^{W^G(\wt{M})}$ le sous--espace de $\bs{SD}_{\rm ell}(\wt{M}(F))$ formé des invariants sous $W^G(\wt{M})$. Alors d'après le corollaire de \cite[2.1]{Moe}, la décomposition (1) entra\^{\i}ne la décomposition
$$
\bs{SD}(\wt{G}(F))= \oplus_{M\in \ES{L}(T)/W} \;
\bs{i}_{\wt{M}}^{\wt{G}}(\bs{SD}_{\rm ell}(\wt{M}(F))^{W^G(\wt{M})}).\leqno{(2)}
$$
}
\end{marema1}

Soit $\bs{G}'=(G',\ES{G}',\tilde{s})$ une donnée endoscopique elliptique non ramifiée pour $(\wt{G},\bs{a})$. On pose $\hat{G}'=\hat{G}_{\tilde{s}}$. 
On fixe des paires de Borel épinglées compatibles 
$\hat{\ES{E}}=(\hat{B},\hat{T},\{\hat{E}_\alpha\}_{\alpha\in \hat{\Delta}})$ de $\hat{G}$ 
et $\hat{\ES{E}}'=
(\hat{B}',\hat{T}',\{\hat{E}'_\alpha\}_{\alpha\in \hat{\Delta}'})$ de $\hat{G}'$, et l'on normalise les actions galoisiennes $\sigma\mapsto \sigma_G$ sur $\hat{G}$ et $\sigma\mapsto \sigma_{G'}$ sur $\hat{G}'$ de manière à ce qu'elles préservent ces paires --- cf. \ref{données endoscopiques}. On a aussi un automorphisme $\hat{\theta}$ de $\hat{G}$ qui préserve $\hat{\ES{E}}$ et commute à l'action galoisienne $\sigma \mapsto \sigma_G$. Rappelons que 
l'élément $\tilde{s}$ appartient à $\hat{T}\hat{\theta}$. Soit $(B',T')$ une paire de Borel de $G'$ définie sur $F$, et soit $\wt{T}'$ le normalisateur de $(B',T')$ dans $\wt{G}'$. Puisque $\wt{G}'$ est à torsion intérieure, l'espace $\wt{T}'$ ne dépend pas du sous--groupe de Borel $B'$ de $G'$, défini sur $F$ et de composante de Levi $T'$. En fait on a
$$\wt{T}'= T'\times_{\ES{Z}(G)}\ES{Z}(\wt{G},\ES{E}))\;(=T'\times_{Z(G')}\ES{Z}(\wt{G}',\ES{E})).$$
D'après la remarque 1, on a décomposition en somme directe
$$
\bs{SD}(\wt{G}'(F))= \oplus_{M'\in \ES{L}(T')/W'} 
\bs{i}_{\wt{M}'}^{\wt{G}'}(\bs{SD}_{\rm ell}(\wt{M}'(F))^{W^{G'}(\wt{M}')}),\leqno{(2)'}
$$
où $\wt{M}'$ est le $F$--Levi semi--standard de $\wt{G}'$
associé à $M'$ (il est donné par $\wt{M}'= N_{\smash{\wt{G}'}}(P'\!,M')$ pour un sous--groupe parabolique 
$P'$ de $G'$, défini sur $F$ et de composante de Levi $M'$). 
D'après le théorème de \cite[3]{Moe}, le transfert spectral (usuel) $\bs{\rm T}_{\bs{G}'}(\Theta')$ 
d'une distribution $\Theta' \in \bs{SD}_{\rm ell}(\wt{G}'(F))$ appartient à $\bs{D}_{\rm ell}(\wt{G}(F),\omega)$. D'où une application de transfert spectral elliptique
$$
\bs{\rm T}_{\bs{G}'\!,{\rm ell}}: \bs{SD}_{\rm ell}(\wt{G}'(F))\rightarrow \bs{D}_{\rm ell}(\wt{G}(F),\omega).
$$
D'ailleurs dans \cite{A} puis \cite{Moe}, c'est cette application de transfert spectral elliptique qui est établie en premier, et ensuite étendue à $\bs{SD}(\wt{G}'(F))$ grâce aux décompositions (1) et $(2)'$. 

Soit $M'\in \ES{L}(T')$. Quitte à remplacer $M'$ par un conjugué par un élément de $W'$, on peut supposer que c'est un $F$--Levi standard de $G'$, \cad qu'il existe un sous--groupe parabolique $P'$ de $G'$, défini sur $F$ et de composante de Levi $M'$, qui contient $B'$ (on a donc $P'=M'U_{B'}$). 
Soit $\wt{M}'= N_{\smash{\wt{G}'}}(P'\!,M')$ le $F$--Levi de $\wt{G}'$ associé à $M'$. 
Soit $(\hat{P}'\!,\hat{M}')$ la paire parabolique standard --- \cad contenant $(\hat{B}',\hat{T}')$ --- de 
$\hat{G}'$ associée à $(P'\!,M')$. Elle est définie sur $F$. Rempla\c{c}ons $\hat{G}'$ par $\hat{M}'$ dans les constructions de \ref{réduction aux données elliptiques}. Notons $\hat{M}$, $\ES{M}$, $\wt{\ES{M}}$ les commutants de $Z(\hat{M}')^{\Gamma_F,\circ}$ dans $\hat{G}$, ${^LG}$, ${^L{\wt{G}}}$. Le groupe $\hat{M}$ est un sous--groupe de Levi semi--standard de $\hat{G}$, dŽfini sur $F$ et $\hat{\theta}$--stable. Fixons un cocaractre $x\in \check{X}(Z(\hat{M}')^{\Gamma_F,\circ})$ en position gŽnŽrale. Il dŽtermine un sous--groupe parabolique $\hat{P}$ de $\hat{G}$, engendrŽ par $\hat{M}$ et les sous--groupes radiciels de $\hat{G}$ associŽs aux racines $\alpha$ de $\hat{T}$ telles que $\langle \alpha, x\rangle >0$. Posons $\ES{P}= \hat{P}\ES{M}$ et $\wt{\ES{P}}= \hat{P}\wt{\ES{M}}$. D'aprs \cite[3.4]{Stab I}, le couple $(\wt{\ES{P}},\wt{\ES{M}})$ est une paire parabolique de ${^L\wt{G}}$: on a
$$
\ES{P}= \hat{P}\rtimes W_F,\quad \ES{M}= \hat{M}\rtimes W_F,
$$
$\wt{\ES{P}}$ est le normalisateur de $\ES{P}$ dans ${^L{\wt{G}}}$, $\wt{\ES{M}}$ est le normalisateur de $(\ES{P},\ES{M})$ dans ${^L{\wt{G}}}$, et 
$\wt{\ES{P}}$ n'est pas vide (il contient $\tilde{s}$). D'aprs \cite[3.1, (4)]{Stab I}, la paire $(\hat{P},\hat{M})$ est conjuguŽe dans $\hat{G}$ ˆ une paire parabolique standard de $\hat{G}$, dŽfinie sur $F$ et $\hat{\theta}$--stable. Quitte ˆ effectuer une telle conjugaison, on peut supposer que la paire $(\hat{P},\hat{M})$ est elle--mme standard, dŽfinie sur $F$ et $\hat{\theta}$--stable. Alors on a
$$
\wt{\ES{P}}=(\hat{P}\rtimes W_F){^L\theta},\quad \wt{\ES{M}}=(\hat{M}\rtimes W_F){^L\theta}.
$$

À $(\hat{P},\hat{M})$ correspond une paire parabolique standard $(P,M)$ de $G$, qui est définie sur $F$ et $\theta_{\ES{E}}$--stable (cf. la remarque de \ref{descente parabolique}). Soit $(\wt{P},\wt{M})$ la paire parabolique standard de $\wt{G}$ associée à $(P,M)$, \cad que $\wt{P}=N_{\smash{\wt{G}}}(P)$ et $\wt{M}=N_{\smash{\wt{G}}}(P,M)$. Elle 
est définie sur $F$, et l'on a
$$
\wt{P}=P\delta_\circ,\quad \wt{M}= M\delta_\circ.
$$
L'espace tordu $(M,\wt{M})$ vérifie les trois hypothèses de \ref{objets}.
Le groupe $\ES{M}$ s'identifie au $L$--groupe ${^LM}$, et $\wt{\ES{M}}$ s'identifie au $L$--espace tordu ${^L{\wt{M}}}= {^LM}{^L\theta}$. Notons $\boldsymbol{a}_M$ l'image de $\boldsymbol{a}$ par l'homomorphisme naturel
$$
{\rm H}^1(W_F,Z(\hat{G}))\rightarrow {\rm H}^1(W_F,Z(\hat{M})),
$$
\cad la classe de cohomologie (non ramifiée) correspondant au caractère $\omega$ 
restreint à $M(F)$. Posons $\ES{M}'= \ES{G}'\cap \ES{M}$. Alors $\bs{M}'=(M',\ES{M}',\tilde{s})$ est une donnée endoscopique non ramifiée pour $(\wt{M},\bs{a}_M)$. Puisque
$$
Z(\hat{M}')^{\Gamma_F,\circ}=[ Z(\hat{M})^{\hat{\theta}}]^{\Gamma_F,\circ},
$$
cette donnée est elliptique. D'après \ref{descente parabolique}, 
la paire $(K_M,\wt{K}_M)$ définie par
$$
K_M = K\cap M(F),\quad \wt{K}_M= \wt{K}\cap \wt{M}(F)\;(=K_M\delta_\circ =\delta_\circ K_M),
$$
est un sous--espace hyperspécial de $\wt{M}(F)$. Soit $\ES{D}^{\wt{M}}(\bs{M}')$ 
le sous--ensemble de $\wt{M}'(F)\times \wt{M}(F)$ obtenu en rempla\c{c}ant $(\wt{G},\bs{G}')$ par 
$(\wt{M},\bs{M}')$ dans la définition de $\ES{D}(\bs{G}')$. D'après le lemme de \cite[6.2]{Stab I}, l'ensemble 
$\ES{D}^{\wt{M}}(\bs{M}')$ n'est pas vide. 
D'après \cite[6.2, 6.3]{Stab I}, à $(K_M,\wt{K}_M)$ 
sont associés un sous--espace hyperspécial $(K'_M,\wt{K}'_M)$ de $\wt{M}'(F)$, bien défini à conjugaison près par $M'_{\rm AD}(F)$, et un facteur de transfert normalisé
$$
\Delta^{\wt{M}}: \ES{D}^{\wt{M}}(\bs{M}')\rightarrow {\Bbb C}^\times.
$$
En rempla\c{c}ant $\wt{G}$ par $\wt{M}$ dans la relation (2) de \ref{données endoscopiques nr}, on voit que le caractère $\omega$ est trivial sur 
$Z(M;F)^{\theta_M}$. La paire $(\wt{M},\bs{a}_M)$ vérifie donc les trois hypothèses de \ref{données endoscopiques nr}. En particulier le $F$--Levi $\wt{M}$ de $\wt{G}$ appartient à $\ES{L}(\wt{T},\omega)$. 

En rempla\c{c}ant $(\wt{G},\bs{G}')$ par $(\wt{M},\bs{M}')$, $(K,\wt{K})$ par $(K_M,\wt{K}_M)$ et $(K',\wt{K}')$ par $(K'_M,\wt{K}'_M)$, dans les définitions de $\bs{\rm t}$, $\bs{\rm T}$, $p^K$, $p'^{K'}$, on définit les homomorphismes 
de transfert sphérique et usuel
$$
\bs{\rm t}^{\wt{M}} = \bs{\rm t}_{\bs{M}'}^{\wt{M};K_M,K'_M}:\bs{D}^{K'_M}(\wt{M}'(F))\rightarrow 
\bs{D}^{K_M}(\wt{M}(F),\omega),
$$
$$
\bs{\rm T}^{\wt{M}}=\bs{\rm T}^{\wt{M}}_{\bs{M}'}:\bs{SD}(\wt{M}'(F))\rightarrow 
\bs{D}(\wt{M}(F),\omega),
$$
et les projections naturelles
$$
p^{K_M}:\bs{D}(\wt{M}(F),\omega) \rightarrow \bs{D}^{K_M}(\wt{M}(F),\omega),
$$
$$
p'^{K'_M}:  \bs{D}(\wt{M}'(F))\rightarrow \bs{D}^{K'_M}(\wt{M}'(F)).
$$

\begin{monlem}
On a:
\begin{enumerate}
\item[(i)]$\bs{\rm t}\circ p'^{K'} \circ \bs{i}_{\wt{M}'}^{\wt{G}'}=p^K \circ \bs{i}_{\wt{M}}^{\wt{G}}\circ 
\bs{\rm t}^{\wt{M}}\circ p'^{K'_M}:\bs{D}(\wt{M}'(F))\rightarrow \bs{D}^K(\wt{G}(F),\omega)$.
\item[(ii)]$\bs{\rm T}\circ \bs{i}_{\wt{M}'}^{\wt{G}'}= \bs{i}_{\wt{M}}^{\wt{G}}\circ \bs{\rm T}^{\wt{M}}: 
\bs{SD}(\wt{M}'(F))\rightarrow \bs{D}(\wt{G}(F),\omega)$.
\item[(iii)]$p^K\circ \bs{i}_{\wt{M}}^{\wt{G}}= 
p^K\circ \bs{i}_{\wt{M}}^{\wt{G}}\circ p^{K_M}:\bs{D}(\wt{M}(F))\rightarrow \bs{D}^K(\wt{G}(F),\omega)$
\end{enumerate}
\end{monlem}

\begin{proof}Prouvons (i). L'homomorphisme de transfert spectral sphérique $\bs{\rm t}^{\wt{M}}$ est donné par
$$
\bs{\rm t}^{\wt{M}}(\Theta')(h* \bs{1}_{\wt{K}_M})=\Theta'(b_M(h)*\bs{1}_{\wt{K}'_M}),\quad \Theta'\in \bs{D}^{K'_M}(\wt{M}'(F)),\, h\in \mathcal{H}_{K_M},
$$
où $b_M: \mathcal{H}_{K_M}\rightarrow \mathcal{H}_{K'_M}$ est l'homomorphisme défini comme en 
\ref{lemme fondamental}, en rempla\c{c}ant $(\wt{G},\bs{G}')$ par $(\wt{M},\bs{M}')$ et $(K,K')$ par $(K_M,K'_M)$. Pour $\Theta'\in \bs{D}(\wt{M}'(F))$ et $f\in \mathcal{H}_{K}$, on a
\begin{eqnarray*}
p^K\circ \bs{i}_{\wt{M}}^{\wt{G}}\circ\bs{\rm t}^{\wt{M}}\circ p'^{K'_M}(\Theta')(f* \bs{1}_{\wt{K}})
&=&\bs{i}_{\wt{M}}^{\wt{G}}\circ\bs{\rm t}^{\wt{M}}\circ p'^{K'_M}(\Theta')(f* \bs{1}_{\wt{K}})\\
&=&\bs{\rm t}^{\wt{M}}\circ p'^{K'_M}(\Theta')((f* \bs{1}_{K})_{\wt{P},\omega})\\
&=& \bs{\rm t}^{\wt{M}}\circ p'^{K'_M}(\Theta')(f_{P,\omega}*\bs{1}_{\wt{K}_M})\\
&=& p'^{K'_M}(\Theta')(b_M(f_{P,\omega})*\bs{1}_{\wt{K}'_M})\\
&=& \Theta'(b_M(f_{P,\omega})*\bs{1}_{\wt{K}'_M}).
\end{eqnarray*}
On a les diagrammes commutatifs suivants
$$
\xymatrix{
\mathcal{H}_K \ar[r]^{\simeq} \ar[d]_b & \hat{\mathcal{H}}_\phi \ar[d]^{\hat{b}}\\
\mathcal{H}_{K'} \ar[r]^{\simeq} & \hat{\mathcal{H}}'_\phi
},\quad \xymatrix{
\mathcal{H}_{K_M} \ar[r]^{\simeq} \ar[d]_{b_M} & \hat{\mathcal{H}}_\phi^M \ar[d]^{\hat{b}_M}\\
\mathcal{H}_{K'_M} \ar[r]^{\simeq} & \hat{\mathcal{H}}^{M'}_\phi
},
$$
où les flèches horizontales sont les isomorphismes de Satake, et les flèches $\hat{b}$ et $\hat{b}_M$ sont les homomorphismes de restriction. On a aussi les diagrammes commutatifs suivants
$$
\xymatrix{
\mathcal{H}_K \ar[r]^{\simeq} \ar[d]_a & \hat{\mathcal{H}}_\phi \ar[d]^{\hat{a}}\\
\mathcal{H}_{K_M} \ar[r]^{\simeq} & \hat{\mathcal{H}}_\phi^M
},\quad 
\xymatrix{
\mathcal{H}_{K'} \ar[r]^{\simeq} \ar[d]_{a'} & \hat{\mathcal{H}}'_\phi \ar[d]^{\hat{a}'}\\
\mathcal{H}_{K'_M} \ar[r]^{\simeq} & \hat{\mathcal{H}}_\phi^{M'}
},
$$
où les flèches horizontales sont les isomorphismes de Satake, les flèches $a$ et $a'$ sont les homomorphismes terme constant (suivant $(P,\omega)$ et suivant $P'$), et les flèches $\hat{a}$ et $\hat{a}'$ 
sont les homomorphismes de restriction. Puisque $\hat{a}'\circ\hat{b}= \hat{b}_M\circ \hat{a}$, on a
$$
a'\circ b = b_M\circ a.
$$
Par conséquent $b_M(h_{P,\omega})= b(h)_{P'}$, et
\begin{eqnarray*}
\Theta'(b_M(f_{P,\omega})*\bs{1}_{\wt{K}'_M})&=&
 \Theta'(b(f)_{P'}*\bs{1}_{\wt{K}'_M})\\
 &=& \Theta'((b(f)* \bs{1}_{\wt{K}'})_{\wt{P}'})\\
 &=& \bs{i}_{\wt{M}'}^{\wt{G}'}(\Theta')(b(f)* \bs{1}_{\wt{K}'})\\
 &=& p'^{K'}\circ \bs{i}_{\wt{M}'}^{\wt{G}'}(\Theta')(b(h)* \bs{1}_{\wt{K}'})\\
 &=& \bs{\rm t}\circ p'^{K'}\circ \bs{i}_{\wt{M}'}^{\wt{G}'}(\Theta')(h*\bs{1}_{\wt{K}}),
\end{eqnarray*}
ce qui démontre (i). 

Prouvons (ii). L'homomorphisme de transfert spectral usuel $\bs{\rm T}^{\wt{M}}$ est donné par
$$
\bs{\rm T}^{\wt{M}}(\Theta')(h)= \Theta'(h'),\quad \Theta'\in \bs{SD}(\wt{M}'(F)),\, h\in C^\infty_{\rm c}(\wt{M}(F)),
$$
où $h'\in C^\infty_{\rm c}(\wt{M}'(F))$ est un transfert de $h$. Pour $\Theta'\in \bs{SD}(\wt{G}'(F))$ et 
$f\in C^\infty_{\rm c}(\wt{G}(F))$, si $f'\in C^\infty_{\rm c}(\wt{G}'(F))$ est un transfert de $f$, alors le terme constant $f'_{\wt{P}'}\in C^\infty_{\rm c}(\wt{M}'(F))$ de $f'$ (suivant $\wt{P}'$) est un transfert de $f_{\wt{P},\omega}$, et
\begin{eqnarray*}
\bs{i}_{\wt{P}}^{\wt{G}}\circ \bs{\rm T}^{\wt{M}}(\Theta')(f) &=& \bs{\rm T}^{\wt{M}}\circ \Theta'(f_{\wt{P},\omega})\\
&=& \Theta'(f'_{\wt{P}'})\\
&=& \bs{i}_{\wt{P}'}^{\wt{G}'}(\Theta')(f')\\
& = & \bs{\rm T}\circ \bs{i}_{\wt{P}'}^{\wt{G}'}(\Theta')(f).
\end{eqnarray*}

Prouvons (iii). Pour $\Theta\in \bs{D}(\wt{M}(F))$ et $f\in \mathcal{H}_{K}$, puisque 
$(f*\bs{1}_K)_{\wt{P},\omega}= f_{P,\omega}* \bs{1}_{\wt{K}_M}$ est dans 
$\mathcal{H}_{K_M}*\bs{1}_{\wt{K}_M}$, on a
\begin{eqnarray*}
p^K\circ \bs{i}_{\wt{M}}^{\wt{G}}(\Theta)(f* \bs{1}_K)&=&\Theta((f*\bs{1}_K)_{\wt{P},\omega})\\
&=& p^{K_M}\circ \Theta ((f*\bs{1}_K)_{\wt{P},\omega})\\
&=& p^K\circ \bs{i}_{\wt{M}}^{\wt{G}}\circ p^{K_M}(\Theta)(f* \bs{1}_K),
\end{eqnarray*}
ce qui démontre (iii). 
\end{proof}

D'après le lemme, si le diagramme suivant
$$
\xymatrix{
\bs{SD}_{\rm ell}(\wt{M}'(F)) \ar@{^{(}->}[r] \ar[d]_{\bs{\rm T}^{\wt{M}}} & \bs{D}(\wt{M}'(F)) \ar[r]^{p'^{K'_M}} & \bs{D}^{K'_M}(\wt{M}'(F))\ar[d]^{\bs{\rm t}^{\wt{M}}}\\
\bs{D}_{\rm ell}(\wt{M}(F),\omega) \ar[rr]^{p^{K_M}} && \bs{D}^{K_M}(\wt{M}(F),\omega)
}
\leqno{(3)}
$$
est commutatif, alors le diagramme suivant
$$
\xymatrix{
\bs{i}_{\wt{M}'}^{\wt{G}'}(\bs{SD}_{\rm ell}(\wt{M}'(F))) \ar@{^{(}->}[r] \ar[d]_{\bs{\rm T}} & 
\bs{i}_{\wt{M}'}^{\wt{G}'}(\bs{D}(\wt{M}'(F))) \ar[r]^{p'^{K'}} & \bs{D}^{K'}(\wt{G}'(F)))\ar[d]^{\bs{\rm t}}\\
\bs{i}_{\wt{M}}^{\wt{G}}(\bs{D}_{\rm ell}(\wt{M}(F),\omega)) \ar[rr]^{p^{K}} && 
\bs{D}^{K}(\wt{G}(F),\omega))
}
\leqno{(4)}
$$
l'est aussi.

Rappelons que le groupe de Weyl $W'$ de $G'$ s'identifie à 
un sous--groupe de $W^{\theta_{\ES{E}}}$. Comme $W^{\theta_{\ES{E}}}$ est contenu dans 
$\wt{W}= N_{G(F)}(\wt{T})/T(F)$, l'application
$$
\ES{L}(T')\rightarrow \ES{L}(\wt{T},\omega),\, M'\mapsto \wt{M}
$$
induit par passage aux quotients une application
$$
\ES{L}(T')/W'\rightarrow \ES{L}(\wt{T},\omega)/\wt{W},\, M'\mapsto \wt{M}. 
$$
D'après les décompositions (1) et $(2)'$, si pour chaque $F$--Levi $M'\in \ES{L}(T')$, le diagramme (3) est commutatif, alors le diagramme du théorème de \ref{transfert spectral} l'est aussi (i.e. le théorème est vrai). On est donc ramené à prouver que le diagramme suivant
$$
\xymatrix{
\bs{SD}_{\rm ell}(\wt{G}'(F)) \ar@{^{(}->}[r] \ar[d]_{\bs{\rm T}} & \bs{D}(\wt{G}'(F)) \ar[r]^{p'^{K'}} & \bs{D}^{K'}(\wt{G}'(F))\ar[d]^{\bs{\rm t}}\\
\bs{D}_{\rm ell}(\wt{G}(F),\omega) \ar[rr]^{p^{K}} && \bs{D}^{K}(\wt{G}(F),\omega)
}
\leqno{(5)}
$$
est commutatif. 

\begin{marema2}
{\rm   
Soit $a :W_F\rightarrow Z(\hat{G})$ un cocycle dans la classe de cohomologie $\bs{a}$. 
Alors $a_M: W_F \rightarrow Z(\hat{G})\hookrightarrow Z(\hat{M})$ est un cocycle dans la classe de cohomologie $\bs{a}_M$. 
D'après \ref{paramètres}, on a une bijection $\varphi_M\mapsto \pi_{\varphi_M}$ entre 
les classes de $\hat{M}$--conjugaison de $L$--homorphismes $\varphi_M: W_F\rightarrow {^LM}$ qui sont tempérés, non ramifiés, et tels que l'ensemble
$$
\wt{S}_{\varphi_M,a_M}=\{\wt{m}\in \hat{M}\hat{\theta}: 
{\rm Int}_{\wt{m}}\circ \varphi_M = a_M\cdot \varphi_M\}$$
n'est pas vide, et 
les classes d'isomorphisme de représentations irréductibles de $M(F)$ qui sont tempérées, $K_M$--sphériques et se prolongent en une représentation de $(\wt{M}(F),\omega)$. On note 
$\wt{\pi}_{\varphi_M}$ le prolongement de $\pi_{\varphi_M}$ à $(\wt{M}(F),\omega)$ normalisé 
comme dans la remarque (1) de \ref{paramètres}, \cad par 
$\Theta_{\wt{\pi}_{\varphi_M}}(\bs{1}_{\wt{K}_M})=1$. 
Pour un $L$--homomorphisme $\varphi_M: W_F\rightarrow {^LM}$, on note $\varphi_M^G:W_F\rightarrow {^LG}$ le $L$--homomorphisme obtenu en composant 
$\varphi_M$ avec ${^LM}\hookrightarrow {^LG}$. Si $\varphi_M$ est tempéré et non ramifié, alors 
$\varphi_M^G$ l'est aussi. D'autre part pour $\varphi=\varphi_M^G$, on a l'égalité
$$
\wt{S}_{\varphi_M, a_M} = \wt{S}_{\varphi,a}\cap \hat{M}\hat{\theta}.
$$
Si de plus $\wt{S}_{\varphi_M, a_M}$ n'est pas vide, alors $\wt{S}_{\varphi,a}$ ne l'est pas non plus, 
et l'on a
$$
p^K\circ \bs{i}_{\wt{M}}^{\wt{G}}(\Theta_{\wt{\pi}_{\varphi_M}})=\Theta_{\wt{\pi}_\varphi},
$$
où $\wt{\pi}_\varphi$ 
est le prolongement de $\pi_\varphi$ à $(\wt{G}(F),\omega)$ normalisé par 
$\Theta_{\wt{\pi}_\varphi}(\bs{1}_{\wt{K}})=1$.
}
\end{marema2}

\section{RŽduction au cas o la donnŽe endoscopique est un tore}

\subsection{Des décompositions}\label{le principe}Les arguments permettant 
de prouver la commutativitŽ du diagramme (5) de \ref{transfert spectral elliptique} ont été brièvement décrits dans l'introduction. Ces arguments utilisent des résultats fins qui seront prouvés dans la section 5, ainsi que les décompositions des espaces $\bs{D}_{\rm ell}(\wt{G}(F),\omega)$ et $\bs{SD}(\wt{G}'(F))$ en partie non ramifiée et partie ramifiée, qui seront démontrées plus loin. On donne ces décompositions dans ce numéro, de manière à fixer les notations. 

Pour $*=\;{\rm ram},\, {\rm nr}$, on pose
$$
\bs{D}_{\rm ell}^*(\wt{G},\omega)= \bs{D}_{\rm ell}(\wt{G}(F),\omega)\cap \bs{D}^*(\wt{G}(F),\omega).
$$
La dŽcomposition (1) de \ref{transfert spectral} donne par restriction une dŽcomposition (cf. \ref{triplets elliptiques essentiels})
$$
\bs{D}_{\rm ell}(\wt{G}(F),\omega)=\bs{D}_{\rm ell}^{\rm nr}(\wt{G}(F),\omega)\oplus \bs{D}_{\rm ell}^{\rm ram}(\wt{G}(F),\omega).\leqno{(1)}
$$
De mme, si $\bs{G}'=(G',\wt{G},\tilde{s})$ est une donnŽe endoscopique elliptique non ramifiŽe pour $(\wt{G},\bs{a})$, on pose 
$$
\bs{SD}_{\rm ell}^*(\wt{G}'(F))= \bs{SD}_{\rm ell}(\wt{G}'(F))\cap \bs{D}^*(\wt{G}'(F)),\quad *=\;{\rm ram},\,{\rm nr}.
$$
Alors on a la dŽcomposition (cf. \ref{le cas à torsion intérieure})
$$
\bs{SD}_{\rm ell}(\wt{G}'(F))= \bs{SD}_{\rm ell}^{\rm nr}(\wt{G}'(F))\oplus \bs{SD}_{\rm ell}^{\rm ram}(\wt{G}'(F)).\leqno{(2)}
$$
De plus, on distingue deux cas: ou bien $G'$ n'est pas un tore, auquel cas 
$\bs{SD}_{\rm ell}^{\rm nr}(\wt{G}'(F))=\{0\}$; ou bien $G'=T'$ est un tore, auquel cas pour 
$*={\rm nr},\,{\rm ram}$, on a $\bs{SD}_{\rm ell}^*(\wt{T}'(F))= \bs{D}^*(\wt{G}'(F))$ et la décomposition (2) n'est autre que la décomposition (1) de \ref{transfert spectral}. Les résultats fins de la section 5 concernent le cas où le groupe $G'$ est un tore. Ils sont utilisés en \ref{preuve} pour prouver 
la commutativité du diagramme (5) de \ref{transfert spectral elliptique} (pour tout $G'$). 

\subsection{$R$--groupes tordus}\label{R-groupes tordus}Rappelons les constructions de \cite[2.8--2.12]{W}. Soit $M$ un $F$--Levi semi--standard de $G$, et soit $\sigma$ une représentation irréductible et de la série discrète de $M(F)$. On note $\ES{N}^{\wt{G},\omega}(\sigma)$ l'ensemble (noté $\ES{N}^{\wt{G}}(\sigma)$ dans loc.~cit.) des couples $(A,\gamma)$ où $A$ est un automorphisme unitaire de l'espace $V_\sigma$ de $\sigma$ et $\gamma$ est un élément de $N_{\wt{G}(F)}(M)$ tels que $A$ entrelace $\sigma^\gamma = \sigma\circ {\rm Int}_\gamma$ et $\omega\sigma=\omega \otimes \sigma$, \cad tels que
$$
\sigma({\rm Int}_\gamma(m))\circ A = \omega(m)A\circ \sigma(m),\quad m\in M(F).
$$
Le groupe $\ES{N}^G(\sigma)$, défini de la même manière en rempla\c{c}ant $(\wt{G},\omega)$ par 
$(G,1)$, opère à gauche et à droite sur $\ES{N}^{\wt{G},\omega}(\sigma)$ par
\begin{eqnarray*}
\ES{N}^{G}(\sigma)\times \ES{N}^{\wt{G},\omega}(\sigma)\times \ES{N}^{G}(\sigma)&\rightarrow &\ES{N}^{\wt{G},\omega}(\sigma),\\ ((A',n'),(A,\gamma),(A'',n''))&\mapsto & (\omega(n'')A'\circ A\circ A'', n'\gamma n'').
\end{eqnarray*}
Cela fait de l'ensemble $\ES{N}^{\wt{G},\omega}(\sigma)$, s'il n'est pas vide, un espace tordu sous 
$\ES{N}^G(\sigma)$. 

Soit $P$ un sous--groupe parabolique de $G$ défini sur $F$ et de composante de Levi $M$. On peut former l'induite parabolique normalisée $\pi=i_P^G(\sigma)$ de $\sigma$ à $G(F)$. C'est une représentation tempérée de $G(F)$. 
\`A tout élément $(A,n)$ de $\ES{N}^{G}(\sigma)$ est associé en \cite[1.9--1.11]{W} un opérateur d'entrelacement normalisé $r_P(A,n)$ de l'espace $V_\pi$ de $\pi$, \cad tel que
$$
r_P(A,n)\circ \pi(g)=\pi(g)\circ r_P(A,n),\quad g\in G(F).
$$
L'application $(A,n)\mapsto r_P(A,n)$ est une représentation unitaire de $\ES{N}^G(\sigma)$, et si l'on remplace $P$ par un autre sous--groupe parabolique $P'$ de $G$ défini sur $F$ et de composante de Levi $M$, alors les représentations $r_P$ et $r_{P'}$ de $\ES{N}^G(\sigma)$ sont équivalentes. Le sous--groupe de $\ES{N}^G(\sigma)$ formé des $(A,n)$ tels que $r_P(A,n)$ est l'identité de $V_\pi$ ne dépend donc pas de $P$. On le note $\ES{N}^G_0(\sigma)$. Ce groupe 
$\ES{N}^G_0(\sigma)$ contient $M(F)$, identifié à un sous--groupe distingué de $\ES{N}^G(\sigma)$ via l'application $m\mapsto (\sigma(m),m)$. Notons $\ES{W}^G(\sigma)$ le groupe quotient 
$\ES{N}^G(\sigma)/M(F)$, et $W^G(\sigma)$ le groupe quotient $N_{G(F)}(\sigma)/M(F)$, où l'on a posé
$$
N_{G(F)}(\sigma)= \{g\in G(F): gMg^{-1}=M,\, \sigma^g \simeq \sigma\}.
$$
La projection sur le second facteur induit une suite exacte courte
$$
1\rightarrow {\Bbb U} \rightarrow \ES{W}^G(\sigma)\rightarrow W^G(\sigma)\rightarrow 1,\leqno{(1)}
$$
où (rappel) ${\Bbb U}$ est le groupe des nombres complexes de modules $1$. Via cette suite, le groupe quotient $\ES{N}^G_0(\sigma)/M(F)\subset \ES{W}^G(\sigma)$ s'identifie à un sous--groupe distingué de $W^G(\sigma)$, disons $W^G_0(\sigma)$. Posons
$$
\ES{R}^G(\sigma)= \ES{W}^G(\sigma)/W^G_0(\sigma),\quad R^G(\sigma)=W^G(\sigma)/W^G_0(\sigma).
$$
Alors $R^G(\sigma)$ est le $R$--groupe de $\sigma$, et la suite (1) induit 
une suite exacte courte
$$
1\rightarrow {\Bbb U}\rightarrow \ES{R}^G(\sigma)\rightarrow R^G(\sigma)\rightarrow 1.\leqno{(2)}
$$
La représentation $r_P$ de $\ES{N}^G(\sigma)$ se 
quotiente en une représentation de $\ES{R}^G(\sigma)$ telle que tout élément $z\in {\Bbb U}$ opère sur $V_\pi$ par l'homothétie des rapport $z$. Une représentation vérifiant une telle condition d'homothétie est appelée {\it ${\Bbb U}$--représentation}. On note ${\rm Irr}(\ES{R}^G(\sigma))$ l'ensemble des classes d'isomorphisme de ${\Bbb U}$--représentations irréductibles de $\ES{R}^G(\sigma)$. D'après la théorie du $R$--groupe, il existe une bijection $\rho\mapsto \pi_\rho$ de 
${\rm Irr}(\ES{R}^G(\sigma))$ sur l'ensemble $\Pi_\sigma$ des classes d'isomorphisme représentations irréductibles de $G(F)$ qui sont des sous--représentations de $\pi\;(=i_P^G(\sigma))$, de sorte que la représentation $r_P\otimes \pi$ de $\ES{R}^G(\sigma)\times G(F)$ se décompose en 
$$
r_P\otimes \pi \simeq \oplus_{\rho\in {\rm Irr}(\ES{R}^G(\sigma))}\, \rho \otimes \pi_\rho.\leqno{(3)}
$$
La correspondance $\rho\mapsto \pi_\rho$ ne dépend pas de $P$ --- cf. \cite[1.11]{W}. 

Supposons que l'ensemble $\ES{N}^{\wt{G},\omega}(\sigma)$ n'est pas vide. 
En \cite[2.8]{W} est défini une application 
$\wt{\nabla}_P$ de $\ES{N}^{\wt{G},\omega}(\sigma)$ dans le groupe des automorphismes unitaires de l'espace $V_\pi$ vérifiant, pour $(A,\gamma)\in \ES{N}^{\wt{G},\omega}(\sigma)$:
\begin{itemize}
\item pour $g\in G(F)$, on a
$$
\pi({\rm Int}_\gamma(g)) \circ \wt{\nabla}_P(A,\gamma) = \wt{\nabla}_P(A,\gamma) \circ (\omega\pi)(g); \leqno{(4)}
$$
\item pour 
$(A',n'),\, (A'',n'')\in \ES{N}^G(\sigma)$ tels que $(A',n')(A,\gamma)=(A,\gamma)(A'',n'')$, on a
$$
r_P(A',n')\circ \wt{\nabla}_P(A,\gamma)= \wt{\nabla}_P(A,\gamma)\circ r_P(A'',n'').\leqno{(5)}
$$
\end{itemize}
Les orbites dans $\ES{N}^{\wt{G},\omega}(\sigma)$ pour l'action de $M(F)\subset \ES{N}^G(\sigma)$ sont les mêmes que l'on considère l'action à gauche ou celle à droite. On note 
$\ES{W}^{\wt{G},\omega}(\sigma)$ l'ensemble de ces orbites. D'après (5), les orbites dans $\ES{W}^{\wt{G},\omega}(\sigma)$ pour l'action de $W^G_0(\sigma)\subset \ES{W}^G(\sigma)$ sont les mêmes que l'on considère l'action à gauche ou celle à droite. On note $\ES{R}^{\wt{G},\omega}(\sigma)$ l'ensemble de ces orbites. C'est un espace tordu sous $\ES{R}^G(\sigma)$. Posons 
$W^{\wt{G},\omega}(\sigma)= N_{\wt{G}(F),\omega}(\sigma)/M(F)$, où
$$
N_{\wt{G}(F),\omega}(\sigma)= \{\gamma\in \wt{G}(F): {\rm Int}_\gamma(M)=M,\, \sigma^\gamma\simeq \omega \sigma\}, 
$$
et $R^{\wt{G},\omega}(\sigma)= W^{\wt{G},\omega}(\sigma)/ W^G_0(\sigma)$. On a une application 
surjective $\ES{R}^{\wt{G},\omega}(\sigma) \rightarrow R^{\wt{G},\omega}(\sigma)$ dont les fibres sont isomorphes à ${\Bbb U}$. Si $g\in G(F)$ est tel que $M'=gMg^{-1}$ est semi--standard, alors 
notant $\sigma'$ la représentation ${^g\sigma}=\sigma\circ {\rm Int}_{g^{-1}}$ de $M'(F)$, l'application
$$
\ES{N}^{\wt{G},\omega}(\sigma)\rightarrow \ES{N}^{\wt{G},\omega}(\sigma'),\,
(A,\gamma)\mapsto (\omega(g)A,g\gamma g^{-1})
$$
est bijective, et se quotiente en des applications bijectives
$$
\ES{R}^{\wt{G},\omega}(\sigma)\rightarrow \ES{R}^{\wt{G},\omega}(\sigma'),
\quad R^{\wt{G},\omega}(\sigma)\rightarrow R^{\wt{G},\omega}(\sigma').
$$
On note encore ${\rm Int}_g$ toutes ces applications, et pour $\tilde{\bs{r}}\in \ES{R}^{\wt{G},\omega}(\sigma)$, on dit que les triplets $(M,\sigma, \tilde{\bs{r}})$ et $(M',\sigma', {\rm Int}_g(\tilde{\bs{r}}))$ 
sont conjugués (par $g$).

Une ${\Bbb U}$--représentation $\tilde{\rho}$ de $\ES{R}^{\wt{G},\omega}(\sigma)$ est par définition la donnée d'un couple $(\rho, \tilde{\rho})$ où $\rho$ est une ${\Bbb U}$--représentation unitaire de longueur finie de $\ES{R}^G(\sigma)$, et $\tilde{\rho}$ est une application de $\ES{R}^{\wt{G},\omega}(\sigma)$ dans le groupe des automorphismes unitaires de l'espace $V_\rho$ de $\rho$ vérifiant la condition
$$
\tilde{\rho}(\bs{r'}\tilde{\bs{r}}\bs{r}'')= \rho(\bs{r}')\tilde{\rho}(\tilde{\bs{r}})\rho(\bs{r}''),\quad 
\tilde{\bs{r}}\in \ES{R}^{\wt{G},\omega}(\sigma),\; \bs{r}',\,\bs{r}''\in \ES{R}^G(\sigma).
$$
On note ${\rm Irr}(\ES{R}^{\wt{G},\omega}(\sigma))$ l'ensemble des classes d'isomorphisme de ${\Bbb U}$--représentations  $(\rho,\tilde{\rho})$ de $\ES{R}^{\wt{G},\omega}(\sigma)$ qui sont 
$\ES{R}^G(\sigma)$--irréductibles, \cad telles que $\rho$ est irréductible. 
Si $(\rho,\tilde{\rho})$ est une ${\Bbb U}$--représentation de $\ES{R}^{\wt{G},\omega}(\sigma)$, alors pour $z\in {\Bbb U}$, $(\rho, z\tilde{\rho})$ en est une autre. Tout élément $\tilde{\bs{r}}$ de 
$\ES{R}^{\wt{G},\omega}(\sigma)$ détermine un ${\Bbb U}$--automorphisme ${\rm Int}_{\tilde{\bs{r}}}$ de $\ES{R}^G(\sigma)$ par la formule $\tilde{\bs{r}}\bs{r}= {\rm Int}_{\tilde{\bs{r}}}(\bs{r})\tilde{\bs{r}}$. La classe de cet 
automorphisme ${\rm Int}_{\tilde{\bs{r}}}$ modulo les automorphismes intérieurs ne dépend pas du choix de $\bs{\tilde{r}}$. On la note $\theta_{\ES{R}}$. L'application $(\rho,\tilde{\rho})\mapsto \rho$ induit une bijection entre:
\begin{itemize}
\item l'ensemble ${\rm Irr}(\ES{R}^{\wt{G},\omega}(\sigma))/{\Bbb U}$ des orbites de ${\Bbb U}$ dans 
${\rm Irr}(\ES{R}^{\wt{G},\omega}(\sigma))$;
\item l'ensemble ${\rm Irr}(\ES{R}^G(\sigma),\theta_{\ES{R}})$ des classes d'isomorphismes de ${\Bbb U}$--représentations irréductibles $\rho$ de $\ES{R}^G(\sigma)$ telles que $\rho\circ \theta_{\ES{R}}\simeq \rho$.
\end{itemize}

Reprenons la décomposition (3) de $r_P\otimes \pi$. Notons $\theta$ la classe de l'automorphisme ${\rm Int}_\gamma$ de $G(F)$ pour $\gamma\in \wt{G}(F)$ modulo les automorphismes intérieurs ${\rm Int}_g$ pour $g\in G(F)$. 
Soit 
$\rho\in {\rm Irr}(\ES{R}^G(\sigma))$. D'après (5), pour $(A,\gamma)\in \ES{N}^{\wt{G},\omega}(\sigma)$, l'automorphisme $\wt{\nabla}_P(A,\gamma)$ de $V_\pi$ envoie la composante $\rho$--isotypique $\rho\otimes \pi_\rho$ sur la composante $(\rho\circ \theta_{\ES{R}})$--isotypique 
$(\rho\circ \theta_{\ES{R}})\otimes \pi_{\rho \circ \theta_{\ES{R}}}$, et d'après (4), on a 
$\pi_{\rho\circ \theta_{\ES{R}}}\circ \theta \simeq \omega \otimes \pi_\rho$. 
En particulier, on a $\pi_\rho\circ \theta \simeq \omega\otimes \pi_\rho$ si et seulement si $\rho$ appartient à ${\rm Irr}(\ES{R}^G(\sigma),\theta_{\ES{R}})$. 
Supposons que tel est le cas, et choisissons un élément $\tilde{\rho}\in {\rm Irr}(\ES{R}^{\wt{G},\omega}(\sigma))$ dans la ${\Bbb U}$--orbite correspondant à $\rho$. Alors il existe une unique 
représentation $\tilde{\pi}_{\tilde{\rho}}$ de $(\wt{G}(F),\omega)$ prolongeant $\pi_\rho$ et telle que pour tout $(A,\gamma)\in \ES{N}^{\wt{G},\omega}(\sigma)$, la restriction de $\wt{\nabla}_P(A,\gamma)$ à l'espace de $\rho\otimes \pi_\rho$ soit égale à $\tilde{\rho}(\tilde{\bs{r}})\otimes \tilde{\pi}_{\tilde{\rho}}(\gamma)$, où $\tilde{\bs{r}}$ est l'image de $(A,\gamma)$ dans $\ES{R}^{\wt{G},\omega}(\sigma)$. \`A multiplication près par un élément de ${\Bbb U}$, la représentation ($G(F)$--irréductible, tempérée) $\tilde{\pi}_{\wt{\rho}}$ de $(\wt{G}(F),\omega)$ ne dépend que de $\rho$, et pas du choix de $\tilde{\rho}$. On obtient ainsi une bijection 
$\rho \mapsto [\tilde{\pi}_{\rho}]$ de ${\rm Irr}(\ES{R}^G(\sigma), \theta_{\ES{R}})$ sur 
les orbites de ${\Bbb U}$ dans l'ensemble $\wt{\Pi}_\sigma$ des classes d'isomorphisme de représentations tempérées $(\pi',\tilde{\pi}')$ de $(\wt{G}(F),\omega)$ telles que $\pi'$ est une sous--représentation irréductibles de $\pi\;(=i_P^G(\sigma))$. Comme pour la correspondance $\rho\mapsto \pi_\rho$, cette correspondance 
$\rho \mapsto [\tilde{\pi}_{\rho}]$ ne dépend pas du choix de $P$ (défini sur $F$ et de composante de Levi $M$).

\vskip1mm
En résumé, à un triplet $(M,\sigma, \rho)$ formé d'un $F$--Levi semi--standard $M$ de $G$, d'une représentation irréductible et de la série discrète $\sigma$ de $M(F)$ telle que $\ES{N}^{\wt{G},\omega}(\sigma)\neq\emptyset$, et d'un élément $\rho\in\ES{R}^G(\sigma, \theta_{\ES{R}})$, est associé un élément de $\wt{\Pi}_\sigma/{\Bbb U}$ --- \cad une ${\Bbb U}$--orbite de classes d'isomorphisme de représentations $G(F)$--irréductibles et tempérées de $(\wt{G}(F),\omega)$. Cette 
${\Bbb U}$--orbite détermine une droite dans $\bs{D}(\wt{G}(F),\omega)$, que l'on note $D_{M,\sigma,\rho}$. Deux tels triplets 
$(M,\sigma,\rho)$ et $(M',\sigma',\rho')$ sont dits conjugués par $g\in G(F)$ si $gMg^{-1}=M'$, 
${^g\sigma}\simeq \sigma'$ et ${^g\rho}= \rho'$, où ${^g\rho}=\rho\circ {\rm Int}_{g^{-1}}$, l'isomorphisme  ${\rm Int}_{g^{-1}}:\ES{R}^G(\sigma')\rightarrow \ES{R}^G(\sigma)$ provenant par passage au quotient de l'isomorphisme 
$(A,n)\mapsto (A, g^{-1}ng)$ de $\ES{N}^G(\sigma')$ sur $\ES{N}^G(\sigma)$. Notons que si $\tilde{\rho}\in {\rm Irr}(\ES{R}^{\wt{G},\omega}(\sigma))$ est un prolongement de $\rho$, alors ${^g\tilde{\rho}}=\tilde{\rho}\circ {\rm Int}_{g^{-1}}\in {\rm Irr}(\ES{R}^{\wt{G},\omega}(\sigma'))$ est un prolongement de $\rho'$, disons 
$\tilde{\rho}'$. De plus posant, posant $P'=gPg^{-1}$ et $\pi'= i_{P'}^G(\sigma')$, les représentations $\tilde{\pi}_{\tilde{\rho}}$ et $\tilde{\pi}'_{\tilde{\rho}'}$ de $(\wt{G}(F),\omega)$ sont isomorphes, où $\tilde{\pi}'_{\tilde{\rho}'}$ est le prolongement de $\pi'_{\rho'}$ à $(\wt{G}(F),\omega)$ associé à $\tilde{\rho}'$ comme ci--dessus. On obtient ainsi une 
décomposition
$$
\bs{D}(\wt{G}(F),\omega)=\oplus_{(M,\sigma,\rho)/{\rm conj.}\;}D_{M,\sigma,\rho}\leqno{(6)}
$$
où $(M,\sigma,\rho)$ parcourt les classes de $G(F)$--conjugaison de triplets comme ci--dessus.
  
\subsection{Triplets elliptiques essentiels}\label{triplets elliptiques essentiels}
La décomposition (6) de \ref{R-groupes tordus} n'est pas compatible à l'induction parabolique. Pour en déduire une décomposition de $\bs{D}_{\rm ell}(\wt{G}(F),\omega)$, il faut commencer par la modifier \cite[2.9]{W}. Soit $(M,\sigma)$ un couple formé d'un $F$--Levi semi--standard $M$ et d'une représentation irréductible et de la série discrète $\sigma$ de $M(F)$ telle que $\ES{N}^{\wt{G},\omega}(\sigma)$ n'est pas vide. Fixons un sous--groupe parabolique $P$ de $G$, défini sur $F$ et de composante de Levi $M$, et posons $\pi=i_P^G(\sigma)$. Un élément $(A,\gamma)\in \ES{N}^{\wt{G},\omega}(\sigma)$ définit un prolongement (tempéré) $\tilde{\pi}=\tilde{\pi}_{A,\gamma}$ de $\pi$ à $(\wt{G}(F),\omega)$ par la formule
$$
\tilde{\pi}(g\gamma)=\pi(g)\circ \wt{\nabla}(A,\gamma),\quad g\in G(F).
$$
Cette représentation, qui n'est en général pas irréductible (encore moins $G(F)$--irréductible), ne dépend pas vraiment de $(A,\gamma)$, mais seulement de l'image $\tilde{\bs{r}}$ de $(A,\gamma)$ dans $\ES{R}^{\wt{G},\omega}(\sigma)$. Elle 
ne dépend donc que du triplet $\bs{\tau}=(M,\sigma,\tilde{\bs{r}})$, et on peut la noter $\tilde{\pi}_{\bs{\tau}}$. La classe d'isomorphisme de $\tilde{\pi}_{\bs{\tau}}$ ne dépend en fait pas de $P$, mais seulement de la classe de $G(F)$--conjugaison du triplet $\bs{\tau}$ (cf. \ref{R-groupes tordus} pour cette notion de conjugaison). Pour $z\in {\Bbb U}$, notant $z\bs{\tau}$ le triplet $(M,\sigma, z\tilde{\bs{r}})$, on a $\tilde{\pi}_{z\bs{\tau}}=z\tilde{\pi}_{\bs{\tau}}$. S'il existe un $z\in {\Bbb U}\smallsetminus \{1\}$ tel que $z\tilde{\bs{r}}= \bs{r}^{-1}\tilde{\bs{r}}\bs{r}$ pour un $\bs{r}\in \ES{R}^G(\sigma)$ --- auquel cas on a $z\tilde{\pi}_{\bs{\tau}}\simeq \tilde{\pi}_{\bs{\tau}}$ et le caractère de $\tilde{\pi}_{\bs{\tau}}$ est nul ---, on dit que le triplet $\bs{\tau}$ est {\it inessentiel}. Dans le cas contraire, on dit que le triplet $\bs{\tau}$ est {\it essentiel}. On note $\ES{E}(\wt{G},\omega)$ l'ensemble 
des triplets $\bs{\tau}$ qui sont essentiels --- il est stable pour l'action de ${\Bbb U}$  ---, et $E(\wt{G},\omega)= \ES{E}(\wt{G},\omega)/{\Bbb U}$ le quotient de $\ES{E}(\wt{G},\omega)$ par l'action de ${\Bbb U}$. Ce quotient s'identifie à l'ensemble des triplets $(M,\sigma,\tilde{r})$ où 
$(M,\sigma)$ est un couple comme ci--dessus, et $\tilde{r}$ est un élément de $R^{\wt{G},\omega}(\sigma)$ qui se relève en un élément $\tilde{\bs{r}}$ de $\ES{R}^{\wt{G},\omega}(\sigma)$ tel que le triplet $(M,\sigma,\tilde{\bs{r}})$ est essentiel. 

D'après la proposition de \cite[2.9]{W}, pour $\bs{\tau}\in \ES{E}(\wt{G},\omega)$, la représentation $\tilde{\pi}_{\bs{\tau}}$ de $(\wt{G}(F),\omega)$ --- bien définie à isomorphisme près ---  a un caractère non nul. Ce caractère définit donc une droite dans $\bs{D}(\wt{G},\omega)$, que l'on note $D_{\bs{\tau}}$. Cette droite ne dépend que de la classe de $G(F)$--conjugaison de l'image 
du triplet $\bs{\tau}$ dans $E(\wt{G},\omega)$, et d'après loc.~cit., on a la décomposition
$$
\bs{D}(\wt{G}(F),\omega)= \bigoplus_{\tau\in E(\wt{G},\omega)/{\rm conj.}}D_\tau\leqno{(1)}
$$
où $\tau$ parcourt les classes de $G(F)$--conjugaison dans $E(\wt{G},\omega)$.

On note $A_G$ le plus grand tore de $G$ contenu dans $Z(G)$ et déployé sur $F$, et l'on pose 
$\ES{A}_G= \check{X}(A_G)\otimes_{\Bbb Z}{\Bbb R}$ où (rappel) $\check{X}(A_G)$ désigne le groupe des cocaractères algébriques de $A_G$. On note $A_{\smash{\wt{G}}}$ le sous--tore de $A_G$ tel que 
$\check{X}(A_{\smash{\wt{G}}})= \check{X}(A_G)^\theta$, où $\check{X}(A_G)^\theta$ désigne le sous--groupe de $\check{X}(A_G)$ formé des cocaractères qui sont invariants sous l'action de $\theta$, et l'on pose $\ES{A}_{\smash{\wt{G}}}=\check{X}(A_{\smash{\wt{G}}})\otimes_{\Bbb Z}{\Bbb R}$. On a donc $\ES{A}_{\smash{\wt{G}}}=\ES{A}_G^\theta$. 

Soit $(M,\sigma)$ un couple formé d'un 
$F$--Levi semi--standard $M$ de $G$, et d'une représentation irréductible et de la série discrète $\sigma$ de $M(F)$. Un élément $\tilde{w}\in W^{\wt{G},\omega}(\sigma)$ opère naturellement sur $\ES{A}_M$, et le sous--espace $\ES{A}_M^{\tilde{w}}\subset \ES{A}_M$ formé des points fixes sous $\tilde{w}$ contient $\ES{A}_{\smash{\wt{G}}}$. On pose
$$
W^{\wt{G},\omega}_{\rm reg}(\sigma)= \{\tilde{w}\in W^{\wt{G},\omega}(\sigma):  \ES{A}_M^{\tilde{w}}=
\ES{A}_{\smash{\wt{G}}}\}.
$$

Un triplet $\tau=(M,\sigma,\tilde{r})\in E(\wt{G},\omega)$ est dit {\it elliptique} si 
$W^G_0(\sigma)=\{1\}$ --- ce qui entra\^{\i}ne l'égalité $R^{\wt{G},\omega}(\sigma)= W^{\wt{G},\omega}(\sigma)$ --- et si $\tilde{r}\in W^{\wt{G},\omega}_{\rm reg}(\sigma)$. Un triplet $\bs{\tau}\in \ES{E}(\wt{G},\omega)$ est dit elliptique si son image dans $E(\wt{G},\omega)$ est elliptique. 
Soit $\ES{E}_{\rm ell}=\ES{E}_{\rm ell}(\wt{G},\omega)$ le sous--ensemble de $\ES{E}(\wt{G},\omega)$ formé des triplets qui sont elliptiques, et soit $E_{\rm ell}= \ES{E}_{\rm ell}/{\Bbb U}\subset E(\wt{G},\omega)$. L'ensemble $E_{\rm ell}$ est stable par conjugaison dans $G(F)$, et le sous--espace $\bs{D}_{\rm ell}(\wt{G}(F),\omega)$ de $\bs{D}(\wt{G}(F),\omega)$ introduit en \ref{transfert spectral elliptique} est défini par
$$
\bs{D}_{\rm ell}(\wt{G}(F),\omega)= \bigoplus_{\tau\in E_{\rm ell}/{\rm conj.}} D_\tau.
\leqno{(2)}
$$

\vskip1mm
Rappelons qu'on a fixé une mesure positive $G(F)$--invariante à gauche et à droite $d\gamma$ sur $\wt{G}(F)$. Pour $\bs{\tau}=(M,\sigma,\tilde{\bs{r}})\in \ES{E}_{\rm ell}$, notons $\Theta_{\bs{\tau}}$ le caractère $\Theta_{\wt{\pi}_{\bs{\tau}}}$, \cad la distribution sur $\wt{G}(F)$ donnée par
$$
\Theta_{\bs{\tau}}(f)= {\rm tr}(\tilde{\pi}_{\bs{\tau}}(fd\gamma)),\quad f\in C^\infty_{\rm c}(\wt{G}(F).
$$
Pour chaque $\rho \in {\rm Irr}(\ES{R}^G(\sigma, \theta_{\ES{R}}))$, choisissons un prolongement
$\tilde{\rho}\in {\rm Irr}(\ES{R}^{\wt{G},\omega}(\sigma))$ de $\rho$. Par définition de $\tilde{\pi}_{\bs{\tau}}$, on a l'égalité
$$
\Theta_{\bs{\tau}}= \sum_{\rho\in {\rm Irr}(\ES{R}^G(\sigma,\theta_{\ES{R}}))} {\rm tr}(\tilde{\rho}(\tilde{\bs{r}}))\Theta_{\tilde{\pi}_{\tilde{\rho}}}.\leqno{(3)}
$$
Ainsi, ou bien $M=T$ et $\sigma$ est un caractère non ramifié unitaire de $T(F)$, auquel cas tous les caractères $\Theta_{\tilde{\pi}_{\tilde{\rho}}}$ apparaissant dans l'égalité (3) sont des éléments de $\bs{D}_{\rm ell}^{\rm nr}(\wt{G}(F), \omega)$; ou bien l'une des deux conditions précédentes n'est pas vérifiée, auquel cas tous les caractères $\Theta_{\tilde{\pi}_{\tilde{\rho}}}$ apparaissant dans l'égalité (3) sont des éléments de $\bs{D}_{\rm ell}^{\rm ram}(\wt{G}(F), \omega)$. Notons $\ES{E}_{\rm ell}^{\rm nr}$ le sous--ensemble de $\ES{E}_{\rm ell}$ formé des triplets $\bs{\tau}=(T,\lambda,\tilde{\bs{r}})$ tels que $\lambda$ est un caractère non ramifié (unitaire) de $T(F)$, et posons 
$\ES{E}_{\rm ell}^{\rm ram}=\ES{E}_{\rm ell}\smallsetminus 
\ES{E}_{\rm ell}^{\rm nr}$. Pour $*=\;{\rm nr},\, {\rm ram}$, posons $E_{\rm ell}^*= \ES{E}_{\rm ell}^*/{\Bbb U}$. On a donc les décompositions
$$
\bs{D}_{\rm ell}^{\rm nr}(\wt{G}(F),\omega)= \bigoplus_{\tau\in E_{\rm ell}^{\rm nr}/{\rm conj.}} D_\tau,\quad \bs{D}_{\rm ell}^{\rm ram}(\wt{G}(F),\omega)= \bigoplus_{\tau\in E_{\rm ell}^{\rm ram}/{\rm conj.}} D_\tau.
\leqno{(4)}
$$
Cela démontre en particulier la décomposition (1) de \ref{le principe}.

\vskip1mm Fixons un caractère unitaire $\mu$ de $A_{\smash{\wt{G}}}(F)$. Soit $\bs{D}_\mu(\wt{G}(F),\omega)$ le sous--espace de $\bs{D}(\wt{G}(F),\omega)$ engendré par les distributions $\Theta_{\tilde{\pi}}$ où 
$(\pi,\tilde{\pi})$ est une représentation de $(\wt{G}(F),\omega)$ telle que $\pi$ est une représentation irréductible et tempérée de $G(F)$ de caractère central $\mu_\pi:Z(G;F)\rightarrow {\Bbb C}^\times$ prolongeant $\mu$. Posons
$$
\bs{D}_{\mu,{\rm ell}}(\wt{G}(F),\omega)= 
 \bs{D}_\mu(\wt{G}(F),\omega)\cap \bs{D}_{\rm ell}(\wt{G}(F),\omega).
 $$ L'espace $\bs{D}_{\mu,{\rm ell}}(\wt{G}(F),\omega)$ est muni d'un produit scalaire hermitien défini positif, défini à l'aide de la mesure $d\gamma$ sur $\wt{G}(F)$ mais qui n'en dépend pas \cite[7.3]{W}. On note $(\cdot , \cdot )_{\mu,{\rm ell}}$ ce produit (dans loc. cit., il est noté sans l'indice $\mu$). Pour $\bs{\tau}=(M,\sigma,\tilde{\bs{r}})\in \ES{E}(\wt{G},\omega)$, la représentation $\tilde{\pi}_{\bs{\tau}}$ a un caractère central $\mu_{\bs{\tau}}: Z(G;F)\rightarrow {\Bbb C}^\times$, qui n'est autre que la restriction à $Z(G;F)$ du caractère central $\mu_\sigma:Z(M;F)\rightarrow {\Bbb C}^\times$ de $\sigma$. Ce caractère $\mu_\tau$ vérifie l'égalité 
$$
(\mu_{\bs{\tau}}\circ (1-\theta))\cdot \omega =1.\leqno{(5)}
$$
On note $\ES{E}_\mu(\wt{G},\omega)$ le sous--ensemble de $\ES{E}(\wt{G},\omega)$ formé des triplets $\bs{\tau}$ tels que $\omega_{\bs{\tau}}$ prolonge $\mu$. On définit de la même manière 
les ensembles $E_\mu(\wt{G},\omega)$, $E_{\mu,{\rm ell}}$, etc. Les décompositions (2) et (3) restent vraies si l'on remplace l'indice $?_{\rm ell}$ par l'indice $?_{\mu,{\rm ell}}$. Bien sûr on a aussi la décomposition
 $$
 \bs{D}_{\rm ell}(\wt{G}(F),\omega)= \bigoplus_{\mu'}\bs{D}_{\mu'\!,{\rm ell}}(\wt{G}(F),\omega)
 $$
 où $\mu'$ parcourt les caractères unitaires de $A_{\smash{\wt{G}}}(F)$. On munit $ \bs{D}_{\rm ell}(\wt{G}(F),\omega)$ du produit scalaire hermitien défini positif qui est la somme directe des $(\cdot ,\cdot)_{\mu',{\rm ell}}$. On note $(\cdot ,\cdot)_{\rm ell}$ ce produit. 
 D'après le théorème de \cite[7.3]{W}, on a:
\begin{itemize}
\item pour $\tau_1,\,\tau_2\in E_{\mu,{\rm ell}}$ tels que $\tau_1$ et $\tau_2$ ne sont pas conjugués dans $G(F)$, les droites $D_{\tau_1}$ et $D_{\tau_2}$ sont orthogonales;
\item pour $\tau=(M,\sigma, \tilde{r})\in E_{\mu,{\rm ell}}$, si $\bs{\tau}$ est un relèvement de $\tau$ dans $\ES{E}(\wt{G},\omega)$, on a l'égalité
$$
(\Theta_{\bs{\tau}},\Theta_{\bs{\tau}})_{\rm ell}= \vert {\rm Stab}(R^G(\sigma),\tilde{r}) \vert \vert \det(1-\tilde{r}; \ES{A}_M/\ES{A}_{\smash{\wt{G}}})\vert,\leqno{(6)}
$$
où ${\rm Stab}(R^G(\sigma),\tilde{r})$ est le stabilisateur de $\tilde{r}$ dans 
$R^G(\sigma)$ (agissant par conjugaison).
\end{itemize}

\begin{marema1}
{\rm Notons $C^\infty_{\rm c,cusp}(\wt{G}(F))$ le sous--espace de $C^\infty_{\rm c}(\wt{G}(F))$ formé des fonctions qui sont {\it cuspidales} au sens de \cite[7.1]{W}, \cad telles que pour tout $\wt{M}\in \ES{L}(\wt{T},\omega)\smallsetminus \{\wt{G}\}$ et toute paire parabolique $(P,M)$ de $G$, semi--standard et définie sur $F$, telle que $\wt{M}_P= \wt{M}$, l'image du terme constant $f_{\wt{P},\omega}$ dans 
$\bs{I}(\wt{M}(F),\omega)$ est nulle. On note $\bs{I}_{\rm cusp}(\wt{G}(F),\omega)$ l'image de 
$C^\infty_{\rm c,cusp}(\wt{G}(F),\omega)$ dans $\bs{I}(\wt{G}(F),\omega)$. Pour un caractère unitaire $\mu$ de $A_{\smash{\wt{G}}}(F)$, on note $C^\infty_{\mu}(\wt{G}(F))$ l'espace des fonctions 
sur $\wt{G}(F)$, localement constantes et à support compact modulo $A_{\smash{\wt{G}}}(F)$, se transformant suivant $\mu^{-1}$ sur $A_{\smash{\wt{G}}}(F)$. On note $C^\infty_{\mu,{\rm cusp}}(\wt{G}(F))$ et $\bs{I}_{\mu,{\rm cusp}}(\wt{G}(F),\omega)$ les variantes à caractère central des espaces $C^\infty_{\rm c,cusp}(\wt{G}(F))$ et $\bs{I}_{\rm cusp}(\wt{G}(F),\omega)$ --- cf. \cite[7.2]{W}. L'espace $\bs{I}_{\mu,{\rm cusp}}(\wt{G}(F),\omega)$ est muni d'un produit scalaire hermitien défini positif $(\cdot ,\cdot)_{\mu,{\rm ell}}$, obtenu de la manière suivante. Fixons une mesure de Haar $da$ sur $A_{\smash{\wt{G}}}(F)$. Pour $f_1,\, f_2\in C^\infty_{{\rm c},{\rm cusp}} (\wt{G}(F))$, on définit l'expression $J^{\wt{G}}(\omega,f_1,f_2)$ comme en \cite[4.17]{Stab I}:
$$
J^{\wt{G}}(\omega,f_1,f_2)=\int_{\wt{G}(F)_{\rm ell}/{\rm conj.}}d_\gamma^{-1}
{\rm vol}(A_{\smash{\wt{G}}}(F))\backslash G_\gamma(F))\overline{I^{\wt{G}}(\gamma,\omega,f_1)}I^{\wt{G}}(\gamma,\omega,f_2)d\gamma\leqno{(7)}
$$
où (rappel) $d_\gamma=[G^\gamma(F):G_\gamma(F)]$ --- on renvoie à loc.~cit. pour les autres définitions. D'après \cite[6.6.(1)]{W}, l'intégrale est absolument convergente. Bien sûr on peut remplacer les fonctions $f_1$ et $f_2$ par leurs images dans $\bs{I}_{\rm cusp}(\wt{G}(F),\omega)$. Notons que l'expression (7) ne dépend que de $dg$ et $da$, et précisément varie comme $dg^2da^{-1}$. L'application
$$
p_\mu: C^\infty_{\rm c}(\wt{G}(F))\rightarrow C^\infty_\mu(\wt{G}(F))
$$
définie par
$$
p_\mu(f)(\gamma)= \int_{A_{\smash{\wt{G}}}(F)}f(a\gamma)\mu(a)da,\quad \gamma\in \wt{G}(F),
$$
est surjective. Elle induit, par passage aux quotients, un homomorphisme surjectif
$$
\bs{p}_\mu: \bs{I}(\wt{G}(F),\omega)\rightarrow \bs{I}_\mu (\wt{G}(F),\omega)
$$
qui envoie $\bs{I}_{\rm cusp}(\wt{G}(F),\omega)$ sur $\bs{I}_{\mu,{\rm cusp}} (\wt{G}(F),\omega)$. 
Pour $\bs{f}_{\!1},\, \bs{f}_{\!2}\in \bs{I}_{\mu,{\rm cusp}}(\wt{G}(F),\omega)$, on choisit des éléments $\bs{h}_1,\,\bs{h}_2\in \bs{I}_{\rm cusp}(\wt{G}(F),\omega)$ tels que $\bs{p}_\mu(\bs{h}_i)=\bs{f}_{\!i}$ 
($i=1,\,2$), et l'on pose
$$
(\bs{f}_{\!1},\bs{f}_{\!2})_{\mu,{\rm ell}}=\int_{A_{\smash{\wt{G}}}(F)}J^{\wt{G}}(\omega,\bs{h}_1,\bs{h}_2^{[a]})\mu(a)da; \leqno{(8)}
$$ 
où l'élément $\bs{h}_2^{[a]}\in \bs{I}_{\rm cusp}(\wt{G}(F),\omega)$ est défini par
$$
\bs{h}_2^{[a]}(\gamma) = \bs{h}_2(\gamma a)= \bs{h}_2(a\gamma),\quad \gamma\in \wt{G}(F).
$$
L'expression (8) est bien définie --- elle ne dépend pas des choix de $\bs{h}_1$ et $\bs{h}_2$ --- et c'est un produit scalaire hermitien défini positif sur $\bs{I}_{\mu,{\rm cusp}}(\wt{G}(F),\omega)$. De plus, ce produit ne dépend que de $dg$ et $da$, et varie comme $dg^2da^{-2}$. 
L'espace $\bs{D}_{\mu,{\rm ell}}(\wt{G}(F),\omega)$ s'identifie, de fa\c{c}on naturelle, à un sous--espace du dual linéaire de $\bs{I}_{\mu,{\rm cusp}}(\wt{G}(F),\omega)$: pour $\Theta\in \bs{D}_{\mu,{\rm ell}}(\wt{G}(F),\omega)$ et $\bs{f}\in \bs{I}_{\mu,{\rm cusp}}(\wt{G}(F),\omega)$, on choisit un élément $\bs{h}\in \bs{I}_{\rm cusp}(\wt{G}(F),\omega)$ tel que $\bs{p}_\mu(\bs{h})=\bs{f}$ et l'on pose $\langle \Theta, \bs{f}\rangle_\mu = \Theta(\bs{h})$. Pour $\Theta\in \bs{D}_{\mu,{\rm ell}}(\wt{G}(F),\omega)$, il existe un unique $\iota_\mu(\Theta)\in \bs{I}_{\mu,{\rm cusp}}(\wt{G}(F),\omega)$ tel que
$$
\langle \Theta,\bs{f}\rangle_\mu= (\iota_\mu(\Theta), \bs{f})_{\mu,{\rm ell}},\quad \bs{f}\in \bs{I}_{\mu,{\rm cusp}}(\wt{G}(F),\omega).
$$
L'application
$$
\iota_\mu : \bs{D}_{\mu,{\rm ell}}(\wt{G}(F),\omega)\rightarrow \bs{I}_{\mu,{\rm cusp}}(\wt{G}(F),\omega)
$$
ainsi définie est un isomorphisme antilinéaire. Il dépend de $dg$ et $da$, et varie comme $da dg^{-1}$. De plus, d'après \cite[7.2, 7.3]{W}, pour 
$\Theta_1,\, \Theta_2\in \bs{D}_{\mu,{\rm ell}}(\wt{G}(F),\omega)$, on a l'égalité
$$
(\iota_\mu(\Theta_2),\iota_\mu(\Theta_1))_{\mu,{\rm ell}}=  (\Theta_1,\Theta_2)_{\rm ell}.
\leqno{(9)}
$$
}
\end{marema1}

\begin{marema2}
{\rm Supposons que $(\wt{G},\bs{a})$ est à torsion intérieure, \cad que l'espace tordu $(G,\wt{G})$ est à torsion intérieure, et que le caractère $\omega$ est trivial. On supprime $\omega$ dans les notations précédentes, et l'on note $\bs{SI}_{\rm cusp}(\wt{G}(F))$ le sous--espace de $\bs{SI}(\wt{G}(F))$ formé des images des fonctions $f\in C^\infty_{\rm c,cusp}(\wt{G}(F))$. On verra plus loin (cf. la remarque 3 de \ref{décomposition endoscopique}) que l'espace $\bs{SI}_{\rm cusp}(\wt{G}(F))$ s'identifie à un sous--espace de $\bs{I}_{\rm cusp}(\wt{G}(F))$. Pour un caractère unitaire $\mu$ de $A_{\smash{\wt{G}}}(F)$, avec une définition naturelle de l'espace $\bs{SI}_\mu(\wt{G}(F))$, on note 
$\bs{SI}_{\mu,{\rm cusp}}(\wt{G}(F))$ le sous--espace de $\bs{SI}_\mu(\wt{G}(F))$ formé des images des fonctions $f\in C^\infty_{\mu,{\rm cusp}}(\wt{G}(F))$. Tout comme pour $\bs{SI}_{\rm cusp}(\wt{G}(F))$, l'espace $\bs{SI}_{\mu,{\rm cusp}}(\wt{G}(F))$ s'identifie à un sous--espace de $\bs{I}_{\mu,{\rm cusp}}(\wt{G}(F))$ --- voir plus loin (\ref{variante à caractère central}, lemme 1). Notons $\bs{SD}_{\mu,{\rm ell}}(\wt{G}(F))$ le sous--espace de $\bs{D}_{\mu,{\rm ell}}(\wt{G}(F))$ formé des distributions qui sont stables. L'isomorphisme antilinéaire $\iota_\mu: \bs{D}_{\mu,{\rm ell}}(\wt{G}(F))\rightarrow \bs{I}_{\mu,{\rm cusp}}(\wt{G}(F))$ de la remarque 1 induit par restriction un isomorphisme antilinéaire
$$
{\jmath}_\mu:\bs{SD}_{\mu,{\rm ell}}(\wt{G}(F))\rightarrow \bs{SI}_{\mu,{\rm cusp}}(\wt{G}(F)).
$$
Les produits scalaires hermitiens sur $\bs{D}_{\mu,{\rm ell}}(\wt{G}(F))$ et 
$\bs{I}_{\mu,{\rm cusp}}(\wt{G}(F))$ se restreignent en des produits scalaires hermitiens sur 
$\bs{SD}_{\mu,{\rm ell}}(\wt{G}(F))$ et 
$\bs{SI}_{\mu,{\rm cusp}}(\wt{G}(F))$, vérifiant l'égalité (9).
}
\end{marema2}

\subsection{DŽcomposition endoscopique de l'espace $\bs{D}_{\rm ell}(\wt{G}(F),\omega)$}
\label{décomposition endoscopique}
On a vu (\ref{transfert spectral} et \ref{transfert spectral elliptique}) que si $\bs{G}'=(G',\ES{G}',\tilde{s})$ est une donnée endoscopique elliptique non ramifiée pour $(\wt{G},\bs{a})$, on a un homomorphisme de transfert spectral (usuel)
$$
\bs{\rm T}_{{\bs G}'}:\bs{SD}(\wt{G}'(F))\rightarrow \bs{D}(\wt{G}(F),\omega),
$$
qui envoie $\bs{SD}_{\rm ell}(\wt{G}'(F))$ dans $\bs{D}_{\rm ell}(\wt{G}(F),\omega)$. 

Pour les données endoscopiques elliptiques relevantes pour $(\wt{G},\bs{a})$ qui sont {\it ramifiées}, \cad qui ne sont pas non ramifiées, on dispose encore d'un tel homomorphisme, mais sa définition est moins directe que dans le cas non ramifié. Soit $\bs{G}'=(G',\ES{G}',\tilde{s})$ une telle donnée. Choisissons des données auxiliaires $G'_1$, $\wt{G}'_1$, $C_1$, $\hat{\xi}_1$, et un facteur de transfert $\Delta_1: \ES{D}_1\rightarrow {\Bbb C}^\times$ comme en \cite[2.1]{Stab I}. Rappelons que $\ES{D}_1$ est l'ensemble des couples $(\delta_1,\gamma)\in 
\wt{G}'_1(F)\times \wt{G}(F)$ tels que $(\delta,\gamma)\in \ES{D}(\bs{G}')$, où $\delta$ est l'image 
de $\delta_1$ dans $\wt{G}'(F)$. Pour $c_1\in C_1(F)$ et $g\in G(F)$, on a l'égalité
$$
\Delta_1(c_1\delta_1, g^{-1}\gamma g)= \lambda_1(c_1)^{-1}\omega(g)\Delta(\delta_1,\gamma),\quad 
(\delta_1,\gamma)\in \ES{D}_1.
$$
Choisissons aussi des mesures de Haar $dg$ sur $G(F)$ et $dg'$ sur $G'(F)$. Ces choix permettent de définir, comme en \cite[2.4]{Stab I}, un homomorphisme de transfert géométrique
$$
\bs{I}(\wt{G}(F),\omega)\rightarrow \bs{SI}_{\lambda_1}(\wt{G}'_1(F)),\, \bs{f}\mapsto \bs{f}^{\wt{G}'_1}.\leqno{(1)}
$$
Ici $\lambda_1$ est le caractère de $C_1(F)$ défini par le plongement $\hat{\xi}_1: \ES{G}'\rightarrow {^L{G_1}}$ (cf. \cite[2.1]{Stab I}), et $\bs{SI}_{\lambda_1}(\wt{G}'_1(F))$ est le quotient de 
$C^\infty_{\lambda_1}(\wt{G}'_1(F))$ par le sous--espace annulé par les distributions 
$S^{\wt{G}'_1\!}(\delta_1,\cdot)$ pour $\delta_1\in \wt{G}'(F)$ fortement régulier; 
où $C^\infty_{\lambda_1}(\wt{G}'_1(F))$ est l'espace des fonctions sur $\wt{G}'_1(F)$ qui sont localement constantes et à support compact modulo $C_1(F)$, et se transforment suivant $\lambda_1^{-1}$ sur $C_1(F)$.

\begin{marema1}
{\rm 
Pour définir l'homomorphisme de transfert géométrique (1), on a choisi des mesures de Haar $dg$ sur $G(F)$ et $dg'$ sur $G'(F)$. Pour éviter d'avoir à choisir ces mesures, on peut y incorporer les espaces ${\rm Mes}(G(F))$ et ${\rm Mes}(G'(F))$, où ${\rm Mes}(G(F))$ désigne la droite complexe portée par une mesure de Haar sur $G(F)$, et voir le transfert comme une application linéaire \cite[2.4]{Stab I}
$$
\bs{I}(\wt{G}(F),\omega)\otimes {\rm Mes}(G(F))\rightarrow \bs{SI}_{\lambda_1}(\wt{G}'_1(F))\otimes {\rm Mes}(G'(F)).
$$
}
\end{marema1}

\begin{marema2}
{\rm Pour une donnée $\bs{G}'$ non ramifiée, on retrouve bien sûr l'homomorphisme de transfert géométrique (1) de \ref{transfert nr} en prenant $G'_1=G'$, $\wt{G}'_1=\wt{G}'$, $\hat{\xi}_1=\hat{\xi}: \ES{G}'\buildrel\simeq\over{\longrightarrow} {^LG}$, et en normalisant les mesures de Haar $dg$ sur $G(F)$ et $dg'$ sur $G'(F)$, et le facteur de transfert $\Delta: \ES{D}(\bs{G}')\rightarrow {\Bbb C}$, à l'aide du sous--espace hyperspécial $(K,\wt{K})$ de $G(F)$. Pour une donnée $\bs{G}'$ ramifiée, on peut aussi, comme dans \cite[2.5]{Stab I}, s'affranchir des données auxiliaires $G'_1$, $\wt{G}'_1$, $C_1$, $\hat{\xi}_1$ et $\Delta_1$ en prenant la limite inductive des espaces $\bs{SI}_{\lambda_1}(\wt{G}'(F))$ sur toutes ces données. Mais nous n'en avons pas vraiment besoin ici.
}
\end{marema2}

On fixe un ensemble de représentants $\mathfrak{E}=\mathfrak{E}(\wt{G},\bs{a})$ des classes d'isomorphisme de données endoscopiques elliptiques et relevantes pour $(\wt{G},\bs{a})$. On note 
$\mathfrak{E}_{\rm nr}$ le sous--ensemble de $\mathfrak{E}$ formé des données qui sont non ramifiées, et l'on pose $\mathfrak{E}_{\rm ram}= \mathfrak{E}\smallsetminus \mathfrak{E}_{\rm nr}$. 
Pour chaque $\bs{G}'\in \mathfrak{G}_{\rm ram}$, on fixe aussi des données auxiliaires 
$G'_1,\,\ldots ,\,\hat{\xi}_1$ de sorte que le caractère $\lambda_1$ de $C_1(F)$ soit unitaire, un facteur de transfert $\Delta_1:\ES{D}_1\rightarrow {\Bbb C}^\times$, et une mesure de Haar $dg'$ sur $G'(F)$. Sur $G(F)$, on prend la mesure Haar $dg$ normalisée par $K$. On note $\bs{D}_{\lambda_1}(\wt{G}'_1(F))$ l'espace des formes linéaires sur l'espace $C^\infty_{\lambda_1}(\wt{G}'_1(F))$ qui sont des combinaisons linéaires (finies, à coefficients complexes) de traces $\Theta_{\tilde{\pi}'_1}$ où $(\pi'_1,\tilde{\pi}'_1)$ est une représentation de $\wt{G}'_1(F)$ telle que $\pi'_1$ est une représentation irréductible et tempérée de $G'_1(F)$ se transformant suivant le caractère $\lambda_1$ sur $C_1(F)$. Pour une telle représentation $(\pi'_1,\tilde{\pi}'_1)$ de $\wt{G}'_1(F)$, la forme linéaire $\Theta_{\tilde{\pi}'_1}$ sur 
$C^\infty_{\lambda_1}(\wt{G}'_1(F))$ est définie par 
$$
\Theta_{\tilde{\pi}'_1}(f'_1)= {\rm tr}(\tilde{\pi}'_1(f'_1d\delta)),\quad 
f'_1\in C^\infty_{\lambda_1}(\wt{G}'_1(F)),
$$
où $d\delta$ est la mesure positive, $G'(F)$--invariante à gauche et à droite, sur $\wt{G}'(F)$ 
associée à $dg'$, et $\tilde{\pi}'_1(f'_1d\delta)$ est l'opérateur sur l'espace de $\pi'_1$ donné par
$$
\tilde{\pi}'_1(f'_1d\delta)= \int_{\wt{G}'(F)}f'_1(\delta)\tilde{\pi}'_1(\delta)d\delta.
$$
Puisque sur $C_1(F)$, la fonction $f'_1$ se transforme suivant $\lambda_1^{-1}$ et la représentation $\tilde{\pi}'_1$ se transforme suivant $\lambda_1$, l'intégrale ci--dessus est bien définie (d'ailleurs c'est une somme finie). 
L'espace $\bs{D}_{\lambda_1}(\wt{G}'_1(F))$ est un sous--espace de $\bs{D}(\wt{G}'_1(F))$. On pose
$$
\bs{SD}_{\lambda_1}(\wt{G}'_1(F))= \bs{D}_{\lambda_1}(\wt{G}'_1(F)
\cap \bs{SD}(\wt{G}'_1(F))$$
et
$$\bs{SD}_{\lambda_1,{\rm ell}}(\wt{G}'_1(F))= \bs{D}_{\lambda_1}(\wt{G}'_1(F)
\cap \bs{SD}_{\rm ell}(\wt{G}'_1(F)).$$
D'après \cite{Moe}, le transfert géométrique (1) définit dualement un homomorphisme de transfert spectral
$$
\bs{\rm T}_{\bs{G}'}:\bs{SD}_{\lambda_1}(\wt{G}'_1(F))\rightarrow \bs{D}(\wt{G}(F),\omega),\leqno{(2)}
$$
qui envoie $\bs{SD}_{\lambda_1,{\rm ell}}(\wt{G}'_1(F))$ dans $\bs{D}_{\rm ell}(\wt{G}(F),\omega)$. Comme dans le cas non ramifié, on note $\bs{\rm T}_{\bs{G}',{\rm ell}}$ la restriction de $\bs{\rm T}_{\bs{G}'}$ à $\bs{SD}_{\lambda_1,{\rm ell}}(\wt{G}'_1(F))$. 

Pour $\bs{G}'\in \mathfrak{E}$, on pose
$$
\bs{SD}_{\rm ell}(\bs{G}')=\left\{
\begin{array}{ll}
\bs{SD}_{\rm ell}(\wt{G}'(F))& \mbox{si $\bs{G}'\in \mathfrak{E}_{\rm nr}$}\\
\bs{SD}_{\lambda_1,{\rm ell}}(\wt{G}'_1(F)) & \mbox{si $\bs{G}'\in \mathfrak{E}_{\rm ram}$}
\end{array}
\right..
$$
De même, on pose
$$
\bs{SI}_{\rm cusp}(\bs{G}')=\left\{
\begin{array}{ll}
\bs{SI}_{\rm cusp}(\wt{G}'(F))& \mbox{si $\bs{G}'\in \mathfrak{E}_{\rm nr}$}\\
\bs{SI}_{\lambda_1,{\rm cusp}}(\wt{G}'_1(F)) & \mbox{si $\bs{G}'\in \mathfrak{E}_{\rm ram}$}
\end{array}
\right.,
$$
où l'indice ``cusp'' désigne le sous--espace formé des images dans $\bs{SI}(\wt{G}'(F))$ ou 
$\bs{SI}_{\lambda_1}(\wt{G}'_1(F))$ des fonctions dans $C^\infty_{\rm c}(\wt{G}'(F))$ ou $C^\infty_{\lambda_1}(\wt{G}'_1(F))$ qui sont cuspidales. 
D'après la proposition de \cite[4.11]{Stab I}, les homomorphismes de transfert géométrique (1) pour $\bs{G}\in \mathfrak{E}$, induisent un isomorphisme
$$
\bs{I}_{\rm cusp}(\wt{G}(F),\omega) \rightarrow \bigoplus_{\bs{G}'\in \mathfrak{E}}\bs{SI}_{\rm cusp}(\bs{G}')^{{\rm Aut}(\bs{G}')},\, \bs{f}\mapsto \oplus_{\bs{G}'} \bs{f}^{\bs{G}'},\leqno{(3)}
$$
où $X^{{\rm Aut}(\bs{G}')}\subset X$ est le sous--espace des invariants sous ${\rm Aut}(\bs{G}')$, et où l'on a posé $\bs{f}^{\bs{G}'}= \bs{f}^{\wt{G}'_1}$. 

\begin{marema3}
{\rm 
Dans le cas où $(\wt{G},\bs{a})$ est à torsion intérieure (cf. la remarque 2 de \ref{triplets elliptiques essentiels}), l'isomorphisme (3) identifie l'espace $\bs{SI}_{\rm cusp}(\bs{G}')= \bs{SI}_{\rm cusp}(\wt{G}(F))$ à un sous--espace de $\bs{I}_{\rm cusp}(\wt{G}(F))$. C'est le sous--espace formé 
des images des fonctions $f\in C^\infty_{c,{\rm cusp}}(\wt{G}(F))$ dont les intégrales orbitales sont constantes sur toute classe de conjugaison stable fortement régulière.
}
\end{marema3}

Dualement, on obtient que les homomorphismes de transfert spectral elliptique $\bs{\rm T}_{{\bs G}'\!,{\rm ell}}$ pour $\bs{G}'\in \mathfrak{E}$, induisent un homomorphisme
$$
\oplus_{{\bs G}'\in \mathfrak{G}}\bs{\rm T}_{{\bs G}'\!,{\rm ell}}:\bigoplus_{\bs{G}'\in \mathfrak{G}}\bs{SD}_{\rm ell}(\bs{G}')_{{\rm Aut}(\bs{G}')}\rightarrow \bs{D}_{\rm ell}(\wt{G}'(F),\omega),\leqno{(4)}
$$ 
où $X_{{\rm Aut}(\bs{G}')}$ est l'espace des coinvariants de ${\rm Aut}(\bs{G}')$ dans $X$. Puisque le groupe ${\rm Aut}(\bs{G}')$ agit sur $\bs{SD}_{\rm ell}(\bs{G}')$ par un quotient fini, l'espace des coinvariants $\bs{SD}_{\rm ell}(\bs{G}')_{{\rm Aut}(\bs{G}')}$ s'identifie au sous--espace des invariants $\bs{SD}_{\rm ell}(\bs{G}')^{{\rm Aut}(\bs{G}')}\subset \bs{SD}_{\rm ell}(\bs{G}')$, et l'on a une projection naturelle $\bs{SD}_{\rm ell}(\wt{G}'(F))\rightarrow \bs{SD}_{\rm ell}(\wt{G}'(F))^{{\rm Aut}(\bs{G}')}$. D'après \cite{Moe} --- voir aussi \ref{variante à caractère central} ---, l'homomorphisme (4) est un isomorphisme. En d'autres termes, pour $\bs{G}'\in \mathfrak{G}$, le transfert spectral elliptique $\bs{\rm T}_{{\bs G}'\!,{\rm ell}}: \bs{SD}_{\rm ell}(\bs{G}')\rightarrow \bs{D}_{\rm ell}(\wt{G}(F),\omega)$ se factorise à travers la projection sur $\bs{SD}_{\rm ell}(\bs{G}')^{{\rm Aut}(\bs{G}')}$, et toute distribution $\Theta\in \bs{D}_{\rm ell}(\wt{G}(F),\omega)$ se décompose de manière unique en
$$
\Theta = \sum_{{\bs G}'\in \mathfrak{G}}\bs{\rm T}_{{\bs G}'}(\Theta^{{\bs G}'})\leqno{(5)}
$$ 
où $\Theta^{{\bs G}'}$ est un élément de 
$\bs{SD}_{\rm ell}(\bs{G}')^{{\rm Aut}({\bs G}')}$. 

\begin{marema4}
{\rm Rappelons que l'on a noté $G_\sharp$ le groupe $G/Z(G)^\theta$. On a vu (\ref{données endoscopiques nr}) qu'à une donnée $\bs{G}'\in \mathfrak{E}_{\rm nr}$ est associé un caractère $\omega'=\omega_{{\bs G}'}$ de $G_\sharp(F)$ vérifiant l'égalité (3) de \ref{données endoscopiques nr}. De la même manière \cite[2.7]{Stab I}, à une donnée $\bs{G}'\in \mathfrak{E}_{\rm ram}$ est associé un caractère $\omega_{{\bs G}'}$ de $G_\sharp(F)$ vérifiant l'égalité
$$
\Delta(\delta_1,{\rm Int}_{g^{-1}}(\gamma))= \omega_{\bs{G}'}(g)\Delta_1(\delta_1,\gamma), \quad (\delta_1,\gamma)\in \ES{D}_1,\, g\in G_\sharp(F).
$$
On en déduit que pour $\bs{G}'\in \mathfrak{E}$ et $\Theta'\in \bs{SD}_{\rm ell}(\bs{G}')$, le transfert $\bs{\rm T}_{{\bs G}'}(\Theta')$ est une distribution sur $\wt{G}(F)$ vérifiant l'égalité
$$
\bs{\rm T}_{{\bs G}'}(\Theta')({^g\!f})= \omega_{{\bs G}'}(g)\bs{\rm T}_{{\bs G}'}(\Theta')(f),\quad 
f\in C^\infty_{\rm c}(\wt{G}(F)),\, g\in G_\sharp(F);
$$
où l'on a posé ${^g\!f}= f\circ {\rm Int}_{g^{-1}}$. Par conséquent, si $\Theta\in \bs{D}_{\rm ell}(\wt{G}(F),\omega)$ est une distribution sur $\wt{G}(F)$ se transformant suivant un caractère $\xi$ de $G_\sharp(F)$ sous l'action de $G_\sharp(F)$ par conjugaison, \cad telle que $\Theta({^g\!f})=\xi(g)\Theta(f)$ pour tout $f\in C^\infty_{\rm c}(\wt{G}(F))$ et tout $g\in G_\sharp(F)$, alors seules les données $\bs{G}'\in \mathfrak{E}$ telles que $\omega_{{\bs G}'}=\xi$ peuvent donner une contribution non triviale à la décomposition (5) de $\Theta$.
On sait, d'après \cite[2.1.(3)]{Stab VII}\footnote{Le numéro 2.1 de \cite{Stab VII} est écrit sous certaines hypothèses sur $F$, mais le résultat cité n'en dépend pas.} appliqué à $(\wt{M},\bs{a}_M)=(\wt{G},\bs{a})$, que pour $\bs{G}'\in \mathfrak{E}$, le caractère $\omega_{{\bs G}'}$ de $G_\sharp(F)$ est non ramifié si et seulement si $\bs{G}'\in \mathfrak{E}_{\rm nr}$. On obtient en particulier que si $\Theta\in \bs{D}_{\rm ell}(\wt{G}(F),\omega)$ est une distribution sur $\wt{G}(F)$ se transformant suivant un caractère {\it non ramifié} $\xi$ de $G_\sharp(F)$ sous l'action de $G_\sharp(F)$ par conjugaison, alors seules les données $\bs{G}'\in \mathfrak{E}_{\rm nr}$ (telles que $\omega_{{\bs G}'}=\xi$) peuvent donner une contribution non triviale à la décomposition 
(5) de $\Theta$.
}
\end{marema4}

\subsection{Variante à caractère central}\label{variante à caractère central}

\vskip1mm
On a aussi une variante à caractère central des isomorphismes (3) et (4) de 
\ref{décomposition endoscopique}. Soit $\mathcal{Z}$ un sous--groupe fermé de $Z(G;F)^\theta$ --- par exemple le groupe $A_{\smash{\wt{G}}}(F)$ ---, que l'on munit d'une mesure de Haar $dz$. Soit $\mu$ un caractère unitaire de $\mathcal{Z}$. On reprend, en les affublant d'un indice $\mathcal{Z}$, les définitions de la remarque 1 de \ref{triplets elliptiques essentiels}. On définit l'espace $C^\infty_{\mathcal{Z},\mu}(\wt{G}(F))$ des fonctions sur $\wt{G}(F)$, localement constantes et à support compact modulo $\mathcal{Z}$, telles que $f^{[z]}(\gamma)=\mu(z)^{-1}f(\gamma)$ pour tout $(z,\gamma)\in \mathcal{Z}\times \wt{G}(F)$, où l'on a posé $f^{[z]}(\gamma)= f(z\gamma)=f(\gamma z)$. On note $\bs{I}_{\mathcal{Z},\mu}(\wt{G}(F),\omega)$ le quotient de $C^\infty_{\mathcal{Z},\mu}(\wt{G}(F))$ par le sous--espace des fonctions dont toutes les $\omega$--intégrales orbitales fortement régulières sont nulles. On note $C^\infty_{\mathcal{Z},\mu,{\rm cusp}}(\wt{G}(F))$ le sous--espace de 
$C^\infty_{\mathcal{Z},\mu}(\wt{G}(F))$ formé des fonctions qui sont cuspidales, et $\bs{I}_{\mathcal{Z},\mu,{\rm cusp}}(\wt{G}(F),\omega)$ sa projection sur $\bs{I}_{\mathcal{Z},\mu}(\wt{G}(F),\omega)$. 
L'application
$$p_{\mathcal{Z},\mu}: C^\infty_{\rm c}(\wt{G}(F))\rightarrow C^\infty_{\mathcal{Z},\mu}(\wt{G}(F))
$$ définie par
$$
p_{\mathcal{Z},\mu}(f)(\gamma)= \int_{\mathcal{Z}}f(z\gamma)\mu(z)dz,\quad \gamma\in \wt{G}(F),
$$
est surjective. Elle induit, par passage aux quotients, un homomorphisme surjectif
$$
\bs{p}_{\mathcal{Z},\mu}: \bs{I}(\wt{G}(F),\omega)\rightarrow \bs{I}_{\mathcal{Z},\mu} (\wt{G}(F),\omega)
$$
qui envoie $\bs{I}_{\rm cusp}(\wt{G}(F),\omega)$ sur $\bs{I}_{\mathcal{Z},\mu,{\rm cusp}} (\wt{G}(F),\omega)$. 

On définit le sous--espace $\bs{D}_{\mathcal{Z},\mu}(\wt{G}(F),\omega)\subset \bs{D}(\wt{G}(F),\omega)$ engendré par les traces $\Theta_{\tilde{\pi}}$ des représentations $G(F)$--irréductibles tempérées $\tilde{\pi}$ de $(\wt{G}(F),\omega)$ telles que $\tilde{\pi}(z\gamma)= \mu(z)\tilde{\pi}(\gamma)$ pour tout $(z,\gamma)\in \mathcal{Z}\times  \wt{G}(F)$. On pose
$$\bs{D}_{\mathcal{Z},\mu,{\rm ell}}(\wt{G}(F),\omega)=
\bs{D}_{\mathcal{Z},\mu}(\wt{G}(F),\omega)\cap \bs{D}_{\rm ell}(\wt{G}(F),\omega).
$$
Pour $\Theta\in \bs{D}_{\mathcal{Z},\mu}(\wt{G}(F),\omega)$ et $\bs{f}\in \bs{I}_{\mathcal{Z},\mu}(\wt{G}(F),\omega)$, on choisit un élément $\bs{h}\in \bs{I}(\wt{G}(F),\omega)$ tel que $\bs{p}_{\mathcal{Z},\mu}(\bs{h})=\bs{f}$, et l'on pose 
$\langle \Theta,\bs{f}\rangle_{\mathcal{Z},\mu}= \Theta(\bs{h})$. Si $\Theta= \Theta_{\wt{\pi}}$ pour une représentation $(\pi,\tilde{\pi})$ de $(\wt{G}(F),\omega)$ telle que $\pi$ est irréductible, tempérée, et de caractère central prolongeant $\mu$, alors pour toute fonction $f\in C^\infty_{\mathcal{Z},\mu}(\wt{G}(F))$ se projetant sur $\bs{f}\in \bs{I}_{\mathcal{Z},\mu}(\wt{G}(F),\omega)$, on a simplement
$$
\langle \Theta, \bs{f}\rangle_{\mathcal{Z},\mu}= {\rm trace}\left(\int_{\mathcal{Z}\backslash \wt{G}(F)}\tilde{\pi}(\gamma)f(\gamma) {d\gamma\over dz}\right).
$$
L'application $\Theta\mapsto \langle \Theta ,\cdot \rangle_{\mathcal{Z},\mu}$ identifie 
$\bs{D}_{\mathcal{Z},\mu,{\rm ell}}(\wt{G}(F),\omega)$ à un sous--espace du dual linéaire de l'espace $\bs{I}_{\mathcal{Z},\mu,{\rm cusp}}(\wt{G}(F),\omega)$.

Soit $\bs{G}'\in \mathfrak{E}$. Rappelons que l'on a fixé des données auxiliaires 
$G'_1,\,\ldots ,\,\hat{\xi}_1$ de sorte que le caractère $\lambda_1$ de $C_1(F)$ soit unitaire, un facteur de transfert $\Delta_1:\ES{D}_1\rightarrow {\Bbb C}^\times$, et une mesure de Haar $dg'$ sur $G'(F)$ --- si $\bs{G}'\in \mathfrak{E}_{\rm nr}$, on a $G'_1=G'$ et $\lambda_1=1$. On a un homomorphisme naturel $Z(G)\rightarrow Z(G')$, de noyau $(1-\theta)(Z(G))$. On peut former le produit fibré
$$\mathfrak{Z}= Z(G'_1)\times_{Z(G')}Z(G).
$$ 
Un élément de $\mathfrak{Z}(F)$ est une paire $(z'_1,z)$ dans $Z(G'_1;F)\times Z(G;F)$ telle que les images de $z'_1$ et $z$ dans $Z(G')$ co\"{\i}ncident. L'application $c_1\mapsto (c_1,1)$ identifie $C_1(F)$ à un sous--groupe de $\mathfrak{Z}(F)$, et il existe un caractère unitaire $\lambda_\mathfrak{Z}$ de $\mathfrak{Z}(F)$ prolongeant $\lambda_1$ tel que
$$
\Delta_1(z'_1\delta_1, z \gamma)= \lambda_\mathfrak{Z}(z'_1,z)^{-1}\Delta_1(\delta,\gamma),\quad (\delta_1,\gamma)\in \ES{D}_1,\, (z'_1,z)\in \mathfrak{Z}(F).
$$

\begin{marema}
{\rm On veut définir un transfert géométrique pour des fonctions localement constantes sur $\wt{G}(F)$ qui sont non plus à support compact, mais dans l'espace $C^\infty_{\mathcal{Z},\mu}(\wt{G}(F))$. 
Le cas qui nous intéresse tout particulièrement (voir \ref{produits scalaires elliptiques}) 
est celui où $\mathcal{Z}\cap A_{\smash{\wt{G}}}(F)$ est d'indice fini dans $A_{\smash{\wt{G}}}(F)$. Notons que l'homomorphisme naturel $Z(G)\rightarrow Z(G')$ induit par restriction un homomorphisme $A_{\smash{\wt{G}}}\rightarrow A_{G'}$ qui n'est en général pas injectif,  mais seulement de noyau fini.
}
\end{marema}

On note $\mathcal{Z}'$ l'image de $\mathcal{Z}$ dans $Z(G';F)$ par l'homomorphisme naturel $Z(G)\rightarrow Z(G')$, et $\mathcal{Z}'_1$ son image réciproque dans $Z(G'_1;F)$. On munit $G'_1(F)$ et $\mathcal{Z}'_1$ de mesures de Haar $dg'_1$ et $dz'_1$. On forme le produit fibré
$$
\mathfrak{Z}_{\mathcal{Z}}= \mathcal{Z}'_1\times_{\mathcal{Z}'}\mathcal{Z}\subset \mathfrak{Z}(F).
$$
Continuons avec le caractère unitaire $\mu$ de $\mathcal{Z}$ fixé plus haut. L'application
$$
(z'_1,z)\mapsto \lambda_{\mathfrak{Z}}(z'_1,z)\mu(z)
$$
est un caractère unitaire de $\mathfrak{Z}_{\mathcal{Z}}$. S'il ne se factorise pas par la projection $(z'_1,z)\mapsto z'_1$, on dit que le transfert à $\wt{G}'_1(F)$ de tout élément de $C^\infty_{\mathcal{Z},\mu}(\wt{G}(F))$ est nul. Supposons qu'il se factorise en un caractère $\mu'_1$ de $\mathcal{Z}'_1$. Avec une définition naturelle des espaces et $C^\infty_{\mathcal{Z}'_1,\mu'_1}(\wt{G}'_1(F))$ et $\bs{SI}_{\smash{\mathcal{Z}'_1,\mu'_1}}(\wt{G}'_1(F))$ --- cf. \ref{décomposition endoscopique} ---, 
on définit l'homomorphisme de transfert géométrique
$$
\bs{I}_{\mathcal{Z},\mu}(\wt{G}(F),\omega)\mapsto \bs{SI}_{\smash{\mathcal{Z}'_1,\mu'_1}}(\wt{G}'_1(F)),\, 
\bs{f}\mapsto \bs{f}^{\wt{G}'_1},\leqno{(1)}
$$
par la formule habituelle: pour $f\in C^\infty_{\mathcal{Z},\mu}(\wt{G}(F))$ et $f'_1\in C^\infty_{\mathcal{Z}'_1,\mu'_1}(\wt{G}'_1(F))$, on dit que $f'_1$ est un transfert de $f$ si pour tout 
$\delta_1\in \wt{G}'_1(F)$ dont l'image dans $\wt{G}'(F)$ est un élément fortement $\wt{G}$--régulier, on a l'égalité
$$
S^{\wt{G}'_1}(\delta_1,f'_1)= d_\theta^{1/2}\sum_\gamma d_\gamma^{-1}\Delta_1(\delta_1,\gamma)I^{\wt{G}}(\gamma,\omega,f)\leqno{(2)}
$$
où $\gamma$ parcourt les éléments de $\wt{G}(F)$ tels que $(\delta_1,\gamma)\in \ES{D}_1$, modulo conjugaison par $G(F)$. Pour que la formule (2) ait un sens, on a besoin de mesures. 
Pour $\gamma\in \wt{G}(F)$ tel que $(\delta_1,\gamma)\in \ES{D}_1$, on munit l'espace quotient 
$\mathcal{Z}\backslash G_\gamma(F)$ d'une mesure $dh$. On a une application
$$
\mathcal{Z}\backslash G_\gamma(F)\rightarrow \mathcal{Z}'_1\backslash G'_{1,\delta_1}(F)
$$
qui est un isomorphisme local, et on munit l'espace d'arrivée de la mesure $dh'_1$ déduite de $dh$. Alors
$$G_\gamma(F)\backslash G(F)\simeq (\mathcal{Z}\backslash G_\gamma(F))\backslash (\mathcal{Z}\backslash G(F))
$$
est muni de la mesure $dh \backslash (dz\backslash dg)$, et 
$$
G'_{1,\delta_1}(F)\backslash G'_1(F) \simeq (\mathcal{Z}'_1\backslash G'_{1,\delta_1}(F))\backslash (\mathcal{Z}'_1\backslash G'_1(F))
$$ 
est muni de la mesure $dh'_1\backslash (dz'_1\backslash dg'_1)$. Une autre définition équivalente est la suivante. Pour $\bs{f}\in \bs{I}_{\mathcal{Z},\mu}(\wt{G}(F),\omega)$, on choisit un élément $\bs{h} \in \bs{I}(\wt{G}(F),\omega)$ tel que $\bs{p}_{\mathcal{Z},\mu}(\bs{h})=\bs{f}$, et l'on note 
$\bs{h}^{\wt{G}'_1}\in \bs{SI}_{\lambda_1}(\wt{G}'_1(F))$ son transfert usuel, défini à l'aide de la mesure de Haar $dg'$ sur $G'(F)$. Rappelons que $\lambda_1$ est la restriction de $\lambda_{\mathfrak{Z}}$ à $C_1(F)$, identifié à un sous--groupe de $\mathfrak{Z}(F)=
Z(\wt{G}'_1;F)\times_{Z(G';F)}Z(G;F)$ par l'application 
$c_1\mapsto (c_1,1)$. Cet élément $\bs{h}^{\wt{G}'_1}$ dépend de $dg$ et $dg'$, et précisément varie comme $dg (dg')^{-1}$. De $dg'$ et $dg'_1$ se déduit une mesure de Haar $dc_1$ sur $C_1(F)$, laquelle définit un homomorphisme surjectif
$$
\bs{p}_{C_1,\lambda_1}: \bs{I}(\wt{G}'_1(F))\rightarrow \bs{I}_{\lambda_1}(\wt{G}'_1(F))
$$
qui envoie $\bs{SI}(\wt{G}'_1(F))$ sur $\bs{SI}_{\lambda_1}(\wt{G}'_1(F))$. On choisit un élément $\bs{h}'_1\in \bs{SI}(\wt{G}'_1(F))$ tel que $\bs{p}_{C_1,\lambda_1}(\bs{h}'_1)= \bs{h}^{\wt{G}'_1}$, et l'on pose $\bs{f}^{\wt{G}'_1}= \bs{p}_{\mathcal{Z}'_1,\mu'_1}(\bs{h}'_1)$. 
Notons que l'élément $\bs{f}^{\wt{G}'_1}$ dépend de $dz\backslash dg$ et $dz'_1\backslash dg'_1$, et varie comme $(dz\backslash dg)(dz'_1\backslash dg'_1)^{-1}$. L'espace $\bs{SI}_{\mathcal{Z}'_1,\mu'_1}(\wt{G}'_1(F))$ est naturellement muni d'une action du groupe ${\rm Aut}(\bs{G}')$, et l'image du transfert (1) est contenu dans le sous--espace $\bs{SI}_{\mathcal{Z}'_1,\mu'_1}(\wt{G}'_1(F))^{{\rm Aut}(\bs{G}')}\subset 
\bs{SI}_{\mathcal{Z}'_1,\mu'_1}(\wt{G}'_1(F))$ 
des invariants par cette action. On définit de fa\c{c}on naturelle le sous--espace 
$\bs{SI}_{\mathcal{Z}'_1,\mu'_1,{\rm cusp}}(\wt{G}'_1(F))$ de 
$\bs{SI}_{\mathcal{Z}'_1,\mu'_1}(\wt{G}'_1(F))$. Il est stable par ${\rm Aut}(\bs{G}')$, et pour $\bs{f}\in \bs{I}_{\mathcal{Z},\mu,{\rm cusp}}(\wt{G}(F),\omega)$, le transfert $\bs{f}^{\wt{G}'_1}$ appartient à l'espace $\bs{SI}_{\mathcal{Z}'_1,\mu'_1,{\rm cusp}}(\wt{G}'_1(F))^{{\rm Aut}(\bs{G}')}$. 

On note $\mathfrak{E}_{\mathcal{Z},\mu}$ le sous--ensemble de $\mathfrak{E}$ formé des données $\bs{G}'$ pour lesquelles le caractère $\mu'_1$ est défini, et pour $\bs{G}'\in \mathfrak{E}_{\mathcal{Z},\mu}$, on pose
$$
\bs{SI}_{\mathcal{Z},\mu, {\rm cusp}}(\bs{G}')= \bs{SI}_{\mathcal{Z}'_1,\mu'_1,{\rm cusp}}(\wt{G}'_1(F)).
$$
Alors on a la variante à caractère central suivante de l'isomorphisme (3) de \ref{triplets elliptiques essentiels}.

\begin{monlem1}
Soit $\mathcal{Z}$ un sous--groupe fermé de $Z(G;F)^\theta$, et soit $\mu$ un caractère unitaire de $\mathcal{Z}$. L'application
$$
\bs{I}_{\mathcal{Z},\mu,{\rm cusp}}(\wt{G}(F),\omega)\rightarrow \bigoplus_{\bs{G}'\in \mathfrak{E}_{\mathcal{Z},\mu}}\bs{SI}_{\mathcal{Z},\mu,{\rm cusp}}(\bs{G}')^{{\rm Aut}(\bs{G}')},\, 
\bs{f}\mapsto \oplus_{\bs{G}'}\bs{f}^{\bs{G}'}
$$
est un isomorphisme; où l'on a posé $\bs{f}^{\bs{G}'}= \bs{f}^{\wt{G}'_1}$.
\end{monlem1}

Pour $\bs{G}'\in \mathfrak{E}_{\mathcal{Z},\mu}$, 
le transfert géométrique (1) définit dualement un homomorphisme de transfert spectral
$$
\bs{\rm T}_{\bs{G}'\!,\mathcal{Z},\mu}: \bs{SD}_{\mathcal{Z}'_1,\mu'_1}(\wt{G}'_1(F))
\rightarrow \bs{D}_{\mathcal{Z},\mu}(\wt{G}(F),\omega)\leqno{(3)}
$$
qui se factorise à travers la projection naturelle
$$\bs{SD}_{\mathcal{Z}'_1,\mu'_1}(\wt{G}'_1(F))
\rightarrow \bs{SD}_{\mathcal{Z}'_1,\mu'_1}(\wt{G}'_1(F))^{{\rm Aut}(\bs{G}')}.
$$
Pour $\Theta'\in \bs{SD}_{\mathcal{Z}'_1,\mu'_1}(\wt{G}'_1(F))$, on a
$$
\langle \bs{\rm T}_{{\bs G}'\!,\mathcal{Z},\mu}(\Theta'),\bs{f}\rangle_{\mathcal{Z},\mu}=
\langle \Theta', \bs{f}^{\wt{G}'_1}\rangle_{\mathcal{Z}'_1,\mu'_1},\quad \bs{f}\in \bs{I}_{\mathcal{Z},\mu}(\wt{G}(F),\omega).\leqno{(4)}
$$
Les définitions impliquent que $\bs{\rm T}_{\bs{G}'\!,\mathcal{Z},\mu}$ n'est autre que la restriction de 
l'homomorphisme de transfert spectral usuel $\bs{\rm T}_{\bs{G}'}: \bs{SD}_{\lambda_1}(\wt{G}'_1(F))\rightarrow \bs{D}(\wt{G}(F),\omega)$ au sous--espace
$$
\bs{SD}_{\mathcal{Z}'_1,\mu'_1}(\wt{G}'_1(F))
\subset \bs{SD}_{\lambda_1}(\wt{G}'_1(F)).
$$
Il envoie $\bs{SD}_{\mathcal{Z}'_1,\mu'_1,{\rm ell}}(\wt{G}'_1(F))$ dans $\bs{D}_{\mathcal{Z},\mu,{\rm ell}}(\wt{G}(F),\omega)$, et l'on note $\bs{\rm T}_{\bs{G}'\!,\mathcal{Z},\mu,{\rm ell}}$ sa restriction à 
$\bs{SD}_{\mathcal{Z}'_1,\mu'_1,{\rm ell}}(\wt{G}'_1(F))$. Tout comme $\bs{\rm T}_{\bs{G}'\!,\mathcal{Z},\mu}$, l'homomorphisme $\bs{\rm T}_{\bs{G}'\!,\mathcal{Z},\mu,{\rm ell}}$ se factorise à travers la projection naturelle
$$\bs{SD}_{\mathcal{Z}'_1,\mu'_1,{\rm ell}}(\wt{G}'_1(F))
\rightarrow \bs{SD}_{\mathcal{Z}'_1,\mu'_1,{\rm ell}}(\wt{G}'_1(F))^{{\rm Aut}(\bs{G}')}.
$$

Pour $\bs{G}'\in \mathfrak{E}_{\mathcal{Z},\mu}$, on pose
$$
\bs{SD}_{\mathcal{Z},\mu,{\rm ell}}(\bs{G}')= \bs{SD}_{\mathcal{Z}'_1,\mu'_1,{\rm ell}}(\wt{G}'_1(F)).
$$
Alors on a la variante à caractère central suivante de l'isomorphisme (4) de \ref{triplets elliptiques essentiels}.

\begin{monlem2}
Soit $\mathcal{Z}$ un sous--groupe fermé de $Z(G;F)^\theta$, et soit $\mu$ un caractère unitaire de $\mathcal{Z}$. L'application
$$
\oplus_{\bs{G}'\in \mathfrak{E}_{\mathcal{Z},\mu}} \bs{\rm T}_{\bs{G}'\!,\mathcal{Z},\mu,{\rm ell}} :\bigoplus_{\bs{G}'\in \mathfrak{E}_{\mathcal{Z},\mu}}\bs{SD}_{\mathcal{Z},\mu,{\rm ell}}(\bs{G}')^{{\rm Aut}(\bs{G}')}\rightarrow \bs{D}_{\mathcal{Z},\mu,{\rm ell}}(\wt{G}(F),\omega)
$$
est un isomorphisme.
\end{monlem2}

\subsection{Produits scalaires elliptiques}\label{produits scalaires elliptiques}On munit $A_{\smash{\wt{G}}}(F)$ d'une mesure de Haar $da$. Pour $f_1,\,f_2\in C^\infty_{c,{\rm cusp}}(\wt{G}(F))$, on a défini dans la remarque 1 de \ref{triplets elliptiques essentiels} une expression $J^{\wt{G}}(\omega,f_1,f_2)$ (formule (7)), qui dépend de $dg$ et $da$ et varie comme $dg^2da^{-1}$. Soit $\mathcal{Z}$ un sous--groupe fermé de $Z(G;F)^\theta$. On suppose que $\mathcal{Z}\cap A_{\smash{\wt{G}}}(F)$ est d'indice fini dans $A_{\smash{\wt{G}}}(F)$, et l'on munit $\mathcal{Z}$ d'une mesure de Haar $dz$. Pour $f_1,\,f_2\in C^\infty_{c,{\rm cusp}}(\wt{G}(F))$, on définit l'expression $J^{\wt{G}}_{\mathcal{Z}}(\omega,f_1,f_2)$ en rempla\c{c}ant ${\rm vol}(A_{\smash{\wt{G}}}(F)\backslash G_\gamma(F))$ par ${\rm vol}(\mathcal{Z}\backslash G_\gamma(F))$ dans la formule pour $J^{\wt{G}}(\omega,f_1,f_2)$. Une définition équivalente consiste à fixer une mesure de Haar sur $\mathcal{Z}\cap A_{\smash{\wt{G}}}(F)$ et à poser
$$
J^{\wt{G}}_{\mathcal{Z}}(\omega,f_1,f_2)= {\rm vol}((\mathcal{Z}\cap A_{\smash{\wt{G}}}(F))\backslash A_{\smash{\wt{G}}}(F)){\rm vol}((\mathcal{Z}\cap A_{\smash{\wt{G}}}(F))\backslash \mathcal{Z})^{-1}J^{\wt{G}}(\omega,f_1,f_2). \leqno{(1)}
$$ 
L'expression (1) dépend de $dg$ et $dz$ (elle ne dépend plus de $da$), et varie comme $dg^2dz^{-1}$. Bien sûr on peut remplacer les fonctions $f_1$ et $f_2$ par leurs images dans $\bs{I}_{\rm cusp}(\wt{G}(F),\omega)$. 

Soit $\mu$ un caractère unitaire de $\mathcal{Z}$. Comme dans la remarque 1 de \ref{triplets elliptiques essentiels}, on munit l'espace $\bs{I}_{\mu,{\rm cusp}}(\wt{G}(F),\omega)$ d'un produit scalaire hermitien défini positif $(\cdot ,\cdot)_{\mathcal{Z},\mu,{\rm ell}}$, défini comme suit. Pour $\bs{f}_{\!1},\,\bs{f}_{\!2}\in \bs{I}_{\mathcal{Z},\mu,{\rm cusp}}(\wt{G}(F),\omega)$, on choisit $\bs{h}_1,\,\bs{h}_2\in \bs{I}_{\rm cusp}(\wt{G}(F),\omega)$ tels que $\bs{p}_{\mathcal{Z},\mu}(\bs{h}_i)=\bs{f}_{\!i}$ ($i=1,\, 2$), et l'on pose
$$
(\bs{f}_{\!1},\bs{f}_{\!2})_{\mathcal{Z},\mu,{\rm ell}}= \int_{\mathcal{Z}}J^{\wt{G}}_{\mathcal{Z}}(\omega,\bs{h}_1,\bs{h}_2^{[z]})\mu(z)dz. \leqno{(2)}
$$
L'expression (2) dépend de $dg$ et $dz$, et varie comme $dg^2dz^{-2}$. Pour $\Theta\in \bs{D}_{\mathcal{Z},\mu,{\rm ell}}(\wt{G}(F),\omega)$, il existe un unique $\iota_{\mathcal{Z},\mu}(\Theta)\in \bs{I}_{\mathcal{Z},\mu,{\rm cusp}}(\wt{G}(F),\omega)$ tel que
$$
\langle \Theta, \bs{f}\rangle_{\mathcal{Z},\mu}= (\iota_{\mathcal{Z},\mu}(\Theta),\bs{f})_{\mathcal{Z},\mu,{\rm ell}},\quad \bs{f}\in \bs{I}_{\mathcal{Z},\mu,{\rm cusp}}(\wt{G}(F),\omega).
$$
L'application
$$
\iota_{\mathcal{Z},\mu}: \bs{D}_{\mathcal{Z},\mu,{\rm ell}}(\wt{G}(F),\omega)\rightarrow \bs{I}_{\mathcal{Z},\mu,{\rm cusp}}(\wt{G}(F),\omega)
$$
ainsi définie est un isomorphisme antilinéaire. Il dépend de $dg$ et $dz$, et varie comme $dzdg^{-1}$. 
Pour $\Theta_1,\,\Theta_2\in \bs{D}_{\mathcal{Z},\mu,{\rm ell}}(\wt{G}(F),\omega)$, on pose
$$
(\Theta_1,\Theta_2)_{\mathcal{Z},\mu,{\rm ell}}= (\iota_{\mathcal{Z},\mu}(\Theta_2), \iota_{\mathcal{Z},\mu}(\Theta_1))_{\mathcal{Z},\mu,{\rm ell}}.\leqno{(3)}
$$
Cela définit un produit scalaire hermitien défini positif sur l'espace $\bs{D}_{\mathcal{Z},\mu,{\rm ell}}(\wt{G}(F),\omega)$. Comme
$$
\bs{D}_{\rm ell}(\wt{G}(F),\omega)= \bigoplus_{\mu'}\bs{D}_{\mathcal{Z},\mu'\!,{\rm ell}}(\wt{G}(F),\omega)
$$
où $\mu'$ parcourt tous les caractères unitaires de $\mathcal{Z}$, on peut munir l'espace l'espace $\bs{D}_{\rm ell}(\wt{G}(F),\omega)$ du produit scalaire hermitien défini positif qui est la somme directe des $(\cdot, \cdot)_{\mathcal{Z},\mu'\!,{\rm ell}}$. Notons $(\cdot,\cdot)_{\mathcal{Z},{\rm ell}}$ ce produit. Le lemme suivant résulte de la définition de ce produit. 

\begin{monlem1}
Le produit $(\cdot ,\cdot)_{\mathcal{Z},{\rm ell}}$ sur $\bs{D}_{\rm ell}(\wt{G}(F),\omega)$ ne dépend ni de $dg$, ni de $dz$, ni de $\mathcal{Z}$.
\end{monlem1}

En appliquant la construction ci--dessus à $\mathcal{Z}=A_{\smash{\wt{G}}}(F)$, on obtient que 
$(\cdot ,\cdot)_{\mathcal{Z},{\rm ell}}=(\cdot ,\cdot)_{\rm ell}$, où $(\cdot ,\cdot)_{\rm ell}$ est le produit de \cite[7.3]{W} introduit en \ref{triplets elliptiques essentiels}, qui a été calculé explicitement (formule (6) de \ref{triplets elliptiques essentiels}). 

\vskip1mm
Soit $\bs{G}'\in \mathfrak{E}_{\mathcal{Z},\mu}$. On définit $\mathcal{Z}'_1$ et le caractère unitaire $\mu'_1$ de $\mathcal{Z}'_1$ comme en \ref{variante à caractère central}. Notons que $\mathcal{Z}'_1$ est un sous--groupe fermé de $Z(G'_1;F)$ tel que $\mathcal {Z}'_1\cap A_{\smash{G'_1}}(F)$ est d'indice fini dans $A_{\smash{G'_1}}(F)$. On munit $G'_1(F)$ et $\mathcal{Z}'_1$ de mesures de Haar $dg'_1$ et $dz'_1$. On a un produit scalaire hermitien défini positif $(\cdot ,\cdot)_{\mathcal{Z}'_1,\mu'_1,{\rm ell}}$ sur $\bs{I}_{\mathcal{Z}'_1,\mu'_1,{\rm cusp}}(\wt{G}'_1(F))$, donné par la formule (2). Grâce au lemme 1 de \ref{variante à caractère central} appliqué à $\wt{G}=\wt{G}'_1$ et $\mathcal{Z}= \mathcal{Z}'_1$, on peut identifier $\bs{SI}_{\mathcal{Z}'_1,\mu'_1,{\rm cusp}}(\wt{G}'_1(F))$ à un sous--espace de $\bs{I}_{\mathcal{Z}'_1,\mu'_1,{\rm cusp}}(\wt{G}'_1(F))$ --- c'est le sous--espace formé des images des fonctions $f'_1\in C^\infty_{\mathcal{Z}'_1,\mu'_1,{\rm cusp}}(\wt{G}'_1(F))$ dont les intégrales orbitales sont constantes sur toute classe de conjugaison stable fortement régulière ---, ce qui munit $\bs{SI}_{\mathcal{Z}'_1,\mu'_1,{\rm cusp}}(\wt{G}'_1(F))$ de la restriction du produit scalaire hermitien sur $\bs{I}_{\mathcal{Z}'_1,\mu'_1,{\rm cusp}}(\wt{G}'_1(F))$. 

\vskip1mm
Pour $\bs{G}'\in \mathfrak{G}$, soit $c(\wt{G},\bs{G})$ la constante donnée par (cf. \cite[4.17]{Stab I} ou \cite[3.1.1]{Stab X})
\begin{eqnarray*}
c(\wt{G},\bs{G}')&=& \vert \det(1-\theta;\ES{A}_G/\ES{A}_{\smash{\wt{G}}})\vert^{-1}\vert \pi_0(Z(\hat{G})^{\Gamma_F})\vert \vert \pi_0(Z(\hat{G}')^{\Gamma_F})\vert^{-1} 
\times \cdots \\
&& \cdots \times 
\vert {\rm Out}(\bs{G}')\vert^{-1} \vert \pi_0(Z(\hat{G})^{\Gamma_F,\circ}\cap \hat{G}')\vert 
\vert \pi_0([Z(\hat{G})/(Z(\hat{G})\cap \hat{G}')]^{\Gamma_F})\vert^{-1}.
\end{eqnarray*}
\begin{monlem2}
Pour $\bs{f}_{\!1},\, \bs{f}_{\!2}\in \bs{I}_{\mathcal{Z},\mu,{\rm cusp}}(\wt{G}(F),\omega)$, 
on a l'égalité (conformément à la décomposition du lemme 1 de \ref{variante à caractère central})
$$
(\bs{f}_{\!1},\bs{f}_{\!2})_{\mathcal{Z}, \mu,{\rm ell}}= \sum_{\bs{G}'\in \mathfrak{E}_{\mathcal{Z},\mu}}
c(\wt{G},\bs{G}')(\bs{f}_{\!1}^{\bs{G}'}\!,\bs{f}_{\!2}^{\bs{G}'})_{\mathcal{Z}'_1,\mu'_1,{\rm ell}}.
$$
\end{monlem2} 

\begin{proof}Notons tout d'abord que l'égalité du lemme ne dépend pas des choix des mesures: le produit $(\cdot ,\cdot)_{\mathcal{Z},\mu,{\rm ell}}$ dépend de $dg$ et $dz$ et varie comme $dg^2 dz^{-2}$; pour $\bs{G}'\in \mathfrak{E}_{\mathcal{Z},\mu}$, le produit $(\cdot , \cdot)_{\mathcal{Z}'_1,\mu'_1,{\rm ell}}$ dépend de $dg'_1$ et $dz'_1$ et varie comme $(dg'_1)^2(dz'_1)^{-2}$, et l'homomorphisme de transfert $\bs{f}\mapsto \bs{f}^{\bs{G}'}= \bs{f}^{\wt{G}'_1}$ 
dépend de $dz\backslash dg$ et $dz'_1\backslash dg'_1$ et varie comme $(dz\backslash dg)(dz'_1\backslash dg'_1)^{-1}$. D'après la proposition de \cite[4.17]{Stab I}, pour $\bs{h}_1,\, \bs{h}_{2}\in 
\bs{I}_{\rm cusp}(\wt{G}(F), \omega)$, on a l'égalité
$$
J^{\wt{G}}(\omega, \bs{h}_{1},\bs{h}_{2})= \sum_{\bs{G}'\in \mathfrak{E}}c(\wt{G},\bs{G}')
J^{\wt{G}'_1}(\bs{h}_{1}^{\bs{G}'}\!,\bs{h}_{2}^{\bs{G}'}),\leqno{(4)}
$$
où l'expresssion $J^{\wt{G}'_1}(\bs{h}_{1}^{\bs{G}'}\!,\bs{h}_{2}^{\bs{G}'})$ est donnée par la formule (7) de \ref{triplets elliptiques essentiels} (remarque 1) pour $\wt{G}= \wt{G}'_1$ et $\omega=1$, ou, ce qui revient au même, par la formule (3) de \cite[4.17]{Stab I} modulo les choix de mesures de Haar $dg'$ sur $\wt{G}'(F)$ et $da'$ sur $A_{G'}(F)$. La mesure $dg'$ est celle définissant l'homomorphisme de transfert $\bs{h}\mapsto \bs{h}^{\bs{G}'}= \bs{h}^{\wt{G}'_1}$, et les mesures $da$ et $da'$ sont supposées se correspondre par l'isomorphisme naturel $\ES{A}_{\smash{\wt{G}}}\rightarrow \ES{A}_{G'}$ --- cf. \cite[4.17]{Stab I}. Pour $\bs{f}_{\!1},\,\bs{f}_{\!2}\in 
\bs{I}_{\mathcal{Z},\mu,{\rm cusp}}(\wt{G}(F),\omega)$, on commence par choisir des éléments $\bs{h}_1,\,\bs{h}_2\in \bs{I}_{\rm cusp}(\wt{G}(F),\omega)$ tels que $\bs{p}_{\mathcal{Z},\mu}(\bs{h}_i)=\bs{f}_{\!i}$ ($i=1,\,2$). Ensuite on remplace $J^{\wt{G}}(\omega,\bs{h}_1,\bs{h}_2)$ par $J^{\wt{G}}_{\mathcal{Z}}(\omega,\bs{h}_1,\bs{h}_2)$ grâce à l'égalité (1), puis on intégre sur $\mathcal{Z}$ comme dans l'égalité (2), ce qui nous donne $(\bs{f}_{\!1},\bs{f}_{\!2})_{\mathcal{Z}, \mu,{\rm ell}}$. Compte--tenu de la définition des transferts $\bs{f}_{\!i}\mapsto \bs{f}_{\!i}^{\bs{G}'}$ ($i=1,\,2$) --- cf. \ref{variante à caractère central} ---, en appliquant la même construction à chacun des termes de la somme dans l'égalité (4), on obtient l'égalité du lemme. 
\end{proof}

Passons au résultat dual qui nous intéresse. Pour $\bs{G}'\in \mathfrak{E}_{\mathcal{Z},\mu}$ et $\Theta'\in \bs{D}_{\mathcal{Z},\mu,{\rm ell}}(\bs{G}')$, on note $\iota_{\mathcal{Z},\mu}^{\bs{G}'}(\Theta')$ l'unique élément de $\bs{I}_{\mathcal{Z},\mu,{\rm cusp}}(\bs{G}')$ tel que
$$
\langle \Theta', \bs{f}'\rangle_{\mathcal{Z}'_1,\mu'_1} = (\iota_{\mathcal{Z},\mu}^{\bs{G}'}(\Theta'), \bs{f}')_{\mathcal{Z}'_1,\mu'_1,{\rm ell}},\quad \bs{f}'\in \bs{I}_{\mathcal{Z},\mu,{\rm cusp}}(\bs{G}'). 
$$
L'application
$$
\iota_{\mathcal{Z},\mu}^{\bs{G}'}: \bs{D}_{\mathcal{Z},\mu,{\rm ell}}(\bs{G}')\rightarrow \bs{I}_{\mathcal{Z},\mu,{\rm cusp}}(\bs{G}')
$$
ainsi définie est un isomorphisme antilinéaire. Il induit par restriction un isomorphisme anti\-linéaire
$$
\jmath_{\mathcal{Z},\mu}^{\bs{G}'}: \bs{SD}_{\mathcal{Z},\mu,{\rm ell}}(\bs{G}')\rightarrow \bs{SI}_{\mathcal{Z},\mu,{\rm cusp}}(\bs{G}')
$$
qui envoie $\bs{SD}_{\mathcal{Z},\mu,{\rm ell}}(\bs{G}')^{{\rm Aut}(\bs{G}')}$ sur 
$\bs{SI}_{\mathcal{Z},\mu,{\rm cusp}}(\bs{G}')^{{\rm Aut}(\bs{G}')}$. D'après le lemme 2 ci--dessus, 
lorsqu'on passe de la décomposition du lemme 1 de \ref{variante à caractère central} à la décomposition du lemme 2 de 
\ref{variante à caractère central} via les isomorphismes $\iota_{\mathcal{Z},\mu}$ et 
$\jmath_{\mathcal{Z},\mu}^{\bs{G}'}$, la constante $c(\wt{G},\bs{G}')$ appara\^{\i}t: pour 
$\bs{G}'\in \mathfrak{E}_{\mathcal{Z},\mu}$, on a
$$
\jmath_{\mathcal{Z},\mu}^{\bs{G}'}(\Theta^{\bs{G}'})=
c(\wt{G},\bs{G}')\iota_{\mathcal{Z},\mu}(\Theta)^{\bs{G}'},\quad 
\Theta\in \bs{D}_{\mathcal{Z},\mu,{\rm ell}}(\wt{G}(F),\omega).
$$
On en déduit que pour $\Theta_1,\,\Theta_2\in \bs{D}_{\mathcal{Z},\mu,{\rm ell}}(\wt{G}(F),\omega)$, 
on a l'égalité (conformément à la décom\-position du lemme 2 de \ref{variante à caractère central})
$$
(\Theta_1,\Theta_2)_{\mathcal{Z},\mu,{\rm ell}}= \sum_{\bs{G}'\in \mathfrak{E}_{\mathcal{Z},\mu}}
c(\wt{G},\bs{G}')^{-1} (\Theta_1^{\bs{G}'}\!,\Theta_2^{\bs{G}'})_{\mathcal{Z}'_1,\mu'_1,{\rm ell}}.
\leqno{(5)}
$$
En vertu du lemme 1 ci--dessus, on peut s'affranchir de $\mathcal{Z}$ et $\mu$. On obtient le résultat suivant.

\begin{monlem3}Pour $\Theta_1,\,\Theta_2\in \bs{D}_{\rm ell}(\wt{G}(F),\omega)$, on a l'égalité (conformément à la décomposition (5) de \ref{décomposition endoscopique})
$$
(\Theta_1,\Theta_2)_{\rm ell}= \sum_{\bs{G}'\in \mathfrak{E}}c(\wt{G},\bs{G}')^{-1}(\Theta_1^{\bs{G}'}\!,\Theta_2^{\bs{G}'})_{\rm ell}.
$$
\end{monlem3}

\subsection{Séries principales non ramifiées}\label{séries principales nr}
Posons $\hat{Z}=Z(\hat{G})$ et $\hat{Z}_{\rm sc}= Z(\hat{G}_{\rm SC})$, où $\hat{G}_{\rm SC}$ est le revêtement simplement connexe du groupe dérivé de $\hat{G}$. Rappelons que Langlands a construit des isomorphismes canoniques ${\rm H}^1(W_F,\hat{Z})\rightarrow G(F)^D$ et ${\rm H}^1(W_F,\hat{Z}_{\rm sc})\rightarrow G_{\rm AD}(F)^D$, où $X^D$ désigne le groupe des caractères (homomorphismes continus dans ${\Bbb C}^\times$) de $X$. On en déduit un isomorphisme canonique
$$
G_{\rm AD}(F)/q(G(F))\rightarrow \ker ({\rm H}^1(W_F,\hat{Z}_{\rm sc})\rightarrow {\rm H}^1(W_F,\hat{Z}))^D,\leqno{(1)}
$$
où $q:G\rightarrow G_{\rm AD}$ est 
l'homomorphisme naturel. 

Soit $\varphi:W_F\rightarrow {^LG}$ un paramètre de Langlands tempéré. 
Posons
$$
S_\varphi= \{g\in \hat{G}: {\rm Int}_g\circ \varphi =\varphi\},\quad 
S_{\varphi,{\rm ad}}= \{g\in \hat{G}_{\rm AD}: {\rm Int}_g\circ \varphi =\varphi\}.
$$
La suite exacte courte $1\rightarrow \hat{Z}_{\rm sc}\rightarrow \hat{G}_{\rm SC}\rightarrow \hat{G}_{\rm AD}\rightarrow 1$ fournit une suite exacte de cohomologie
$$
{\rm H}^0(W_F,\varphi, \hat{G}_{\rm SC})\rightarrow {\rm H}^0(W_F,\varphi,\hat{G}_{\rm AD})\rightarrow {\rm H}^1(W_F,\hat{Z}_{\rm sc}),
$$
où ${\rm H}^0(W_F,\varphi,\cdot)$ désigne l'ensemble des invariants sous $W_F$ pour l'action donnée par $\varphi$. En composant la flèche de droite avec l'application naturelle ${\rm H}^0(W_F,\varphi,\hat{G})\rightarrow {\rm H}^0(W_F,\varphi, \hat{G}_{\rm AD})$, on obtient une application 
$$
S_\varphi ={\rm H}^0(W_F,\varphi,\hat{G})\rightarrow \ker({\rm H}^1(W_F,\hat{Z}_{\rm sc})\rightarrow {\rm H}^1(W_F,\hat{Z}))
$$
qui est un homomorphisme continu. 
Son noyau contient donc la composante neutre $S_\varphi^\circ$ du groupe algébrique affine complexe $S_\varphi$. Comme il contient aussi $Z(\hat{G})^{\Gamma_F}$, cet homomorphisme se factorise en un homomorphisme
$$
\bs{S}(\varphi)=S_\varphi/S_\varphi^\circ Z(\hat{G})^{\Gamma_F}\rightarrow 
\ker ({\rm H}^1(W_F,\hat{Z}_{\rm sc})\rightarrow {\rm H}^1(W_F,\hat{Z})).\leqno{(2)}
$$
Par dualité de Pontryagin, on en déduit un homomorphisme
$$
G_{\rm AD}(F)/q(G(F))\rightarrow \bs{S}(\varphi)^D.\leqno{(3)}
$$

Supposons que le paramètre $\varphi$ est de la forme $\varphi = ({^LT}\hookrightarrow {^LG})\circ \varphi_\lambda^T$ pour un caractère unitaire $\lambda$ de $T(F)$ de paramètre $\varphi_\lambda^T:W_F\rightarrow {^LT}$ (cf. \ref{paramètres}). D'après Keys \cite{K}, on sait que le $R$--groupe $R^G(\lambda)$ est naturellement isomorphe au $S$-groupe $\bs{S}(\varphi)$. Posons $\hat{U}= 
\hat{T}^{\Gamma_F,\circ}$ et notons $\hat{N}$ le normalisateur de $\hat{T}$ dans $\hat{G}$. D'après \cite[2.5]{K}, $\hat{T}$ est le centralisateur de $\hat{U}$ dans $\hat{G}$, $\hat{U}$ est un tore maximal de $S_\varphi^\circ$, $\hat{T}\cap S_\varphi= \hat{U}\hat{Z}^{\Gamma_F}$, $\hat{T}\cap S_\varphi^\circ = \hat{U}$, et $\hat{N}\cap S_\varphi^\circ$ est le normalisateur de $\hat{U}$ dans $S_\varphi^\circ$. On en déduit une suite exacte courte \cite[2.6]{K}
$$
1\rightarrow (\hat{N}\cap S_\varphi^\circ)/\hat{U}\rightarrow (\hat{N}\cap S_\varphi)/\hat{U}\hat{Z}^{\Gamma_F}\rightarrow \bs{S}(\varphi)\rightarrow 1.\leqno{(4)}
$$
De plus (loc.~cit.), on a un isomorphisme naturel
$$
W^G(\lambda)\simeq (\hat{N}\cap S_\varphi)/\hat{U}\hat{Z}^{\Gamma_F},
$$
qui se restreint en un isomorphisme
$$
W^G_0(\lambda)\simeq (\hat{N}\cap S_\varphi^\circ )/\hat{U}.
$$
D'où l'isomorphisme $ R^G(\lambda)\simeq \bs{S}(\varphi)$. Notons aussi que
$$
W^G_0(\lambda)=\{1\} \Leftrightarrow  S_\varphi^\circ = \hat{T}^{\Gamma_F,\circ}.\leqno{(5)}
$$

\begin{marema1}
{\rm Tout élément $w\in W^{\Gamma_F}=N_{G(F)}(T)/T(F)$ définit un $L$--isomorphisme ${^Lw}: {^LT}\rightarrow {^LT}$, donné par ${^Lw}(t\rtimes \sigma)= \hat{w}(t)\rtimes \sigma$, et 
d'après la correspondance de Langlands pour les tores, $W^G(\lambda)$ est le groupe des éléments $w\in W^{\Gamma_F}$ tels que ${^Lw}\circ \varphi_\lambda^T$ est $\hat{T}$--conjugué à $\varphi_\lambda^T$. On sait que tout élément $\Gamma_F$--invariant de $W^{\hat{G}}(\hat{T})=\hat{N}/\hat{T}$ se relève en un élément $\Gamma_F$--invariant de $\hat{N}$. Pour $w\in W^G(\lambda)$, choisissons un relèvement 
$n\in \hat{N}^{\Gamma_F}$ de $\hat{w}$. Alors ${\rm Int}_n \circ \varphi_\lambda^T$ est $\hat{T}$--conjugué à $\varphi_\lambda^T$. En d'autres termes, il existe un $t\in \hat{T}$ tel que $tn\in S_\varphi$. L'application $w\mapsto tn$ donne une bijection
$$
\iota_\lambda: W^G(\lambda)\buildrel \simeq\over{\longrightarrow} (\hat{N}\cap S_\varphi)\hat{T}/\hat{T}= (\hat{N}\cap S_\varphi)/ \hat{U}\hat{Z}
$$
qui est un isomorphisme sur l'opposé du groupe d'arrivée (par construction, $w^{-1}\;(= \hat{w})$ et ${\rm Int}_n$ ont la même action sur $\hat{T}$). Par passage aux quotients, 
$\iota_\lambda$ induit une bijection
$$
\bs{\iota}_\lambda: R^G(\lambda)\rightarrow (\hat{N}\cap S_\varphi)\hat{T}/(\hat{N}\cap S_\varphi^\circ)\hat{T}= (\hat{N}\cap S_\varphi)/ (\hat{N}\cap S_\varphi^\circ) \hat{U} \hat{Z}^{\Gamma_F}= \bs{S}(\varphi)
$$
qui est un isomorphisme sur le groupe opposé de $\bs{S}(\varphi)$. L'isomorphisme $R^G(\lambda)\simeq \bs{S}(\varphi)$ s'obtient en composant $\bs{\iota}_\lambda$ avec le passage à l'inverse.

}\end{marema1}

\begin{marema2}
{\rm On a introduit en \ref{paramètres} un ensemble $\wt{S}_{\varphi,a}=\{\tilde{g}\in \hat{G}\hat{\theta}: {\rm Int}_{\tilde{g}}\circ \varphi = a\cdot \varphi\}$. S'il n'est pas vide, ce que l'on suppose, alors c'est un $S_\varphi$--espace tordu, et l'on peut définir le ``$S$--groupe tordu''
$$
\wt{\bs{S}}(\varphi,\bs{a})= \wt{S}_{\varphi,a}/S_\varphi^0 Z(\hat{G})^{\Gamma_F}.
$$
C'est un $\bs{S}(\varphi)$--espace tordu. 
Notons $\theta$ l'automorphisme de $T$ associé à $(B,T)$ (\cad la restriction $\theta_{\ES{E}}\vert_T$). Alors, tout comme dans le cas non tordu, $W^{\wt{G},\omega}(\lambda)$ est l'ensemble des éléments $w\theta\in W^{\Gamma_F}\theta$ tels que 
${^L\theta}\circ {^Lw}\circ \varphi_\lambda^T$ est $\hat{T}$--conjugué à $a\cdot \varphi_\lambda^T$. 
Pour $w\theta\in W^{\wt{G},\omega}(\lambda)$, l'élément $\hat{\theta} \hat{w}= \hat{\theta}(\hat{w})\hat{\theta}$ appartient à $W^{\hat{G}}(\hat{T})^{\Gamma_F}$, et l'on peut choisir un relèvement $n$ de 
$\hat{\theta}(\hat{w})$ dans $\hat{N}^{\Gamma_F}$. Alors ${\rm Int}_n\circ {^L\theta}\circ \varphi_\lambda^T$ est $\hat{T}$--conjugué à $a\cdot \varphi_\lambda^T$. En d'autres termes, il existe un $t\in \hat{T}$ tel que $tn \hat{\theta}\in \wt{S}_{\varphi,a}$. L'application $w\theta\mapsto tn\hat{\theta}$ donne une application
$$
\tilde{\iota}_\lambda :W^{\wt{G},\omega}(\lambda)\rightarrow (\hat{N}\hat{\theta}\cap \wt{S}_{\varphi,a})\hat{T}/\hat{T}
= (\hat{N}\hat{\theta}\cap \wt{S}_{\varphi,a})/\hat{U}\hat{Z}^{\Gamma_F}.
$$
Par construction, $\hat{\theta}\circ w^{-1}$ et ${\rm Int}_n\circ \hat{\theta}$ ont la même action sur $\hat{T}$. Puisque l'application $\iota_\lambda$ vérifie
$$
\tilde{\iota}_\lambda (\tilde{w} w') = \iota_\lambda(w')\tilde{\iota}_\lambda(\tilde{w}),\quad \tilde{w} \in W^{\wt{G},\omega}(\lambda),\, w'\in W^G(\lambda),
$$
elle est bijective. De plus, elle induit par passage aux quotients une application bijective
$$
\tilde{\bs{\iota}}_\lambda: R^{\wt{G},\omega}(\lambda)\rightarrow \wt{\bs{S}}(\varphi,\bs{a})
$$
qui vérifie
$$
\tilde{\bs{\iota}}_\lambda(\tilde{r}r')= \bs{\iota}_\lambda(r')\tilde{\bs{\iota}}_\lambda(\tilde{r}),\quad \tilde{r}\in R^{\wt{G},\omega}(\lambda), \, 
r'\in R^G(\lambda).
$$
}
\end{marema2}

Supposons de plus que le caractère $\lambda$ est non ramifié. 
Posons $T_{\rm ad}=q(T)$, et notons $\hat{T}_{\rm sc}$ la pré--image de $\hat{T}$ dans $\hat{G}_{\rm SC}$ (par l'homomorphisme naturel $\hat{q}:\hat{G}_{\rm SC}\rightarrow \hat{G}$, dual de $q$). On utilise maintenant le fait que $G$ est non ramifié. Soit
\begin{eqnarray*}
\Omega &=& {\rm coker}(\check{X}(T)^{\Gamma_F} \rightarrow \check{X}(T_{\rm ad})^{\Gamma_F})\\
&\simeq & \ker(\hat{T}_{\rm sc}/(1-\phi)(\hat{T}_{\rm sc})\rightarrow \hat{T}/(1-\phi)(\hat{T}))^D.
\end{eqnarray*}
C'est un groupe abélien fini, qui opère transitivement et librement sur les classes de $G(F)$--conjugaison de sous--groupes hyperspéciaux de $G(F)$. Cette action provient de celle du groupe $G_{\rm AD}(F)/q(G(F))= T_{\rm ad}(F)/q(T(F))$ par conjugaison, via la projection naturelle
$$
\ker({\rm H}^1(W_F,\hat{Z}_{\rm sc})\rightarrow {\rm H}^1(W_F,\hat{Z}))^D
\simeq \ker({\rm H}^1(W_F,\hat{T}_{\rm sc})\rightarrow {\rm H}^1(W_F,\hat{T}))^D \rightarrow 
\Omega.
$$
D'après \cite[prop. 9]{Mi}, le $R$--groupe $R^G(\lambda)$ est naturellement isomorphe à un sous--groupe de $\Omega^D$. En particulier il est abélien, et si $G=G_{\rm AD}$, alors $R^G(\lambda)=\{1\}$ et l'induite $\pi=i_B^G(\lambda)$ est irréductible --- ce que l'on savait déjà grâce à des travaux antérieurs de Keys. On note $\zeta=\zeta_\lambda: \Omega \rightarrow R^G(\lambda)^D$ l'homomorphisme dual. En le composant avec la projection naturelle $G_{\rm AD}(F)\rightarrow \Omega$, on obtient un homomorphisme surjectif $G_{\rm AD}(F)\rightarrow R^G(\lambda)^D$, que l'on note encore $\zeta$. Un caractère $\rho$ de $R^G(\lambda)$ s'identifie à un élément de ${\rm Irr}(\ES{R}^G(\lambda))$ par la formule
$$
\rho(z,r)= z\rho(n),\quad (z,r)\in \ES{N}^G(\lambda)={\Bbb U}\times R^G(\lambda).
$$
D'après \cite[theorem 1]{Mi}, on a 
$$
\pi_\rho \circ {\rm Int}_{x-1}\simeq \pi_{\rho + \zeta(x)}, \quad \rho\in R^G(\lambda)^D, \, 
x\in G_{\rm AD}(F),\leqno{(6)}
$$
où $\rho+\zeta(x)$ désigne le caractère $r\mapsto \rho(r)\zeta(x)(r)$ de $R^G(\lambda)$. 
Ainsi, les sous--représentations irréductibles de $\pi$ sont permutées transitivement par $G_{\rm AD}(F)$, et pour $\rho\in R^G(\lambda)^D$, $K_1$ un sous--groupe hyperspécial de $G(F)$ et 
$x\in G_{\rm AD}(F)$, $\pi_{\rho}$ est l'élément $K_1$--sphérique de $\Pi_\lambda$ si et seulement si 
$\pi_{\rho + \zeta(x)}$ est l'élément ${\rm Int}_g(K_1)$--sphérique de $\Pi_\lambda$.

\begin{marema3}
{\rm La bijection $R^G(\lambda)^D\rightarrow \Pi_\lambda,\, \rho\mapsto \pi_\rho$ dépend du choix d'une normalisation des opérateurs d'entrelacement. Deux telles normalisations différent par un caractère de $R^G(\lambda)$, \cad que si $\bs{r}\mapsto r'_B(\bs{r})$ est la ${\Bbb U}$--représentation de $\ES{R}^G(\lambda)$ associée à une autre normalisation, il existe un caractère 
$\rho'$ de $R^G(\lambda)$ tel que $r'_B(\bs{r})= \rho'(r)r_B(\bs{r})$ pour $\bs{r}\in \ES{R}^G(\lambda)$, où $r$ est l'image de $\bs{r}$ dans $R^G(\lambda)$. La 
représentation $r'_P\otimes \pi$ de $\ES{R}^G(\lambda)\times G(F)$ se décompose alors en
$$
r'_P\otimes \pi = \sum_{\rho\in R^G(\lambda)^D} \rho\otimes \pi_{\rho-\rho'}.
$$
Quitte à changer de normalisation des opérateurs d'entracement, on peut donc supposer que la correspondance $\rho\mapsto \pi_\rho$ est telle que l'élément $\pi_0\in \Pi_\lambda$ associée au caractère trivial de $R^G(\lambda)$ est la représentation $K$--sphérique.}
\end{marema3}

\subsection{Le cas où $(\wt{G},\bs{a})$ est à torsion intérieure}\label{le cas à torsion intérieure}
On suppose dans ce numéro que l'espace tordu $(G,\wt{G})$ est à torsion intérieure, et que $\omega=1$ (on suppose toujours que le groupe $G$ est quasi--déployé sur $F$ et déployé sur $F^{\rm nr}$). Rappelons que $\wt{K}=K\delta_0=\delta_0K$ pour un $\delta_0\in \wt{T}(F)$, et que le $F$--automorphisme $\theta_0= {\rm Int}_{\delta_0}$ de $G$ est trivial sur $T(F)$. Soit $\lambda$ un caractère unitaire de $T(F)$. 
L'application $\ES{N}^{G}(\lambda)\rightarrow \ES{N}^{\wt{G}}(\lambda),\, (z,n)\mapsto (z,n\delta_0)$ est bijective, et elle induit une application bijective $\ES{W}^G(\lambda)\rightarrow \ES{W}^{\wt{G}}(\lambda)$. 
En quotientant par les orbites sous $W^G_0(\lambda)$, on obtient une application bijective 
$\ES{R}^G(\lambda)\rightarrow \ES{R}^{\wt{G}}(\lambda)$. L'automorphisme $\theta_{\ES{R}}$ de 
$\ES{R}^G(\lambda)$ --- qui n'est défini que modulo les automorphismes intérieurs de $\ES{R}^G(\lambda)$ --- opère trivialement sur ${\rm Irr}(\ES{R}^G(\lambda))$, et l'application
$${\rm Irr}(\ES{R}^{\wt{G}}(\lambda))\rightarrow {\rm Irr}(\ES{R}^G(\lambda)),\, (\rho, \tilde{\rho})\mapsto \rho$$
est surjective. En d'autres termes, toute représentation irréductible  
$\rho$ de $R^G(\lambda)$ se prolonge en une ${\Bbb U}$--repré\-sentation $\tilde{\rho}$ de $\ES{R}^{\wt{G}}(\lambda)$, laquelle définit un prolongement $\tilde{\pi}_{\tilde{\rho}}$ de $\pi_\rho$ à $\wt{G}(F)$. On peut normaliser ces prolongements $\tilde{\rho}$ en imposant que
$$
\tilde{\rho}(z,n\delta_0)=\rho(z,n)=z\rho(n),\quad (z,n)\in \ES{N}^G(\lambda).
$$
Pour $\rho\in {\rm Irr}(\ES{R}^G(\lambda))$, notons $\tilde{\rho}_0\in {\rm Irr}(\ES{R}^{\wt{G}}(\lambda))$ la classe d'isomorphisme du prolongement de $\rho$ normalisé de cette manière, et posons $\tilde{\pi}_\rho=\tilde{\pi}_{\tilde{\rho}_0}$. Cela revient à imposer que la restriction de $\widetilde{\nabla}_B(1,\delta_0)$ à l'espace de $\rho\otimes \pi_\rho$ soit égale à ${\rm id}\otimes \tilde{\pi}_\rho(\delta_0)$. Le $F$--automorphisme $\theta_0$ de $G$ induit un $F$--automorphisme $\bar{\theta}_0$ de $G_{\rm AD}$, qui est trivial sur $T_{\rm ad}$. Puisque l'inclusion $T_{\rm ad}\subset G_{\rm AD}$ induit un isomorphisme $T_{\rm ad}(F)/q(T(F))\rightarrow G_{\rm AD}(F)/q(G(F))$, on a
$$
x^{-1}\bar{\theta}_0(x)\in q(G(F)),\quad x\in G_{\rm AD}(F).
$$
 
Si de plus le caractère $\lambda$ est non ramifié, alors on a
$$
\tilde{\pi}_\rho \circ {\rm Int}_{x-1}\simeq \tilde{\pi}_{\rho +\zeta_\lambda(x)}, 
\quad \rho\in R^G(\lambda)^D, \, 
x\in G_{\rm AD}(F).\leqno{(1)}
$$
De plus (toujours si $\lambda$ est non ramifié), la représentation $\tilde{\pi}=\tilde{\pi}_{(1,\delta_0)}$ de $\wt{G}(F)$ donnée par
$$
\tilde{\pi}(g\delta_0)=\pi(g)\wt{\nabla}_B(1,\delta_0),\quad 
g\in G(F),$$
est isomorphe à une somme directe $\oplus_x\tilde{\pi}_{\rho_0}\circ {\rm Int}_{x^{-1}}$ où $x$ parcourt un ensemble de représentants de $G_{\rm AD}(F)/\ker(\zeta_\lambda)$. 

\begin{monlem}
On suppose que l'espace tordu $(G,\wt{G})$ est à torsion intérieure, et que $\omega=1$. Alors on a la décomposition
$$
\bs{SD}_{\rm ell}= \bs{SD}_{\rm ell}^{\rm nr}(\wt{G}(F))\oplus \bs{SD}_{\rm ell}^{\rm ram}(\wt{G}(F)).
$$
De plus, si $G$ n'est pas un tore, alors on a $\bs{SD}_{\rm ell}^{\rm nr}(\wt{G}(F))=\{0\}$.
\end{monlem}

\begin{proof}Soit $\Theta\in \bs{SD}_{\rm ell}(\wt{G}(F))$. On écrit $\Theta= \Theta^{\rm nr}+ \Theta^{\rm ram}$ conformément à la décomposition (1) de \ref{caractères tempérés}. Cette décomposition (1) de \ref{caractères tempérés} est compatible à l'action de $G_{\rm AD}(F)$ par conjugaison. Puisque la distribution $\Theta$ est stable, elle est invariante par cette action, et sa décomposition en $\Theta =\Theta^{\rm nr}+\Theta^{\rm ram}$ l'est aussi. Pour un caractère  
unitaire non ramifié $\lambda$ de $T(F)$, on peut choisir un prolongement unitaire $\tilde{\lambda}$ de $\lambda$ à $\wt{T}(F)$, et d'après ce qui précède, 
les sous--représentations $G(F)$--irréductibles de l'induite $i_{\wt{B}}^{\wt{G}}(\tilde{\lambda})$ sont permutées transitivement par l'action de $G_{\rm AD}(F)$ par conjugaison. On en déduit que $\Theta^{\rm nr}$ est une combinaison linéaire finie, à coefficients complexes, de caractères de telles induites. Une telle distribution ne peut être elliptique que si $G=T$. Donc $\bs{SD}_{\rm ell}^{\rm nr}(\wt{G}(F))=\{0\}$ si $G\neq T$. Cela prouve la décomposition du lemme dans le cas où $G\neq T$. Si $G=T$, alors tout est stable, et la décomposition du lemme équivaut à la décomposition (1) de \ref{caractères tempérés}.
\end{proof}

\subsection{Action de $\ES{C}$ sur les caractères de support non ramifié}
\label{action du groupe C}
Soit $C=C(G,\wt{G})$ le sous--groupe de $Z(G)\times G$ formé des paires $(z,g)$ telles que pour tout $\sigma\in \Gamma_F$, l'élément $g \sigma(g)^{-1}$ appartient à $Z(G)$ et $(1-\theta)(g\sigma(g)^{-1})= z\sigma(z)^{-1}$. Il contient le sous--groupe
$$
Z(G)^*=\{(1-\theta)(z),z): z\in Z(G)\}\simeq Z(G).
$$
On pose
$$
\ES{C}=C/Z(G)^*,
$$
et l'on note $\bs{q}:C\rightarrow \ES{C}$ la projection naturelle. 
Notons que pour $(z,g)\in C$, l'image $g_{\rm ad}$ de $g$ dans $G_{\rm AD}$ appartient à $G_{\rm AD}(F)$. 
Un élément $(z,g)\in C$ définit un $F$--automorphisme $g'\mapsto {\rm Int}_{\smash{g^{-1}}}(g')$ de 
$G$ et un $F$--automorphisme $\gamma \mapsto z {\rm Int}_{\smash{g^{-1}}}(\gamma)$ de $\wt{G}$. On obtient 
ainsi une application de $C$ dans le groupe ${\rm Aut}_F(G,\wt{G})$ des automorphismes de $(G,\wt{G})$ qui sont définis sur $F$. Cette application se quotiente en un homomorphisme injectif de $\ES{C}$ dans le groupe opposé ${\rm Aut}_F^{\rm op}(G,\wt{G})$. On a un plongement de $Z(G;F)$ dans $\ES{C}$ donné par $z \mapsto \bs{q}(z,1)$. On a aussi un plongement de $G_\sharp(F)$ dans $\ES{C}$ défini comme suit. Pour $g_\sharp\in G_\sharp(F)$, on choisit un relèvement $g$ de $g_\sharp$ dans $G$, et l'on envoie $g_\sharp$ sur $\bs{q}(1,g)\in \ES{C}$.

\begin{marema1}
{\rm  Le groupe $\ES{C}$ est en général plus gros que le sous--groupe de ${\rm Aut}_F^{\rm op}(G,\wt{G})$ engendré par $Z(G;F)$ et $G_\sharp(F)$, comme le montre l'exemple suivant. Soit $F'/F$ une extension cyclique de degré $4$ (par exemple la sous--extension de degré $4$ de $F^{\rm nr}/F$), et soit $\tau$ un générateur de ${\rm Gal}(F'/F)$. Soit $E=F'^{\langle \tau^2 \rangle}$ l'extension 
quadratique de $F$ intermédiaire. Soit $G$ le groupe ${\rm Res}_{E/F}({\rm SL}(2)_E)$, \cad la restriction des scalaires de $E$ à $F$ du groupe $SL(2)$ vu comme groupe défini et déployé sur $E$. 
On a donc
$$
G(\overline{F})= SL(2,\overline{F})\times SL(2,\overline{F})
$$
avec action de $\sigma \in \Gamma_F$ donnée par:
\begin{itemize}
\item $\sigma(g_1,g_2)= (\sigma(g_1),\sigma(g_2))$ si $\sigma \in \Gamma_E$;
\item $\sigma(g_1,g_2)= (\sigma(g_2),\sigma(g_1))$ si $\sigma \notin \Gamma_E$.
\end{itemize}
Le groupe $G$ est défini et quasi--déployé sur $F$, et déployé sur $E$, et l'on a
$$G(F)= \{(g,g): g\in SL(2,E)\}.
$$
Soit $\theta$ le $F$--automorphisme $(g_1,g_2)\mapsto (g_2,g_1)$ de $G$, et soit $\wt{G}= G\theta$ (si l'extension $E/F$ est non ramifiée, alors l'espace tordu $(G,\wt{G})$ vérifie toutes les hypothèses de \ref{objets}). Fixons un élément $\xi\in E^\times$ tel que $F'= E[\sqrt{\xi}]$. Considérons la matrice $h= {\rm diag}(\sqrt{\xi},\sqrt{\xi}^{-1})\in SL(2,F')$, l'élément $g=(\tau(h),h)\in G(F')$, et l'élément $z=(1,-1)\in Z(G;E)$. On a $\tau^2(h)=-h$, d'où $\tau(g)= zg $ et $\tau^2(g)= (-1,-1)g$. On en déduit que l'élément $(z,g)$ appartient à $C$. Son image $\bs{c}=\bs{q}(z,g)\in \ES{C}$ n'est pas dans le produit des images de $Z(G;F)$ et de $G_\sharp(F)$. En effet si elle l'était, alors $z$ appartiendrait au produit de $Z(G;F)$ et de $(1-\theta)(Z(G;\overline{F}))$. Or ces deux groupes sont égaux, et ont pour éléments $(1,1)$ et $(-1,-1)$.
}
\end{marema1}

\begin{monlem1}
Le groupe $\ES{C}$ opère transitivement sur les sous--espaces hyperspéciaux de $\wt{G}(F)$.
\end{monlem1}

\begin{proof}Rappelons qu'on a choisi le sous--espace hyperspécial $(K,\wt{K})$ de $\wt{G}(F)$ tel que $K=K_\ES{E}$ est le sous--groupe hyperspécial de $G(F)$ associé à la paire de Borel épinglée $\ES{E}= (B,T, \{E_\alpha\}_{\alpha\in \Delta})$ de $G$ définie sur $F$, et que l'espace $\wt{K}$ s'écrit $\wt{K}=K\delta_0=\delta_0K$ pour un élément $\delta_0\in \wt{T}(F)\cap N_{\smash{\wt{G}(F)}}(K)$ de la forme $\delta_0= t_0\epsilon_0$ avec $t_0\in T(F^{\rm nr})_1$ et $\epsilon_0\in Z(\wt{G}, \ES{E}; F^{\rm nr})$. 
Soit $(K_1,\wt{K}_1)$ un autre sous--espace hyperspécial de $\wt{G}(F)$. On peut aussi choisir une paire de Borel épinglée $\ES{E}_1=(B_1,T_1,\{E_{\alpha_1}\}_{\alpha_1\in \Delta_1})$ de $G$ définie sur $F$ telle que $K_1=K_{\ES{E}_1}$, et un élément $\delta_1\in \wt{K}_1$ de la forme $
\delta_1= t_1\epsilon_1$ avec $t_1\in T_1(F^{\rm nr})_1$ et $\epsilon_1\in Z(\wt{G},\ES{E}_1;F^{\rm nr})$. On se ramène tout de suite au cas où $(B_1,T_1)=(B,T)$: il suffit pour cela de conjuguer $(K_1,\wt{K}_1)$ par un élément de $G(F)$, puisque deux paires de Borel de $G$ définies sur $F$ sont toujours conjuguées par un tel élément. On peut alors fixer un élément $x\in T(\overline{F})$ tel que $x_{\rm ad}\in T_{\rm ad}(F)$ et ${\rm Int}_x(\ES{E})=\ES{E}_1$; où $x_{\rm ad}$ désigne la projection de $x$ sur $T_{\rm ad}=T/Z(G)$. Alors ${\rm Int}_x(\epsilon_0)$ conserve $\ES{E}_1$, donc on peut écrire ${\rm Int}_x(\epsilon_0)= z\epsilon_1$ avec $z\in Z(G;\overline{F})$. Notons $y$ l'élément de $G(F)$ tel que $\delta_1=y\delta_0$. En notant $\theta = {\rm Int}_{\epsilon_0}$, on obtient l'égalité $t_1 x \theta(x)^{-1}= zy t_0$. Elle entra\^{\i}ne que $y\in T(F)$. On projette l'égalité dans $T_{\rm ad}(\overline{F})$: elle donne $t_{1,{\rm ad}} x_{\rm ad}\theta(x_{\rm ad})^{-1}= y_{\rm ad}t_{0,{\rm ad}}$. Remarquons que pour $i=0,\,1$, ${\rm Int}_{\delta_i}$ est défini sur $F$ et ${\rm Int}_{\epsilon_i}$ aussi (puisque ${\rm Int}_{\epsilon_0}$ conserve $\ES{E}$ et ${\rm Int}_{\epsilon_i}$ conserve $\ES{E}_1$); donc ${\rm Int}_{t_i}$ est aussi défini sur $F$, \cad que $t_{i,{\rm ad}}\in T_{\rm ad}(F)_1$. Puisque $T$ et $T_{\rm ad}$ sont des tores non ramifiés, le choix d'une uniformisante $\varpi$ de $F$ permet d'identifier $T(F)$ à $T(F)_1\times \check{X}(T)^{\phi}$ et $T_{\rm ad}(F)$ à 
$T_{\rm ad}(F)_1\times \check{X}(T_{\rm ad})^\phi$. Conformément à ces deux décompositions, on écrit $y=y'y''$ et $x_{\rm ad}= x'_{\rm ad} x''_{\rm ad}$. L'égalité $t_{1,{\rm ad}} x_{\rm ad}\theta(x_{\rm ad})^{-1}= y_{\rm ad}t_{0,{\rm ad}}$ se décompose alors en deux égalité:
$$
x''_{\rm ad}\theta(x''_{\rm ad})^{-1}= y''_{\rm ad},\quad t_{1,{\rm ad}}x'_{\rm ad}\theta(x'_{\rm ad})^{-1}= y'_{\rm ad}t_{0,{\rm ad}}.
$$
On relève $x''_{\rm ad}$ en un élément $x''\in T(\overline{F})$. Il existe alors un élément $z''\in Z(G;\overline{F})$ tel que $x''\theta(x'')^{-1}=z''y''$. Parce que $x''_{\rm ad}$ et $y''$ sont $F$--rationnels, cela implique que le couple $(z''\!,x'')$ appartient à $C$. Notons $\bs{c}$ l'image dans $\ES{C}$ de l'inverse de ce couple. L'action de $\bs{c}$ envoie $K$ sur ${\rm Int}_{x''}(K)$. Puisque l'élément $x''_{\rm ad}$ appartient au sous--groupe compact maximal $T_{\rm ad}(F)_1$ de $T_{\rm ad}(F)$, la conjugaison par $xx''^{-1}$ conserve $K$. Le groupe ${\rm Int}_{x''}(K)$ est donc égal à ${\rm Int}_x(K)$, qui est égal à $K_1$ par définition de $x$. L'action de $\bs{c}$ envoie $\delta_0$ sur
$$
z''^{-1}{\rm Int}_{x''}(\delta_0)= z''^{-1}x''\theta(x'')^{-1}\delta_0= y'' \delta_0= y'^{-1}\delta_1.
$$
Mais $y'$ appartient au sous--groupe compact maximal $T(F)_1$ de $T(F)$, donc est contenu dans $K_1$. Donc $\bs{c}$ envoie $\delta_0$ sur un élément de $\wt{K}_1$, et il envoie $\wt{K}= K\delta_0$ sur $\wt{K}_1= K_1\delta_1$.
\end{proof}

\begin{marema2}
{\rm La preuve du lemme 2 est tirée de celle du lemme 1.9.(iv) de l'article:

-- \textsc{Waldspurger J.--L.}, \textit{À propos du lemme fondamental pondéré tordu}, Math. Ann. \textbf{343} (2009), 
103--174.

On aurait tout aussi bien pu s'y référer. Mais comme la propriété à prouver est énoncée dans cet article d'une fa\c{c}on assez différente, il aurait alors fallu expliquer que les deux énoncés sont équivalents. On a préféré la redémontrer ici.
}
\end{marema2}

\begin{monlem2}
Pour $x\in G_{\rm AD}(F)$, on a $\omega \circ {\rm Int}_x =\omega$.
\end{monlem2} 

\begin{proof}Soit $\pi:G_{\rm SC}\rightarrow G$ l'homomorphisme naturel. Tout caractère de $G(F)$ se factorise à travers $G(F)/\pi(G_{\rm SC}(F))$. Il suffit donc de prouver que $G_{\rm AD}(F)$ agit trivialement (par conjugaison) sur ce quotient, ou encore que, pour 
$g\in G(F)$ et $x\in G_{\rm AD}(F)$, on a $g{\rm Int}_x(g^{-1})\in \pi(G_{\rm SC}(F))$. Relevons $x$ en $x'\in G$, et écrivons $g= z\pi(g_{\rm sc})$ et $x'=z'\pi(x'_{\rm sc})$ avec $z,\,z'\in Z(G)$ et $g_{\rm sc},\, x'_{\rm sc}\in G_{\rm SC}$. Alors $g{\rm Int}_x(g)= \pi(g_{\rm sc}x'_{\rm sc}g_{\rm sc}^{-1}x'^{-1}_{\rm sc})$ et il suffit de prouver que l'élément $g_{\rm sc}x'_{\rm sc}g_{\rm sc}^{-1}x'^{-1}_{\rm sc}$ de $G_{\rm SC}$ est $F$--rationnel. Soit $\sigma\in \Gamma_F$. Puisque $g\in G(F)$ et $x\in G_{\rm AD}(F)$, il existe deux éléments $z_{\rm sc}(\sigma),\, z'_{\rm sc}(\sigma)\in Z(G_{\rm sc})$ tels que
$$
\sigma(g_{\rm sc})g_{\rm sc}^{-1}= z_{\rm sc}(\sigma),\quad \sigma(x'_{\rm sc})x'^{-1}_{\rm sc}= z'_{\rm sc}(\sigma).
$$
Alors on a
$$
\sigma(g_{\rm sc}x'_{\rm sc}g_{\rm sc}^{-1}x'^{-1}_{\rm sc})= z_{\rm sc}(\sigma)g_{\rm sc}z'_{\rm sc}(\sigma)x'_{\rm sc}z_{\rm sc}(\sigma)^{-1}g_{\rm sc} x'^{-1}_{\rm sc}z'_{\rm sc}(\sigma)^{-1}
= g_{\rm sc}x'_{\rm sc}g_{\rm sc}^{-1}x'^{-1}_{\rm sc}.
$$
Donc $g_{\rm sc}x'_{\rm sc}g_{\rm sc}^{-1}x'^{-1}_{\rm sc}$ est fixe par $\Gamma_F$, d'où l'assertion.
\end{proof}

D'après le lemme 2, l'action de $\ES{C}$ sur $(G,\wt{G})$ induit une action sur l'espace $\bs{I}(\wt{G}(F),\omega)$, et aussi une action sur l'espace $\bs{D}(\wt{G}(F),\omega)$. Cette action stabilise $\bs{D}_{\rm ell}(\wt{G}(F),\omega)$ et préserve la décomposition (1) de \ref{le principe}. 
Précisément, pour $(z,g)\in C$ d'image $\bs{c}=\bs{q}(z,g)\in \ES{C}$, on pose
$$
{^{\bs{c}\!}f}(\gamma)= f(zg^{-1}\gamma g),\quad f\in C^\infty_{\rm c}(\wt{G}(F)),\, \gamma\in \wt{G}(F),
$$
et
$$
{^{\bs{c}}\Theta}(f)=\Theta({^{\bs{c}^{-1}\!\!}f}),\quad \Theta\in \bs{D}(\wt{G}(F),\omega),\, 
f\in C^\infty_{\rm c}(\wt{G}(F)).
$$
En d'autres termes, si $(\pi,\tilde{\pi})$ est une représentation de $(\wt{G}(F),\omega)$ telle que la représentation $\pi$ de $G(F)$ soit tempérée et de longueur finie, alors l'application $\gamma \mapsto {^{\bs c}\tilde{\pi}}(\gamma)= \tilde{\pi}(zg^{-1}\gamma g)$ définit une autre représentation de $(\wt{G}(F),\omega)$, dont le caractère $\Theta_{^{\bs{c}}\tilde{\pi}}$ est donné par $\Theta_{^{\bs{c}}\tilde{\pi}}= {^{\bs{c}}(\Theta_{\tilde{\pi}})}$. La fonction localement constante 
$\theta_{^{\bs{c}}\tilde{\pi}}$ sur $\wt{G}_{\rm reg}(F)$ est donnée par 
$\theta_{^{\bs{c}}\tilde{\pi}}(\gamma)= \Theta_{\tilde{\pi}}(zg^{-1}\gamma g)$.

\vskip1mm
Soit $\lambda$ un caractère unitaire et non ramifié de $T(F)$. Posons $\pi= i_B^G(\lambda)$. On a vu (\ref{séries principales nr}) qu'il existe un homomorphisme surjectif $\zeta_\lambda:G_{\rm AD}(F)\rightarrow R^G(\lambda)^D$ tel que pour $\rho\in {\rm Irr}(\ES{R}^G(\lambda))$, identifié à l'élément $r\mapsto \rho(1,r)$ de $R^G(\lambda)^D$, on a
$$
\pi_{\rho}\circ {\rm Int}_{x^{-1}}\simeq \pi_{\rho+\zeta_\lambda(x)},\quad x\in G_{\rm AD}(F).
$$ On voudrait étendre cette construction au cadre tordu qui nous intéresse ici. On abandonne pour cela l'identification ${\rm Irr}(\ES{R}^G(\lambda))=R^G(\lambda)^D$, qui s'adapte très mal au cadre tordu.

Supposons de plus que l'ensemble $\ES{N}^{\wt{G},\omega}(\lambda)$ n'est pas vide (ou, ce qui revient au même, que l'ensemble $N_{\smash{\wt{G}(F),\omega}}(\lambda)$ n'est pas vide). Alors l'ensemble $\ES{R}^{\wt{G},\omega}(\lambda)$ est une extension (scindée) de $R^{\wt{G},\omega}(\lambda)$ par ${\Bbb U}$, et un espace tordu sous $\ES{R}^G(\lambda)$. Pour $\tilde{\bs{r}}\in \ES{R}^{\wt{G},\omega}(G)$, l'automorphisme $\theta_{\tilde{\bs{r}}}$ de $\ES{R}^G(\lambda)$ défini par $\tilde{\bs{r}}\bs{r}=\theta_{\tilde{\bs{r}}}(\bs{r}) \tilde{\bs{r}}$ ne dépend pas de $\bs{\tilde{r}}$, et est noté $\theta_{\ES{R}}$. Un élément de ${\rm Irr}(\ES{R}^{\wt{G},\omega}(\lambda))$ est par définition la donnée d'un couple $(\rho,\tilde{\rho})$ où $\rho$ est un élément de ${\rm Irr}(\ES{R}^G(\lambda),\theta_{\ES{R}})$, \cad un élément de ${\rm Irr}(\ES{R}^G(\lambda))$ tel que $\rho\circ \theta_{\ES{R}}=\rho$, et $\tilde{\rho}: \ES{R}^{\wt{G},\omega}(\lambda) \rightarrow {\Bbb U}$ est une application telle que
$$
\tilde{\rho}(\bs{r}\tilde{\bs{r}}\bs{r}')= \rho(\bs{r})\tilde{\rho}(\tilde{\bs{r}})\rho(\bs{r}'),\quad \tilde{\bs{r}}\in 
\ES{R}^{\wt{G},\omega}(\lambda),\, \bs{r},\,\bs{r}'\in \ES{R}^G(\lambda).
$$
L'application $\tilde{\rho}$ est une ${\Bbb U}$--représentation de $\ES{R}^{\wt{G},\omega}(\lambda)$, \cad qu'elle vérifie
$$
\tilde{\rho}(z,\tilde{r})= z\tilde{\rho}(1,\tilde{r}),\quad (z,\tilde{r})\in \ES{R}^{\wt{G},\omega}(\lambda)={\Bbb U}\times R^{\wt{G},\omega}(\lambda).
$$
Pour $\bs{r}\in \ES{R}^G(\lambda)$ et $\rho\in {\rm Irr}(\ES{R}^G(\lambda))$, en posant $\bs{r}^{-1}\theta_{\ES{R}}(\bs{r})= (z,r')$ avec $z\in {\Bbb U}$ et $r'\in R^G(\lambda)$, on a $\rho(\bs{r}^{-1}\theta_{\ES{R}}(\bs{r}))=z\rho(1,r')$. Par conséquent, si pour un (\cad pour tout) $\tilde{\bs{r}}\in \ES{R}^{\wt{G},\omega}(\lambda)$, le triplet $\bs{\tau}=(T,\lambda, \tilde{\bs{r}})$ est inessentiel, alors l'ensemble ${\rm Irr}(\ES{R}^G(\lambda),\theta_{\ES{R}})$ est vide (et, ce qui revient au même,  l'ensemble ${\rm Irr}(\ES{R}^{\wt{G},\omega}(\lambda))$ est vide). Réciproquement, si pour $\tilde{\bs{r}}\in \ES{R}^{\wt{G},\omega}(\lambda)$, le triplet $\bs{\tau}=(T,\lambda, \tilde{\bs{r}})$ est essentiel, alors puisque la représentation $\tilde{\pi}_{\bs{\tau}}$ de $(\wt{G}(F),\omega)$ a un caractère non nul \cite[2.9]{W}, l'ensemble ${\rm Irr}(\ES{R}^G(\lambda),\theta_{\ES{R}})$ n'est pas vide. D'autre part, pour $\rho\in {\rm Irr}(\ES{R}^G(\lambda))$, si la représentation $\pi_\rho$ de $G(F)$ est $K_1$--sphérique pour un sous--groupe hyperspécial $K_1$ de $G(F)$ tel que l'ensemble $N_{\smash{\wt{G}(F)}}(K_1)$ ne soit pas vide, alors on a $\rho\circ \theta_{\ES{R}}= \rho$. En effet,  pour $\delta_1\in N_{\smash{\wt{G}(F)}}(K_1)$, la représentation $\omega^{-1}(\pi_{\rho})^{\delta_1}$ de $G(F)$ est encore $K_1$--sphérique, donc isomorphe à $\pi_\rho$, ce qui signifie que $\pi_\rho$ se prolonge à $(\wt{G}(F),\omega)$, et donc que $\rho\in {\rm Irr}(\ES{R}^G(\lambda),\theta_{\ES{R}})$. L'existence du sous--espace hyperspécial $(K,\wt{K})$ de $\wt{G}(F)$ assure donc que
\begin{enumerate}
\item[(1)]pour $\tilde{\bs{r}}\in \ES{R}^{\wt{G},\omega}(\lambda)$, le triplet $(T,\lambda,\tilde{\bs{r}})$ est essentiel.
\end{enumerate}

Tout élément $\tilde{\bs{r}}\in \ES{R}^{\wt{G},\omega}(\lambda)$ définit un triplet (essentiel) $\bs{\tau}=(T,\lambda,\tilde{\bs{r}})$ et une représentation $\tilde{\pi}_{\bs{\tau}}$ de $(\wt{G}(F),\omega)$ qui prolonge $\pi$ --- cf. \ref{triplets elliptiques essentiels}. Le caractère $\Theta_{\bs{\tau}}$ de $\tilde{\pi}_{\bs{\tau}}$ est non nul, et l'on veut décrire la distribution ${^{\bs{c}}(\Theta_{\bs{\tau}})}$ pour tout  élément $\bs{c}\in \ES{C}$. Puisque $G_{\rm AD}(F)$ est le produit de $T_{\rm ad}(F)$ et de l'image naturelle de $G_{\rm SC}(F)$ dans $G_{\rm AD}(F)$ --- cela résulte de la décomposition de Bruhat--Tits ---, il suffit de le faire pour les éléments $\bs{c}$ de la forme $\bs{c}=\bs{q}(z,g)$ avec $g_{\rm ad}\in T_{\rm ad}(F)$.  

\begin{monlem3}
(On rappelle que $\lambda$ est un caractère unitaire et non ramifié de $T(F)$ tel que l'ensemble $\ES{N}^{\wt{G},\omega}(\lambda)$ n'est pas vide.) Soient $\tilde{\bs{r}}\in \ES{R}^{\wt{G},\omega}(\lambda)$, et $\bs{c}\in \ES{C}$ de la forme $\bs{c}=\bs{q}(z,g)$ avec $g_{\rm ad}\in T_{\rm ad}(F)$. Choisissons un relèvement $\tilde{w}\in W^{\wt{G},\omega}(\lambda)$ de la projection de $\tilde{\bs{r}}$ sur l'ensemble $R^{\wt{G},\omega}(\lambda)= W^{\wt{G},\omega}(\lambda)/W^G_0(\lambda)$. Posons $\bs{\tau}=(T,\lambda, \tilde{\bs{r}})$ et $t= z g^{-1}{\rm Int}_{\tilde{w}}(g)$. Alors $t$ appartient à $T(F)$, et on a l'égalité
$$
{^{\bs{c}}(\Theta_{\bs{\tau}})}= \lambda(t)\Theta_{\bs{\tau}}.
$$
\end{monlem3}

\begin{proof}Puisque $g_{\rm ad}\in T_{\rm ad}(F)$, les éléments $g$ et $t$ appartiennent à $T=T(\overline{F})$, et un calcul simple montre que $t$ est $F$--rationnel. Rappelons que pour $z'\in {\Bbb U}$, on a $\Theta_{z'\bs{\tau}}= z'\Theta_{\bs{\tau}}$. Par conséquent, quitte à remplacer $\tilde{\bs{r}}$ par $z'\tilde{\bs{r}}$ pour un $z'\in {\Bbb U}$, on peut supposer que $\tilde{\bs{r}}$ est de la forme $\tilde{\bs{r}}=(1,\tilde{r})$ pour un $\tilde{r}\in R^{\wt{G},\omega}(\lambda)$. 
Relevons $\tilde{w}$ en un élément $\gamma_{\tilde{w}}\in N_{\smash{\wt{G}(F),\omega}}(\lambda)$. L'automorphisme ${\rm Int}_{\gamma_{\tilde{w}}}$ de $G$ conserve $T$, et opère sur $T$ comme ${\rm Int}_{\tilde{w}} = {\rm Int}_w \circ \theta$ pour un élément $w\in W^{\Gamma_F}$, où $\theta$ est le $F$--automorphisme de $T$ défini par $\wt{T}$. 
Considérons la série principale $\pi=i_B^G(\lambda)$ et la représentation $\tilde{\pi}_{\bs{\tau}}$ de $(\wt{G}(F),\omega)$ agissant sur l'espace $V_\pi$ de $\pi$. Rappelons que cette dernière est définie par
$$
\tilde{\pi}_{\bs{\tau}}(g\gamma_{\tilde{w}})= \pi(g) \circ \wt{\nabla}(1,\gamma_{\tilde{w}}),\quad g\in G(F).
$$
Soit $\Psi$ l'automorphisme de $V_\pi$ défini par
$$
\Psi(\phi)(x)=\phi(g^{-1}xg),\quad \phi \in V_\pi,\, x\in G(F).
$$
C'est un automorphisme ${\Bbb C}$--linéaire, et pas un automorphisme de $G(F)$--module. 
L'égalité de traces de l'énoncé résulte en fait de l'égalité plus précise suivante
$$
\tilde{\pi}_{\bs{\tau}}(zg^{-1}\gamma g) = \lambda(t)\Psi^{-1}\circ \tilde{\pi}_{\bs{\tau}}(\gamma)\circ \Psi,\quad \gamma\in \wt{G}(F). \leqno{(2)}
$$
Prouvons (2). Il est clair que $\pi(g^{-1}xg)= \Psi^{-1}\circ \pi(x) \circ \Psi$ pour tout $x\in G(F)$. Il suffit donc de prouver que $\tilde{\pi}_{\bs{\tau}}(zg^{-1}\gamma_{\tilde{w}} g) = \lambda(t)\Psi^{-1}\circ \tilde{\pi}_{\bs{\tau}}(\gamma_{\tilde{w}})\circ \Psi$, ce qui équivaut à
$$
\tilde{\pi}_{\bs{\tau}}(t\gamma_{\tilde{w}})= \lambda(t)\Psi^{-1}\circ \tilde{\pi}_{\bs{\tau}}(\gamma_{\tilde{w}})\circ \Psi.
$$
Posons ${^wB}={\rm Int}_w(B)$. Par définition \cite[2.8]{W}, on a $\tilde{\pi}_{\bs{\tau}}(\gamma_{\tilde{w}})=I_1\circ I_2\circ I_3$ où:
\begin{itemize}
\item $I_1\in {\rm Isom}_{G(F)}(i_B^G(\lambda),i_B^G(\omega\lambda))$ est défini par $I_1(\phi)(x)= \omega(x)\phi(x)$;
\item $I_2\in  {\rm Isom}_{G(F)}(i_B^G(\omega\lambda), i_{^w\!B}^G(\lambda))$ est défini par $I_2(\phi)(x)= \partial_B(\gamma_{\tilde{w}})^{1/2}\phi({\rm Int}_{\gamma_{\tilde{w}}}^{-1}(x))$;
\item $I_3= R_{B\vert {^w\!B}}(\lambda)\in  {\rm Isom}_{G(F)}(i_{^w\!B}^G(\lambda), i_{B}^G(\lambda))$; 
\end{itemize} 
où $\phi\in V_\pi$ et $x\in G(F)$, $\partial_B(\gamma_{\tilde{w}})$ est le Jacobien de l'application 
${\rm Int}_{\gamma_{\tilde{w}}}:U_B(F)\rightarrow U_{^w\!B}(F)$, et $R_{B,{^w\!B}}(\lambda)$ est l'opérateur d'entrelacement normalisé (cf. \cite[1.10]{W}). On note encore $\Psi$ l'automorphisme défini de la même manière sur les espaces des induites $i_B^G(\omega\lambda)$, 
$i_{^w\!B}^G(\lambda)$, et l'on note $\Psi'$ l'automorphisme de ces différents espaces obtenu en rempla\c{c}ant $g$ par ${\rm Int}_{\tilde{w}}(g)$ dans la définition de $\Psi$. Alors on a (calcul simple)
$$
I_1\circ \Psi =\Psi \circ I_1,\quad I_2\circ \Psi = \Psi' \circ I_2,\quad I_3\circ \Psi' = d \Psi' \circ I_3,
$$
où $d$ est un certain module. La relation cherchée équivaut donc à
$$\Psi \circ \pi(t)= d\lambda(t) \Psi'.$$
Deux autres calculs simples donnent $\Psi \circ \pi(t)= \delta_B(t)^{1/2}\lambda(t)\Psi'$ et $d= \delta_B(t)^{1/2}$. Cela prouve (2), et donc aussi le lemme.

\end{proof}

Soit $\omega_{\ES{R}}: \ES{R}^G(\lambda)\rightarrow {\Bbb U}$ l'application définie par
$$
\omega_{\ES{R}}= p_{\Bbb U}\circ (\theta_{\ES{R}}-1)
$$
où $p_{\Bbb U}: \ES{R}^G(\lambda)\rightarrow {\Bbb U}$ désigne la projection naturelle, et 
$(\theta_{\ES{R}}-1)(\bs{r})= \theta_{\ES{R}}(\bs{r})\bs{r}^{-1}$, $\bs{r}\in \ES{R}^G(\lambda)$. On a donc
$$
\omega_{\ES{R}}(z,r) = p_{\Bbb U}\circ\theta_{\ES{R}}(1,r), \quad (z,r)\in \ES{R}^G(\lambda)={\Bbb U}\times R^G(\lambda).
$$ On note ${\rm Irr}(\ES{R}^G(\lambda),\theta_{\ES{R}},\omega_{\ES{R}})$ le sous--ensemble de ${\rm Irr}(\ES{R}^G(\lambda))$ formé des $\rho$ tels que $\rho \circ \theta_{\ES{R}}= \omega_{\ES{R}}\rho$. 
Ce sous--ensemble n'est pas vide, il contient l'application $p_{\Bbb U}$. Il est muni d'une structure de groupe commutatif, déduite par restriction de celle sur ${\rm Irr}(\ES{R}^G(\lambda))$: pour $\rho,\, \rho'\in {\rm Irr}(\ES{R}^G(\lambda))$, l'élément $\rho+ \rho'$ de ${\rm Irr}(\ES{R}^G(\lambda))$ est donné par
$$
(\rho + \rho')(z,r)= z^{-1}\rho(z,r)\rho'(z,r);
$$
$p_{\Bbb U}$ est l'élément neutre, et l'inverse $(-\rho)$ de $\rho$ est donné par $(-\rho)(z,r)= z\rho(1,r)^{-1}$. Pour $\rho\in {\rm Irr}(\ES{R}^G(\lambda),\theta_{\ES{R}})$ et $\rho'\in 
{\rm Irr}(\ES{R}^G(\lambda),\theta_{\ES{R}},\omega_{\ES{R}})$, l'élément $\rho + \rho'$ est dans 
${\rm Irr}(\ES{R}^G(\lambda),\theta_{\ES{R}})$. Cela munit ${\rm Irr}(\ES{R}^G(\lambda),\theta_{\ES{R}})$ d'une structure de ${\rm Irr}(\ES{R}^G(\lambda),\theta_{\ES{R}},\omega_{\ES{R}})$--ensemble, et fait de ${\rm Irr}(\ES{R}^G(\lambda),\theta_{\ES{R}})$ un espace principal homogène sous ${\rm Irr}(\ES{R}^G(\lambda),\theta_{\ES{R}},\omega_{\ES{R}})$. 

On note ${\rm Irr}(\ES{R}^{\wt{G},\omega}(\lambda),\omega_{\ES{R}})$ l'ensemble des couples $(\rho,\tilde{\rho})$ où $\rho$ est un élément de l'ensemble ${\rm Irr}(\ES{R}^G(\lambda),\theta_{\ES{R}},\omega_{\ES{R}})$, et $\tilde{\rho}: \ES{R}^{\wt{G},\omega}(\lambda)\rightarrow {\Bbb U}$ est une application telle que
$$
\tilde{\rho}(\bs{r}\tilde{\bs{r}}\bs{r}')= \rho(\bs{r})\tilde{\rho}(\tilde{\bs{r}})\omega_{\ES{R}}(\bs{r}')\rho(\bs{r}'),\quad \tilde{\bs{r}}\in 
\ES{R}^{\wt{G},\omega}(\lambda),\, \bs{r},\,\bs{r}'\in \ES{R}^G(\lambda).
$$
L'application $\tilde{\rho}$ est encore une ${\Bbb U}$--représentation de $\ES{R}^{\wt{G},\omega}(\lambda)$. Cet ensemble ${\rm Irr}(\ES{R}^{\wt{G},\omega}(\lambda),\omega_{\ES{R}})$ n'est pas vide, il contient la projection naturelle $\tilde{p}_{\Bbb U}: \ES{R}^{\wt{G},\omega}(\lambda)\rightarrow {\Bbb U}$. On a une action de ${\Bbb U}$ sur ${\rm Irr}(\ES{R}^{\wt{G},\omega}(\lambda),\omega_{\ES{R}})$, donnée par $z(\rho,\tilde{\rho})= (\rho, z\tilde{\rho})$, $z\in {\Bbb U}$. Tout élément $\rho\in {\rm Irr}(\ES{R}^G(\lambda),\theta_{\ES{R}},\omega_{\ES{R}})$ se prolonge en un élément $(\rho,\tilde{\rho})$ de ${\rm Irr}(\ES{R}^{\wt{G},\omega}(\lambda),\omega_{\ES{R}})$, et les fibres de l'application $(\rho,\tilde{\rho})\mapsto \rho$ sont les orbites de ${\Bbb U}$ dans ${\rm Irr}(\ES{R}^{\wt{G},\omega}(\lambda),\omega_{\ES{R}})$. L'ensemble 
${\rm Irr}(\ES{R}^{\wt{G},\omega}(\lambda), \omega_{\ES{R}})$ est muni d'une structure de groupe commutatif, d'élément neutre $\tilde{p}_{\Bbb U}$, définie comme celle sur ${\rm Irr}(\ES{R}^G(\lambda),\theta_{\ES{R}},\omega_{\ES{R}})$: pour $\tilde{\rho},\, \tilde{\rho}'\in {\rm Irr}(\ES{R}^{\wt{G},\omega}(\lambda),\omega_{\ES{R}})$, on note $\tilde{\rho}+ \tilde{\rho}'$ l'élément de 
${\rm Irr}(\ES{R}^{\wt{G},\omega}(\lambda),\omega_{\ES{R}})$ défini par
$$(\tilde{\rho}+ \tilde{\rho}')(z,\tilde{r})= z^{-1}\tilde{\rho}(z,\tilde{r})\tilde{\rho}'(z,\tilde{r}),\quad 
(z,\tilde{r})\in \ES{R}^{\wt{G},\omega}(\lambda).
$$
Cette structure prolonge celle sur ${\rm Irr}(\ES{R}^G(\lambda),\theta_{\ES{R}},\omega_{\ES{R}})$, au sens où si $\tilde{\rho}$ et $\tilde{\rho}'$ prolongent les éléments $\rho$ et $\rho'$ de ${\rm Irr}(\ES{R}^G(\lambda),\theta_{\ES{R}},\omega_{\ES{R}})$, alors $\tilde{\rho}+ \tilde{\rho}'$ prolonge $\rho + \rho'$. Pour $\tilde{\rho}\in {\rm Irr}(\ES{R}^{\wt{G},\omega}(\lambda), \omega_{\ES{R}})$ et 
$\tilde{\rho}'\in  {\rm Irr}(\ES{R}^{\wt{G},\omega}(\lambda))$, on définit comme ci--dessus 
un élément $\tilde{\rho} + \tilde{\rho}' = \tilde{\rho}' + \tilde{\rho}$ de ${\rm Irr}(\ES{R}^{\wt{G},\omega}(\lambda))$. Cela munit ${\rm Irr}(\ES{R}^{\wt{G},\omega}(\lambda))$ d'une structure de 
${\rm Irr}(\ES{R}^{\wt{G},\omega}(\lambda),\omega_{\ES{R}})$--ensemble, qui prolonge la structure 
de ${\rm Irr}(\ES{R}^G(\lambda),\theta_{\ES{R}},\omega_{\ES{R}})$--ensemble sur 
${\rm Irr}(\ES{R}^G(\lambda),\theta_{\ES{R}})$, et fait de ${\rm Irr}(\ES{R}^{\wt{G},\omega}(\lambda))$ un espace principal homogène sous ${\rm Irr}(\ES{R}^{\wt{G},\omega}(\lambda),\omega_{\ES{R}})$.

\vskip1mm 
Rappelons que pour $\rho\in {\rm Irr}(\ES{R}^G(\lambda),\theta_{\ES{R}})$, le choix d'un prolongement $\tilde{\rho}$ de $\rho$ à $\ES{R}^{\wt{G},\omega}(\lambda)$ définit un prolongement $\tilde{\pi}_{\tilde{\rho}}$ de $\pi_\rho$ à $(\wt{G}(F),\omega)$, déterminé par le fait que pour $\gamma\in N_{\smash{\wt{G}(F),\omega}}(\lambda)$, la restriction de $\wt{\nabla}_B(1,\gamma)$ à l'espace de $\pi_\rho$ est donnée par $\tilde{\rho}(\tilde{r})\tilde{\pi}_{\tilde{\rho}}(\gamma)$, où $\tilde{r}$ est l'image de $\gamma$ dans $R^{\wt{G},\omega}(\lambda)$. Pour $\bs{c}\in \ES{C}$ et $(\rho,\tilde{\rho})\in {\rm Irr}(\ES{R}^{\wt{G},\omega}(\lambda))$, la représentation ${^{\bs{c}}(\tilde{\pi}_{\tilde{\rho}}})$ de $(\wt{G}(F),\omega)$ prolonge la représentation ${^x(\pi_\rho)}= \pi_\rho \circ {\rm Int}_{x^{-1}}$ de $G(F)$, où l'on a posé $x= g_{\rm ad}\in G_{\rm AD}(F)$. Notons $\bs{c}\cdot \tilde{\rho}$ l'élément de ${\rm Irr}(\ES{R}^{\wt{G},\omega}(\lambda))$ prolongeant $x\cdot \rho =\rho + \zeta_\lambda(x)$ défini par 
$$
{^{\bs{c}}(\Theta_{\tilde{\pi}_{\tilde{\rho}}})}= \Theta_{\tilde{\pi}_{\bs{c}\cdot \tilde{\rho}}}.\leqno{(3)}
$$
Cela définit une action de $\ES{C}$ sur ${\rm Irr}(\ES{R}^{\wt{G},\omega}(\lambda))$. 
Comme $\rho$ et $x\cdot \rho$ appartiennent à 
${\rm Irr}(\ES{R}^G(\lambda),\theta_{\ES{R}})$, l'élément $\zeta_\lambda(x)$ appartient à ${\rm Irr}(\ES{R}^G(\lambda),\theta_{\ES{R}}, \omega_{\ES{R}})$. En composant l'homomorphisme $\zeta_\lambda$ avec la projection 
$\ES{C}\rightarrow G_{\rm AD}(F),\, (z,g)\mapsto g_{\rm ad}$, on obtient donc un homomorphisme
$$
\zeta_{\lambda,\ES{C}}: \ES{C}\rightarrow {\rm Irr}(\ES{R}^G(\lambda),\theta_{\ES{R}},\omega_{\ES{R}}).
$$
Le lemme suivant est une simple conséquence du lemme 3.

\begin{monlem4}
(On rappelle que $\lambda$ est un caractère unitaire et non ramifié de $T(F)$ tel que l'ensemble $\ES{N}^{\wt{G},\omega}(\lambda)$ n'est pas vide.) Il existe un unique homomorphisme
$$
\tilde{\zeta}_\lambda: \ES{C} \rightarrow {\rm Irr}(\ES{R}^{\wt{G},\omega}(\lambda),\omega_{\ES{R}})
$$
qui vérifie l'égalité
$$
\bs{c}\cdot \tilde{\rho}= \tilde{\rho}+ \tilde{\zeta}_\lambda(\bs{c}),\quad \tilde{\rho}\in {\rm Irr}(\ES{R}^{\wt{G},\omega}(\lambda)),\, \bs{c}\in \ES{C}.
$$
Cet homomorphisme relève $\zeta_{\lambda,\ES{C}}$ au sens où, composé avec la projection naturelle
$${\rm Irr}(\ES{R}^{\wt{G},\omega}(\lambda),\omega_{\ES{R}})\rightarrow {\rm Irr}(\ES{R}^G(\lambda),\theta_{\ES{R}},\omega_{\ES{R}}),
$$
il redonne $\zeta_{\lambda,\ES{C}}$. 
Pour $\tilde{r}\in R^{\wt{G},\omega}(\lambda)$ et $\tau=(T,\lambda,\tilde{r})$, l'application
$\bs{c}\mapsto \tilde{\zeta}_\lambda(\bs{c})(1,\tilde{r})^{-1}$ est un caractère unitaire du groupe $\ES{C}$, que l'on note $\omega_\tau$, et pour $\bs{\tau}=(T,\lambda, \tilde{\bs{r}})$ se projetant sur $\tau$, la distri\-bution $\Theta_{\bs{\tau}}$ est un vecteur propre pour l'action de $\ES{C}$ relativement au caractère $\omega_\tau$, au sens où l'on a
$$
{^{\bs{c}}(\Theta_{\bs{\tau}})}= \omega_\tau(\bs{c})\Theta_{\bs{\tau}},\quad \bs{c}\in \ES{C}.
$$
Enfin, le caractère $\omega_\tau$ ne dépend que de la classe de conjugaison du triplet $\tau$. 
\end{monlem4}

\begin{proof}Notons qu'il existe au plus une application $\tilde{\zeta}: \ES{C}\rightarrow 
{\rm Irr}(\ES{R}^{\wt{G},\omega}(\lambda),\omega_{\ES{R}})$ telle que $\bs{c}\cdot \tilde{\rho} = 
\tilde{\rho} + \tilde{\zeta}(\bs{c})$ pour tout $\tilde{\rho}\in {\rm Irr}(\ES{R}^{\wt{G},\omega}(\lambda))$. Donc si une telle application existe, elle est unique. Soit $\tilde{\bs{r}}\in \ES{R}^{\wt{G},\omega}(\lambda)$, d'image $\tilde{r}\in R^{\wt{G},\omega}(\lambda)$. Posons $\bs{\tau}=(T,\lambda,\tilde{\bs{r}})$ et $\tau=(T,\lambda,\tilde{r})$. D'après le lemme 3, pour $\bs{c}\in \ES{C}$, il existe une constante $\omega_\tau(\bs{c})\in {\Bbb U}$ telle que
$$
{^{\bs c}(\Theta_{\bs{\tau}})}= \omega_\tau(\bs{c}) \Theta_{\bs{\tau}}.\leqno{(4)}
$$
Concrètement, on choisit un couple $(z,g)\in C$ tel que $\bs{c}=\bs{q}(z,g)$, un élément $y\in G_{\rm SC}(F)$ tel $y_{\rm ad}g_{\rm ad}\in T_{\rm ad}(F)$, et l'on pose $g'=\pi(y)g\in T$, où $\pi: G_{\rm SC}\rightarrow G$ est l'homomorphisme naturel. On choisit aussi un relèvement $\tilde{w}\in W^{\wt{G},\omega}(\lambda)$ de $\tilde{r}$, et l'on pose $t= z g'^{-1}{\rm Int}_{\tilde{w}}(g')$. Alors $t\in T(F)$, et l'on a $\omega_\tau(\bs{c})= \lambda(t)$. D'après (4), l'application $\bs{c}\mapsto \omega_\tau(\bs{c})$ est un caractère unitaire du groupe $\ES{C}$, et la distribution  $\Theta_{\bs{\tau}}$ est un vecteur propre pour l'action de $\ES{C}$ relativement à ce caractère $\omega_\tau$. De plus, tout comme la distribution $\Theta_{\tau}$, le caractère $\omega_\tau$ ne dépend que de la classe de conjugaison du triplet $\tau$. Rappelons que la distribution $\Theta_{\bs{\tau}}$ est donnée par (\ref{triplets elliptiques essentiels}.(3))
$$
\Theta_{\bs{\tau}}= \sum_{\rho \in {\rm Irr}(\ES{R}^G(\lambda),\theta_{\ES{R}})}
\tilde{\rho}(\tilde{\bs{r}})\Theta_{\tilde{\pi}_{\tilde{\rho}}}.
$$
Pour $\bs{c}\in \ES{C}$, on a donc
$$
{^{\bs{c}}(\Theta_{\bs{\tau}})}= \sum_{\rho \in {\rm Irr}(\ES{R}^G(\lambda),\theta_{\ES{R}})}
(\bs{c}^{-1}\cdot \tilde{\rho})(\tilde{\bs{r}})\Theta_{\tilde{\pi}_{\tilde{\rho}}},
$$
et d'après (4), on a
$$
(\bs{c}^{-1}\cdot \tilde{\rho})(\tilde{\bs{r}})= \omega_{\tau}(\bs{c}) \tilde{\rho}(\tilde{\bs{r}}),\quad \rho \in {\rm Irr}(\ES{R}^G(\lambda), \theta_{\ES{R}}).\leqno{(5)}
$$
Pour $\bs{c}\in \ES{C}$, posons
$$
\tilde{\zeta}_\lambda(\bs{c})(\tilde{\bs{r}}) = p_{\Bbb U}(\tilde{\bs{r}})\omega_{\tau}(\bs{c})^{-1}.
$$
D'après (5), pour $(z,\tilde{r})\in \ES{R}^{\wt{G},\omega}(\lambda)$, on a
$$
(\bs{c}\cdot \tilde{\rho})(z,\tilde{r})= z^{-1}\tilde{\rho}(z,\tilde{r})\tilde{\zeta}_\lambda(\bs{c})(z,\tilde{r}),\quad 
\rho\in {\rm Irr}(\ES{R}^{\wt{G},\omega}(\lambda),\theta_{\ES{R}}).
$$
On en déduit que:
\begin{itemize}
\item pour $\bs{c}\in \ES{C}$, l'application $\tilde{\bs{r}}\mapsto \tilde{\zeta}_{\lambda}(\bs{c})(\tilde{\bs{r}})$ est un élément de ${\rm Irr}(\ES{R}^{\wt{G},\omega}(\lambda),\omega_{\ES{R}})$; 
\item pour $\bs{c}\in \ES{C}$ et $\tilde{\rho}\in {\rm Irr}(\ES{R}^{\wt{G},\omega}(\lambda))$, on a 
$\bs{c}\cdot \tilde{\rho} = \tilde{\rho} + \tilde{\zeta}_{\lambda}(\bs{c})$;
\item l'application $\ES{C}\rightarrow {\rm Irr}(\ES{R}^{\wt{G},\omega}(\lambda),\omega_{\ES{R}}),\, 
\bs{c} \mapsto \tilde{\zeta}_\lambda(\bs{c})$ est un homomorphisme;
\item pour $\bs{c}=\bs{q}(z,g)\in \ES{C}$, $\tilde{\zeta}_\lambda(\bs{c})$ est un prolongement de $\zeta_{\lambda,\ES{C}}(\bs{c})= \zeta_\lambda(g_{\rm ad})$.
\end{itemize}
Cela démontre le lemme.
\end{proof}

Pour $\bs{c}\in \ES{C}$ et $\tilde{\rho}\in {\rm Irr}(\ES{R}^{\wt{G},\omega}(\lambda))$, on a donc
$$
{^{\bs c}(\Theta_{\tilde{\pi}_{\tilde{\rho}}})}= \Theta_{\tilde{\pi}_{\tilde{\rho} + \tilde{\zeta}_{\lambda}(\bs{c})}}. \leqno{(6)}
$$

\vskip1mm
Pour $\rho\in {\rm Irr}(\ES{R}^G(\lambda))$, on a vu plus haut que si la représentation $\pi_\rho$ de $G(F)$ est $K_1$--sphérique pour un sous--groupe hyperspécial $K_1$ de $G(F)$ tel que l'ensemble $N_{\smash{\wt{G}(F)}}(K_1)$ ne soit pas vide, alors elle se prolonge à $(\wt{G}(F),\omega)$. Pour définir un tel prolongement $\tilde{\pi}_{\tilde{\rho}}$, il suffit de choisir un élément $\delta_1\in N_{\smash{\wt{G}(F)}}(K_1)$, et d'imposer la condition
$$
\Theta_{\tilde{\pi}_{\tilde{\rho}}}(\bs{1}_{K_1\delta_1})= 1.
$$
Cela détermine $\tilde{\pi}_{\tilde{\rho}}$ et, ce qui revient au même, le relèvement $\tilde{\rho}$ de $\rho$ à $\ES{R}^{\wt{G},\omega}(\lambda)$. 

Soit $\rho_0\in {\rm Irr}(\ES{R}^G(\lambda))$ l'élément tel que la représentation $\pi_{\rho_0}$ est $K$--sphérique. La représentation $\pi_{\rho_0}$ de $G(F)$ se prolonge à $(\wt{G}(F),\omega)$, et l'on note $\tilde{\rho}_0$ le prolongement de $\rho_0$ à $\ES{R}^{\wt{G},\omega}(\lambda)$ normalisé par $\wt{K}$, \cad tel que
$$
\Theta_{\tilde{\pi}_{\tilde{\rho}_0}}(\bs{1}_{\wt{K}})= 1.
$$
Pour $\bs{c}=\bs{q}(z,g)\in \ES{C}$, notons $({^{\bs{c}}K},{^{\bs{c}}\wt{K}})$ le sous--espace hyperspécial de $\wt{G}(F)$ défini par
$$
{^{\bs{c}}K}= {\rm Int}_g(K),\quad {^{\bs{c}}\wt{K}}= 
z^{-1}{\rm Int}_g(\wt{K}).
$$
La représentation ${^{\bs{c}}(\tilde{\pi}_{\tilde{\rho}_0})}\simeq \tilde{\pi}_{\tilde{\rho}_0+ \tilde{\zeta}_\lambda(\bs{c})}$ de $(\wt{G}(F),\omega)$ 
est un prolongement de la représentation $\pi_{\rho_0}\circ {\rm Int}_{\smash{g_{\rm ad}^{-1}}}\simeq \pi_{\rho_0+ \zeta(g_{\rm ad})}$ de $G(F)$. Son caractère 
$\Theta_{\tilde{\pi}_{\tilde{\rho}_0 + \tilde{\zeta}_\lambda(\bs{c})}}={^{\bs{c}}(\Theta_{\tilde{\pi}_{\tilde{\rho}_0}})}$ vérifie
$$
\Theta_{\tilde{\pi}_{\tilde{\rho}_0 + \tilde{\zeta}_\lambda(\bs{c})}}(\bs{1}_{{^{\bs{c}}\wt{K}}})= {^{\bs{c}}(\Theta_{\tilde{\pi}_{\tilde{\rho}_0}})}({^{\bs{c}}(\bs{1}_{\wt{K}})})= \Theta_{\tilde{\pi}_{\tilde{\rho}_0}}(\bs{1}_{\wt{K}})=1.
$$
Plus généralement, pour $\tilde{\rho}\in {\rm Irr}(\ES{R}^{\wt{G},\omega}(\lambda))$, $\bs{c}=\bs{q}(z,g)\in \ES{C}$, et $\wt{K}_1$ un sous--espace hyperspécial de $\wt{G}(F)$, on a
$$
\Theta_{\tilde{\pi}_{\tilde{\rho}+\tilde{\zeta}_\lambda(\bs{c})}}(\bs{1}_{z^{-1}{\rm Int}_g(\wt{K}_1)})
= \Theta_{\tilde{\pi}_{\tilde{\rho}}}(\bs{1}_{\wt{K}_1}).\leqno{(7)}
$$

\begin{marema3}
{\rm 
D'après le lemme 1, les sous--espaces hyperspéciaux de $\wt{G}(F)$ forment une seule orbite sous l'action de $\ES{C}$ par conjugaison. Par suite les éléments $\rho_0 + \zeta_\lambda(g_{\rm ad})$,  pour $(z,g)\in C$, décrivent les éléments de ${\rm Irr}(\ES{R}^G(\lambda),\theta_{\ES{R}})$ tels que la représentation $\pi_\rho$ de $G(F)$ associée à $\rho$ est $K_1$--sphérique pour un sous--espace hyperspécial $(K_1,\wt{K}_1)$ de $\wt{G}(F)$. On peut se demander si ces éléments décrivent tout l'ensemble ${\rm Irr}(\ES{R}^G(\lambda),\theta_{\ES{R}})$. Il semble que ce soit le cas si l'ensemble $E_{\rm ell}^{\rm nr}$ n'est pas vide, mais en général nous ne 
savons pas répondre à cette question.
}
\end{marema3}

\begin{marema4}
{\rm 
Soit $\lambda$ un caractère unitaire et non ramifié de $T(F)$, et soit $\tilde{r}\in R^{\wt{G},\omega}(\lambda)$. Posons $\tau = (T,\lambda, \tilde{r})$. Alors la restriction de $\omega_{\tau}$ au sous--groupe $G_\sharp(F)$ de $\ES{C}$ est un caractère non ramifié de $G_\sharp(F)$, \cad qu'il correspond à une classe de cohomologie, disons 
$\bs{a}_{\tau}\in {\rm H}^1(W_F,Z(\hat{G}_\sharp))$, qui est non ramifiée. En effet, pour $\bs{c}=\bs{q}(z,g)\in \ES{C}$, on a
$$\omega_\tau(\bs{c})= \tilde{\zeta}_\lambda(\bs{c})(1,\tilde{r})^{-1}.
$$
Or $\tilde{\zeta}_\lambda(\bs{c})\in {\rm Irr}(\ES{R}^{\wt{G},\omega}(\lambda),\omega_{\ES{R}})$ est un prolongement de
$$\zeta_\lambda(g_{\rm ad})\in {\rm Irr}(\ES{R}^G(\lambda),\theta_{\ES{R}},\omega_{\ES{R}}) \subset 
{\rm Irr}(\ES{R}^G(\lambda))= R^G(\lambda)^D,
$$ et l'application $\zeta_\lambda: G_{\rm AD}(F)\rightarrow R^G(\lambda)^D$ se factorise à travers le groupe abélien fini
$$
\ker(\hat{Z}_{\rm sc}/(1-\phi)(\hat{Z}_{\rm sc})\rightarrow 
\hat{Z}/(1-\phi)(\hat{Z}))^D,
$$
qui paramétrise les caractères non ramifiés de $G_{\rm AD}(F)$ qui sont triviaux sur $q(G(F))$ --- cf. \ref{séries principales nr}. D'où le résultat, puisqu'un caractère non ramifié de $G_{\rm AD}(F)$ composé avec l'homomorphisme naturel $G_\sharp(F)\rightarrow G_{\rm AD}(F)$, est un caractère non ramifié de $G_\sharp(F)$.
}
\end{marema4}

\subsection{L'ensemble $\mathfrak{E}_{\rm t-nr}$}\label{l'ensemble t-nr}On a fixé en \ref{décomposition endoscopique} un ensemble $\mathfrak{E}_{\rm nr}$ de représentants des classes d'isomorphisme de données endoscopiques elliptiques et non ramifiées pour $(\wt{G},\bs{a})$. 
Soit $\mathfrak{E}_{\rm t-nr}$ le sous--ensemble de $\mathfrak{E}_{\rm nr}$ formé des données $\bs{G}'=(G',\ES{G}',\tilde{s})$ telles que le groupe sous--jacent $G'$ est un tore. 

Soit $\bs{T}'=(T',\ES{T}',\tilde{s})\in \mathfrak{E}_{\rm t-nr}$. Un {\it caractère affine} de $\wt{T}'(F)$ est une représentation 
$(\chi',\wt{\chi}')$ de $\wt{T}'(F)$ telle que $\chi'$ est un caractère de $T'(F)$. Un caractère affine $\wt{\chi}'$ de $T'(F)$ est unitaire s'il prend ses valeurs dans ${\Bbb U}$, et on dit qu'il est non ramifié si le caractère $\chi'$ de $T'(F)$ est non ramifié (\cad trivial sur le sous--groupe compact maximal $T'(F)_1$ de $T'(F)$). Au sous--espace hyperspécial $(K,\wt{K}')$ de $\wt{G}(F)$ est associé un sous--espace hyperspécial $(K',\wt{K}')$ de $T'(F)$. On a forcément $K'=T'(F)_1$ et $N_{\smash{\wt{T}'\!(F)}}(K')= \wt{T}'(F)$. 
\'Ecrivons $\wt{K}'= K'\delta'_0$ pour un élément $\delta'_0\in \wt{T}'(F)$. Puisque $\wt{T}'$ est à 
torsion intérieure, tout caractère de $T'(F)$ se prolonge en un caractère affine de $\wt{T}'(F)$, et un tel prolongement est déterminé par sa valeur en $\delta'_0$. Un caractère affine unitaire et non ramifié $\tilde{\chi}'$ de $\wt{T}'(F)$ est dit {\it normalisé par $\wt{K}$} si $\chi'(\delta'_0)=1$, \cad si la restriction de $\tilde{\chi}'$ à $\wt{K}'$ vaut $1$. Rappelons que pour transférer un caractère affine de $\wt{T}'(F)$ en une représen\-tation de $(\wt{G}(F),\omega)$, il faut fixer un isomorphisme 
${^LT'}\buildrel \simeq\over{\longrightarrow} \ES{T}'$. Pour cela on choisit un élément $(h,\phi)\in \ES{T}'$, et l'on prend l'isomorphisme défini par
$$
(t',w\phi^k)\mapsto (t',w)(h,\phi)^k,\quad t'\in \hat{T}',\, w\in I_F,\,k\in {\Bbb Z}.
$$
Si le lemme fondamental pour les unités des algèbres de Hecke sphériques est vrai pour la donnée $\bs{T}'$ (par exemple si la caractéristique résiduelle de $F$ est assez grande), 
un caractère affine unitaire et non ramifié $\tilde{\chi}'$ de $\wt{T}'(F)$ est normalisé par 
$\wt{K}$ si et seulement si
$$
\bs{\rm T}_{\bs{T}'}(\wt{\chi}')(\bs{1}_{\wt{K}})=1.
$$

\'Ecrivons $\tilde{s}=s\hat{\theta}$ et posons $\bs{h}= h\phi$. En considéront $\tilde{s}$ et $\bs{h}$ comme des automorphismes de $\hat{G}$, on a l'égalité $\tilde{s}\bs{h}= a(\phi)\bs{h}\tilde{s}$, où $a:W_F\rightarrow Z(\hat{G})$ est un cocycle dans la classe $\bs{a}$. 

{\bf Jusqu'à la fin de ce numéro, on suppose $\hat{G}= \hat{G}_{\rm AD}$.} Alors les automorphismes $\tilde{s}$ et $\bs{h}$ de $\hat{G}$ vérifient les propriétés suivantes:
\begin{enumerate}
\item[(1)]$\tilde{s}$ est semisimple et la composante neutre $\hat{S}= \hat{G}_{\tilde{s}}$ du centralisateur de $\tilde{s}$ dans $\hat{G}$ est un tore;
\item[(2)]$\tilde{s}\bs{h}=\bs{h}\tilde{s}$ (c'est l'égalité (1) de \ref{données endoscopiques nr});
\item[(3)]le sous--groupe des points fixes de $\hat{S}^{\bs{h}}$ est fini (cela traduit l'ellipticité de $\bs{T}'$);
\item[(4)]le commutant commun $Z_{\hat{G}}(\tilde{s},\bs{h})$ de $\tilde{s}$ et $\bs{h}$ dans $\hat{G}$ est fini;
\item[(5)]pour $n,\, m\in {\Bbb Z}$, l'élément $\tilde{s}^n\bs{h}^m$ est semisimple.
\end{enumerate}
En effet, (1), (2) et (3) résultent des définitions. Puisque la composante neutre $\hat{S}^{\bs{h},\circ}$ de $\hat{S}^{\bs{h}}$ est aussi la composante neutre de $Z_{\hat{G}}(\tilde{s},\bs{h})$, (3) est équivalent à (4). On sait qu'à conjugaison près, on peut supposer que $s$ appartient à $\hat{T}$. Alors $\hat{S}=\hat{T}^{\hat{\theta},\circ}$, et comme $\bs{h}$ normalise ce tore, il normalise aussi son commutant $\hat{T}$. Pour $n,\, m\in {\Bbb Z}$, on a alors $\tilde{s}^n\bs{h}^m= g\hat{\theta}^n\phi^m$ pour un élément $g\in N_{\hat{G}}(\hat{T})$. Une puissance d'un tel élément appartient au tore $\hat{T}$, d'où (5). 

Inversement, on a:
\begin{enumerate}
\item[(6)] si les automorphismes $\tilde{s}$ et $\bs{h}$ vérifient (2), (4) et (5), alors il vérifient (1).
\end{enumerate}
En effet, si $\tilde{s}$ et $\bs{h}$ vérifient (2), (4) et (5), alors comme $\bs{h}$ commute à $\tilde{s}$, il stabilise $\hat{G}_{\tilde{s}}$. Puisque $\bs{h}$ est semisimple, sa restriction à $\hat{G}_{\tilde{s}}$  l'est aussi, et $\bs{h}$ stabilise une paire de Borel $(\hat{B}',\hat{T}')$ de $\hat{G}_{\tilde{s}}$. Si $\hat{G}_{\tilde{s}}$ n'est pas réduit à $\hat{T}'$, alors $\bs{h}$ conserve le cocaractère de $\hat{T}'$ qui est la somme des coracines positives relativement à $\hat{B}'$. L'image de ce cocaractère est donc contenue dans le groupe $Z_{\hat{G}}(\tilde{s},\bs{h})$, et ce dernier n'est pas fini, ce qui  contredit (4). 

Notons aussi que puisque les rôles des automorphismes $\tilde{s}$ et $\bs{h}$ sont symétriques, s'ils vérifient (2), (4) et (5), alors on a aussi:
\begin{enumerate}
\item[$(1)'$]$\bs{h}$ est semisimple et la composante neutre $\hat{G}_{\bs{h}}$ du centralisateur de $\bs{h}$ dans $\hat{G}$ est un tore.
\end{enumerate}

\subsection{L'application $(\bs{T}',\tilde{\chi}')\mapsto \bs{\tau}_{\tilde{\chi}'}$}\label{l'application tau}
(On ne suppose plus que le groupe $\hat{G}$ est adjoint.) Soit $\bs{T}'=(T',\ES{T}',\tilde{s})\in \mathfrak{E}_{\rm t-nr}$. On suppose que $\tilde{s}$ appartient à $\hat{T}\hat{\theta}$. Choisissons un élément $(h,\phi)\in \ES{T}'$. Il définit comme en \ref{l'ensemble t-nr} un isomorphisme ${^LT'}\buildrel\simeq\over{\longrightarrow} \ES{T}'$. 
Posons $\tilde{s}=s\hat{\theta}$ et $\bs{h}= h\phi$. On a donc $\tilde{s}\bs{h}= a(\phi)\bs{h}\tilde{s}$, où $a:W_F\rightarrow Z(\hat{G})$ est un cocycle dans la classe $\bs{a}$. 

Soit $\chi'$ un caractère unitaire et non ramifié de $T'(F)$. \`A $\chi'$ est associé un paramètre non ramifié $\varphi^{T'}: W_F\rightarrow {^LT'}$, et un paramètre non ramifié
$$
\varphi': W_F\xrightarrow{\varphi^{T'}} {^LT}' \simeq
\ES{T}'\hookrightarrow {^LG}.
$$
\`A ce paramètre (tempéré et non ramifié) $\varphi'$ sont associés 
un $S$--groupe
$$
\bs{S}(\varphi')=S_{\varphi'}/S_{\varphi'}^\circ Z(\hat{G})^{\Gamma_F},\quad S_{\varphi'}=\{g\in \hat{G}: 
{\rm Int}_g(\bs{h})=\bs{h}\},
$$
et un ``$S$--groupe tordu''
$$
\wt{\bs{S}}(\varphi'\!,\bs{a})= \wt{S}_{\varphi'\!,a}/S_{\varphi'}^\circ Z(\hat{G})^{\Gamma_F},\quad 
\wt{S}_{\varphi'\!,a}= \{\tilde{g}\in \hat{G}\hat{\theta}: {\rm Int}_{\tilde{g}}(\bs{h})= a(\phi)\bs{h}\}.
$$
Notons que $\wt{S}_{\varphi'\!,a}$ est un $S_{\varphi'}$--espace tordu, et que $\wt{\bs{S}}(\varphi'\!,\bs{a})$ est un $\bs{S}({\varphi'})$--espace tordu.
Par définition, l'élément $\tilde{s}$ appartient à $\wt{S}_{\varphi'\!,a}$.  
\`A $\varphi'$ est aussi associé un caractère unitaire et non ramifié $\lambda_{\varphi'}$ de $T(F)$, bien défini à conjugaison près par $N_{G(F)}(T)$. Soit $b'\in \hat{T}'$ l'élément défini par 
$\varphi^{T'}(\phi)= b'\rtimes \phi $. Rappelons que (par ellipticité) on a $\hat{T}'^{\Gamma_F}\subset 
Z(\hat{G})$, ce qui entra\^{\i}ne l'égalité $\hat{T}'=(Z(\hat{G})\cap \hat{T}')(1-\phi_{T'})(\hat{T}')$. On peut donc supposer que $b'$ appartient à $Z(\hat{G})\cap \hat{T}'$. D'autre part, choisissons un élément $y\in \hat{G}_{\rm SC}$ tel que ${\rm Int}_{y^{-1}}(\bs{h})= \underline{h}\phi$ pour un élément $\underline{h}\in \hat{T}$, et notons $\psi^T: W_F\rightarrow {^LT}$ le paramètre tempéré et non ramifié défini par
$$
\psi^T(\phi)= b'\underline{h}\rtimes \phi.
$$
\`A $\psi^T$ sont associés un caractère unitaire et non ramifié $\lambda=\lambda_\psi$ de $T(F)$, et un paramètre (tempéré et non ramifié)
$$
\psi: W_F \xrightarrow{\psi^T} {^LT}\hookrightarrow {^LG}.
$$
On a donc un $R$--groupe $R^G(\lambda)$, et un $R$--groupe tordu $R^{\wt{G},\omega}(\lambda)$ --- cf. \ref{R-groupes tordus}. On a aussi un $S$--groupe
$\bs{S}(\psi)= S_\psi/ S_\psi^\circ Z(\hat{G})^{\Gamma_F}$ et un $S$--groupe tordu 
$\wt{\bs{S}}(\psi,\bs{a})= \wt{S}_{\psi,a}/S_\psi^\circ Z(\hat{G})^{\Gamma_F}$, où $\wt{S}_{\psi,a}$ est défini en \ref{paramètres} (cf. aussi la remarque 2 de \ref{séries principales nr}). 
Ces ensembles sont ceux obtenus en rempla\c{c}ant $\bs{h}$ par $b'\underline{h}\phi$ dans les définitions de $\bs{S}(\varphi')$ et $\wt{\bs{S}}(\varphi'\!,\bs{a})$. Puisque $b'\in Z(\hat{G})$, l'automorphisme ${\rm Int}_{y}$ de $\hat{G}$ induit un isomorphisme de $S_\psi$ sur $S_{\varphi'}$, et donc un isomorphisme de $S_\psi^\circ$ sur $S_{\varphi'}^\circ$. D'autre part, puisque $b'\in Z(\hat{G})$ et $\hat{\theta}(b')=b'$, l'automorphisme ${\rm Int}_{y}$ de $\hat{G}\hat{\theta}$ induit aussi une bijection de $\wt{S}_{\psi,a}$ sur $\wt{S}_{\varphi'\!,a}$. D'où, par passage aux quotients, une application bijective
$$
{\rm Int}_{y}: \wt{\bs{S}}(\psi,\bs{a})\rightarrow \wt{\bs{S}}(\varphi'\!,\bs{a}),\leqno{(1)}
$$
qui prolonge l'isomorphisme ${\rm Int}_{y}:{\bs{S}}(\psi)\rightarrow {\bs{S}}(\varphi')$.  

Posons $\hat{U}= \hat{T}^{\Gamma_F,\circ}$, et notons $\hat{N}$ le normalisateur de $\hat{T}$ dans $\hat{G}$. Puisque $\hat{U}$ est un tore maximal de $S_\psi^\circ$ et que ${\rm Int}_{y}(S_{\psi,{\rm ad}}^\circ)= S_{\varphi'\!,{\rm ad}}^\circ$ est un tore (\ref{l'ensemble t-nr}, relation $(1)'$), on a $S_\psi^\circ = \hat{U}$. Par suite $W^G_0(\lambda)=\{1\}$, et l'on a un isomorphisme naturel (cf. la remarque 1 de \ref{séries principales nr})
$$
R^G(\lambda)=W^G(\lambda)\buildrel\simeq\over{\longrightarrow} (\hat{N}\cap S_{\psi})/ (\hat{T}\cap S_\psi)= \bs{S}(\psi),\, r\mapsto\iota_\lambda(r)^{-1}
$$
avec
$$
\hat{T}\cap S_\psi = \hat{U}Z(\hat{G})^{\Gamma_F},\quad \hat{T}\cap S_\psi^\circ =\hat{U}.
$$
D'après la remarque 2 de \ref{séries principales nr}, on a aussi une application bijective
$$
\tilde{\iota}_\lambda:R^{\wt{G},\omega}(\lambda)=W^{\wt{G},\omega}(\lambda)\rightarrow (\hat{N}\hat{\theta}\cap \wt{S}_{\psi,a})/ (\hat{T}\cap S_\psi)= \wt{\bs{S}}(\psi,\bs{a})\leqno{(2)}
$$
qui vérifie
$$
\tilde{\iota}_\lambda(\tilde{r} r')= \iota_\lambda(r)\tilde{\iota}_\lambda(\tilde{r}),\quad 
\tilde{r}\in R^{\wt{G},\omega}(\lambda),\, r\in R^G(\lambda).
$$
En composant (1) et (2), on obtient une application bijective
$$
\tilde{\jmath}_\lambda :R^{\wt{G},\omega}(\lambda) \rightarrow \wt{\bs{S}}(\varphi'\!,\bs{a}),\leqno{(3)}
$$
qui vérifie
$$
\tilde{\jmath}_\lambda(\tilde{r} r) = {\rm Int}_y(\iota_\lambda(r))\tilde{\jmath}_\lambda(\tilde{r}),\quad 
\tilde{r}\in R^{\wt{G},\omega}(\lambda),\, r\in R^G(\lambda).
$$

\begin{marema1}
{\rm 
La bijection (3) dépend bien sûr du choix de $y$ (d'ailleurs le caractère $\lambda$ aussi en dépend). Soit 
$y_1\in \hat{G}_{\rm SC}$ un autre élément tel que ${\rm Int}_{\smash{y_1^{-1}}}(\bs{h})= t_1\phi$ pour un $t_1\in \hat{T}$. Cet élément définit comme ci--dessus un paramètre $\psi^T_1: W_F \rightarrow {^LT}$, auquel est associé un caractère unitaire et non ramifié $\lambda_1$ de $T(F)$, et un paramètre (tempéré et non ramifié) $\psi_1: W_F \rightarrow {^LG}$. Posons $x= y_1^{-1}y$. On a donc 
${\rm Int}_x(S_\psi)= S_{\psi_1}$. Puisque $S_{\psi_1}^\circ = \hat{U} = S_\psi^\circ$ et $Z_{\hat{G}}(\hat{U})= \hat{T}$, l'élément $x$ appartient à $\hat{N}$, et sa projection sur $W=\hat{N}/\hat{T}$ appartient à $W^{\Gamma_F}$. On en déduit que la bijection
$$\tilde{\jmath}_{\lambda_1}:R^{\wt{G},\omega}(\lambda_1)\rightarrow \wt{\bs{S}}(\varphi'\!,\bs{a})$$
se déduit de (3) par transport de structure grâce à l'automorphisme ${\rm Int}_x$ de $\hat{T}$.
}
\end{marema1}

Soit $\tilde{r}$ l'élément de $R^{\wt{G},\omega}(\lambda)$ correspondant à la projection de $\tilde{s}\in \wt{S}_{\varphi'\!,a}$ sur $\wt{\bs{S}}(\varphi'\!,\bs{a})$ par la bijection (3). La condition d'ellipticité sur $\tilde{s}$ assure que $\tilde{r}$ appartient à $R^{\wt{G},\omega}_{\rm reg}(\lambda)$. En effet, posons $\hat{U}'= S_{\varphi'}^\circ\;(= {\rm Int}_y(S_\psi^\circ))$. C'est un tore de $\hat{G}$, qui co\"{\i}ncide avec le centralisateur connexe $\hat{G}_{\bs{h}}= Z_{\hat{G}}(\bs{h})^\circ$. L'application naturelle $\hat{U}\rightarrow \hat{T}/(1-\phi_T)(\hat{T})$ est surjective, de noyau fini. Ainsi l'isomorphisme ${\rm Int}_y$  identifie $X(A_T)\otimes_{\Bbb Z} {\Bbb R}$ à $\check{X}(\hat{U}')\otimes_{\Bbb Z}{\Bbb R}$, et l'action de $\tilde{r}$ à celle de $\tilde{s}=s\hat{\theta}$. D'autre part, on a l'égalité
$$
\det(1-\tilde{r}; \ES{A}_T/\ES{A}_{\wt{G}})= \det (1-\theta; \ES{A}_G/\ES{A}_{\wt{G}})\det (1-\tilde{r}; \ES{A}_T/\ES{A}_G).\leqno{(4)}
$$
Soit $\hat{U}'_{\rm ad}= S_{\varphi'\!,{\rm ad}}^\circ$ l'image de $\hat{U}'$ dans $\hat{G}_{\rm AD}$ par la projection naturelle 
$q:\hat{G}\rightarrow \hat{G}_{\rm AD}$. D'après la condition d'ellipticité, le commutant commun de $\tilde{s}$ et $\bs{h}$ dans $\hat{G}_{\rm AD}$ est fini (d'après \ref{l'ensemble t-nr}, relation (4)). Par suite le sous--groupe $(\hat{U}'_{\rm ad})^{\tilde{s}}\subset \hat{U}'_{\rm ad}$ formé des points fixes sous $\tilde{s}$ est fini, $1-\tilde{r}$ est un automorphisme de $\ES{A}_T/\ES{A}_G$, et $\ES{A}_T^{\tilde{r}}= \ES{A}_G^\theta = \ES{A}_{\wt{G}}$.  
On obtient donc un triplet elliptique essentiel 
(non ramifié) $\tau_{\chi'}=(T,\lambda,\tilde{r})\in E_{\rm ell}^{\rm nr}$, bien défini à conjugaison près. 

Soit $\tilde{\chi}'$ un caractère affine unitaire de $\wt{T}'(F)$ prolongeant $\chi'$. 
Alors on peut relever $\tau_{\chi'}$ en un triplet $\bs{\tau}_{\tilde{\chi}'}= (T,\lambda, \bs{\tilde{r}})\in \ES{E}_{\rm ell}^{\rm nr}$ --- lui aussi bien défini à conjugaison près --- en imposant la condition
$$
\Theta_{\bs{\tau}_{\tilde{\chi}'}}({\bs 1}_{\wt{K}})= \tilde{\chi}'(\bs{1}_{\wt{K}'}).\leqno{(5)}
$$
Si le lemme fondamental pour les unités des algèbres de Hecke sphériques est vrai pour la donnée 
$\bs{T}'$, cela revient à imposer que
$$
\Theta_{\bs{\tau}_{\tilde{\chi}'}}({\bs 1}_{\wt{K}}) = \bs{\rm T}_{\bs{T}'}(\tilde{\chi}')(\bs{1}_{\wt{K}}).
$$

\begin{marema2}
{\rm Tout élément $x\in {\rm Aut}(\bs{T}')$ définit un automorphisme $\alpha_x$ de $T'$ et un automorphisme $\tilde{\alpha}_x$ de $\wt{T}'=T'\times_{\ES{Z}(G)}\ES{Z}(\wt{G},\ES{E})$. Soient $(\chi'_1,\tilde{\chi}'_1)$ et $(\chi'_2,\tilde{\chi}'_2)$ deux caractères affines unitaires et non ramifiés de $\wt{T}'(F)$. D'après la construction, si $\chi'_1$ et $\chi'_2$ se déduisent l'un de l'autre par un élément de ${\rm Aut}(\bs{T}')$, alors $\tau_{\chi'_1}$ et $\tau_{\chi'_2}$ sont dans la même classe de conjugaison de $E_{\rm ell}^{\rm nr}$ . De même, si $\tilde{\chi}'_1$ et $\tilde{\chi}'_2$ se déduisent l'un de l'autre par un élément de ${\rm Aut}(\bs{T}')$, alors $\bs{\tau}_{\tilde{\chi}'_1}$ et $\bs{\tau}_{\tilde{\chi}'_2}$ sont dans la même classe de conjugaison de $\ES{E}_{\rm ell}^{\rm nr}$ . 

}
\end{marema2}

\begin{monlem1}
Soit $\bs{T}'= (T',\ES{T}',\tilde{s})\in \mathfrak{E}_{\rm t-nr}$. Soit $(\chi',\tilde{\chi}')$ un caractère affine unitaire et non ramifié de $\wt{T}'(F)$, et soit $\bs{\tau}'= \bs{\tau}_{\tilde{\chi}'}\;(\in \ES{E}_{\rm ell}^{\rm nr}/{\rm conj.})$.  
On a l'égalité
$$
(\Theta_{\bs{\tau}'},\Theta_{\bs{\tau}'})_{\rm ell} = c(\wt{G},\bs{G}')^{-1}\vert {\rm Aut}(\bs{T}')/B \vert^{-1},
$$
où $B$ est le stabilisateur de $\tilde{\chi}'$ dans ${\rm Aut}(\bs{T}')$.
\end{monlem1}

\begin{proof}Reprenons les constructions (et aussi les notations) précédentes. D'après la relation (6) de \ref{triplets elliptiques essentiels}, on a
$$
(\Theta_{\bs{\tau}'},\Theta_{\bs{\tau}'})_{\rm ell}= \vert {\rm Stab}(R^G(\sigma),\tilde{r})\vert 
\vert \det(1-\tilde{r}; \ES{A}_T/\ES{A}_{\wt{G}})\vert,
$$
et d'après la relation (4) ci--dessus, on a
$$
\vert \det(1-\tilde{r}; \ES{A}_T/\ES{A}_{\wt{G}})\vert=
\vert \det (1-\theta; \ES{A}_G/\ES{A}_{\wt{G}})\vert \vert \det (1-\tilde{r}; \ES{A}_T/\ES{A}_G)\vert.
$$
Le terme $\vert \det (1-\tilde{r}; \ES{A}_T/\ES{A}_G)\vert$ est égal à 
$\vert \det (1-\tilde{s}; \check{X}(\hat{U}'_{\rm ad})\otimes_{\Bbb Z}{\Bbb R})\vert$. Par un résultat général concernant les automorphismes d'un tore complexe, ce terme est aussi égal à  
$\vert (\hat{U}'_{\rm ad})^{\tilde{s}}\vert$.

On a un isomorphisme
$$
{\rm Stab}(R^G(\lambda), \tilde{r})\simeq {\rm Stab}(\bs{S}(\varphi'), \tilde{s}), 
$$
où ${\rm Stab}(\bs{S}(\varphi'),\tilde{s})$ désigne le commutant de $\tilde{s}$ dans $\bs{S}(\varphi')$. 
Notons $\mathfrak{X}$ le groupe des $x\in \hat{G}$ tels que ${\rm Int}_{x^{-1}}(\bs{h})= \bs{h}$ et 
${\rm Int}_{x^{-1}}(\tilde{s})\in Z(\hat{G})\tilde{s}$. Par définition, $\mathfrak{X}$ est un sous--groupe de $S_{\varphi'}$. 
En se rappelant que $\tilde{s}\bs{h}\in Z(\hat{G})\bs{h}\tilde{s}$, on voit que si $x\in \mathfrak{X}$, alors ${\rm Int}_{x^{-1}}(\tilde{s})\in Z(\hat{G})^{\Gamma_F}\tilde{s}$. Cela implique que l'image de $\mathfrak{X}$ dans $\bs{S}(\varphi')=S_{\varphi'}/\hat{U}'Z(\hat{G})^{\Gamma_F}$ est contenue dans ${\rm Stab}(\bs{S}(\varphi'),\tilde{s})$. On en déduit une suite
$$
1\rightarrow (\mathfrak{X}\cap \hat{U}')Z(\hat{G})^{\Gamma_F}
\rightarrow \mathfrak{X} \rightarrow {\rm Stab}(\bs{S}(\varphi'),\tilde{s})\rightarrow 1.
\leqno{(6)}
$$
Montrons que cette suite est exacte. L'injectivité à gauche et l'exactitude au centre sont faciles, puisque $\mathfrak{X}\cap (\hat{U}'Z(\hat{G})^{\Gamma_F})=
(\mathfrak{X}\cap \hat{U}')Z(\hat{G})^{\Gamma_F}$. Il s'agit de prouver la surjectivité à droite. Soit $x\in S_{\varphi'}$ dont l'image dans $\bs{S}(\varphi')$ est fixée par $\tilde{s}$. On a donc 
${\rm Int}_{\tilde{s}}(x)\in \hat{U}'Z(\hat{G})^{\Gamma_F}x$. Comme l'automorphisme $\tilde{s}$ de $\hat{U}'_{\rm ad}$ n'a qu'un nombre fini de points fixes, on a
$$
\hat{U}'= (Z(\hat{G})\cap \hat{U}')(1-\tilde{s})(\hat{U}').
$$
Par suite, quitte à multiplier $x$ par un élément convenable de $\hat{U}'$, on peut supposer que 
${\rm Int}_{\tilde{s}}(x)\in Z(\hat{G})$. Mais alors $x$ appartient à $\mathfrak{X}$, d'où la surjectivité cherchée. 

Remarquons que le groupe $\mathfrak{X}$ est d'image finie dans $\hat{G}_{\rm AD}$ (toujours parce que le commutant commun de $\tilde{s}$ et $\bs{h}$ dans $\hat{G}_{\rm AD}$ est fini). Il en résulte que la composante neutre de $\mathfrak{X}$ est $Z(\hat{G})^{\Gamma_F,\circ}$. Dans la suite (6), on peut quotienter les deux premiers termes par $Z(\hat{G})^{\Gamma_F,\circ}$. On obtient ainsi des groupes finis, d'où l'égalité
$$
\vert {\rm Stab}(\bs{S}(\varphi'),\tilde{s})\vert = \vert \mathfrak{X}/ Z(\hat{G})^{\Gamma_F,\circ}\vert 
\vert (\mathfrak{X}\cap \hat{U}')Z(\hat{G})^{\Gamma_F}/Z(\hat{G})^{\Gamma_F,\circ}\vert^{-1}.
$$
Par définition de $\mathfrak{X}$, l'ensemble $\mathfrak{X}\cap \hat{U}'$ s'envoie surjectivement sur $(\hat{U}'_{\rm ad})^{\tilde{s}}$. On en déduit une suite exacte courte
$$
1\rightarrow Z(\hat{G})^{\Gamma_F}/Z(\hat{G})^{\Gamma_F,\circ}\rightarrow 
(\mathfrak{X}\cap \hat{U}')Z(\hat{G})^{\Gamma_F}/Z(\hat{G})^{\Gamma_F,\circ}
\rightarrow (\hat{U}'_{\rm ad})^{\tilde{s}}\rightarrow 1,
$$
puis l'égalité
$$
 \vert (\hat{U}'_{\rm ad})^{\tilde{s}}\vert=  
\vert (\mathfrak{X}\cap \hat{U}')Z(\hat{G})^{\Gamma_F}/Z(\hat{G})^{\Gamma_F,\circ}\vert 
\vert \pi_0(Z(\hat{G})^{\Gamma_F})\vert^{-1}.
$$
On obtient
$$
\vert {\rm Stab}(\bs{S}(\varphi'),\tilde{s})\vert \vert (\hat{U}'_{\rm ad})^{\tilde{s}}\vert 
= \vert \mathfrak{X}/ Z(\hat{G})^{\Gamma_F,\circ}\vert \vert \pi_0(Z(\hat{G})^{\Gamma_F})\vert^{-1}.
$$
En rassemblant les calculs ci--dessus, on obtient l'égalité
$$
(\Theta_{\bs{\tau}'},\Theta_{\bs{\tau}'})_{\rm ell}=\vert \det (1-\theta; \ES{A}_G/\ES{A}_{\wt{G}})\vert 
\vert \mathfrak{X}/ Z(\hat{G})^{\Gamma_F,\circ}\vert \vert \pi_0(Z(\hat{G})^{\Gamma_F})\vert^{-1}.
\leqno{(7)}
$$

Le stabilisateur $B$ de $\tilde{\chi}'$ dans ${\rm Aut}(\bs{T}')$ est aussi celui de $\chi'$, autrement dit du paramètre $\varphi'$ modulo action de $\hat{T}'$. Rappelons que $\varphi'(\phi)= b'(h,\phi)$, avec $b'\in Z(\hat{G})\cap \hat{T}'$. C'est--à--dire que $B$ est le groupe des $x\in \hat{G}$ tels que ${\rm Int}_{x^{-1}}(b'\bs{h})\in (1-\bs{h})(\hat{T}')b'\bs{h}$ et ${\rm Int}_{x^{-1}}(\tilde{s})\in Z(\hat{G})\tilde{s}$. On voit que $B= \hat{T}'\mathfrak{X}$. On obtient une suite exacte
$$
1\rightarrow (\mathfrak{X}\cap \hat{T}')Z(\hat{G})^{\Gamma_F,\circ}/Z(\hat{G})^{\Gamma_F,\circ}
\rightarrow \mathfrak{X}/Z(\hat{G})^{\Gamma_F,\circ}
\rightarrow B/ \hat{T}'Z(\hat{G})^{\Gamma_F,\circ}
\rightarrow 1,
$$
d'où (puisque tous ces groupes sont finis)
$$
\vert \mathfrak{X}/Z(\hat{G})^{\Gamma_F,\circ}\vert =
\vert (\mathfrak{X}\cap \hat{T}')Z(\hat{G})^{\Gamma_F,\circ}/Z(\hat{G})^{\Gamma_F,\circ}\vert 
\vert B/ \hat{T}'Z(\hat{G})^{\Gamma_F,\circ}\vert.
$$
Le groupe $\hat{T}'Z(\hat{G})^{\Gamma_F,\circ}$ est la composante neutre de $B$, comme de ${\rm Aut}(\bs{T}')$; d'où
$$
\vert B/ \hat{T}'Z(\hat{G})^{\Gamma_F,\circ}\vert= \vert {\rm Aut}(\bs{T}')/B\vert^{-1} 
\vert \pi_0({\rm Aut}(\bs{T}'))\vert.
$$
En utilisant \cite[2.1, page 19]{KS1}, on obtient
$$
\vert \pi_0({\rm Aut}(\bs{T}'))\vert = \vert {\rm Out}(\bs{T}')\vert \vert \pi_0([Z(\hat{G})/(Z(\hat{G})\cap \hat{T}')]^{\Gamma_F}).
$$
D'autre part, on a $\mathfrak{X}\cap \hat{T}'= \hat{T}'^{\Gamma_F}$, d'où une suite exacte
$$
1\rightarrow (\hat{T}'\cap Z(\hat{G})^{\Gamma_F,\circ})/Z(\hat{G})^{\Gamma_F,\hat{\theta},\circ}
\rightarrow \hat{T}'^{\Gamma_F}/ Z(\hat{G})^{\Gamma_F,\hat{\theta},\circ}
\rightarrow (\mathfrak{X}\cap \hat{T}')Z(\hat{G})^{\Gamma_F,\circ}/Z(\hat{G})^{\Gamma_F,\circ}
\rightarrow 1.
$$
On en déduit (à nouveau puisque ces groupes sont finis)
$$
\vert (\mathfrak{X}\cap \hat{T}')Z(\hat{G})^{\Gamma_F,\circ}/Z(\hat{G})^{\Gamma_F,\circ}\vert =
\vert \pi_0(\hat{T}'^{\Gamma_F})\vert \vert \pi_0(\hat{T}'\cap Z(\hat{G})^{\Gamma_F,\circ})\vert^{-1}
$$
L'égalité (7) devient
\begin{eqnarray*}
(\Theta_{\bs{\tau}'},\Theta_{\bs{\tau}'})_{\rm ell} &=&\vert \det (1-\theta; \ES{A}_G/\ES{A}_{\wt{G}})\vert  \vert \pi_0(Z(\hat{G})^{\Gamma_F})\vert^{-1}\times \cdots \\
&& \cdots \times \vert {\rm Out}(\bs{T}')\vert \vert 
\pi_0([Z(\hat{G})/(Z(\hat{G})\cap \hat{T}')]^{\Gamma_F})
\vert  \times \cdots \\
&& \cdots \times \vert  \pi_0(\hat{T}'^{\Gamma_F})\vert \vert \pi_0(\hat{T}'\cap Z(\hat{G})^{\Gamma_F,\circ})\vert^{-1}
\vert {\rm Aut}(\bs{T}')/B\vert^{-1} 
\end{eqnarray*}
On reconnaît le produit des premiers termes (tous sauf le dernier): c'est $c(\wt{G},\bs{G})^{-1}$. D'où le lemme. 
\end{proof}

Par linéarité, l'application $(\bs{T}',\tilde{\chi}')\mapsto \Theta_{\bs{\tau}_{\tilde{\chi}'}}$ définit un 
homomorphisme (normalisé grâce au choix de l'espace hyperspécial $\wt{K}$, \cad par la condition (5))
$$
\bigoplus_{\bs{T}'\in \mathfrak{E}_{\rm t-nr}}\bs{D}^{\rm nr}(\wt{T}'(F))\rightarrow \bs{D}_{\rm ell}^{\rm nr}(\wt{G}(F),\omega).\leqno{(8)}
$$

\begin{monlem2}
L'homomorphisme (8) est surjectif.
\end{monlem2}

\begin{proof}Il s'agit d'inverser la construction de l'application $(\bs{T}',\tilde{\chi}')\mapsto \bs{\tau}_{\tilde{\chi}'}$. On part d'un triplet $\tau=(T,\lambda,\tilde{r})\in E_{\rm ell}^{\rm nr}$. Soit $\psi^T: W_F\rightarrow {^LT}$ un paramètre non ramifié associé à $\lambda$, et soit $\psi:W_F \rightarrow {^LG}$ le paramètre non ramifié obtenu en composant $\psi^T$ avec le plongement naturel ${^LT}\hookrightarrow {^LG}$. Soit $\tilde{s}=s\hat{\theta}$ un élément de $\hat{N}\hat{\theta}\cap\wt{S}_{\psi, a}$ se projetant sur l'image de $\tilde{r}$ dans $\wt{\bs{S}}(\psi,\bs{a})$ par la bijection (2). Soit $h\in \hat{T}$ l'élément défini par $\psi^T(\phi)= h\rtimes \phi$. Posons $\bs{h}= h\phi$. Alors on a $\tilde{s}\bs{h}=a(\phi)\bs{h}\tilde{s}$. Puisque $W_0^G(\lambda)=\{1\}$, le centralisateur connexe $\hat{G}_{\bs{h}} = S_{\psi}^\circ$ co\"{\i}ncide avec le tore $\hat{U}= \hat{T}^{\Gamma_F,\circ}$ (\ref{séries principales nr}.(5)). L'application naturelle $\hat{U}\rightarrow \hat{T}/ (1-\phi_T)(\hat{T})$ est surjective, de noyau fini. On identifie $X(A_T)\otimes_{\Bbb Z}{\Bbb R}$ à $\check{X}(\hat{U})\otimes_{\Bbb Z}{\Bbb R}$, et l'action de $\tilde{r}$ à celle de $\tilde{s} = s \hat{\theta}$. Soit $\hat{U}_{\rm ad}= S_{\psi,{\rm ad}}^\circ$ l'image de $\hat{U}$ dans $\hat{G}_{\rm AD}$ par la projection naturelle $\hat{G}\rightarrow \hat{G}_{\rm AD}$. La condition d'ellipticité sur $\tilde{r}$ assure (d'après (4)) que l'automorphisme $\tilde{s}$ de $\hat{U}_{\rm ad}$ n'a qu'un nombre fini de points fixes, par conséquent le commutant commun de $\tilde{s}$ et $\bs{h}$ dans $\hat{G}_{\rm AD}$ est fini, et les automorphismes $\tilde{s}$ et $\bs{h}$ de $\hat{G}_{\rm AD}$ vérifient les propriétés (1) à (6) de \ref{l'ensemble t-nr} (rappelons que dans ces propriétés, les rôles de $\tilde{s}$ et $\bs{h}$ sont symétriques). En particulier, $\tilde{s}$ est semisimple --- comme automorphisme de $\hat{G}_{\rm AD}$, et donc aussi comme élément de $\hat{G}$ --- et le centralisateur connexe $\hat{G}_{\tilde{s}}$ est un tore, disons $\hat{T}'$. Soit 
$\ES{T}'$ le sous--groupe de ${^LG}$ engendré par $\hat{T}'$ et $\psi(W_F)$. C'est un sous--groupe fermé de ${^LG}$, et une extension scindée de $W_F$ par $\hat{T}'$. On munit $\hat{T}'$ de l'action galoisienne $\sigma \mapsto \sigma_{\ES{T}'}$ définie comme suit: pour $\sigma = w\phi^k$ avec $w\in I_F$ et $k\in {\Bbb Z}$, on pose $\sigma_{\ES{T}'}=\phi_{\ES{T}'}^k$ avec 
$$
\phi_{\ES{T}'}(x)= {\rm Int}_{\bs{h}}(x)= h\phi_G(x)h^{-1},\quad x\in \hat{T}'.
$$
Soit $T'$ un tore défini sur $F$ tel que ${^LT}= \hat{T}'\rtimes W_F$, \cad tel que l'action galoisienne $\sigma \mapsto \sigma_{T'}$ sur $\hat{T}'$ donnée par $T'$ co\"{\i}ncide avec $\sigma \mapsto \sigma_{\ES{T}'}$. Par construction, le triplet $(T',\ES{T}',\tilde{s})$ est une donnée endoscopique non ramifiée pour $(\wt{G},\bs{a})$, et puisque le commutant commun de $\bs{h}$ et $\tilde{s}$ dans $\hat{G}_{\rm AD}$ est fini, cette donnée est elliptique: on a $\hat{T}'^{\Gamma_F,\circ}= [Z(\hat{G})^{\hat{\theta}}]^{\Gamma_F,\circ}$. Elle est donc isomorphe à un unique élément $\bs{T}'_1=(T'_1,\ES{T}'_1,\tilde{s}_1)$ de $\mathfrak{E}_{\rm t-nr}$. Soit $x\in \hat{G}$ tel que
$$
x\ES{T}'x^{-1} = \ES{T}'_1,\quad x\tilde{s}x^{-1}\in Z(\hat{G})\tilde{s}_1.
$$
Le choix d'un élément dans $\ES{T}'_1$ définit un isomorphisme $\ES{T}'_1\buildrel\simeq\over{\longrightarrow}{^LT'_1}$ qui, composé avec l'homomorphisme $W_F \buildrel \varphi\over{\longrightarrow}\ES{T}'\xrightarrow{{\rm Int}_x} \ES{T}'_1$, donne un paramètre $\varphi^{T'_1}:W_F\rightarrow {^L{T'_1}}$. \`A ce paramètre correspond un caractère unitaire et non ramifié $\chi'_1$ de $T'_1(F)$, que l'on prolonge en un caractère affine unitaire $\tilde{\chi}'_1$ de $\wt{T}'_1(F)$. Par construction, les éléments $\tau$ et $\tau_{\chi'_1}$ de $E_{\rm ell}^{\rm nr}$ sont dans la même classe de conjugaison par $G(F)$. Si de plus on choisit un relèvement $\bs{\tau}\in \ES{E}_{\rm ell}^{\rm nr}$ de $\tau$, alors on peut normaliser le prolongement $\tilde{\chi}'_1$ en imposant la condition  
$\tilde{\chi}'_1(\bs{1}_{\wt{K}'_1})= \Theta_{\bs{\tau}}(\bs{1}_{\wt{K}})$, où $(K'_1,\wt{K}'_1)$ est le sous--espace hyperspécial de $\wt{T}'_1(F)$ associé à $(K,\wt{K})$. Alors les éléments $\bs{\tau}$ et $\bs{\tau}_{\tilde{\chi}'_1}$ de $\ES{E}_{\rm ell}^{\rm nr}$ sont dans la même classe de conjugaison par $G(F)$. Cela démontre le lemme.
\end{proof}

\begin{marema3}
{\rm 
On verra plus loin (\ref{preuve}) que l'homomorphisme (8) se factorise en un isomorphisme
$$
\bigoplus_{\bs{T}'\in \mathfrak{E}_{\rm t-nr}}\bs{D}^{\rm nr}(\wt{T}'(F))^{{\rm Aut}(\bs{T}')}\rightarrow \bs{D}_{\rm ell}^{\rm nr}(\wt{G}(F),\omega).
$$ 
}
\end{marema3}

\subsection{Preuve de la commutativité du diagramme (5) de \ref{transfert spectral elliptique}}\label{preuve}
On rappelle qu'à une donnée $\bs{G}'\in \mathfrak{E}$ est associé un caractère $\omega_{\bs{G}'}$ de $G_\sharp(F)$. Si l'ensemble $\mathfrak{E}_{\rm t-nr}$ n'est pas vide, alors les propriétés suivantes sont vérifiées:
\begin{enumerate}
\item[(1)] Soient $\bs{T}'_1,\, \bs{T}'_2\in \mathfrak{E}_{\rm t-nr}$. Si $\omega_{\bs{T}'_1}= \omega_{\bs{T}'_2}$, alors $\bs{T}'_1=\bs{T}'_2$. 
\item[(2)]Soit $\bs{T}'=(T',\ES{T}',\tilde{s})\in \mathfrak{E}_{\rm t-nr}$, et soit $(\chi',\tilde{\chi}')$ un caractère affine unitaire et non ramifié de $\wt{T}'(F)$. 
La distribution $\bs{\rm T}_{\bs{T}'}(\tilde{\chi}')$ est un vecteur propre pour l'action du groupe $\ES{C}$ relativement à un caractère $\omega_{\chi'}$ de ce groupe, \cad qu'on a
$$
{^{\bs{c}}(\bs{\rm T}_{\bs{T}'}(\tilde{\chi}'))}= \omega_{\chi'}(\bs{c})\bs{\rm T}_{\bs{T}'}(\tilde{\chi}'),\quad \bs{c}\in \ES{C}.
$$
\item[(3)]Soit $\bs{T}=(T',\ES{T}',\tilde{s})\in \mathfrak{E}_{\rm t-nr}$, et soient $(\chi'_1,\tilde{\chi}'_1)$ et $(\chi'_2,\tilde{\chi}'_2)$ deux caractères affines unitaires et non ramifiés de $\wt{T}'(F)$. Si $\omega_{\chi'_1}=\omega_{\chi'_2}$, alors à homothétie près, $\tilde{\chi}'_1$ et $\tilde{\chi}'_2$ se déduisent l'un de l'autre par l'action d'un élément de ${\rm Aut}(\bs{T}')$.
\item[(4)]Soit $\bs{T}=(T',\ES{T}',\tilde{s})\in \mathfrak{E}_{\rm t-nr}$, et soit $\chi'$ un caractère unitaire et non ramifié de $T'(F)$. On a $\omega_{\chi'}= \omega_{\tau_{\chi'}}$. 
\end{enumerate}
Ces quatre propriétés seront démontrées (sous l'hypothèse $\mathfrak{E}_{\rm t-nr}\neq \emptyset$) dans la section 5. Admet\-tons ce résultat, ainsi que le lemme fondamental pour les unités des algèbres de Hecke sphériques pour toutes les données $\bs{T}'\in \mathfrak{E}_{\rm t-nr}$, et déduisons--en que le diagramme (5) de \ref{transfert spectral elliptique} est commutatif.

Supposons pour commencer que l'ensemble $E_{\rm ell}^{\rm nr}$ est vide. Alors d'après \ref{triplets elliptiques essentiels}, on a
$$\bs{D}_{\rm ell}(\wt{G}(F),\omega)= \bs{D}_{\rm ell}^{\rm ram}(\wt{G}(F),\omega),$$
et dans ce cas, le diagramme (5) de \ref{transfert spectral elliptique} est trivialement commutatif.

Supposons maintenant que l'ensemble $E_{\rm ell}^{\rm nr}$ n'est pas vide. Grâce au lemme 2 de \ref{l'application tau}, cela équivaut à 
ce que l'ensemble $\mathfrak{E}_{\rm t-nr}$ ne soit pas vide. Soit un couple $(\bs{T}',\tilde{\chi}')$ formé d'un élément $\bs{T}'\in \mathfrak{E}_{\rm t-nr}$ et d'un caractère affine unitaire et non ramifié $(\chi',\tilde{\chi}')$ de $\wt{T}'(F)$. On veut calculer le transfert $\bs{\rm T}_{\bs{T}'}(\tilde{\chi}')$. Soit $\bs{\tau}'= \bs{\tau}_{\tilde{\chi}'}$ un élément de $\ES{E}_{\rm ell}^{\rm nr}$ associé à $(\bs{T}',\tilde{\chi}')$ comme en \ref{l'application tau} --- il est bien défini à conjugaison près --- et soit $\tau'=\tau_{\chi'}$ son image dans $E_{\rm ell}^{\rm nr}$. Le caractère $\omega_{\tau'}$ de $\ES{C}$ est bien défini, et d'après la propriété (4), on a
$$
\omega_{\tau'}= \omega_{\chi'}.
$$
\'Ecrivons
$$
\bs{\rm T}_{\bs{T}'}(\tilde{\chi}')= \sum_{\tau \in E_{\rm ell}/{\rm conj.}}\bs{\rm T}_{\bs{T}'}(\tilde{\chi}')_\tau, \quad 
\bs{\rm T}_{\bs{T}'}(\tilde{\chi}')_\tau\in D_\tau,\leqno{(5)}
$$
conformément à la décomposition (2) de \ref{triplets elliptiques essentiels}. D'après le lemme 4 et la remarque 4 de \ref{action du groupe C}, pour $\tau\in E_{\rm ell}^{\rm nr}/{\rm conj.}$, tout élément non nul de l'espace $D_\tau$ est un vecteur propre pour l'action du groupe $\ES{C}$ relativement à un caractère $\omega_{\tau}$ de ce groupe, et la restriction de $\omega_{\tau}$ au sous--groupe $G_\sharp(F)\subset \ES{C}$ est un caractère non ramifié de $G_\sharp(F)$. D'autre part, puisque l'homomorphisme (8) de 
\ref{l'application tau} est surjectif (\ref{l'application tau}, lemme 2), d'après les propriétés (1), (2), (3), (4) et la remarque 2 de \ref{l'application tau}, si deux triplets $\tau_1,\, \tau_2\in E_{\rm ell}^{\rm nr}$ vérifient $\omega_{\tau_1}=\omega_{\tau_2}$, alors il sont conjugués par un élément de $G(F)$. On en déduit que
$$
\bs{\rm T}_{\bs{T}'}(\tilde{\chi}')= c\, \Theta_{\bs{\tau}'}+\sum_{\tau\in E_{\rm ell}^{\rm ram}/{\rm conj.}} \bs{\rm T}_{\bs{T}'}(\tilde{\chi}')_{\tau}\leqno{(6)}
$$
pour une constante $c\in {\Bbb C}$, où (rappel) $E_{\rm ell}^{\rm ram}=E_{\rm ell}\smallsetminus E_{\rm ell}^{\rm nr}$. Pour 
$\tau\in E_{\rm ell}^{\rm ram}/{\rm conj.}$, la distribution $\bs{\rm T}_{\bs{T}'}(\tilde{\chi}')_\tau$ est dans $D_\tau
\subset \bs{D}_{\rm ell}^{\rm ram}(\wt{G}(F),\omega)$. En appliquant le lemme fondamental pour les unités des algèbres de Hecke sphériques à la donnée $\bs{T}'$, on obtient que
$$
0\neq \tilde{\chi}'(\bs{1}_{\wt{K}'})= \bs{\rm T}_{\bs{T}'}(\tilde{\chi}')(\bs{1}_{\wt{K}})= c\,\Theta_{\bs{\tau}'}(\bs{1}_{\wt{K}}).
$$
Puisque $\bs{\tau}'=\bs{\tau}_{\tilde{\chi}'}$, on a $\bs{\rm T}_{\bs{T}'}(\tilde{\chi}')(\bs{1}_{\wt{K}})= \Theta_{\bs{\tau}'}(\bs{1}_{\wt{K}})$, d'où $c=1$. 

Pour calculer le produit scalaire elliptique de la distribution $\bs{\rm T}_{\bs{T}'}(\tilde{\chi}')$, il faut commencer par rendre le caractère $\tilde{\chi}'$ invariant par le groupe des automorphismes de $\bs{T}'$. Notons $B$ le stabilisateur de $\tilde{\chi}'$ dans ${\rm Aut}(\bs{T}')$, et posons
$$
\tilde{\xi}'= \vert{\rm Aut}(\bs{T}')/B\vert^{-1} \sum_{x\in {\rm Aut}(\bs{T}')/B}{^x(\tilde{\chi}')}.
$$
La distribution $\tilde{\xi}'$ sur $\wt{T}'(F)$ appartient à l'espace $\bs{D}^{\rm nr}(\wt{T}'(F))^{{\rm Aut}(\bs{T}')}= 
\bs{SD}_{\rm ell}^{\rm nr}(\wt{T}'(F))^{{\rm Aut}(\bs{T}')}$, et comme 
$\bs{\rm T}_{\bs{T}'}(\tilde{\chi}')= \bs{\rm T}_{\bs{T}'}(\tilde{\xi}')$, on a
$$
\bs{\rm T}_{\bs{T}'}(\tilde{\xi}')= \Theta_{\bs{\tau}'} + \bs{\rm T}_{\bs{T}'}(\tilde{\xi}')^{\rm ram},\quad 
\bs{\rm T}_{\bs{T}'}(\tilde{\xi}')^{\rm ram}= \sum_{\tau \in E_{\rm ell}^{\rm ram}/{\rm conj.}} \bs{\rm T}_{\bs{T}'}(\tilde{\xi}')_\tau.\leqno{(7)}
$$
Puisque 
$$
(\tilde{\xi}',\tilde{\xi}')_{\rm ell}^{\bs{T}'}= \vert {\rm Aut}(\bs{T}')/B\vert^{-1}
(\tilde{\chi}',\tilde{\chi}')_{\rm ell}^{\bs{T}'}= \vert {\rm Aut}(\bs{T}')/B\vert^{-1},
$$
d'après le lemme 1 de \ref{l'application tau}, on a
$$
(\Theta_{\bs{\tau}'},\Theta_{\bs{\tau}'})_{\rm ell}= c(\wt{G},\bs{T}')^{-1}(\tilde{\xi}',\tilde{\xi}')_{\rm ell}^{\bs{T}'}.
$$
D'autre part, d'après le lemme 3 de \ref{produits scalaires elliptiques} (formule des produits scalaires elliptiques), on a
$$
(\bs{\rm T}_{\bs{T}'}(\tilde{\xi}'), \bs{\rm T}_{\bs{T}'}(\tilde{\xi}'))_{\rm ell}= c(\wt{G},\bs{T}')^{-1}(\tilde{\xi}',\tilde{\xi}')_{\rm ell}^{\bs{T}'}=
(\Theta_{\bs{\tau}'},\Theta_{\bs{\tau}'})_{\rm ell}.
$$
D'après (7), on a donc $\bs{\rm T}_{\bs{T}'}(\tilde{\xi}')^{\rm ram}=0$ et
$$
\bs{\rm T}_{\bs{T}'}(\tilde{\chi}')= \Theta_{\bs{\tau}'}.\leqno{(8)}
$$

\vskip1mm
En définitive, on a prouvé que l'application $(\bs{T}',\tilde{\chi}')\mapsto \bs{\tau}_{\tilde{\chi}'}$ est une bijection entre:
\begin{itemize}
\item l'ensemble des paires $(\bs{T}',\tilde{\chi}')$ où $\bs{T}'\in \mathfrak{E}_{\rm t-nr}$ et $\tilde{\chi}'$ est un caractère affine unitaire et non ramifié de $\wt{T}'(F)$, où $\tilde{\chi}'$ est pris modulo l'action de ${\rm Aut}(\bs{T}')$;
\item l'ensemble $\ES{E}_{\rm ell}^{\rm nr}/{\rm conj.}$ des triplets $\bs{\tau}'\in \ES{E}_{\rm ell}^{\rm nr}$, pris à conjugaison près dans $G(F)$.
\end{itemize}
Cette bijection est complètement déterminée par l'égalité (8), ou, ce qui revient au même, par les deux conditions
$$
\omega_{\tau'} = \omega_{\chi'},\quad \Theta_{\bs{\tau}'}(\bs{1}_{\wt{K}})= \tilde{\chi}'(\bs{1}_{\wt{K}'}),\leqno{(9)}
$$
où $\tau'$ est la projection de $\bs{\tau}'$ sur $E_{\rm ell}^{\rm nr}$ et $\chi'$ est le caractère de $T'(F)$ sous--jacent à $\tilde{\chi}'$. 

Revenons à notre propos, la démonstration de la commutativité du diagramme (5) de \ref{transfert spectral elliptique}. Remarquons tout d'abord que la bijection donnée par (8) implique que l'homomorphisme (8) de \ref{l'application tau} se factorise en un isomorphisme
$$
\bigoplus_{\bs{T}'\in \mathfrak{E}_{\rm t-nr}}\bs{D}^{\rm nr}(\wt{T}'(F))^{{\rm Aut}(\bs{T}')}\rightarrow \bs{D}_{\rm ell}^{\rm nr}(\wt{G}(F),\omega)\leqno{(10)}
$$
qui n'est autre que celui déduit par restriction de l'isomorphisme (4) de \ref{décomposition endoscopique}. On en déduit l'inclusion
$$
\bs{T}_{\bs{G}'}(\bs{SD}_{\rm ell}^{\rm ram}(\wt{G}'(F))\subset \bs{D}_{\rm ell}^{\rm ram}(\wt{G},\omega),\quad \bs{G}'\in \mathfrak{E}_{\rm nr}.\leqno{(11)}
$$
Par conséquent pour $\bs{G}'\in \mathfrak{E}_{\rm nr}\smallsetminus \mathfrak{E}_{\rm t-nr}$, le diagramme (5) de \ref{transfert spectral elliptique} est trivialement commutatif. Soit donc 
$\bs{T}'\in \mathfrak{E}_{\rm t-nr}$, et soit $\tilde{\chi}'$ un caractère affine unitaire et non ramifié de $\wt{T}'(F)$, que l'on peut supposer normalisé par $\wt{K}$ (cf. \ref{l'ensemble t-nr}). Soit $\bs{\tau} = \bs{\tau}_{\tilde{\chi}'}\in \ES{E}_{\rm ell}^{\rm nr}$ un triplet associé à $\tilde{\chi}'$ comme en \ref{l'application tau}. \'Ecrivons $\bs{\tau}=(T,\lambda,\tilde{\bs{r}})$. On a vu (8) que $\Theta_{\bs{\tau}}= \bs{\rm T}_{\bs{T}}(\tilde{\chi}')$. On a
$$
p^K(\Theta_{\bs{\tau}})= \tilde{\rho}_0(\tilde{\bs{r}})\Theta_{\tilde{\pi}_{\tilde{\rho}_0}},
$$
où $\rho_0$ est l'élément de ${\rm Irr}(\ES{R}^G(\lambda),\theta_{\ES{R}})$ tel que $\pi_{\rho_0}$ est l'élément $K$--sphérique de $\Pi_\lambda$, et $\tilde{\rho}_0$ est le prolongement de $\rho_0$ à $\ES{R}^{\wt{G},\omega}(\lambda)$ normalisé par $\Theta_{\tilde{\pi}_{\tilde{\rho}_0}}(\bs{1}_{\wt{K}})=1$. Puisque la distribution $\Theta_{\bs{\tau}}$ elle--même a été normalisée de sorte que $\Theta_{\bs{\tau}}(\bs{1}_{\wt{K}})=1$, cela entra\^{\i}ne $\tilde{\rho}_0(\tilde{\bs{r}})=1$. 
D'autre part, on a $p^{K'}(\tilde{\chi}')=\tilde{\chi}'$ et $\bs{\rm t}_{\bs{T}'}(\tilde{\chi}')= c_0\Theta_{\tilde{\pi}_{\tilde{\rho}_0}}$ pour une constante $c_0\in {\Bbb U}$. 
Comme
$$
\bs{\rm t}_{\bs{T}'}(\tilde{\chi}')(\bs{1}_{\wt{K}})= \tilde{\chi}'(\bs{1}_{\wt{K}'})= 1,
$$
on obtient $c_0=1$, puis
$$
p^K(\Theta_{\bs{\tau}})= \bs{\rm t}_{\bs{T}'}(\tilde{\chi}').
$$
Cela achève la démonstration de la commutativité du diagramme (5) de \ref{transfert spectral elliptique}.

\begin{marema}
{\rm 
Pour établir la commutativité du diagramme (5) de \ref{transfert spectral elliptique}, on a été amené à prouver la paramétrisation endoscopique des représentations tempérées elliptiques non ramifiées de $(\wt{G}(F),\omega)$, \cad l'égalité (8). On peut bien sûr, grâce au dictionnaire de \ref{l'application tau}, écrire cette égalité en termes duaux, \cad remplacer le $R$--groupe tordu $R^{\wt{G},\omega}(\lambda)$ par le $S$--groupe tordu $\wt{\bs{S}}(\varphi'\!,\bs{a})$ --- où $\varphi'$ est le paramètre $W_F \xrightarrow{\varphi^{T'}} {^LT'}\simeq \ES{T}'\hookrightarrow {^LG}$ --- dans la décomposition du caractère elliptique $\Theta_{\bs{\tau}_{\tilde{\chi}'}}$ associé à $(\bs{T}',\tilde{\chi}')$. 
}
\end{marema}

\subsection{le cas du groupe $GL(n)$ tordu}\label{le cas de GL(n) tordu}Dans ce numéro, on considère le cas du groupe $GL(n,F)$ tordu, \cad du groupe linéaire déployé $G=GL(n)$ pour un entier $n\geq 1$, avec comme automorphisme extérieur $g\mapsto {^{\rm t}g^{-1}}$. 

\begin{mapropo} Le lemme fondamental tordu pour tous les éléments des algèbres de Hecke sphériques est vrai sans restriction sur la caract\'eristique r\'esiduelle dans le cas du groupe $GL(n,F)$ tordu.
\end{mapropo}

\begin{proof}Les r\'eductions faites permettent de ne consid\'erer que le cas d'un espace de Levi qui admet une donn\'ee endoscopique elliptique dont le groupe sous--jacent est un tore. Et en plus, seul le lemme fondamental pour les unit\'es des alg\`ebres de Hecke sph\'eriques doit \^etre prouv\'e sans restriction sur la caract\'eristique r\'esiduelle. Un espace de Levi est  produit d'au plus un groupe $GL(n',F)$ tordu (avec $n'\leq n$) et d'un certain nombre fini de paires $GL(m,F)\times GL(m,F)$ sur lesquelles l'automorphisme agit par permutation. Pour ces paires, le lemme \`a prouver est imm\'ediat. Cela nous ram\`ene au cas d'un groupe $GL(n,F)$ tordu qui admet  une donn\'ee endoscopique elliptique dont le groupe sous--jacent est un tore. Cela ne peut se produire que si $n=1$ ou $n=2$. Au lieu de prouver le lemme fondamental pour les unit\'es des algèbres de Hecke sphériques, 
on peut aussi, ce que l'on va faire dans le cas $n=1$, simplement prouver le transfert spectral cherch\'e (cf. ci--dessous). 

Pour cela, il faut d'abord v\'erifier que le transfert spectral se traduit ais\'ement en termes des fonctions--caract\`eres des repr\'esentations. En toute g\'en\'eralit\'e, donc en particulier pour un $G$--espace tordu $\wt{G}$ vérifiant les hypothèses de \ref{objets}, pour une donnée endoscopique elliptique $\bs{G}'=(G',\ES{G}',\tilde{s})$ pour $\wt{G}$, et pour des distributions $\Theta\in \bs{D}(\wt{G}(F))$ et $\Theta'\in \bs{SD}(\wt{G}'(F))$, on a le transfert spectral
$$
\Theta(f) = \Theta'(f^{\wt{G}'}), \quad f\in C^\infty_{\rm c}(\wt{G}(F)),
$$
si et seulement si pour pour tout \'el\'ement $\gamma\in \tilde{G}(F)$ fortement régulier, on a l'égalité
$$
d_\theta^{-1/2}D^{\tilde{G}}(\gamma)^{1/2}\Theta(\gamma)= \sum_{\delta} \Delta(\delta,\gamma)\, D^{\wt{G}'}(\delta)^{1/2} \Theta'(\delta),\leqno{(1)}
$$
où $\delta$ parcourt les éléments fortement $\wt{G}$--réguliers de $\wt{G}'(F)$ qui correspondent à $\gamma$, pris à conjugaison stable près.

\begin{marema}
{\rm 
On a \'enonc\'e (1) dans le cas simple o\`u il n'est pas n\'ecessaire d'introduire de donn\'ees auxiliaires, ce qui est ici le cas. 
}
\end{marema}

Consid\'erons donc le cas o\`u $n=1$: les donn\'ees endoscopiques elliptiques sont form\'ees uniquement d'un caract\`ere quadratique du groupe de Galois de $F$. Fixons une telle donn\'ee, \cad un caract\`ere quadratique $\eta_{0}$ de $\Gamma_F$. On suppose cette donn\'ee non ramifi\'ee, ce qui ici revient \`a dire que le caract\`ere $\eta_0$ est non ramifi\'e.  

Le caract\`ere $\eta_{0}$, vu comme un caractère de $G(F)=F^\times$, s'\'etend trivialement en un carac\-tère affine $\tilde{\pi}$ de $\wt{G}(F)= F^\times \theta$: pour $x\in F^\times$, on pose $\tilde{\pi}(x\theta)= \eta_0(x)$. C'est, à homothétie près, l'unique représentation $G(F)$--irréductible de $\wt{G}(F)$ de caractère central $\eta_0$. Soit $\gamma=x\theta$ un élément de $\wt{G}(F)$. Le caract\`ere $\Theta_{\tilde{\pi}}$ en ce point $\gamma$ a pour valeur $\eta_{0}(x)$, et le facteur de transfert 
$\Delta(\delta,\gamma)$ est lui aussi égal à $\eta_{0}(x)$. En effet, le facteur de transfert est normalis\'e de sorte qu'il vaille 1 pour $\theta$ et, en notant $x'$ une racine carr\'ee de $x$, vue comme un \'el\'ement de $G_\sharp(F)= (GL(1)/{\pm 1})(F)$, on a $\gamma= {\rm Int}_{x'}(\theta)$. Le facteur de transfert se transforme sous cette action par $\eta_{0}(x'^2)=\eta_{0}(x)$. D'o\`u trivialement un transfert de traces de repr\'esentations. C'est bien le transfert donné par (1) car $d_\theta^{-1/2}D^{\tilde{G}}(\gamma)^{1/2}=1$ et $D^{\wt{G}'}(\delta)=1$.

\vskip1mm
On consid\`ere maintenant le cas o\`u $n=2$. Dans ce cas on montre que le lemme fondamental tordu que l'on cherche à démontrer est \'equivalent au lemme fondamental (non tordu) pour $(PGL(2),\omega)$, où $\omega$ est l'unique caract\`ere non ramifi\'e d'ordre 2 de $PGL(2,F)$, que l'on identifie à un caractère de $F^\times$. Ce dernier lemme fondamental a \'et\'e prouv\'e par Hales \cite{H}. Le cas que nous devons consid\'erer est celui o\`u la donn\'ee endoscopique a pour groupe sous-jacent $SO(2)$ avec \'evidemment le caract\`ere quadratique non ramifi\'e de $\Gamma_F$, que l'on identifie au caract\`ere $\omega$ de $F^\times$. C'est lui qui d\'etermine la forme du groupe endoscopique $SO(2)$. On la note $G'=T'$. On a donc $G'(F)=E^{\times,1}$, où $E^{\times,1}$ est le groupe des \'el\'ements de norme $1$ de l'extension quadratique non ramifi\'ee $E$ de $F$.

On note $\theta$ l'\'el\'ement de $\wt{G}(F)$ qui agit sur $G(F)$ par $g\mapsto \det (g)^{-1}g$. 
Soit $f\in C^{\infty}_{c}(\wt{G}(F))$ et soit $\gamma=x\theta$ un élément fortement régulier de $\wt{G}(F)$. L'int\'egrale orbitale de $f$ en le point $\gamma$ est le produit de $D^{\wt{G}}(\gamma)^{1/2}$ avec l'int\'egrale (pour des mesures dont on parlera ci-dessous)
$$
\int_{G(F)/G_{\gamma}(F)}f(g\gamma g^{-1})d\bar{g}_\gamma
=\int_{G(F)/G_{\gamma}(F)} f( \det(g)gx  g^{-1}\theta )d\bar{g}_\gamma.\leqno{(2)}
$$
Posons $K=GL(2,\mathfrak{o})$, où (rappel) $\mathfrak{o}$ est l'anneau des entiers de $F$, et supposons que $f=f_{0}$ est la fonction caract\'eristique du sous--espace hyperspécial $\wt{K}= K\theta$ de $\wt{G}(F)$. L'int\'egrale ci--dessus ne porte que sur les \'el\'ements $g$ tel que $\det (g)^2\det (x)^{-1}$ est une unit\'e de $F^\times$. Cette int\'egrale est donc nulle sauf si la valuation (normalisée) de $\det(x)$ est paire. Dans ce cas, en posant $Z=Z(G)$, l'int\'egrale vaut
$$
{\rm vol}(G_{\gamma}(F))^{-1}\int_{
g\in GL(2,F),\, v_F(\det (g))=-{1\over 2} v_F(\det(x))}
{\mathbf 1}_{Z(F)K}(g x g^{-1})dg, 
$$
ou encore
$$
{\rm vol}(G_{\gamma}(F))^{-1}\int_{g\in PGL(2,F),\, v_F(\det(g))\equiv -{1\over 2} v_F(\det(x))\;[2]}{\mathbf 1}_{Z(F)K}(g x g^{-1})dg.\leqno{(3)}
$$

D'autre part, un \'el\'ement fortement régulier $\gamma=x\theta$ de $\wt{G}(F)$ a sa classe de conjugaison stable qui  correspond \`a un \'el\'ement $\delta$ de $E^{\times,1}$ si $x^2/\det(x)$ est conjugu\'e d'un \'el\'ement  de $E^{\times,1}$, vu comme sous--groupe de $E^\times =T_0(F)\subset GL(2,F)$, et si $\delta$ appartient \`a la classe de conjugaison stable de cet \'el\'ement de $E^{\times,1}$; cette classe de conjugaison stable est en fait r\'eduite \`a un point par commutativit\'e. 

Quitte \`a conjuguer $\gamma$,  on peut supposer $x^2/\det(x)=\delta$. Alors on a $x^2 \in E^\times$. Mais $\gamma$ est un \'el\'ement fortement r\'egulier elliptique de $\wt{G}(F)$, et $x$ est un \'el\'ement r\'egulier d'un tore de $GL(2)$. Comme $\delta$ est un \'el\'ement lui aussi r\'egulier, $x^2$ est un \'el\'ement r\'egulier du m\^eme tore que celui auquel appartient $x$, \cad $T_0$, et donc $x$ est (comme $x^2$) dans $E^\times$. On remarque  que la valuation de $\det(x)$ est alors n\'ecessairement paire: en effet $\det(x)= x\sigma(x)$ o\`u $\sigma$ est 
l'élément non trivial de ${\rm Gal}(E/F)$, mais comme l'extension $E/F$ est non ramifi\'ee, on a l'assertion.

Avant de continuer le calcul de l'int\'egrale, montrons qu'il existe exactement deux classes de conjugaison dans la classe de conjugaison stable de $\gamma$ comme ci--dessus. On a $\det(x)= x\sigma(x)$ et $x^2/\det(x)=x/\sigma(x)$. On \'ecrit $\delta=x/\sigma(x)$ avec $x\in \mathfrak{o}_E^\times$, où $\mathfrak{o}_E^\times$ est le groupe des unités de $E^\times$, en utilisant le fait que 
$E^\times=\varpi^{\mathbb Z}\mathfrak{o}_E^\times$ pour une uniformisante $\varpi$ de $F$. Pour un autre élément $\gamma'=x'\theta$ vérifiant nos conditions, en posant $z=x' x^{-1}$, on obtient que $z=\sigma(z)$, \cad que $z\in F^\times$. Les \'el\'ements de $N_{E/F}(E^\times)\gamma$ sont les conjugués de $\gamma$ par un élément de $E^\times$, et les éléments de la forme $z\gamma$ avec $z\in F^\times \smallsetminus N_{E/F}(E^\times)$ sont tous conjugu\'es mais ne sont pas des conjugu\'es de $\gamma$. Cela donne les deux classes de $G(F)$--conjugaison dans $\wt{G}(F)$ correspondant \`a la classe de conjugaison stable de $\delta$. On a fix\'e $\gamma$, et l'on note $\gamma'=x'\theta$ un \'el\'ement de l'autre classe de conjugaison correspondant à $\delta$. Alors on a la relation entre les facteurs de transfert:
$$
\Delta(\delta,\gamma)=-\Delta(\delta,\gamma').
$$

Revenons aux int\'egrales, et calculons le c\^ot\'e tordu du lemme fondamental, en fixant $\delta$ et $\gamma$ comme ci--dessus. 
On a l'égalité:
$$
D^{\wt{G}}(\gamma)^{1/2}\Delta(\delta,\gamma) {\rm vol}(G_{\gamma}(F))^{-1}=-
D^{\wt{G}}(\gamma')^{1/2}\Delta(\delta,\gamma') {\rm vol}(G_{\gamma'}(F))^{-1}
$$
De plus en rempla\c{c}ant $\gamma'=x'\theta$ par $z\gamma= zx\theta$ avec $z\in F^\times$ de valuation $1$, on voit que 
l'intégrale
$$
\int_{g\in PGL(2,F),\, v_F(\det(g))\equiv-{1\over 2} v_F(\det(x'))\;[2]}{\mathbf 1}_{Z(F)K}(g x' g^{-1})dg
$$
est égale à
$$
\int_{g\in PGL(2,F),\, v_F(\det(g))\equiv-{1\over 2}v_F(\det(x))+1\;[2]}{\mathbf 1}_{Z(F)K}(g x g^{-1})dg.
$$
D'o\`u le c\^ot\'e tordu du lemme fondamental:
$$
d_\theta^{1/2}D^{\wt{G}}(\gamma)^{1/2}\Delta(\delta,\gamma) {\rm vol}(G_{\gamma}(F))^{-1}\int_{g\in PGL(2,F)}\omega(\det(g)){\mathbf 1}_{Z(F)K}(g \gamma g^{-1})dg.\leqno{(4)}
$$
On sait d'autre part (d'après \cite{H}) que l'expression
$$
D^{PGL(2)}(x)^{1/2}\Delta^{PGL(2)}(\delta,x)\int_{g\in PGL(2,F)}\omega(\det(g)){\mathbf 1}_{Z(F)K}(g x g^{-1}) dg\leqno{(5)}
$$
vaut l'int\'egrale orbitale stable de $\delta$ pour la fonction caract\'eristique du compact $E^{\times,1}$; \cad $1$. Il reste donc \`a prouver que les expressions (4) et (5) sont les mêmes, \cad à vérifier:
\begin{itemize}

\item l'\'egalit\'e des facteurs de transfert pour $(PGL(2),\omega)$ et $\wt{G}= GL(2)\theta$;
 
\item l'\'egalit\'e $d_\theta^{-1/2}D^{\wt{G}}(x\theta)=D^{PGL(2)}(x)$; 
 
\item l'\'egalit\'e $d_\theta={\rm vol}(G_{\gamma}(F))/{\rm vol} (E^{\times,1})$ dans la normalisation pour le transfert tordu; pour le transfert non tordu de $PGL(2)$, le transfert des mesures vaut bien 1. 
 
 \end{itemize}
 
\noindent L'\'egalit\'e des facteurs de transfert r\'esulte imm\'ediatement de leur d\'efinition, cf. \cite[6.3]{Stab I}. La deuxième égalité est claire. On d\'etaille la derni\`ere \'egalit\'e: dans l'identification entre les stabilisateurs, dans le cas tordu, on a 
$G_{\gamma}(F)=E^{\times ,1}$ et $ E^{\times, 1}=G'_{\delta}(F)$, mais l'identification n'est pas l'identit\'e, c'est l'homomorphisme naturel $E^{\times,1}\rightarrow (T_0/(1-\theta)(T_0))(F)=E^\times /F^\times$, \cad l'application naturelle de $E^{\times,1}$ dans $E^\times/F^\times$. L'identification  de $E^\times/F^\times$ avec $E^{\times,1}=G'_\delta(F)$ est l'application $x\mapsto x/\sigma(x)$. Quand on part de $x\in E^{\times,1}$, c'est l'application $x\mapsto x/\sigma(x)=x^2$. D'o\`u le Jacobien $d_\theta$, cf. \cite{Stab I} 2.4.
\end{proof}

\section{Le cas o la donnée endoscopique est un tore}

\subsection{La proposition--clé dans le cas où $\hat{G}_{\rm AD}$ est simple}\label{le cas simple}Dans cette section 5, on s'intéresse exclusivement aux données endoscopiques elliptiques pour $(\wt{G},\bs{a})$ telles que le groupe sous--jacent est un tore, \cad à l'ensemble $\mathfrak{E}_{\rm t-nr}$. 

Soit $\bs{T}=(T',\ES{T}',\tilde{s})\in \mathfrak{E}_{\rm t-nr}$. On reprend les notations de \ref{données endoscopiques nr}. En particulier, l'action du groupe de Weyl $W_F$ sur $\hat{G}$ est donnée par celle de l'élément de Frobenius $\phi$ qui stabilise une paire de Borel épinglée $\hat{\ES{E}}=(\hat{B},\hat{T},\{\hat{E}_\alpha\}_{\alpha\in \hat{\Delta}})$ de $\hat{G}$. Cette paire $\hat{\ES{E}}$ est stabilisée par ${\rm Int}_{\tilde{s}}$, et l'on note $\hat{\theta}$ l'automorphisme de $\hat{G}$ qui stabilise $\hat{\ES{E}}$ et commute à l'action galoisienne $\sigma\mapsto \sigma_{G}$. On a donc $\tilde{s}=s\hat{\theta}$ pour un $s\in \hat{T}$. Soit un élément $(h,\phi)\in \ES{T}'$. Il définit (comme en \ref{l'ensemble t-nr}) un isomorphisme $\ES{T}'\simeq {^LT'}$. Du plongement $\hat{T}'\hookrightarrow \hat{T}$ se déduit par dualité un homomorphisme $\xi: T\rightarrow T/(1-\theta_{\ES{E}})(T)\simeq T'$. Cet homomorphisme n'est pas $\Gamma_F$--équivariant, mais sa restriction à $Z(G)$, notée $\xi_Z: Z(G)\rightarrow Z(T')=T'$, l'est. Pour un caractère affine $\tilde{\chi}'=(\chi',\tilde{\chi}')$ de $\wt{T}'(F)$ tel que le caractère $\chi'$ de $T'(F)$ soit unitaire, le transfert $\bs{\rm T}_{\bs{T}'}(\wt{\chi}')$ est une combinaison linéaire de caractères $\Theta_{\tilde{\pi}}$ pour des représentations $G(F)$--irréductibles tempérées $\tilde{\pi}$ de $(\wt{G}(F),\omega)$ qui ont même caractère central $Z(G;F)\rightarrow {\Bbb C}^\times$ (pour l'action de $Z(G;F)$ à gauche sur $\wt{G}(F)$). Ce caractère est le produit de $\chi'\circ \xi_{Z}$ et d'un caractère indépendant de $\tilde{\chi}'$ (dépendant du facteur de transfert) --- cf. \ref{décomposition endoscopique}. Comme on peut toujours multiplier $\tilde{\chi}'$ par un nombre complexe non nul pour le rendre unitaire, on peut se limiter à ne considérer que des caractères affines unitaires de $\wt{T}(F)$. 

La donnée $\bs{T}'$ a un groupe d'automorphismes, que l'on a noté ${\rm Aut}(\bs{T}')$ --- 
cf. \ref{données endoscopiques}. Deux caractères $\tilde{\chi}'_1$ et $\tilde{\chi}'_2$ de $\wt{T}'(F)$ qui se déduisent l'un de l'autre par un automorphisme de $\bs{T}'$ ont même transfert. Dans le cas tordu, certains éléments de ${\rm Aut}(\bs{T}')$ agissent subtilement par multiplication par des caractères affines (voir la preuve du point (iii) de la proposition ci--dessous). 

{\bf Jusqu'à la fin de ce numéro on suppose que $\hat{G}_{\rm AD}$ est simple}.

\begin{marema1}
{\rm 
L'hypothèse que $\hat{G}_{\rm AD}$ est simple est vérifiée dans le cas de $GL(n)$ tordu par son automorphisme extérieur, et pour les groupes classiques non tordus à l'exception de $SO(4)$.}
\end{marema1}

\begin{mapropo}
On suppose que $\hat{G}_{\rm AD}$ est simple. On suppose aussi que l'ensemble $\mathfrak{E}_{\rm t-nr}$ n'est pas vide. Alors on a:
\begin{enumerate}
\item[(i)]Les actions de $\phi$ et $\hat{\theta}$ sur $\hat{G}_{\rm AD}$ sont triviales.
\item[(ii)]Soient $\bs{T}'_1,\,\bs{T}'_2\in \mathfrak{E}_{\rm t-nr}$. Si $\omega_{\bs{T}'_1}=\omega_{\bs{T}'_2}$, alors $\bs{T}'_1=\bs{T}'_2$.
\item[(iii)]Soit $\bs{T}'\in \mathfrak{E}_{\rm nr}$, et soient $(\chi'_1,\tilde{\chi}'_1)$ et $(\chi'_2,\tilde{\chi}'_2)$ deux caractères affines unitaires et non ramifiés de $\wt{T}'(F)$. Si $\chi'_1\circ \xi_Z = \chi'_2\circ \xi_Z$, où $\xi_Z: Z(G)\rightarrow T'$ est l'homomorphisme naturel, alors à homothétie près, $\tilde{\chi}_1$ et $\tilde{\chi}_2$ se déduisent l'un de l'autre par l'action d'un élément du groupe d'automorphismes de $\bs{T}'$.
\end{enumerate}
\end{mapropo}

\begin{marema2}
{\rm Le fait que $\hat{\theta}$ et $\phi$ soient triviaux sur $\hat{G}_{\rm AD}$ n'implique évidemment pas que $\hat{\theta}$ et $\phi$ soient triviaux sur $\hat{G}$. En particulier, $\mathfrak{E}_{\rm t-nr}$ n'est pas vide dans le cas de $GL(2)$ tordu et aussi dans le cas de $U(2)$ non tordu. }
\end{marema2}

\begin{marema3}
{\rm 
La proposition a comme conséquence immédiate que les transferts de caractères affines unitaires non ramifiés provenant de données dans $\mathfrak{E}_{\rm t-nr}$ se séparent à l'aide de l'action du groupe $G_\sharp(F)$ et du caractère central.
}
\end{marema3}

\begin{proof}Prouvons (i). Dans ce point (i) du lemme, tout se passe dans le groupe $\hat{G}_{\rm AD}$. Pour simplifier les notations, on suppose que $\hat{G}=\hat{G}_{\rm AD}$. Soit $\bs{T}'=(T',\ES{T}',\tilde{s})\in \mathfrak{E}_{\rm t-nr}$. Choisissons un élément $(h,\phi)\in \ES{T}'$. \'Ecrivons $\tilde{s}=s\hat{\theta}$ et $\bs{h}=h\phi$.

{\it Supposons tout d'abord que $\hat{\theta}=1$.} On va utiliser les résultats de Langlands \cite{La}. 
On suppose que $s\in \hat{T}$ (c'est toujours possible, quitte à conjuguer $\tilde{s}=s$). On a alors $G_{s}= \hat{T}$. On note $\hat{\Sigma}\;(= \check{\Sigma})$ l'ensemble des racines de $\hat{T}$ dans $\hat{G}$. Langlands montre tout d'abord (\cite{La}, page 705) que $s$ est d'ordre fini (sinon $\bs{T}'$ ne pourrait pas être elliptique). Soit $m\geq 1$ l'ordre de $s$. On pose $\zeta= e^{2\pi i/m}\in {\Bbb U}$. Pour $k=0,\ldots ,m-1$, on note $\mathfrak{N}_k$ l'ensemble des $\alpha\in \hat{\Sigma}$ tels que $\alpha(s)=\zeta^k$. La propriété (1) signifie que $\mathfrak{N}_0$ est vide. On note $\mathfrak{X}_k$ l'ensemble des $\alpha\in \mathfrak{N}_k$ qui ne sont pas combinaisons linéaires à coefficients entiers d'éléments de $\bigcup_{j=0}^{k-1}\mathfrak{N}_j$ (ces ensembles sont notés $\mathfrak{Z}_k$ dans \cite{La}, mais puisque $\mathfrak{N}_0$ est vide, ils co\"{\i}ncident avec ceux notés $\mathfrak{X}_k$). Tous ces ensembles sont invariants par $\bs{h}$. Langlands montre (\cite{La}, page 709) que l'ensemble $\mathfrak{D}=\bigcup_{k=0}^{m-1}\mathfrak{X}_k$ est justiciable du lemme 2 de \cite{La}, page 705. Il existe donc une base $\hat{\Delta}'$ de $\hat{\Sigma}$ telle que l'une des deux situations suivantes est vérifiée:
\begin{itemize}
\item $\mathfrak{D}= \hat{\Delta}'$;
\item $\mathfrak{D}= \hat{\Delta}'\cup \{-\alpha'_0\}$ où $\alpha'_0$ est la plus grande racine pour cette base $\hat{\Delta}'$.
\end{itemize}
Soit $\hat{B}'$ le sous--groupe de Borel de $\hat{G}$ contenant $\hat{T}$ associé à cette base $\hat{\Delta}'$. Quitte à conjuguer toute la situation par un élément de $N_{\hat{G}}(\hat{T})$ envoyant $\hat{B}$ sur $\hat{B}'$, on peut supposer que $\hat{\Delta}'= \hat{\Delta}$. Alors on a:
\begin{enumerate}
\item[(1)]$\mathfrak{D}\neq \hat{\Delta}$ et les éléments de $\mathfrak{D}$ forment une seule orbite sous l'action de $\bs{h}$.
\end{enumerate}
En effet, si (1) n'est pas vrai, alors une orbite de $\mathfrak{D}$ sous l'action de $\bs{h}$ est contenue dans $\hat{\Delta}$. La somme des éléments de cette orbite est non nulle et invariante par $\bs{h}$, par conséquent $X(\hat{T})^{\bs{h}}\neq \{0\}$ et donc aussi 
$\check{X}(\hat{T})^{\bs{h}}\neq \{0\}$. L'image d'un cocaractère de $\hat{T}$ appartenant à cet ensemble $\check{X}(\hat{T})^{\bs{h}}$ est contenue dans $Z_{\hat{G}}(s,\bs{h})^\circ$, ce qui contredit la propriété de finitude (4) de \ref{l'ensemble t-nr}. La propriété (1) est donc vérifiée. A fortiori $\mathfrak{D}$ est réduit à un seul $\mathfrak{X}_k$. L'indice $k$ en question est forcément le plus petit entier $k\geq 1$ tel que $\mathfrak{N}_k\neq \emptyset$. Puisque ce $\mathfrak{X}_k$ contient $\hat{\Delta}$, toutes les valeurs $\alpha(s)$ ($\alpha\in \check{\Sigma}$) sont des puissances de $\zeta^k$. Quitte à remplacer $\zeta$ par $\zeta^k$, on ne perd rien à supposer que $\mathfrak{D}=\mathfrak{X}_1$. Puisque $\mathfrak{X}_1= \hat{\Delta} \cup \{-\alpha_0\}$, on voit que $\bs{h}$ définit un automorphisme du diagramme de Dynkin complété de $\hat{G}$, et les éléments de ce diagramme forment une seule orbite pour cet automorphisme. En inspectant tous les diagrammes possibles, on voit que seuls ceux de type $A_{n-1}$ possèdent de tels automorphismes. L'élément de Frobenius $\phi$, vu comme un automorphisme de $\hat{G}$, agit sur $\hat{\Delta}$, soit par l'identité, soit par l'automorphisme $\alpha_i\mapsto \alpha_{n-i}$. Ce dernier cas n'est possible que si $n\geq 3$ (rappelons que $\hat{G}$ est adjoint; si $n=2$, l'automorphisme $\phi$ de $\hat{G}$ est forcément l'identité). L'élément $h$ normalise le tore $\hat{T}$. Si $\phi$ n'est pas l'identité, comme $h\phi(h)$ agit trivialement sur $\hat{\Delta}\cup \{-\alpha_0\}$, les orbites de $\bs{h}$ dans $\hat{\Delta}\cup\{-\alpha_0\}$ sont au plus d'ordre $2$. Dans ce cas, il en résulte que $\Delta\cup\{-\alpha_0\}$ a au plus deux éléments. Mais alors $n\leq 2$, contradiction. Donc $\phi$ agit trivialement, et (sous l'hypothèse $\hat{\theta}=1$) le point (i) est démontré.  

{\it Supposons maintenant que $\phi=1$.} Il suffit d'échanger les rôles de $\tilde{s}$ et $\bs{h}$ dans le raisonnement précédent. Ce raisonnement prouvait que $\phi=1$ (sous l'hypothèse $\hat{\theta}=1$), il prouve maintenant que $\hat{\theta}=1$ (sous l'hypothèse $\phi=1$). 

{\it Enfin supposons que $\hat{\theta}\neq 1$ et $\phi\neq 1$.}
On rappelle que $\hat{\theta}$ et $\phi$ sont deux automorphismes de $\hat{G}$ qui stabilisent $\hat{\ES{E}}$ et commutent entre eux. Notons $\mathfrak{A}$ le groupe des automorphismes de $\hat{G}$ stabilisant $\hat{\ES{E}}$. Les seuls types pour lesquels $\mathfrak{A}\neq \{1\}$ sont $A_{n-1}$ ($n\geq 3$), $D_n$ ($n\geq 3$), et $E_6$. Pour tous ces types, à l'exception de $D_4$, ce groupe est isomorphe à ${\Bbb Z}/2{\Bbb Z}$, et donc forcément $\hat{\theta}=\phi$. En type $D_4$, le groupe $\mathfrak{A}$ est isomorphe au groupe symétrique $\mathfrak{S}_3$, et parce que $\hat{\theta}$ et $\phi$ commutent, on a forcément $\phi= \hat{\theta}^k$ avec $k\in \{1,2\}$. Dans tous les cas, on a $\phi= \hat{\theta}^k$ pour un entier $k\geq 1$. On modifie les actions $\hat{\theta}'=\hat{\theta}$ et $\phi'=1$. Posons $\tilde{s}'=\tilde{s}$ et $\bs{h}'= \tilde{s}^{-k}\bs{h}$. Ces éléments vérifient encore les propriétés (1), (2), (4) et (5) de \ref{l'ensemble t-nr}. En particulier (pour (4)), le commutant commun $Z_{\hat{G}}(\tilde{s}',\bs{h}')$ co\"{\i}ncide avec $Z_{\hat{G}}(\tilde{s},\bs{h})$. On applique à ces nouvelles données $(\tilde{s}',\bs{h}')$ ce que l'on a déjà prouvé pour $(\tilde{s},\bs{h})$. On obtient que $\hat{\theta}'=1$; contradiction. 

Cela achève la preuve de (i).

\begin{marema4}
{\rm 
Dans le cas où $\hat{G}=\hat{G}_{\rm AD}$ est simple et de type $A_{n-1}$, avec actions triviales de $\hat{\theta}$ et $\phi$, on retrouve la description familière. L'élément $\zeta\in {\Bbb U}$ est une racine primitive $n$--ième de $1$ et, avec les notations habituelles, les éléments $s$ et $h$ sont, à conjugaison près, les images $u_{\rm ad}(\zeta)$ et $v_{\rm ad}$ dans $PGL(n,{\Bbb C})$ des éléments suivants de $GL(n,{\Bbb C})$:
$$
u(\zeta)= {\rm diag}(\zeta^{n-1},\zeta^{n-2},\ldots ,\zeta, 1),
$$
$$
v=
\begin{pmatrix}0&\cdots &\cdots & 0& (-1)^{n-1}\\
1& 0 & \cdots &\cdots& 0  \\
0& 1 & \ddots &&\vdots\\
\vdots&\ddots&\ddots&\ddots&\vdots\\
0&\cdots & 0& 1& 0
\end{pmatrix}.
$$
Notons que $v\in SL(n,{\Bbb C})$. 
}
\end{marema4}

\vskip1mm 
Prouvons (ii). On ne suppose plus que $\hat{G}=\hat{G}_{\rm AD}$. En revanche (d'après (i)), on peut supposer que $\hat{G}_{\rm AD}= PGL(n,{\Bbb C})$ avec actions triviales de $\hat{\theta}$ et $\phi$ (mais comme on l'a déjà dit, ces actions ne sont pas forcément triviales sur $Z(\hat{G})$). On considère deux éléments $\bs{T}'_1,\, \bs{T}'_2\in \mathfrak{E}_{\rm t-nr}$. Pour $i=1,\,2$, on a des éléments $\tilde{s}_i= s_i\hat{\theta}_i$ et $\bs{h}_i= h_i\phi$. On peut supposer que $s_i\in \hat{T}$ et que les images de ces éléments dans $PGL(n,{\Bbb C})$ sont de la forme de ceux de la remarque 4, avec des racines $\zeta_i$ ($i=1,\,2$). Posons $\hat{Z}=Z(\hat{G})$, 
$\hat{Z}_{\rm sc}=Z(\hat{G}_{\rm SC})$ et $\hat{Z}_\sharp = Z(\hat{G}_\sharp)$. 
Puisque $\hat{\theta}$ est trivial sur $\hat{G}_{\rm AD}$, son relèvement à $\hat{G}_{\rm SC}$ l'est aussi, et donc il est a fortiori trivial sur $\hat{Z}_{\rm sc}$. La description de $\hat{Z}_\sharp $ donnée dans \cite[2.7]{Stab I} se simplifie: on a une identification
$$
\hat{Z}_\sharp= \hat{Z}/(\hat{Z}\cap \hat{T}^{\hat{\theta},\circ}) \times \hat{Z}_{\rm sc}.
$$
Pour $i=1,\,2$, le caractère $\omega_{\bs{T}'_i}$ correspond à un cocycle $a'_i$ de $W_F$ dans $\hat{Z}_\sharp$ qui est non ramifié, donc déterminé par l'image $a'_i(\phi)$. Cette image est un élément de $\hat{Z}_\sharp$ dont seule compte la projection sur $\hat{Z}_\sharp/(1-\phi)(\hat{Z}_\sharp)$. D'après \cite[2.7]{Stab I}, cet élément se calcule comme suit. On choisit un relèvement $s_{i,{\rm sc}}$ de $s_{i,{\rm ad}}$ (la projection de $s_i$ sur $\hat{G}_{\rm AD}$) dans $\hat{G}_{\rm SC}$, et on écrit $h_i = z_i\pi(h_{i,{\rm sc}})$ avec $z_i\in \hat{Z}$ et $h_{i,{\rm sc}}\in \hat{G}_{\rm SC}$; où $\pi:\hat{G}_{\rm SC}\rightarrow \hat{G}$ est l'homomorphisme naturel. On note $a'_{i,{\rm sc}}$ l'élément de $\hat{Z}_{\rm sc}$ défini par
$$
a'_{i,{\rm sc}}= s_{i,{\rm sc}}\hat{\theta}(h_{i,{\rm sc}})\phi(s_{i,{\rm sc}})^{-1}h_{i,{\rm sc}}^{-1}.
$$
Alors $a'_i(\phi)$ est l'image de $(z_i,a'_{i,{\rm sc}})$ dans $\hat{Z}_\sharp$. Remarquons que puisque $\hat{\theta}$ et $\phi$ sont triviaux sur $\hat{G}_{\rm SC}$, la définition de $a'_{i,{\rm sc}}$ se simplifie en
$$
a'_{i,{\rm sc}}= s_{i,{\rm sc}}h_{i,{\rm sc}}s_{i,{\rm sc}}^{-1}h_{i,{\rm sc}}^{-1}.
$$
Un calcul matriciel immédiat montre que
$$
a'_{i,{\rm sc}}= {\rm diag}(\zeta_i^{-1},\ldots ,\zeta_i^{-1})\in SL(n,{\Bbb C}).
$$
Supposons que $\omega_{\bs{T}'_1}=\omega_{\bs{T}'_2}$. Alors $a'_1(\phi)=a'_2(\phi)(1-\phi)(z_\sharp)$ pour un élément $z_\sharp\in \hat{Z}_\sharp$. Puisque $\phi$ est trivial sur $\hat{G}_{\rm SC}$, cela se simplifie: on a $\zeta_1=\zeta_2$ et il existe un élément $z\in \hat{Z}$ tel que
$$
z_1\equiv z_2(1-\phi)(z) \quad ({\rm mod}\; \hat{Z}\cap \hat{T}^{\hat{\theta},\circ}).
$$
En conjuguant la donnée $\bs{T}'_1$ par $z$, on se ramène au cas où $z_1z_2^{-1}\in \hat{Z}\cap \hat{T}^{\hat{\theta},\circ}$. On a supposé que $h_1$ et $h_2$ ont même image dans $PGL(n,{\Bbb C})$, à savoir l'élément $v_{\rm ad}$ de la remarque 4. Par suite $h_{1,{\rm sc}}$ et $h_{2,{\rm sc}}$ diffèrent par un élément de $\hat{Z}_{\rm sc}$, qui est inclus dans $\hat{T}_{\rm sc}=\hat{T}_{\rm sc}^{\hat{\theta},\circ}$. Donc $\pi(h_{1,{\rm sc}})$ et $\pi(h_{2,{\rm sc}})$ diffèrent par un élément de $\hat{Z}\cap \hat{T}^{\hat{\theta},\circ}$. On obtient finalement que $h_1$ et $h_2$ diffèrent par un élément de $\hat{Z}\cap \hat{T}^{\hat{\theta},\circ}$. Puisque $\ES{T}'_i$ est engendré par $(h_i,\phi)$, $\hat{T}^{\hat{\theta},\circ}$ et les élément $(1,w)$ pour $w\in I_F$, on obtient que $\ES{T}'_1=\ES{T}'_2$. Enfin puisque $\zeta_1=\zeta_2$, les éléments $s_1$ et $s_2$ ont même image dans $PGL(n,{\Bbb C})$, par suite ils diffèrent par un élément de $\hat{Z}$. Les données $\bs{T}'_1$ et $\bs{T}'_2$ sont donc isomorphes, et puisqu'on les a prises dans $\mathfrak{E}_{\rm t-nr}$, elles sont égales. 

\vskip1mm
Preuve de (iii). On conserve les hypothèses sur $\hat{G}_{\rm AD}$, et on fixe un élément $\bs{T}'\in \mathfrak{E}_{\rm t-nr}$. Notons $\wt{T}'(F)^{\wt{G}}$ l'ensemble des éléments de $\wt{T}'(F)$ qui correspondent à une classe de conjugaison stable semisimple dans $\wt{G}(F)$. On sait qu'il n'est pas vide et qu'il existe un sous--groupe $T'(G)^G$ de $T'(F)$, ouvert et d'indice fini, de sorte que $\wt{T}'(F)^{\wt{G}'}$ soit une seule classe modulo multiplication par ce sous--groupe. Soit $\mu$ un caractère de $T'(F)/T'(F)^G$, et soit $\tilde{\mu}$ le caractère affine de $\wt{T}'(F)$ qui vaut $1$ sur $\wt{T}'(F)^{\wt{G}}$. D'après \cite[2.6]{Stab I}, on sait qu'il existe un automorphisme de $\bs{T}'$ dont l'action associée sur les fonctions sur $\wt{T}'(F)$ est la multiplication par $\tilde{\mu}$. On en déduit:
\begin{enumerate}
\item[(2)]Soient $(\chi'_1,\tilde{\chi}'_1)$ et $(\chi'_2,\tilde{\chi}'_2)$ deux caractères affines de $\wt{T}'(F)$ tels que $\chi_1$ et $\chi_2$ co\"{\i}ncident sur $T'(F)^G$. Alors, à homothétie près, $\tilde{\chi}'_1$ et $\tilde{\chi}'_2$ se déduisent l'un de l'autre par un automorphisme de 
$\bs{T}'$.
\end{enumerate}

On va identifier le groupe $T'(F)^G$. Pour cela choisissons une paire de Borel $({\Bbb B},{\Bbb T})$ de $G$ définie sur $F$ --- on peut bien sûr prendre $({\Bbb B},{\Bbb T})=(B,T)$ ---, et notons $\theta$ le $F$--automorphisme de ${\Bbb T}$ associé à cette paire.   On note $\sigma\mapsto \sigma_G$ l'action galoisienne naturelle sur ${\Bbb T}$. On peut identifier $T'$ à ${\Bbb T}/(1-\theta)({\Bbb T})$ muni d'une action galoisienne $\sigma \mapsto \sigma'_G= w(\sigma)\sigma_G$, où $w\mapsto w(\sigma)$ est un cocycle non ramifié de $W_F$ à valeurs dans $W^\theta$, où $W=W^G({\Bbb T})$ est le groupe de Weyl (ici on a $W^\theta=W$ puisque $\hat{\theta}$ est trivial sur $\hat{G}_{\rm AD}$). On note $T_0$ le tore ${\Bbb T}$ muni de l'action galoisienne $\sigma\mapsto \sigma'_G$. On a un homomorphisme naturel $T_0(F)\rightarrow T'(F)$. Montrons que:
\begin{enumerate}
\item[(3)] $T'(F)^G$ est l'image de cet homomorphisme naturel.
\end{enumerate}
On a une suite d'homomorphismes
$$
(\hat{Z}/(\hat{Z}\cap \hat{T}^{\hat{\theta},\circ})^{\Gamma_F}\rightarrow {\rm H}^1(\Gamma_F,\hat{Z}\cap \hat{T}^{\hat{\theta},\circ})={\rm H}^1(\Gamma_F,\hat{Z}\cap \hat{T}')\rightarrow {\rm H}^1(W_F,\hat{T}').\leqno{(4)}
$$
On a noté $\hat{T}'$ le tore $\hat{T}^{\hat{\theta},\circ}$ muni de l'action galoisienne provenant de $T'$. Les deux actions galoisiennes (celle sur $\hat{T}'$ et celle sur $\hat{T}^{\hat{\theta},\circ}$) co\"{\i}ncident sur $\hat{Z}$, d'où l'égalité centrale. On a d'autre part une dualité entre ${\rm H}^1(W_F,\hat{T}')$ et $T'(F)$. D'après \cite[1.13]{Stab I}, $T'(F)^G$ est l'annulateur dans $T'(F)$ de l'image de la suite (4). Il suffit de montrer que cette image est le noyau de l'homomorphisme ${\rm H}^1(W_F,\hat{T}')\rightarrow {\rm H}^1(W_F,\hat{T}_0)$ dual de l'homomorphisme naturel $T_0(F)\rightarrow T'(F)$; évidemment $\hat{T}_0$ s'identifie à $\hat{T}$ muni de l'action galoisienne convenable. L'homomorphisme $p:T_0\rightarrow T'$ se complète en la suite exacte
$$
1\rightarrow (1-\theta)(T_0)\rightarrow T_0\rightarrow T'\rightarrow 1.
$$
Dualement, on a la suite exacte
$$
1\rightarrow \hat{T}'= \hat{T}_0^{\hat{\theta},\circ}\rightarrow \hat{T}_0\rightarrow \hat{T}_0/ \hat{T}_0^{\hat{\theta},\circ}\rightarrow 1,
$$
laquelle fournit une suite exacte de cohomologie
$$
{\rm H}^0(W_F,\hat{T}_\circ / \hat{T}_0^{\hat{\theta},\circ})\rightarrow {\rm H}^1(W_F, \hat{T}')\rightarrow {\rm H}^1(W_F,\hat{T}_\circ).\leqno{(5)}
$$
Comme $\hat{\theta}$ est trivial sur $\hat{G}_{\rm AD}$, on a $\hat{T}_{0,{\rm ad}}= \hat{T}_{0,{\rm ad}}^{\hat{\theta},\circ}$. On en déduit que $\hat{T}_0/\hat{T}_0^{\hat{\theta},\circ}$ n'est autre que $\hat{Z}/ (\hat{Z}\cap \hat{T}_0^{\hat{\theta},\circ})$, ou encore que $\hat{Z}/(\hat{Z}\cap \hat{T}^{\hat{\theta},\circ})$. Il est clair qu'alors le premier homomorphisme de la suite (5) co\"{\i}ncide avec l'homomorphisme composé de la suite (4). Son image est donc bien le noyau de ${\rm H}^1(W_F,\hat{T}')\rightarrow {\rm H}^1(W_F,\hat{T}_0)$, ce qui prouve (3). 

L'homomorphisme $\xi_Z:Z(G)\rightarrow T'$ est le composé du plongement naturel $Z(G)\rightarrow T_0$ et de la projection $p:T_0\rightarrow T'$. En vertu de (2) et (3), il reste à prouver l'assertion suivante:
\begin{enumerate}
\item[(6)] Supposons que $\chi'_1$ et $\chi'_2$ sont non ramifiés et qu'ils co\"{\i}ncident sur $p(Z(G;F))$. Alors ils co\"{\i}ncident sur $p(T_0(F))$. 
\end{enumerate} 
Pour $i=1,\, 2$, le caractère $\chi_i\circ p$ de $T_0(F)$ est non ramifié, \cad qu'il est trivial sur le sous--groupe compact maximal $T_0(F)_1$ de $T_0(F)$. Fixons une uniformisante $\varpi_F$ de $F$ et identifions $\check{X}(T_0)$ à un sous--groupe de $T_0$ via l'application $\check{x}\mapsto \check{x}(\varpi_F)$. On sait que $T_0(F)$ est le produit direct de $T_0(F)_1$ et de $\check{X}(T_0)^{\Gamma_F}$. Par ellipticité, on sait que $\check{X}(T_{0,{\rm ad}})^{\Gamma_F,\theta}=\{0\}$. Mais $\hat{\theta}$ est trivial sur $\hat{G}_{\rm AD}$, par conséquent $\check{X}(T_{0,{\rm ad}} )^{\Gamma_F}=\{0\}$. Il en résulte que $\check{X}(T_0)^{\Gamma_F}$ est contenu dans $\check{X}(Z(G)^\circ)^{\Gamma_F}$. Donc $T_0(F)=T_0(F)_1Z(G;F)$. L'assertion (6) en résulte. Cela achève la démonstration du point (iii) et de la proposition.   
\end{proof}

\subsection{Réduction au cas où $\hat{G}_{\rm AD}$ est simple}\label{réduction au cas simple}
On souhaite supprimer l'hypothèse que $\hat{G}_{\rm AD}$ est simple dans la proposition de \ref{le cas simple}. 
Notons $\Omega$ le groupe d'automorphismes de $\hat{G}$ engendré par $\phi$ et $\hat{\theta}$. Il opère sur l'ensemble des composantes connexes du diagramme de Dynkin $\bs{\Delta}$ de $\hat{G}_{\rm AD}$. Soient $\EuScript{O}_1,\ldots ,\EuScript{O}_d$ les orbites sous $\Omega$ dans cet ensemble, et pour $i=1,\ldots ,d$, soit $\hat{G}_{{\rm AD},i}$ le sous--groupe de $\hat{G}_{\rm AD}$ correspondant à $\EuScript{O}_i$. On a la décomposition
$$
\hat{G}_{\rm AD}= \hat{G}_{{\rm AD},1}\times \cdots \times \hat{G}_{{\rm AD},d}.\leqno{(1)}
$$
Pour $i=1,\ldots ,d$, les automorphismes $\phi$ et $\hat{\theta}$ de $\hat{G}$ induisent des automorphismes de $\hat{G}_{{\rm AD},i}$, que l'on note encore $\phi$ et $\hat{\theta}$. 

Du point (i) de la proposition de \ref{le cas simple}, on déduit le résultat suivant.

\begin{monlem}(On ne suppose pas que $\hat{G}_{\rm AD}$ est simple.) On suppose que l'ensemble 
$\mathfrak{E}_{\rm t-nr}$ n'est pas vide. Alors 
pour $i=1,\ldots ,d$, le groupe $\hat{G}_{{\rm AD},i}$ est isomorphe à un produit de copies d'un groupe adjoint simple $\hat{H}$ de type $A_{n-1}$ (l'entier $n$, ainsi que le nombre de copies de $\hat{H}$, dépendent de $i$), et les automorphismes $\phi$ et $\hat{\theta}$ de $\hat{G}_{{\rm AD},i}$ sont de la forme suivante: il existe des entiers $m, \, r,\, q \geq 1$ tels que $\hat{G}_{{\rm AD},i}\simeq 
((\hat{H}^{\times m})^{\times r})^{\times q}$, et notant $\alpha$ l'automorphisme de $\hat{H}^{\times m}$ 
donné par 
$$
\alpha(x_1,\ldots ,x_m)= (x_2,\ldots ,x_m,x_1)
$$
et $\beta$ l'automorphisme de $\hat{H}^*=(\hat{H}^{\times m})^{\times r}$ donné par
$$
\beta(y_1,\ldots ,y_r)=(y_2,\ldots ,y_r,\alpha(y_1)),
$$
les automorphismes $\hat{\theta}$ et $\phi$ de 
$\hat{G}_{{\rm AD},i}\simeq (\hat{H}^*)^{\times q}$ sont donnés par
$$
\hat{\theta}(g_1,\ldots, g_q)= (\beta(g_1),\ldots ,\ldots \beta(g_q))
$$
et
$$
\phi(g_1,\ldots ,g_q) =(g_2,\ldots ,g_{q-1}, \beta^{er}(g_1))
$$
pour un entier $e\in \{1,\ldots ,m-1\}$ premier à $m$ (si $m=1$, on prend $e=0$).
\end{monlem}

\begin{proof}Comme dans la preuve du point (i) de la proposition de \ref{le cas simple}, on peut supposer que $\hat{G}=\hat{G}_{\rm AD}$. On a donc
$$
\hat{G}= \hat{G}_1\times \cdots \times \hat{G}_{d},\quad \hat{\theta}=\hat{\theta}^{\otimes d},
$$
où l'on a posé $\hat{G}_i=\hat{G}_{{\rm AD},i}$. On note
$$
G=G_1\times \cdots \times G_d,\quad \widetilde{G}= \widetilde{G}_{1}\times \cdots \times \widetilde{G}_{d}
$$
la décomposition duale: $G_i$ est un groupe semisimple simplement connexe, défini et quasi--déployé sur $F$, et déployé sur une extension non ramifiée de $F$; $\widetilde{G}_{i}$ est un espace tordu sous $G_{i}$, défini sur $F$ et tel que $\widetilde{G}_{i}(F)$ n'est pas vide. Remarquons que les espaces $\wt{G}_i$ peuvent être définis canoniquement par la formule
$\wt{G}_i= \wt{G}/ (\prod_{j\neq i}G_j)$.
On note 
$\EuScript{E}_{{\rm t-nr},i}$ l'ensemble des données endoscopiques $\boldsymbol{T}'_i= (T'_i,\EuScript{T}'_i, \tilde{s}_i)$ pour $(G_{i}, \widetilde{G}_{i})$ qui sont elliptiques, non ramifiées, et telles que $T'_i$ est un tore.

Alors on a la décomposition
$$
\EuScript{E}_{\rm t-nr}= \EuScript{E}_{{\rm t-nr},1}\times \cdots \times \EuScript{E}_{{\rm t-nr},d}.
$$
On peut donc supposer que $d=1$, \cad que le groupe $\Omega$ opère transitivement sur les composantes connexes du diagramme de Dynkin $\bs{\Delta}$ de $\hat{G}$. 

Soit $\boldsymbol{T}'=(T',\EuScript{T}',\tilde{s})\in \EuScript{E}_{\rm t-nr}$ et $(h,\phi)\in \EuScript{T}'$. On écrit $\tilde{s}=s\hat{\theta}$ et $\boldsymbol{h}=h\phi$. On procède par étapes, en allant du cas particulier vers le cas général.

\vskip1mm
{\it \'Etape 1}. On suppose que chacun des deux automorphismes $\hat{\theta}$ et $\phi$ de $\hat{G}$ opère transitivement sur l'ensemble des composantes connexes de $\bs{\Delta}$. On note 
$\bs{\Delta}_1,\ldots ,\bs{\Delta}_r$ ces composantes connexes, ordonnées de telle manière que 
$\phi(\bs{\Delta}_{i+1})=\bs{\Delta}_i$, $i=1,\ldots r-1$. Pour $i=1,\ldots ,r$, on note $\hat{G}_i$ la composante simple de $\hat{G}$ correspondant à $\bs{\Delta}_i$, et pour $i=1,\ldots ,r-1$, on identifie $\hat{G}_{i+1}$ à $\hat{G}_1$ via $\phi^i$. Avec ces identifications, l'automorphisme $\phi$ de $\hat{G}=\hat{G}_1\times \cdots \times \hat{G}_1$ est donné par
$$
\phi(x_1,\ldots ,x_r)= \phi(x_2,\ldots ,x_r,\phi_1(x_1)),
$$
où $\phi_1$ est l'automorphisme de $\hat{G}_1$ induit par $\phi^r$. 
Puisque $\hat{\theta}$ opère transitivement sur les composantes connexes de $\Delta$, il existe un entier $e\in \{1,\ldots ,r-1\}$ premier à $r$ (si $r=1$, on prend $e=0$), et des automorphismes $\hat{\theta}_1,\ldots ,\hat{\theta}_r$ de $\hat{G}_1$, tels que
$$
\phi^{-e}\hat{\theta}(x_1,\ldots ,x_r)= (\hat{\theta}_1(x_1),\ldots ,\hat{\theta}_r(x_r)).    
$$
Comme 
$\hat{\theta}\phi=\phi \hat{\theta}$, on a $\hat{\theta}_r=\hat{\theta}_{r-1} =\cdots = \hat{\theta}_1$ et 
$\hat{\theta}_1\phi_1 = \phi_1\hat{\theta}_1$. Notons $\alpha$ l'automorphisme d'ordre $r$ de $\hat{G}_1$ donné par
$$
\alpha(x_1,\ldots ,x_r)= (x_2,\ldots, x_r,x_1).
$$
On verra plus loin que l'on peut se ramener aux deux cas particuliers suivants:
\begin{itemize}
\item {\bf Cas 1:} $e=1$, $\phi_1=1$, i.e. $\phi =\alpha$ et $\hat{\theta} = \hat{\theta}_1^{\otimes r}\alpha$;
\item {\bf Cas 2:} $e=1$, $\hat{\theta}_1=1$, i.e. $\hat{\theta}=\phi$.
\end{itemize}
Montrons que dans les deux cas, les actions $\phi_1$ et $\hat{\theta}_1$ sur $\hat{G}_1$ sont triviales, et $\hat{G}_1$ est de type $A_{n-1}$.

\vskip1mm
Commen\c{c}ons par le {\it cas 1}. Posons $\mu =\hat{\theta}_1$. Rappelons que l'on a posé $\tilde{s}=s\hat{\theta}$. On écrit $s=(s_1,\ldots ,s_r)$. Pour $g=(g_1,\ldots ,g_r)\in \hat{G}$, on a
$$
g^{-1}\tilde{s}g= (g_1^{-1}s_1\mu(g_2),\ldots ,g_{r-1}^{-1}s_{r-1}\mu(s_r), g_r^{-1}s_r \mu(g_1)).
$$
En prenant $g=(1,s_2\mu(s_3)\cdots \mu^{r-2}(s_r),\ldots ,s_{r-1}\mu(s_r) ,s_r)$, on obtient
$$
g^{-1}\tilde{s}g= (s_1\mu(s_2)\cdots \mu^{r-1}(s_r),1,\ldots ,1).
$$
Quitte à remplacer $\tilde{s}$ et $\boldsymbol{h}$ par $g^{-1}\tilde{s}g$ et $g^{-1}\boldsymbol{h}g$ pour un $g\in \hat{G}$, on peut donc supposer que $s=(s_1,1,\ldots ,1)$. 
L'équation $h^{-1}s\hat{\theta}(h)=\phi(s)$ entra\^{\i}ne que
$$
h=(h_1, \mu^{r-2}(\mu(h_1)s_1^{-1}),\ldots,  \mu(\mu(h_1)s_1^{-1}),\mu(h_1)s_1^{-1})
$$
pour un $h_1\in \hat{G}_1$ vérifiant $h_1^{-1}s_1\mu^{r-1}(\mu(h_1)s_1^{-1})=1$. Posons $\tilde{s}_1= s_1\mu^{r}$ et $\boldsymbol{h}_1= h_1 \mu^{r-1}$. On a
$$
\tilde{s}_1\boldsymbol{h}_1=\boldsymbol{h}_1\tilde{s}_1. 
$$
L'application
$$
g_1\mapsto (g_1,\mu^{r-1}(g_1),\ldots ,\mu(g_1))
$$
induit un isomorphisme de $Z_{\hat{G}_1}(\tilde{s}_1)$ sur $Z_{\hat{G}}(\tilde{s})$ (donc aussi de 
$Z_{\hat{G}_1}(\tilde{s}_1)^\circ$ sur $Z_{\hat{G}}(\tilde{s})^\circ$), et induit aussi un isomorphisme de $Z_{\hat{G}_1}(\tilde{s}_1,\boldsymbol{h}_1)$ sur $Z_{\hat{G}}(\tilde{s},\boldsymbol{h})$. On en déduit que les automorphismes $\tilde{s}_1$ et $\boldsymbol{h}_1$ de $\hat{G}_1$ vérifient les propriétés (1), (2) et (4) de \ref{le cas simple}. Puisque le groupe $\hat{G}_1$ est adjoint et simple, on peut appliquer la proposition de \ref{le cas simple} (point (i)): les actions $\mu^r$ et $\mu^{r-1}$ sur $\hat{G}_1$ sont triviales, et $\hat{G}_1$ est de type $A_{n-1}$. On a aussi $\hat{\theta}_1= \mu^r(\mu^{r-1})^{-1}=1$. 
 
 \vskip1mm
Traitons maintenant le {\it cas 2}. Cette fois--ci posons $\mu= \phi_1$. On peut comme dans le cas 1 supposer que $s=(s_1,1,\ldots ,1)$. L'équation $h^{-1}s\hat{\theta}(h)= \phi(s)$ entra\^{\i}ne que
$$
h=(h_1,\mu(h_1s_1^{-1}),\ldots ,\mu(h_1s_1^{-1}))
$$
pour un élément $h_1\in \hat{G}_1$ vérifiant $h_1^{-1}s_1\mu(h_1s_1^{-1})=1$. Ainsi, posant $\tilde{s}_1=s_1\mu$ et $\boldsymbol{h}_1= h_1\mu$, on a
$$
\tilde{s}_1\boldsymbol{h}_1= \boldsymbol{h}_1\tilde{s}_1.
$$
L'application
$$
g_1\mapsto (g_1,\mu(g_1),\ldots ,\mu(g_1))
$$
induit un isomorphisme de $Z_{\hat{G}_1}(\tilde{s}_1)$ sur $Z_{\hat{G}}(\tilde{s})$, et de $Z_{\hat{G}}(\tilde{s}_1,\boldsymbol{h}_1)$ sur $Z_{\hat{G}}(\tilde{s},\boldsymbol{h})$. On en déduit que les automorphismes $\tilde{s}_1$ et $\boldsymbol{h}_1$ de $\hat{G}_1$ vérifient les propriétés (1), (2) et (4) de \ref{le cas simple}. On peut donc, comme dans le cas 1, appliquer 
la proposition de \ref{le cas simple} (point (i)). On obtient que $\mu=1$ et que $\hat{G}_1$ est de type $A_{n-1}$. 

\vskip1mm
Revenons au cas général (toujours dans l'étape 1). Supposons $\phi_1\neq 1$. 
Alors $\hat{\theta}_1$ est une puissance $\phi_1^k$ de 
$\phi_1$ pour un entier $k\geq 0$, d'où $\hat{\theta}= \phi^{e+kr}$. On modifie les actions de $\phi$ et $\hat{\theta}$ en posant $\phi'=\phi$ et $\hat{\theta}'= \phi$. On pose $\boldsymbol{h}'= \boldsymbol{h}$ et $\tilde{s}'= \boldsymbol{h}^{1-e -kr} \tilde{s}$. Les automorphismes $\boldsymbol{h}'$ et $\tilde{s}'$ de $\hat{G}$ vérifient encore les conditions (2), (4) et (5) de \ref{le cas simple}. En particulier puisque leur commutant commun $Z_{\hat{G}}(\tilde{s}',\boldsymbol{h}')$ co\"{\i}ncide avec $Z_{\hat{G}}(\tilde{s},\boldsymbol{h})$, il est fini. On est dans le cas 2 pour les actions $\phi'$ et $\hat{\theta}'$. 
On a donc $\phi_1=1$, contradiction. Supposons maintenant 
$\phi_1=1$. On modifie les actions de $\phi$ et $\hat{\theta}$ en posant $\phi'=\phi\;(=\alpha)$ et $\hat{\theta}'= \phi_1^{\otimes r}\alpha$. On pose $\boldsymbol{h}'= \boldsymbol{h}$ et $\tilde{s}'= 
\boldsymbol{h}^{1-e}\tilde{s}$. Les automorphismes $\boldsymbol{h}'$ et $\tilde{s}'$ de $\hat{G}$ vérifient les conditions (2), (4) et (5) de \ref{le cas simple}, et l'on est dans le cas 2 pour les actions $\phi'$ et $\hat{\theta}'$. On a donc $\hat{\theta}_1=1$, d'où $\phi= \alpha$ et $\hat{\theta}= \alpha^e$, et $\hat{G}_1$ est de type $A_{n-1}$. 

Récapitulons: le groupe $\hat{G}$ est isomorphe à un produit direct de $r$ copies 
d'un groupe adjoint simple de type $A_{n-1}$, et $\phi$ et $\hat{\theta}$ sont deux $r$--cycles d'ordre $r$.

\vskip1mm
{\it \'Etape 2.} On suppose que $\hat{\theta}$ opère transitivement sur l'ensemble des composantes connexes de $\bs{\Delta}$. On note $\bs{\Delta}_1,\ldots ,\bs{\Delta}_r$ les $\phi$--orbites dans cet ensemble. Comme $\phi$ et $\hat{\theta}$ commutent, ces orbites sont $\hat{\theta}$--stables, et l'on peut supposer que $\hat{\theta}(\bs{\Delta}_{i+1})= \bs{\Delta}_i$, $i=1,\ldots ,r-1$. Pour $i=1,\ldots ,r$, on note $\hat{G}_i$ le sous--groupe de $\hat{G}$ correspondant à $\bs{\Delta}_i$, et pour $i=1,\ldots ,r-1$, on identifie $\hat{G}_{i+1}$ à $\hat{G}_1$ via $\theta^i$. Avec ces identifications, l'automorphisme $\hat{\theta}$ de $\hat{G}=\hat{G}_1\times \cdots \times \hat{G}_1$ est donné par
$$
\hat{\theta}(x_1,\ldots ,x_r)= (x_2,\ldots ,x_r, \hat{\theta}_1(x_1)),
$$
où $\hat{\theta}_1$ est la restriction de $\hat{\theta}^r$ à $\hat{G}_1$. Quant à l'automorphisme $\phi$, puisqu'il commute à $\hat{\theta}$, il est donné par $\phi = \phi_1^{\otimes r}$ pour un automorphisme $\phi_1$ de $\hat{G}_1$ qui commute à $\hat{\theta}_1$. Par construction, 
les automorphismes $\hat{\theta}_1$ et $\phi_1$ de $\hat{G}_1$ opèrent chacun transitivement sur l'ensemble des composantes connexes de $\bs{\Delta}_1$. Comme dans le cas 2 de l'étape 1, on peut supposer que $s=(s_1,1,\ldots ,1)$. L'équation $h^{-1}s\hat{\theta}(h)=\phi(s)$ entra\^{\i}ne que
$$
h=(h_1, \hat{\theta}_1(h_1),\ldots ,\hat{\theta}_1(h_1))
$$
pour un élément $h_1\in \hat{G}_1$ vérifiant $h_1^{-1}s_1\hat{\theta}_1(h_1)= \phi_1(s_1)$. Posons 
$\tilde{s}_1=s_1\hat{\theta}_1$ et $\boldsymbol{h}_1= h_1\phi_1$. On a donc
$$
\tilde{s}_1\boldsymbol{h}_1=\boldsymbol{h}_1\tilde{s}_1.
$$
L'application
$$
g_1\mapsto (g_1,\hat{\theta}_1(g_1),\ldots ,\hat{\theta}_1(g_1))
$$
induit un isomorphisme de $Z_{\hat{G}}(\tilde{s}_1)$ sur $Z_{\hat{G}}(\tilde{s})$, et de $Z_{\hat{G}}(\tilde{s}_1,\boldsymbol{h}_1)$ sur $Z_{\hat{G}}(\tilde{s},\boldsymbol{h})$. On en déduit que les automorphismes $\tilde{s}_1$ et $\boldsymbol{h}_1$ de $\hat{G}_1$ vérifient les propriétés (1), (2) et (4) de \ref{l'ensemble t-nr}. On peut donc leur appliquer le résultat de l'étape 1: le groupe $\hat{G}_1$ est isomorphe à un produit direct de $m$ copies d'un groupe adjoint simple de type $A_{n-1}$ pour un entier $m\geq 1$, et $\phi_1$ et $\hat{\theta}_1$ sont deux $m$--cycles d'ordre $m$. En particulier $\phi_1 = (\hat{\theta}_1)^e$ pour un entier $e\in \{1,\ldots ,m-1\}$ premier à $m$ (si $m=1$, on prend $e=0$).

\vskip1mm
{\it \'Etape 3.} On peut maintenant traiter le cas général (en supposant toujours que $\Omega$ opère transitivement sur les composantes connexes de $\bs{\Delta}$). On note $\bs{\Delta}_1,\ldots ,\bs{\Delta}_q$ les $\hat{\theta}$--orbites dans l'ensemble des composantes connexes de $\bs{\Delta}$. Puisque $\phi$ et $\hat{\theta}$ commutent, ces ensembles $\bs{\Delta}_i$ sont permutés (transitivement) par $\phi$, et l'on 
peut supposer que $\phi(\bs{\Delta}_{i+1})=\bs{\Delta}_i$, $i=1,\ldots ,q-1$. Pour $i=1,\ldots ,q$, on note $\hat{G}_i$ le sous--groupe de $\hat{G}$ correspondant à $\bs{\Delta}_i$, et pour $i=1,\ldots ,q-1$, on identifie $\hat{G}_{i+1}$ à $\hat{G}_1$ via $\phi^i$. Avec ces identifications, l'automorphisme $\phi$ de $\hat{G}=\hat{G}_1\times \cdots \times \hat{G}_1$ est donné par
$$
\phi(x_1,\ldots ,x_q)=(x_2,\ldots ,x_{q-1},\phi_1(x_1)),
$$
où $\phi_1$ est la restriction de $\phi^q$ à $\hat{G}_1$. Quant à l'automorphisme $\hat{\theta}$, puisqu'il commute à $\phi$, il est donné par $\hat{\theta}= \hat{\theta}_1^{\otimes q}$ pour un automorphisme $\hat{\theta}_1$ de $\hat{G}_1$ qui commute à $\hat{G}_1$. Comme dans le cas 2 de l'étape 1 (en rempla\c{c}ant $\hat{\theta}$ par $\phi$), on peut supposer que $h=(h_1,1,\ldots ,1)$. Comme dans l'étape 2 (en échangeant les rôles de $\hat{\theta}$ et de $\phi$), on obtient que
$$
s=(s_1, \phi_1(s_1),\ldots ,\phi_1(s_1))
$$
pour un élément $s_1$ de $\hat{G}_1$ vérifiant $s_1^{-1}h_1\phi_1(s_1)=\hat{\theta}_1(h_1)$. 
Posons $\tilde{s}_1=s_1\hat{\theta}_1$ et $\boldsymbol{h}_1= h_1\phi_1$. On a encore
$$
\tilde{s}_1\boldsymbol{h}_1=\boldsymbol{h}_1\tilde{s}_1.
$$
Toujours comme dans l'étape 2 (en échangeant les rôles de $\tilde{s}$ et de $\boldsymbol{h}$), on en déduit que les automorphismes $\tilde{s}'_1= \boldsymbol{h}_1$ et $\boldsymbol{h}'_1=\tilde{s}_1$ vérifient les conditions (1), (2) et (4) de \ref{l'ensemble t-nr}. On peut donc leur appliquer le résultat de l'étape 2. 
Cela achève la preuve du lemme.
\end{proof}

On suppose que l'ensemble $\mathfrak{E}_{\rm t-nr}$ n'est pas vide. D'après le lemme, on peut supposer que
$$
\hat{G}_{\rm AD}= (\hat{H}_1^*)^{\times q_1}\times \cdots \times (\hat{H}_d^*)^{\times q_d},\quad 
\hat{H}_i^* = (\hat{H}_i^{\times m_i})^{\times r_i},\quad \hat{H}_i= PGL(n_i,{\Bbb C}),\leqno{(2)}
$$
et que pour $i=1,\ldots ,d$, notant $\alpha_i$ l'automorphisme
$$
(x_1,\ldots, x_{m_i})\mapsto 
(x_2,\ldots ,x_{m_i},x_1)
$$
de $(\hat{H}_i)^{\times m_i}$, et $\beta_i$ l'automorphisme 
$$
(y_1,\ldots ,y_{r_i})\mapsto (y_2,\ldots ,y_{r_i},\alpha_i(y_1))
$$
de $(\hat{H}_i^{\times m_i})^{\times r_i}$, les automorphismes $\hat{\theta}$ et $\phi$ de 
$(\hat{H}_i^*)^{\times q_i}$ sont donnés par
$$
\hat{\theta}(g_1,\ldots ,g_{q_i})= (\beta_i(g_1),\ldots ,\beta_i(g_{q_i})), \leqno{(3)}
$$
$$
\phi(g_1,\ldots ,g_{q_i})=(g_2,\ldots ,g_{q_i},\beta_i^{e_i r_i}(g_1))\leqno{(4)}
$$
pour un entier $e_i\in \{1,\ldots ,m_i-1\}$ premier à $m_i$ (si $m_i=1$, on prend $e_i=0$). Notons que
$$\beta_i^{e_i r_i}= (\alpha_i^{e_i})^{\otimes r_i}.
$$ 

Soit $\boldsymbol{T}'=(T',\EuScript{T}',\tilde{s})$ un élément de $\mathfrak{E}_{\rm t-nr}$ et soit $(h,\phi)\in \EuScript{H}'$. D'après la preuve du lemme et la remarque 4 de \ref{le cas simple}, on peut supposer que les éléments $\tilde{s}=s\hat{\theta}$ et $\boldsymbol{h}=h\phi$ sont de la {\it forme standard} suivante. Pour $i=1,\ldots ,d$, on a une racine primitive $n_i$--ième de l'unité $\zeta_i$. On note $\bar{u}_i=u_{{\rm ad},i}(\zeta_i)$ et $\bar{v}_i=v_{{\rm ad},i}$ les images dans $PGL(n_i,{\Bbb C})$ des éléments suivants de $GL(n_i,{\Bbb C})$:
$$
u_i(\zeta_i)={\rm diag} (\zeta_i^{n_i-1},\zeta_i^{n_i-2},\ldots, \zeta, 1),
$$
$$
v_i=
\begin{pmatrix}0&\cdots &\cdots & 0& (-1)^{n_i-1}\\
1& 0 & \cdots &\cdots& 0  \\
0& 1 & \ddots &&\vdots\\
\vdots&\ddots&\ddots&\ddots&\vdots\\
0&\cdots & 0& 1& 0
\end{pmatrix}.
$$
Notons $\bar{s}=(\bar{s}_1,\ldots ,\bar{s}_d)$ et $\bar{h}=(\bar{h}_1,\ldots ,\bar{h}_d)$ 
les images de $s$ et $h$ dans $\hat{G}_{\rm AD}$. Pour $i=1,\ldots ,d$, on a
$$
\bar{s}_i=(\bar{s}_{i,1},\beta_i^{e_ir_i}(\bar{s}_{i,1}),\ldots , \beta_i^{e_ir_i}(\bar{s}_{i,1}))
\in (\hat{H}_i^*)^{q_i},
$$
$$ 
\bar{s}_{i,1}= (\bar{a}_i,1,\ldots ,1)\in \hat{H}^*_i= (\hat{H}_i^{\times m_i})^{\times r_i},
$$
$$
\bar{a}_i=(\bar{u}_i,1,\ldots ,1)\in \hat{H}_i^{\times m_i}
$$
et
$$
\bar{h}_i=(\bar{h}_{i,1},1,\ldots ,1) \in (\hat{H}_i^*)^{q_i};
$$
$$
\bar{h}_{i,1}= (\bar{b}_i, \alpha_i(\bar{b}_i),\ldots ,\alpha_i(\bar{b}_i))\in \hat{H}^*_i.
$$
$$
\bar{b}_i= (\bar{v}_i,\bar{v}_i\bar{u}_i^{-1},\ldots ,\bar{v}_i\bar{u}_i^{-1},\bar{v}_i,\ldots,\bar{v}_i)\in \hat{H}_i^{\times m_i}, 
$$
où le nombre de facteurs $\bar{v}_i\bar{u}_i^{-1}$ dans $\bar{b}_i$ est égal à $m_i-e_i$ si $m_i>1$, et à $0$ sinon. 
Puisque $\beta_i^{e_ir_i}= (\alpha_i^{e_i})^{\otimes r_i}$, on a
$$\beta_i^{e_ir_i}(\bar{s}_{i,1})= (\alpha_i^{e_i}(\bar{a}_i),1,\ldots ,1).
$$
On a donc (par construction)
$$
\bar{v}_i^{-1}\bar{u}_i\bar{v}_i=\bar{u}_i,
$$
$$
\bar{b}_i^{-1}\bar{a}_i\alpha_i(\bar{b}_i)= \alpha_i^{e_i}(\bar{a}_i),
$$
$$
\bar{h}_{i,1}^{-1}\bar{s}_{i,1}\beta_i(\bar{h}_{i,1})=  \beta_i^{e_ir_i}(\bar{s}_{i,1}),
$$
$$
\bar{h}_i\bar{s}_i\hat{\theta}(\bar{h}_i)= \phi(\bar{s}_i).
$$
La paire $(\bar{s},\bar{h})$ étant uniquement déterminée par le $d$--uplet $\zeta= (\zeta_1,\ldots ,\zeta_d)$, on la note aussi $(s_{\rm ad}(\zeta),h_{\rm ad}(\zeta))$.

Pour $i=1,\ldots ,d$, soit
$$
\pi_i: \hat{H}_{{\rm SC},i}= SL(n_i,{\Bbb C})\rightarrow PGL(n_i,{\Bbb C})=\hat{H}_i
$$
le revêtement simplement connexe de $\hat{H}_i$. Posant 
$\hat{G}_{{\rm SC},i}= ((\hat{H}_i^{\times m_i})^{\times r_i})^{\times q_i}$, l'application
$$
\pi= \otimes_{i=1}^d((\pi_i^{\otimes m_i})^{\otimes r_i})^{\otimes q_i}: 
\hat{G}_{\rm SC}=\hat{G}_{{\rm SC},1}\times\cdots \times \hat{G}_{{\rm SC},d}\rightarrow \hat{G}_{\rm AD} 
$$
est le revêtement simplement connexe de $\hat{G}_{\rm AD}$.
Les automorphismes $\phi$ et $\hat{\theta}$ de $\hat{G}_{\rm AD}$ se relèvent de manière unique en des automorphismes de $\hat{G}_{\rm SC}$, et la description des actions de $\phi$ et $\hat{\theta}$ sur $\hat{G}_{\rm SC}$ est identique à celle  des actions sur $\hat{G}_{\rm AD}$. Pour $i=1,\ldots ,d$, choisissons un relèvement $u_{{\rm sc},i}$ de $\bar{u}_i$ dans $\hat{H}_{{\rm SC},i}$, et notons 
$s_{{\rm sc},i}$ l'élément de $\hat{G}_{{\rm SC},i}$ obtenu en rempla\c{c}ant $\bar{u}_i$ par $u_{{\rm sc},i}$ dans la construction de $\bar{s}_i$. Posons
$$
s_{\rm sc}=(s_{{\rm sc},1},\ldots ,s_{{\rm sc},d})\in \hat{G}_{\rm SC}.
$$
C'est un relèvement de $\bar{s}=s_{\rm ad}(\zeta)$ dans $\hat{G}_{\rm SC}$. De même, pour $i=1,\ldots ,d$, choisissons un relèvement $v_{{\rm sc},i}$ de $\bar{v}_i$ dans $\hat{H}_{{\rm SC},i}$ --- on peut bien sûr prendre $v_{{\rm sc},i}=v_i\in SL(n_i,{\Bbb C})$ ---, et notons $h_{{\rm sc},i}$ l'élément de $\hat{G}_{{\rm SC},i}$ obtenu en rempla\c{c}ant $\bar{v}_i$ par $v_{{\rm sc},i}$ dans la construction de $h_i$. Posons
$$
h_{\rm sc}=(h_{{\rm sc},1},\ldots ,h_{{\rm sc},d})\in \hat{G}_{\rm SC}.
$$
C'est un relèvement de $\bar{h}=h_{\rm ad}(\zeta)$ dans $\hat{G}_{\rm SC}$. De tels relèvements $s_{\rm sc}$ de 
$s_{\rm ad}(\zeta)$ et $h_{\rm sc}$ de $h_{\rm ad}(\zeta)$ sont appelés des {\it bons relèvements}.

\begin{marema1}{\rm Supposons $q=1$ et supprimons l'indice $i$ dans les notations précédentes. On a $\hat{H}= PGL(n,{\Bbb C})$, $\hat{H}_{\rm SC}=SL(n,{\Bbb C})$, $\hat{G}_{\rm AD}= ((\hat{H}^{\times m})^{\times r})^{\times q}$ et $\hat{G}_{\rm SC}= ((\hat{H}_{\rm SC}^{\times m})^{\times r})^{\times q}$. On a aussi un entier $e\in \{0,\ldots ,m-1\}$ qui est $>0$ et premier à $m$ si $m>1$, et vaut $0$ sinon. On note $(g_{i,j,k})_{1\leq i\leq m, 1\leq j \leq r, 1\leq k\leq q}$ un élément de $\hat{G}_{\rm AD}$ ou $\hat{G}_{\rm SC}$. Les indices $i,\,j,\,k$ doivent plutôt être considérés comme des entiers modulo $m,\, r,\, q$. Ainsi on a:
\begin{itemize}
\item $\hat{\theta}(g)_{i,j,k}= g_{i,j+1,k}$ si $j\neq r$;
\item $\hat{\theta}(g)_{i,r,k}= g_{i+1,1,k}$;
\item $\phi(g)_{i,j,k}= g_{i,j,k+1}$ si $k\neq q$;
\item $\phi(g)_{i,j,q}= g_{i+e, j,1}$.
\end{itemize}
Pour une racine primitive $n$--ième de l'unité $\zeta$, choisissons des relèvements $u_{\rm sc}$ et 
$v_{\rm sc}$ de $u_{\rm ad}(\zeta)$ et $v_{\rm ad}$ dans $SL(n,{\Bbb C})$. Alors les bons relèvements $s_{\rm sc}=(s_{{\rm sc};i,j,k})$ et $h_{\rm sc}=(h_{{\rm sc};i,j,k})$ de $s_{\rm ad}(\zeta)$ et $h_{\rm ad}(\zeta)$ dans $\hat{G}_{\rm SC}$ associés à ces choix sont donnés par:
\begin{itemize}
\item $s_{{\rm sc};1,1,1}= u_{\rm sc}$;
\item $s_{{\rm sc};e,1,k}= u_{\rm sc}$ si $k\neq 1$;
\item $s_{{\rm sc};i,j,k}=1$ dans tous les autres cas;
\item $h_{{\rm sc};i,j,1}=v_{\rm sc}$ si $j=1$ et $i\neq 2, \ldots , m-e+1$, ou si $j\neq 1$ et $i\neq 1, \ldots m-e$;
\item $h_{{\rm sc};i,j,1}= v_{\rm sc}u_{\rm sc}^{-1}$ si $j=1$ et $i=2,\ldots , m-e+1$, ou si $j \neq 1$ et $i=1,\ldots ,m-e$; 
\item $h_{{\rm sc};i,j,k}=1$ si $k\neq 1$.
\end{itemize}
}
\end{marema1}

\begin{marema2}
{\rm On a envie de définir l'espace tordu $(G_{\rm SC},\wt{G}_{\rm SC})$, mais on se garde bien de le faire! En effet, pour $\gamma\in \wt{G}(F)$, le $F$--automorphisme $\theta_\gamma={\rm Int}_\gamma$ de $G$ se relève en un unique $F$--automorphisme de $\wt{G}_{\rm SC}$, encore noté $\theta_\gamma$. On peut alors définir $\wt{G}_{\rm SC}$ comme étant $G_{\rm SC}\theta_\gamma$. La structure galoisienne est telle que $\theta_\gamma$ est fixe par $\Gamma_F$, et cet espace s'envoie naturellement dans $\wt{G}$ par l'application $\wt{G}_{\rm SC}\rightarrow g_{\rm sc}\theta_\gamma \mapsto \rho(g_{\rm sc})\gamma$, où $\rho:G_{\rm SC}\rightarrow G$ est l'homomorphisme naturel. Le problème est que l'espace $\wt{G}_{\rm SC}$ ainsi défini dépend de $\gamma$. De plus, il se peut qu'il ne soit pas ``non ramifié'', \cad que $\wt{G}_{\rm SC}(F)$ ne possède aucun sous--espace hyperspécial. Par exemple si $\wt{G}=G= SL(n)$, l'espace $\wt{G}_{\rm SC}= G_{\rm SC}\theta_\gamma$ est isomorphe à l'ensemble des $g\in GL(n)$ tels que $\det(g) = \det(\gamma)$. Il dépend clairement (pour sa structure galoisienne) de $\det(\gamma)$. L'espace $\wt{G}_{\rm SC}(F)$ possède un sous--espace hyperspécial si et seulement s'il existe un élément $\gamma'_{\rm sc}\in \wt{G}_{\rm SC}(F)$ tel que ${\rm Int}_{\gamma'_{\rm sc}}$ stabilise un sous--groupe hyperspécial de $SL(n,F)$, ce qui n'est possible que si la valuation normalisée de $\det(\gamma)$ appartient à $n{\Bbb Z}$.}
\end{marema2} 

\subsection{Séparation des données dans $\mathfrak{E}_{\rm t-nr}$}\label{séparation des données}La proposition suivante est une généralisation du point (ii) de la proposition de \ref{le cas simple}.

\begin{mapropo}
(On ne suppose pas que $\hat{G}_{\rm AD}$ est simple.) Supposons que l'ensemble $\mathfrak{E}_{\rm t-nr}$ n'est pas vide. Soient $\bs{T}'$ et $\bs{T}''$ deux éléments de $\mathfrak{E}_{\rm t-nr}$. Si $\omega_{\bs{T}'}=\omega_{\bs{T}''}$, alors $\bs{T}'=\bs{T}''$. 
\end{mapropo}

\begin{proof}Soit $\bs{T}'=(T',\ES{T}',\tilde{s})$ un élément de $\mathfrak{E}_{\rm t-nr}$ et soit $(h,\phi)\in \ES{T}'$. On suppose que les automorphismes $\hat{\theta}$ et $\phi$ de $\hat{G}_{\rm AD}$ sont donnés par les formules (2), (3), (4) de \ref{réduction au cas simple}. On suppose aussi que les éléments $\tilde{s}=s\hat{\theta}$ et $\bs{h}= h\phi$ sont de la forme standard  décrite en \ref{réduction au cas simple}, \cad que les projections $\bar{s}$ et $\bar{h}$ sur $\hat{G}_{\rm AD}$ des éléments $s$ et $h$ de $\hat{G}$ sont de la forme $\bar{s}=s_{\rm ad}(\zeta)$ et $\bar{h}= h_{\rm ad}(\zeta)$ pour un $d$--uplet $\zeta=(\zeta_i,\ldots ,\zeta_d)$, où $d$ est le nombre de composantes connexes de $\hat{G}_{\rm AD}$ (formule (1) de \ref{réduction au cas simple}) et $\zeta_i$ est une racine primitive $n_i$--ième de l'unité. Fixons des bons relèvement $s_{\rm sc}$ et $h_{\rm sc}$ de $\bar{s}$ et $\bar{h}$ dans $\hat{G}_{\rm SC}$. Ces bons relèvements sont construits en choisissant des relèvements $u_{{\rm sc},i}$ et $v_{{\rm sc},i}$ de $\bar{u}_i$ et $\bar{v}_i$ dans $\hat{H}_{{\rm SC},i}$ ($i=1,\ldots ,d$).
 
Rappelons que l'on a noté $\hat{Z}$, $\hat{Z}_{\rm sc}$ et $\hat{Z}_\sharp$ les centres de $\hat{G}$, $\hat{G}_{\rm SC}$ et $\hat{G}_\sharp$. D'après \cite[2.7]{Stab I}, on a une suite exacte courte
$$
1 \rightarrow \hat{Z}_{\rm sc}/\hat{Z}_{\rm sc}^{\hat{\theta}}
\xrightarrow{(\pi,1-\hat{\theta})} \hat{Z}/(\hat{Z}\cap \hat{T}^{\hat{\theta},\circ})
\times \hat{Z}_{\rm sc}\rightarrow \hat{Z}_\sharp\rightarrow 1;\leqno{(1)}
$$
où $\pi:\hat{G}_{\rm SC}\rightarrow \hat{G}$ est l'homomorphisme naturel.
Le caractère $\omega_{\boldsymbol{T}'}$ de $G_\sharp(F)$ correspond à un cocycle $a'$ de $W_F$ dans $\hat{Z}_\sharp$. Ce cocycle est non ramifié, donc déterminé par l'élément $a'(\phi)$ de $\hat{Z}_\sharp$, et sa classe est donnée par l'image de 
$a'(\phi)$ dans $\hat{Z}_\sharp/(1-\phi)(\hat{Z}_\sharp)$. D'après loc.~cit., 
l'élément $a'(\phi)$ se calcule comme suit. On écrit
$$
h= z\pi(h_{\rm sc})
$$
où $\pi:\hat{G}_{\rm SC}\rightarrow \hat{G}$ est l'homorphisme naturel et $z\in \hat{Z}$. Puisque
$$
\bar{h}^{-1}\bar{s}\hat{\theta}(\bar{h})= \phi(\bar{s}),
$$
l'élément $h_{\rm sc}^{-1}s_{\rm sc}\hat{\theta}(h_{\rm sc})\phi(s_{\rm sc})^{-1}$ 
appartient à $\hat{Z}_{\rm sc}$. On le note $a'_{\rm sc}$. Alors $a'(\phi)$ est 
l'image de $(z,a'_{\rm sc})$ dans $Z(\hat{G}_\sharp)$. Le centre $\hat{Z}_{\rm sc}$ de $\hat{G}_{\rm SC}$ se décompose en
$$
\hat{Z}_{\rm sc}= \hat{Z}_{{\rm sc},1}\times \cdots \times \hat{Z}_{{\rm sc},d},\quad 
\hat{Z}_{{\rm sc},i}=
((Z(\hat{H}_{{\rm SC},i})^{m_i})^{r_i})^{q_i}.
$$
On écrit $a'_{\rm sc}=(a'_{{\rm sc},1},\ldots ,a'_{{\rm sc},d})$, où $a'_{{\rm sc},i}\in \hat{Z}_{{\rm sc},i}$. On a par construction
$$
a'_{{\rm sc},i}= (z_{{\rm sc},i},1,\ldots ,1),\quad 
z_{{\rm sc},i}=v_{{\rm sc},i}^{-1}u_{{\rm sc},i}v_{{\rm sc},i}u_{{\rm sc},i}^{-1},
$$
et $z_{{\rm sc},i}$ est l'élément 
${\rm diag}(\zeta_i^{-1},\ldots ,\zeta_i^{-1})$ de $Z(\hat{H}_{{\rm SC},i})$. 
L'élément $a'_{\rm sc}$ ne dépend donc que du $d$--uplet $\zeta=(\zeta_1,\ldots ,\zeta_d)$. On le note 
aussi $z_{\rm sc}(\zeta)$. 

\begin{marema}
{\rm Pour calculer l'élément $a'_{\rm sc}=h_{\rm sc}^{-1}s_{\rm sc}\hat{\theta}(h_{\rm sc})\phi(s_{\rm sc})^{-1}$, on a choisi des bons relèvements $s_{\rm sc}$ et $h_{\rm sc}$ de $\bar{s}$ et $\bar{h}$ dans $\hat{G}_{\rm SC}$. Comme on vient de le voir, $a'_{\rm sc}$ ne dépend pas du choix de ces bons relèvements, mais si l'on choisit des relèvements quelconques, ce n'est en général plus vrai (sauf si $\phi$ et $\hat{\theta}$ sont triviaux sur $\hat{Z}_{\rm sc}$). Si l'on remplace $h_{\rm sc}$ et $s_{\rm sc}$ par $h'_{\rm sc}=y^{-1}h_{\rm sc}\phi(y)$ et $s'_{\rm sc}=y^{-1}s_{\rm sc}\hat{\theta}(y)$ pour un $y\in \hat{G}_{\rm SC}$, alors on a l'égalité
$$h'^{-1}_{\rm sc}s'_{\rm sc}\hat{\theta}(h'_{\rm sc})\phi(s'_{\rm sc})^{-1}=\phi(y)^{-1}h_{\rm sc}^{-1}s_{\rm sc}\hat{\theta}(h_{\rm sc})\phi(s_{\rm sc})^{-1}\phi(y).
$$
En particulier pour $y\in \hat{Z}_{\rm sc}$, les éléments $h'_{\rm sc}$ et $s'_{\rm sc}$ sont encore des relèvements de $\bar{h}$ et $\bar{s}$ dans $\hat{G}_{\rm SC}$ (pas forcément bons), et on l'égalité
$$
h'^{-1}_{\rm sc}s'_{\rm sc}\hat{\theta}(h'_{\rm sc})\phi(s'_{\rm sc})^{-1}=h_{\rm sc}^{-1}s_{\rm sc}\hat{\theta}(h_{\rm sc})\phi(s_{\rm sc})^{-1}.
$$
} 
\end{marema}

\vskip2mm
Soient maintenant deux éléments $\boldsymbol{T}'=(T',\EuScript{T}',\tilde{s}')$ et 
$\boldsymbol{T}''=(T'',\EuScript{T}'',\tilde{s}'')$ de $\mathfrak{E}_{t-nr}$, et soient $(h',\phi)\in \EuScript{T}'$ et $(h'',\phi)\in \EuScript{T}''$. On reprend les constructions ci--dessus en affublant les objets d'exposants ``$\,'\,$'' et ``$\,''\,$'' pour les distinguer. On écrit $\tilde{s}'=s'\hat{\theta}$, $\boldsymbol{h}'=h'\phi$, $\tilde{s}''=s''\hat{\theta}$ et $\boldsymbol{h}''= h''\phi'$. On suppose que les projections des éléments $s'$, $h'$, $s''$, $h''$ de $\hat{G}$ sur $\hat{G}_{\rm AD}$ sont de la forme standard (\ref{réduction au cas simple}). En particulier, à ces projections sont associées deux $d$--uplets $\zeta'=(\zeta'_1,\ldots ,\zeta'_d)$ et $\zeta''=(\zeta''_1,\ldots ,\zeta''_d)$ où, pour $i=1,\ldots ,d$, les éléments $\zeta'_i$ et $\zeta''_i$ sont des racines 
primitives $n_i$--ièmes de l'unité. On fixe des bons relèvements $s'_{\rm sc}$, $h'_{\rm sc}$, $s''_{\rm sc}$, $h''_{\rm sc}$ dans  
$\hat{G}_{\rm SC}$ des projections de $s'$, $h'$, $s''$, $h''$ sur $\hat{G}_{\rm AD}$. On écrit $h'=z'\pi(h'_{\rm sc})$ et $h''=z''\pi(h''_{\rm sc})$ pour des éléments 
$z'$ et $z''$ de $\hat{Z}$. Les caractères $\omega_{\boldsymbol{T}'}$ et 
$\omega_{\boldsymbol{T}''}$ de $G_\sharp(F)$ correspondent à des cocycles $a'$ et 
$a''$ de $W_F$ dans $\hat{Z}_\sharp$, déterminés par les éléments $a'(\phi)$ et $a''(\phi)$ de $\hat{Z}_\sharp$, dont seules comptent 
les images dans $\hat{Z}_\sharp/ (1-\phi)(\hat{Z}_\sharp)$. On a vu que $a'(\phi)$ est l'image de $(z',z_{\rm sc}(\zeta'))$ dans $\hat{Z}_\sharp$, et que $a''(\phi)$ est l'image de $(z'',z_{\rm sc}(\zeta''))$ dans 
$\hat{Z}_\sharp$. 

Supposons que $\omega_{\boldsymbol{T}'}=\omega_{\boldsymbol{T}''}$. On a donc $a''(\phi)= a'(\phi)(1-\phi)(z_\sharp)$ pour un élément $z_\sharp\in \hat{Z}_\sharp$. Choisissons un relèvement $(y,y_{\rm sc})$ de $z_\sharp$ dans 
$\hat{Z}\times \hat{Z}_{\rm sc}$. D'après la suite exacte courte (1) décrivant $\hat{Z}_\sharp$, il existe un élément $z^0_{\rm sc}\in \hat{Z}_{\rm sc}$ tel que:
$$
z_{\rm sc}(\zeta'')= z_{\rm sc}(\zeta') (1-\phi)(y_{\rm sc})(1-\hat{\theta})(z_{\rm sc}^0),\leqno{(2)}
$$
$$
z'' \equiv z' (1-\phi)(y)\pi(z_{\rm sc}^0) \quad ({\rm mod}\; \hat{Z}\cap \hat{T}^{\hat{\theta},\circ}).
\leqno{(3)}
$$
Montrons que l'égalité (2) n'est possible que si
$$(1-\phi)(y_{\rm sc})(1-\hat{\theta})(z_{\rm sc}^0)=1.\leqno{(4)}
$$
Rappelons que pour $i=1,\ldots ,d$, on a posé $\hat{Z}_{{\rm sc},i}= ((Z(\hat{H}_{{\rm SC},i})^{\times m_i})^{\times r_i})^{\times q_i}$. 
Il s'agit de vérifier (pour $i=1,\ldots ,d$) que si $y$ et $z$ sont deux éléments de $\hat{Z}_{{\rm sc},i}$  tels que les coordonnées de $(1-\phi)(y)(1-\hat{\theta})(z)$ sur les $m_ir_iq_i$ facteurs $Z(\hat{H}_{{\rm SC},i})$ de $\hat{Z}_{{\rm sc},i}$ sont toutes égales à $1$, sauf peut--être sur le premier facteur, alors $(1-\phi)(y)(1-\hat{\theta})(z)=1$. Pour cela on peut supposer $d=1$ et supprimer l'indice $i$ dans les notations. Soient $y=(y_1,\ldots ,y_q)$ et $z=(z_1,\ldots ,z_q)$ deux éléments de $\hat{Z}_{\rm sc}$, où les coordonnées $y_k$ et $z_k$ sont dans $Z(\hat{H}^*)=(Z(\hat{H}_{\rm SC})^{\times m})^{\times r}$. Un calcul explicite donne
$$
(1-\phi)(y)(1-\hat{\theta})(z)= ((1-\beta^{er})(y_1)(1-\beta)(z_1\cdots z_{q}), 1,\ldots ,1)
$$
On peut donc supposer $q=1$. Alors on a
$$
(1-\phi)(y)(1-\hat{\theta})(z)= (1-\beta^{re})(y)(1-\beta)(z)
=(1-\beta)(x),
$$
où l'on a posé
$$
x=y\beta(y)\cdots \beta^{re-1}(y)z.
$$
Puisque $\beta$ est un $mr$--cycle d'ordre $mr$, si les $rm-1$ dernières coordonnées de $(1-\beta)(x)$ sont égales à $1$, alors la première l'est aussi, ce qui démontre (4). 

Revenons à $d$ quelconque. D'après (2) et (4), on a $z_{\rm sc}(\zeta'')=
z_{\rm sc}(\zeta')$, d'où $\zeta''=\zeta'$. On note simplement $\zeta$ ce $d$--uplet. Puisque les projections de $s'$ et $s''$ sur $\hat{G}_{\rm AD}$ sont de la forme standard, elles sont égales (à $s_{\rm ad}(\zeta))$. On peut donc supposer que $s'_{\rm sc}=s''_{\rm sc}$. De la même manière, on peut supposer que $h'_{\rm sc}=h''_{\rm sc}$. On note simplement $s_{\rm sc}$ et $h_{\rm sc}$ ces bons relèvements de $s_{\rm ad}(\zeta)$ et  $h_{\rm ad}(\zeta)$ dans $\hat{G}_{\rm SC}$. Posons $\tilde{s}_{\rm sc}= s_{\rm sc}\hat{\theta}$ et $\bs{h}_{\rm hs}= h_{\rm sc}\phi$. Pour $x\in \hat{G}_{\rm SC}$, notons $\eta(x,\tilde{s}_{\rm sc})$ et $\eta(x,\bs{h}_{\rm sc})$ les éléments de $\hat{G}_{\rm SC}$ définis par
$$
\eta(x,\tilde{s}_{\rm sc})= x s_{\rm sc}\hat{\theta}(x)^{-1}s_{\rm sc}^{-1},\quad \eta(x,\bs{h}_{\rm sc})= xh_{\rm sc}\phi(x)^{-1}x^{-1},
$$
\cad par
$$
x\tilde{s}_{\rm sc}x^{-1}= \eta(x,\tilde{s}_{\rm sc})\tilde{s}_{\rm sc},\quad 
x \bs{h}_{\rm sc}  x^{-1} = \eta(x,\bs{h}_{\rm sc})\bs{h}_{\rm sc}.
$$
Soit $\mathfrak{Z}(\tilde{s}_{\rm sc},\bs{h}_{\rm sc})$ l'ensemble des $x\in \hat{G}_{\rm SC}$ tels que $\eta(x,\tilde{s}_{\rm sc})$ et $\eta(x,\bs{h}_{\rm sc})$ appartiennent à $\hat{Z}_{\rm sc}$. Cet ensemble est un groupe, qui co\"{\i}ncide avec la pré--image dans $\hat{G}_{\rm SC}$ du commutant commun de $s_{\rm ad}(\zeta)\hat{\theta}$ et $h_{\rm ad}(\zeta)\phi$ dans $\hat{G}_{\rm AD}$. 

\begin{monlem}
Soit $z\in \hat{Z}_{\rm sc}$. Supposons que $(1-\hat{\theta})(z) \in (1-\phi)(\hat{Z}_{\rm sc})$. Il existe un $x\in \mathfrak{Z}(\tilde{s}_{\rm sc},\bs{h}_{\rm sc})$ tel que $z\in \eta(x,\bs{h}_{\rm sc})(1-\phi)(\hat{Z}_{\rm sc})$.
\end{monlem} 

\begin{proof}On peut supposer $d=1$. Reprenons les notations de la remarque 1 de \ref{réduction au cas simple}. Les éléments $u_{\rm ad}(\zeta)$ et $v_{\rm ad}$ de $\hat{H}= PGL(n,{\Bbb C})$ sont les projections des éléments $u(\zeta)$ et $v$ de $GL(n,{\Bbb C})$ définis dans la remarque 4 de \ref{le cas simple}. On a choisi un relèvement $u_{\rm sc}$ de $u_{\rm ad}(\zeta)$ dans $\hat{H}_{\rm SC}= SL(n,{\Bbb C})$. L'élément $v$ est déjà dans $\hat{H}_{\rm SC}$. On note $\mathfrak{Y}$ le sous--groupe de $\hat{H}_{\rm SC}$ engendré par $\hat{Z}_{\rm sc}$, $u_{\rm sc}$ et $v$. Soit $\mathfrak{X}$ l'image de $\mathfrak{Y}$ dans $\hat{G}_{\rm SC}$ par le plongement diagonal: un élément $x= (x_{i,j,k})$ de $\hat{G}_{\rm SC}$ appartient à $\mathfrak{X}$ si et seulement si $x_{i,j,k}$ est indépendant de $i,\,j,\,k$ et appartient à $\mathfrak{Y}$. On note $\overline{\mathfrak{Y}}$ et $\overline{\mathfrak{X}}$ les projections de $\mathfrak{Y}$ et $\mathfrak{X}$ sur $\hat{H}$ et $\hat{G}_{\rm AD}$. Le groupe $\overline{\mathfrak{Y}}$ est le commutant commun de $u_{\rm ad}(\zeta)$ et $v_{\rm ad}$ dans $\hat{H}$. Il en résulte qu'un élément de $\overline{\mathfrak{X}}$ commute à $s_{\rm ad}(\zeta)\hat{\theta}$ et à $h_{\rm ad}(\zeta)\phi$. En revanche $u_{\rm sc}$ et $v$ ne commutent pas dans $\hat{H}_{\rm SC}$. Les éléments de $\mathfrak{X}$ ne commutent pas à $\tilde{s}_{\rm sc}$ ni à $\bs{h}_{\rm sc}$, mais $\mathfrak{X}$ est contenu dans $\mathfrak{Z}(\tilde{s}_{\rm sc},\bs{h}_{\rm sc})$. On note $\hat{Z}_{\mathfrak{X}}$ le sous--groupe de $\hat{Z}_{\rm sc}$ engendré par les $\eta(x,\bs{h}_{\rm sc})$ quand $x$ décrit $\mathfrak{X}$. On peut vérifier qu'il ne dépend pas de $\zeta$, mais ce n'est pas utile ici. 

Pour $z=(z_{i,j,k})\in \hat{Z}_{\rm sc}$, posons $p(z)_j= \prod_{i,k}z_{i,j,k}$ ($j=1,\ldots ,r$). On vérifie que $(1-\phi)(\hat{Z}_{\rm sc})$ est le groupe formé des $z\in \hat{Z}_{\rm sc}$ tels que $p(z)_j=1$ pour tout $j\in \{1,\ldots ,r\}$. On vérifie aussi que $p((1-\hat{\theta})(z))_j= p(z)_jp(z)_{j+1}^{-1}$ (avec $j+1=1$ si $j=r$). L'hypothèse sur $z$ dans l'énoncé est donc que $p(z)_j$ est indépendant de $j$. Pour prouver le lemme, il suffit de prouver que pour $\xi\in Z(\hat{H}_{\rm SC})$, il existe un 
$z'\in \hat{Z}_{\mathfrak{X}}$ tel que $p(z')_j=\xi$ pour tout $j\in \{1,\ldots ,r\}$. Soient $a,\, b\in {\Bbb Z}$. Considérons l'élément $y= (u_{\rm sc})^a v^b$ de $\mathfrak{Y}$. Notons $x$ l'élément de $\mathfrak{X}$ tel que $x_{i,j,k}=y$, et posons $\eta= \eta(x,\bs{h}_{\rm sc})$. À l'aide de l'égalité
$$
v u_{\rm sc}v^{-1} = \bs{\zeta} u_{\rm sc},\quad \bs{\zeta}={\rm diag}(\zeta,\ldots , \zeta)\in Z(\hat{H}_{\rm SC}),
$$
on vérifie que $p(\eta)_j = \bs{\zeta}^{-ma -eb}$ pour tout $j\in \{1,\ldots ,r\}$. Puisque $m$ et $e$ sont premiers entre eux si $m>1$, et que $\bs{\zeta}$ engendre $Z(\hat{H}_{\rm SC})$, on peut choisir $a$ et $b$ de sorte que $p(\eta)_j=\xi$ pour tout $j\in \{1,\ldots ,r\}$. Cela démontre le lemme. 
\end{proof}

Reprenons la démonstration de la proposition. D'après l'égalité (4), l'élément $z_{\rm sc}^0\in \hat{Z}_{\rm sc}$ vérifie l'hypothèse du lemme. On peut donc écrire $z_{\rm sc}^0= \eta(x,\bs{h}_{\rm sc})(1-\phi)(y_{\rm sc})$ pour des éléments $x\in \hat{G}_{\rm SC}$ et $y_{\rm sc}\in \hat{Z}_{\rm sc}$ tels que $x$ vérifie les conditions du lemme. La relation (3) entra\^{\i}ne alors que
$$
z''\equiv z'(1-\phi)(y')\pi(\eta(x,\bs{h}_{\rm sc}))\quad ({\rm mod}\; \hat{Z}\cap \hat{T}^{\hat{\theta},\circ}).\leqno{(5)}
$$
pour un élément $y'\in \hat{Z}$. Posons $g=y'\pi(x)$. 
Puisque $\eta(x,\tilde{s}_{\rm sc})\in \hat{Z}_{\rm sc}$, on a 
$$
g\tilde{s}'g^{-1}=(1-\hat{\theta})(y') 
\pi(\eta(x,\tilde{s}_{\rm sc})) \tilde{s}'\in \hat{Z} \tilde{s}''.
$$
D'autre part, d'après (5), on a
$$
g\bs{h}'g^{-1}= z'(1-\phi)(y') \pi(\eta(x,\bs{h}_{\rm sc}) \pi(h_{\rm sc}) \phi \in \bs{h}''(\hat{Z}\cap \hat{T}^{\hat{\theta},\circ}).
$$
Comme dans la preuve du point (ii) de la proposition de \ref{le cas simple}, on en déduit que $\ES{T}'=\ES{T}''$. Par conséquent l'élément $g$ est un isomorphisme de $\bs{T}'$ sur $\bs{T}''$, et puisque ces données ont été prises dans $\mathfrak{E}_{\rm t-nr}$, elles sont égales.
\end{proof}

\subsection{Action de $\ES{C}$ sur les transferts usuels non ramifiés}\label{action de C sur les transferts usuels}Soit $\bs{T}'=(T',\ES{T}',\tilde{s})$ un élément de $\mathfrak{E}_{\rm t-nr}$. Comme en \ref{le cas simple}, on suppose que $\tilde{s}=s\hat{\theta}$ pour un $s\in \hat{T}$ (ainsi $\hat{T}'$ s'identifie à $\hat{T}^{\hat{\theta},\circ}$ muni de l'action galoisienne convenable). 
Pour transférer un caractère affine unitaire $\tilde{\chi}'$ de $\wt{T}'(F)$ à $(\wt{G}(F),\omega)$, on a besoin de fixer un isomorphisme ${^LT'}\buildrel \simeq\over{\longrightarrow} \ES{T}'$. 
On choisit un élément $(h,\phi)\in \ES{T}'$, et on prend l'isomorphisme ${^LT}\simeq \ES{T}'$ défini par cet élément comme en \ref{l'ensemble t-nr}. 
\'Ecrivons $\bs{h}=h\phi$. On suppose que les projections $\bar{s}$ et $\bar{h}$ de $s$ et $h$ sur $\hat{G}_{\rm AD}$ sont de la forme standard $\bar{s}=s_{\rm ad}(\zeta)$ et $\bar{h}=h_{\rm ad}(\zeta)$ pour un $d$--uplet $\zeta=(\zeta_1,\ldots ,\zeta_d)$, où $\zeta_i$ est une racine primitive $n_i$--ième de l'unité --- cf. \ref{réduction au cas simple}. On fixe aussi 
des bons relèvements $s_{\rm sc}$ et $h_{\rm sc}$ de $s_{\rm ad}(\zeta)$ et $h_{\rm ad}(\zeta)$ dans $\hat{G}_{\rm SC}$, et l'on pose $\tilde{s}_{\rm sc}=s_{\rm sc}\hat{\theta}$ et $\bs{h}_{\rm sc}= h_{\rm sc}\phi$. 

Le sous--espace hyperspécial $(K,\wt{K})$ de $\wt{G}(F)$ fixé en \ref{données endoscopiques nr}  détermine un sous--espace hyperspécial $(K',\wt{K}')$ de $\wt{T}'(F)$ --- on a forcément $K'=T'(F)_1 $ ---, et un facteur de transfert normalisé $\Delta: \ES{D}(\bs{T}')\rightarrow {\Bbb C}^\times$.

La proposition suivante est une généralisation du point (iii) de la proposition de \ref{le cas simple}. 

\begin{mapropo}(On ne suppose pas que $\hat{G}_{\rm AD}$ est simple.) Soit $\bs{T}'=(T',\ES{T}',\tilde{s})\in \mathfrak{E}_{\rm t-nr}$. 
\begin{enumerate}
\item[(i)]Soit $(\chi',\tilde{\chi}')$ un caractère affine unitaire et non ramifié de $\wt{T}'(F)$. La distribution $\Theta_{\tilde{\chi}'}=\bs{\rm T}_{\bs{T}'}(\tilde{\chi}')$ est un vecteur propre 
pour l'action du groupe $\ES{C}$ relativement à un caractère unitaire $\omega_{\chi'}$ de ce groupe, \cad qu'on a
$$
{^{\bs{c}}(\Theta_{\tilde{\chi}'})}= \omega_{\chi'}(\bs{c})\Theta_{\tilde{\chi}'},\quad \bs{c}\in \ES{C}.
$$
\item[(ii)]Soient $(\chi_1,\tilde{\chi}'_1)$ et $(\chi'_2,\tilde{\chi}'_2)$ deux caractères affines unitaires et non ramifiés de $\wt{T}'(F)$. Supposons que $\omega_{\chi_1}=\omega_{\chi_2}$. Alors, à homothétie près, $\tilde{\chi}'_1$ et $\tilde{\chi}'_2$ se déduisent l'un de l'autre par l'action d'un élément du groupe d'automorphisme de $\bs{T}'$.  
\end{enumerate}
\end{mapropo}

\begin{proof}
Preuve de (i). Notons $\theta_{\tilde{\chi}'}$ la fonction localement constante sur $\wt{G}_{\rm reg}(F)$ associée à la distribution $\Theta_{\tilde{\chi}'}$ (cf. \ref{caractères tempérés}). Pour $\gamma\in \wt{G}(F)$ fortement régulier, on a une égalité
$$
\Theta_{\tilde{\chi}'}(\gamma)= [G^\gamma(F):G_\gamma(F)]^{-1}\sum_{\delta} \Delta(\delta,\gamma) \tilde{\chi}(\delta),\leqno{(1)}
$$
où $\delta$ parcourent les éléments fortement $\wt{G}$--réguliers de $\wt{T}'(F)$ qui correspondent à $\gamma$ --- en général, $\delta$ est pris à conjugaison stable près mais, puisque $T'$ est un tore, la conjugaison de $T'$ sur $\wt{T}'$ est triviale. Soit $c=(z,g)\in C$. Pour $\gamma\in \wt{G}$, posons $\gamma^{c}= zg^{-1}\gamma g$. Pour $\gamma\in \wt{G}(F)$ fortement régulier, on vérifie que, si $\delta\in \wt{T}'(F)$ correspond à $\gamma$, alors $\delta^c= \xi_Z(z)\delta$ correspond à $\gamma^c$, où $\xi_Z: Z(G)\rightarrow T'$ est l'homomorphisme naturel; remarquons que $\xi_Z(z)\in T'(F)$. Ainsi on a
$$
\Theta_{\tilde{\chi}'}(\gamma^c)= [G^{\gamma^c}(F):G_{\gamma^c}(F)]^{-1}\sum_{\delta} \Delta(\delta^c,\gamma^c) \tilde{\chi}(\delta^c),\leqno{(2)}
$$
où $\delta$ parcourt le même ensemble que dans (1). Il est clair que l'on a l'égalité
$$
[G^\gamma(F):G_\gamma(F)]= [G^{\gamma^c}(F):G_{\gamma^c}(F)].
$$
On est donc ramené à prouver l'existence d'un caractère $\omega_{\chi'}$ de $\ES{C}$ tel que, 
pour $(\delta,\gamma)\in \ES{D}(\bs{T}')$ et $c=(z,g)\in C$ d'image $\bs{c}=\bs{q}(z,g)\in 
\ES{C}$, on ait
$$
\Delta(\delta^c,\gamma^c)\tilde{\chi}'(\delta^c)=\omega_{\chi'}(\bs{c})\Delta(\delta,\gamma)\tilde{\chi}'(\delta).\leqno{(3)}
$$

Pour un élément $c=(z,g)\in C$, écrivons $g= z_g \pi(g_{\rm sc})$ avec $z_g\in Z(G)$ et $g_{\rm sc}\in G_{\rm SC}$, où $\pi:G_{\rm SC}\rightarrow G$ est l'homomorphisme naturel. On définit un cocycle galoisien $\sigma \mapsto \alpha(\sigma)$ à valeurs dans $Z(G_{\rm SC})$ par $\alpha(\sigma)= 
g_{\rm sc}\sigma(g_{\rm sc})^{-1}$. Le couple $(\alpha, z^{-1}(1-\theta)(z_g))$ définit un élément du groupe de cohomologie\footnote{Ce groupe est celui noté ${\rm H}^1$ par Kottwitz et Shelstad, et ${\rm H}^0$ par Labesse --- cf. la remarque de \cite[1.12]{Stab I}.} ${\rm H}^{1,0}(\Gamma_F; Z(G_{\rm SC}) \buildrel 1-\theta\over{\longrightarrow} Z(G))$, que l'on note $b_c$. Cet élément ne dépend pas du choix de la décomposition de $g$. L'application ainsi définie $c\mapsto b_c$ de $C$ dans 
${\rm H}^{1,0}(\Gamma_F; Z(G_{\rm SC}) \buildrel 1-\theta\over{\longrightarrow} Z(G))$ se factorise à travers $\ES{C}$. On a un homomorphisme naturel
$$
\pi_{\ES{C}}: G_{\rm SC}(F)\rightarrow G_\sharp(F)\rightarrow \ES{C}. 
$$
Notons $\overline{\ES{C}}= \ES{C}/\pi_{\ES{C}}(G_{\rm SC}(F))$ le quotient de $\ES{C}$ par le sous--groupe image de $G_{\rm SC}(F)$. On voit que l'homomorphisme $\bs{c}\mapsto b_{\bs{c}}$ se quotiente en un homomorphisme
$$
\overline{\ES{C}}\rightarrow {\rm H}^{1,0}(\Gamma_F; Z(G_{\rm SC}) \xrightarrow{1-\theta} Z(G)),\, \bar{\bs{c}}\mapsto b_{\bar{\bs{c}}}.\leqno{(4)}
$$
Une preuve standard montre que:
\begin{enumerate}
\item[(5)]l'homomorphisme (4) est bijectif.
\end{enumerate} 

Fixons un couple $(\delta,\gamma)\in \ES{D}(\bs{T}')$, et reprenons la construction du facteur de transfert non ramifié $\Delta(\delta,\gamma)$ donnée en \cite[6.3]{Stab I}. Ici, puisqu'on a identifié ${^LT}'$ à $\ES{T}'$, il n'y a pas de données auxiliaires. Notons $T_0$ le commutant de $G_\gamma$ dans $G$. C'est un tore maximal de $G$ défini sur $F$. Comme dans la démonstration du point (iii) de la proposition de \ref{le cas simple}, $T_0$ s'identifie au tore ``quasi--déployé'' maximal ${\Bbb T}$ (\cad d'une paire de Borel $({\Bbb B},{\Bbb T})$ de $G$ définie sur $F$) tordu par un cocycle galoisien $\sigma \mapsto w(\sigma)$ à valeur dans $W^\theta$, de sorte que l'application naturelle $T_0\rightarrow T_0/(1-\theta)(T_0)\simeq T'$ soit définie sur $F$. Ici $\theta$ est le $F$--automorphisme de ${\Bbb T}$ associé à la paire $({\Bbb B},{\Bbb T})$, et $W= W^G({\Bbb T})$ est le groupe de Weyl. Bien sûr, le tore dual $\hat{T}_0$ (resp. $\hat{T}_0^{\hat{\theta},\circ}$) s'identifie à $\hat{T}$ (resp. 
$\hat{T}^{\hat{\theta},\circ}$) muni de l'action galoisienne convenable. On a une formule \cite[6.3]{Stab I}
$$
\Delta(\delta,\gamma)= \Delta_{\rm II}(\delta,\gamma)\tilde{\lambda}_{\underline{\zeta}}(\delta)^{-1}
\tilde{\lambda}_{\underline{z}}(\gamma) \langle (V_{T_0},\nu_{\rm ad}),(t_{T_0,{\rm sc}},s_{\rm ad})\rangle^{-1},
$$
le produit étant donné par l'accouplement \cite[A.3]{KS1}
$$
\langle\cdot ,\cdot \rangle: 
{\rm H}^{1,0}(\Gamma_F, T_{0,{\rm sc}}\xrightarrow{1-\theta} T_{0,{\rm ad}})\times 
{\rm H}^{1,0}(W_F; \hat{T}_{0,{\rm sc}}\xrightarrow{1-\hat{\theta}}\hat{T}_{0,{\rm ad}})\rightarrow {\Bbb C}^\times.
$$
Dans cette formule, $\underline{z}$ est un cocycle non ramifié de $W_F$ dans $Z(\hat{G})$, et $\underline{\zeta}$ est un cocycle non ramifié de $W_F$ dans $\hat{T}'$. Ces cocycles déterminent des caractères $\lambda_{\underline{z}}$ de $G(F)$ et $\lambda_{\underline{\zeta}}$ de $T'(F)$, qui sont triviaux sur $K$ et sur $K'$. On les prolonge en des caractères affines $\tilde{\lambda}_{\underline{z}}$ de $\wt{G}(F)$ et $\tilde{\lambda}_{\underline{\zeta}}$ de $\wt{T}'(F)$, tels que $\tilde{\lambda}_{\underline{z}}$ soit trivial sur $\wt{K}$ et $\tilde{\lambda}_{\underline{\zeta}}$ soit trivial sur $\wt{K}'$. 
Soit $c=(z,g)\in C$, d'image $\bs{c}=\bs{q}(z,g)\in \ES{C}$. Posons $(\gamma',\delta')=(\gamma^c,\delta^c)$. \'Ecrivons $g= z_g \pi(g_{\rm sc})$ avec $z_g\in Z(G)$ et $g_{\rm sc}\in G_{\rm SC}$ comme ci--dessus. Quand on remplace $\gamma$ par $\gamma'$, le tore $T_0$ est changé en $T'_0={\rm Int}_{\smash{g_{\rm sc}^{-1}}}(T_0)$, mais on se ramène à $T_0$ par l'isomorphisme ${\rm Int}_{g_{\rm sc}}$ qui est défini sur $F$. Pour définir le facteur de transfert $\Delta(\delta,\gamma)$, on a fixé une paire de Borel épinglée $\ES{E}_0$ de $G$ de paire de Borel sous--jacente $(B_0,T_0)$ de sorte que $(\delta,T',T',B_0,T_0,\gamma)$ soit un diagramme au sens de \cite[1.10]{Stab I}, et on a décomposé $\gamma$ en $\gamma = \nu \epsilon$ avec $\epsilon \in Z(\wt{G},\ES{E}_0)$ et $\nu\in T_0$. Pour $\gamma'$, on peut prendre $\ES{E}'_0= {\rm Int}_{\smash{g_{\rm sc}^{-1}}}(\ES{E}_0)$, $\epsilon'= {\rm Int}_{\smash{g_{\rm sc}^{-1}}}(\epsilon)$ et $\nu'= z (\theta-1)(z_g){\rm Int}_{\smash{g_{\rm sc}^{-1}}}(\nu)$. Quand on se ramène à $T_0$ par ${\rm Int}_{g_{\rm sc}}$, on obtient $\nu'= z(\theta-1)(z_g)\nu$, donc $\nu'_{\rm ad}=\nu_{\rm ad}$. On fixe un élément $g_1\in G_{\rm SC}$ tel que ${\rm Int}_{g_1}(\ES{E}_0)$ soit une paire de Borel épinglée de $G$ {\it définie sur $F$}  (par exemple la paire $\ES{E}$ de \ref{L-groupe}). Le cocycle $V_{T_0}$ est de la forme
$$
V_{T_0}(\sigma)= r_{T_0}(\sigma)n_{\ES{E}_0}(\omega_{T_0}(\sigma))u_{\ES{E}_0}(\sigma),
$$
où $u_{\ES{E}_0}(\sigma)= g_1^{-1}\sigma(g_1)$, $\omega_{T_0}(\sigma)$ est l'élément de $W^\theta$ défini par $u_{\ES{E}_0}(\sigma)^{-1}$, $n_{\ES{E}_0}:W\rightarrow G_{\rm SC}$ est la section de Springer associée à $\ES{E}_0$, et $r_{T_0}: \Gamma_F\rightarrow T_{0,{\rm sc}}^\theta$ est une fonction associée à $(B_0,T_0)$ définie en \cite[2.2]{Stab I}. On peut prendre les mêmes $a$--data pour $\gamma$ et pour $\gamma'$, modulo l'isomorphisme ${\rm Int}_{g_{\rm sc}}$. On voit alors que, pour $\gamma'$, les deux premiers termes $r_{T_0}(\sigma)$ et $n_{\ES{E}_0}(\omega_{T_0}(\sigma))$ sont les mêmes que pour $\gamma$, modulo l'isomorphisme ${\rm Int}_{g_{\rm sc}}$. En effet pour $\gamma'$, on peut prendre $g'_1= g_1g_{\rm sc}$. Alors ${\rm Int}_{g'_1}(\ES{E}'_0)= {\rm Int}_{g_1}(\ES{E}_0)$, et posant $u_{\ES{E}'_0}= g'^{-1}_1\sigma(g'_1)$, on a
$${\rm Int}_{g_{\rm sc}}(u_{\ES{E}'_0}(\sigma))=u_{\ES{E}_0}(\sigma)\alpha(\sigma),\quad \alpha(\sigma)= g_{\rm sc}\sigma(g_{\rm sc})^{-1}.
$$
On a un diagramme naturel de complexes
$$
\xymatrix{
Z(G_{\rm SC}) \ar[r]^{1-\theta} \ar[d] & Z(G) \ar[d] \\
T_{0,{\rm sc}} \ar[d] \ar[r]^{1-\theta}&  T_0 \ar[d]\\
T_{0,{\rm sc}} \ar[r]^{1-\theta} & T_{0,{\rm ad}}
}
$$
qui donne naissance à deux homomorphismes
$$
\iota_0: {\rm H}^{1,0}(\Gamma_F; Z(G_{\rm SC})\xrightarrow{1-\theta} Z(G))
\rightarrow {\rm H}^{1,0}(\Gamma_F; T_{0,{\rm sc}}\xrightarrow{1-\theta} T_0),
$$
$$
\iota_1: {\rm H}^{1,0}(\Gamma_F; T_{0,{\rm sc}}\xrightarrow{1-\theta} T_0)
\rightarrow {\rm H}^{1,0}(\Gamma_F; T_{0,{\rm sc}}\xrightarrow{1-\theta} T_{0,{\rm ad}}).
$$
On obtient que, quand on remplace $\gamma$ par $\gamma'$, le terme $(V_{T_0},\nu_{\rm ad})$ est multiplié par $\iota_1\iota_0(b_{\bs{c}})^{-1}$. Côté dual, rien n'a changé (pourvu que l'on conserve les mêmes $\chi$--data). En définitive, quand on remplace $\gamma$ par $\gamma'$, le terme $\langle (V_{T_0},\nu_{\rm ad}),(t_{T_0,{\rm sc}},s_{\rm ad})\rangle^{-1}$ est multiplié par 
$\langle \iota_1\iota_0(b_{\bs{c}}), (t_{T_0,{\rm sc}},s_{\rm ad})\rangle$, ou encore par 
$\langle \iota_0(b_{\bs{c}}), \hat{\iota}_1(t_{T_0,{\rm sc}},s_{\rm ad})\rangle$, où $\hat{\iota}_1$ désigne l'application ``duale'' de $\iota_1$. 

\begin{marema1}
{\rm 
On entend par application ``duale'' l'homomorphisme naturel
$$
\hat{\iota}_1: {\rm H}^{1,0}(W_F; \hat{T}_{0,{\rm sc}}\xrightarrow{1-\hat{\theta}}\hat{T}_{0,{\rm ad}})\rightarrow {\rm H}^{1,0}(W_F; \hat{T}_0\xrightarrow{1-\hat{\theta}}\hat{T}_{0,{\rm ad}}).
$$
Ces groupes ne sont pas les duaux des précédents, ces duaux sont les quotients par les images naturelles de $\hat{T}_{\rm ad}^{\Gamma_F,\circ}$.
}
\end{marema1}

Le terme $\Delta_{\rm II}(\delta,\gamma)$ ne dépend que des $a$--data, des $\chi$--data, et de $\nu_{\rm ad}$. 

En reprenant la définition du caractère affine $\tilde{\lambda}_{\underline{\zeta}}:\wt{T}'(F)\rightarrow {\Bbb C}^\times$, on voit qu'il est trivial à cause de notre identification de ${^LT'}$ avec $\ES{T}'$.  

On a $\gamma'= z g^{-1}{\rm Int}_{\gamma}(g)\gamma$, donc
$$
\tilde{\lambda}_{\underline{z}}(\gamma')= \lambda_{\underline{z}} (g^{-1}{\rm Int}_{\gamma}(g))\tilde{\lambda}_{\underline{z}}(\gamma).
$$
\'Ecrivons $h=z_h \pi(h_{\rm sc})$ avec $z_h\in Z(\hat{G})$ et $h_{\rm sc}\in \hat{G}_{\rm SC}$; ici, $\pi$ est l'homomorphisme naturel $ \hat{G}_{\rm SC}\rightarrow \hat{G}$. Le cocycle non ramifié $\underline{z}:W_F\rightarrow Z(\hat{G})$ est déterminé par la valeur $\underline{z}(\phi)=z_h$. On le pousse en un élément $(\underline{z},1)$ de ${\rm H}^{1,0}(W_F; \hat{T}_0\rightarrow \hat{T}_{0,{\rm ad}})$. Calculons l'image de 
$z g^{-1}{\rm Int}_{\gamma}(g)$ dans $G_{\rm ab}(F)= {\rm H}^{1,0}(\Gamma_F;T_{0,{\rm sc}}\rightarrow T_0)$. On a
$$
z g^{-1}{\rm Int}_{\gamma}(g)= z(\theta-1)(z_g) \pi(g_{\rm sc}^{-1}{\rm Int}_\gamma(g_{\rm sc})).
$$
On définit le cocycle galoisien
$$\sigma \mapsto \beta(\sigma)= g_{\rm sc}^{-1}{\rm Int}_\gamma(g_{\rm sc})\sigma(g_{\rm sc}^{-1}{\rm Int}_\gamma(g_{\rm sc}))^{-1}.$$
L'image de $z g^{-1}{\rm Int}_{\gamma}(g)$ dans $G_{\rm ab}(F)$ est la classe du cocycle $(\beta,z(\theta-1)(z_g))$. Or puisque ${\rm Int}_\gamma$ est défini sur $F$, on a $\beta(\sigma)= (\theta-1)(\alpha(\sigma))$. On a un diagramme de complexes
$$
\xymatrix{
T_{0,{\rm sc}} \ar[r]^{1-\theta}\ar[d]_{1-\theta} & T_0 \ar[d]\\
T_{0,{\rm sc}} \ar[r] & T_0
}
$$
qui donne naissance à un homomorphisme
$$
\iota_2: {\rm H}^{1,0}(\Gamma_F; T_{0,{\rm sc}} \xrightarrow{1-\theta} T_0)
\rightarrow G_{\rm ab}(F).
$$
Alors l'image de $z g^{-1}{\rm Int}_{\gamma}(g)$ dans $G_{\rm ab}(F)$ est égale à $\iota_2 \iota_0( b_{\bs{c}})^{-1}$. On en déduit que, quand on remplace $\gamma$ par $\gamma'$, $\tilde{\lambda}_{\underline{z}}(\gamma)$ est remplacé par $\langle \iota_2\iota_0(b_{\bs{c}}), (\underline{z},1)\rangle^{-1}=\langle \iota_0(b_{\bs{c}}), \hat{\iota}_2(\underline{z},1)\rangle^{-1}$. 

Enfin on a $\tilde{\chi}'(\delta')= \chi'(\xi_Z(z))\tilde{\chi}'(\delta)$. Le caractère $\tilde{\chi}'$ de $T'(F)$ est repéré par un cocycle $\mu_{\chi'}: W_F\rightarrow \hat{T}'=\hat{T}_0^{\hat{\theta},\circ}$. Alors $\chi'(\xi_Z(z))$ est le produit $\langle \xi_Z(z),\mu_{\chi'}\rangle $, pour l'accouplement
$$
\langle\cdot ,\cdot \rangle: {\rm H}^0(\Gamma_F; T_0/(1-\theta)(T_0))\times {\rm H}^1(W_F;\hat{T}_0^{\hat{\theta},\circ})\rightarrow {\Bbb C}^\times.
$$
D'après Kottwitz--Shelstad \cite[relation (A.3.13), page 137]{KS1}, où l'inversion doit disparaître d'après une correction ultérieure \cite{KS2}, c'est aussi le produit $\langle (1,\xi_Z(z)), (\mu_{\chi'},1)\rangle$ 
pour l'accouplement
$$
\langle\cdot, \cdot \rangle: 
{\rm H}^{1,0}(\Gamma_F; 1\rightarrow T_0/(1-\theta)(T_0))\times  {\rm H}^{1,0}(W_F; \hat{T}_0^{\hat{\theta},\circ}\rightarrow 1)\rightarrow{\Bbb C}^\times.
$$
On a un diagramme de complexes
$$
\xymatrix{
T_{0,{\rm sc}}\ar[r]^{1-\theta} \ar[d]& T_0 \ar[d]\\
1 \ar[r] & T_0/(1-\theta)(T_0)
}
$$
qui donne naissance à un homomorphisme
$$
\iota_3:{\rm H}^{1,0}(\Gamma_F; T_{0,{\rm sc}}\xrightarrow{1-\theta}T_0)\rightarrow
{\rm H}^{1,0}(\Gamma_F; 1 \rightarrow T_0/(1-\theta)(T_0)).
$$
Le terme $(1,\xi_Z(z))$ est égal à $\iota_3\iota_0(b_{\bs{c}})^{-1}$. On en déduit que, quand on remplace $\gamma$ par $\gamma'$, $\tilde{\lambda}_z(\gamma)$ est multiplié par 
$\langle \iota_3\iota_0 (b_{\bs{c}}), (\mu_{\chi'},1)\rangle^{-1}=\langle \iota_0(b_{\bs{c}}), \hat{\iota}_3(\mu_{\chi'},1)\rangle^{-1}$.

Définissons l'élément $\eta_{\chi'}\in {\rm H}^{1,0}(W_F; \hat{T}_0 \xrightarrow{1-\hat{\theta}} \hat{T}_{0,{\rm ad}})$ par
$$
\eta_{\chi'}= \hat{\iota}_1(t_{T_0,{\rm sc}},s_{\rm ad})\hat{\iota}_2(\underline{z},1)^{-1}\hat{\iota}_3(\mu_{\chi'},1)^{-1}.
$$
En rassemblant les calculs ci--dessus, on obtient la relation (3), le caractère $\omega_{\chi'}$ de $\ES{C}$ étant défini par
$$
\omega_{\chi'}(\bs{c})= \langle \iota_0(b_{\bs{c}}), \eta_{\chi'}\rangle, \quad \bs{c}\in \ES{C}.
$$
Le point (i) de la proposition est démontré.

\begin{marema2}
{\rm En reprenant les définitions, on calcule
$$
\eta_{\chi'}= (\beta_{\chi'},s_{\rm ad}),
$$
où $\beta_{\chi'}$ est le cocycle non ramifié qui, en $\phi$, vaut $\pi(\hat{r}_{T_0}(\phi)\hat{n}(\omega_{T_0}(\phi)))h^{-1}\mu_{\chi'}(\phi)^{-1}$. Les deux premiers termes sont ceux de \cite[6.3]{Stab I}; on choisit bien sûr des $\chi$--data non ramifiés.}
\end{marema2}

\begin{marema3}
{\rm On peut donner une formule plus explicite. Le terme $\hat{r}_{T_0}(\phi)\hat{n}(\omega_{T_0}(\phi))$ est un élément du normalisateur de $\hat{T}_{0,{\rm sc}}$ dans $\hat{G}_{\rm SC}^{\hat{\theta}}$ qui agit sur $\hat{T}_{0,{\rm sc}}$ de la même fa\c{c}on que ${\rm Int}_{h_{\rm sc}}$. On peut construire un autre élément qui a la même propriété. Rappelons que $h_{\rm sc}$ est un bon relèvement dans $\hat{G}_{\rm SC}$ de l'élément $h_{\rm ad}(\zeta)$. Il est de la forme 
$h_{{\rm sc},i}= (h_{{\rm sc},1},\ldots ,h_{{\rm sc},d})$. Pour $i=1,\ldots ,d$, les composantes de $h_{{\rm sc},i}$ sont égales à $1$, $v_{{\rm sc},i}$, ou $v_{{\rm sc},i} u_{{\rm sc},i}^{-1}$; où $u_{{\rm sc},i}$ et $v_{{\rm sc},i}$ sont des relèvements de $u_{{\rm ad},i}(\zeta_i)$ et $v_{{\rm ad},i}$ dans $SL(n_i,{\Bbb C})$. Pour $v_{{\rm sc},i}$, on peut prendre l'élément $v_i\in SL(n_i,{\Bbb C})$ --- cf. \ref{réduction au cas simple}. 
On définit un élément $h_{\rm sc}^0\in \hat{G}_{\rm SC}$ en prenant $v_{{\rm sc},i}=v_i$ et en rempla\c{c}ant les coordonnées $v_iu_{{\rm sc},i}^{-1}$ par $v_i$ dans la définition de $h_{{\rm sc},i}$ 
(pour $i=1,\ldots ,d$). C'est l'élément en question. Alors $\hat{r}_{T_0}(\phi)\hat{n}(\omega_{T_0}(\phi))$ appartient à $\hat{T}_{0,{\rm sc}}^{\hat{\theta}}h_{\rm sc}^0$, et $\beta_{\chi'}(\phi)$ appartient à $\pi(\hat{T}_{0,{\rm sc}}^{\hat{\theta}})\pi(h_{\rm sc}^0)h^{-1}\mu_{\chi'}(\phi)^{-1}$. Mais, la donnée $\bs{T}'$ étant elliptique, le groupe $\hat{T}_{0,{\rm sc}}^{\hat{\theta},\Gamma_F,\circ}$ est réduit à $\{1\}$. Il en résulte que
$$\hat{T}_{0,{\rm sc}}^{\hat{\theta}}= 
(1-\phi_{T_0})(\hat{T}_{0,{\rm sc}}^{\hat{\theta}}).
$$
Ici, $\phi_{T_0}$ est l'action du Frobenius $\phi$ sur $\hat{T}_0$ ou $\hat{T}_{0,{\rm sc}}$, le tore $\hat{T}_0$ étant muni de l'action galoisienne qui en fait le tore dual de $T_0$. On voit alors que, sans changer la classe du cocycle 
$(\beta_{\chi'}, s_{\rm ad})$, on peut modifier la définition de $\beta_{\chi'}$ et supposer que
$$
\beta_{\chi'}(\phi)= \pi(h_{\rm sc}^0)h^{-1}\mu_{\chi'}(\phi)^{-1}.
$$}
\end{marema3}

\vskip1mm
Prouvons (ii). Considérons deux caractères affines unitaires et non ramifiés $(\chi'_1,\tilde{\chi}'_1)$ et 
$(\chi'_2,\tilde{\chi}'_2)$ de $\wt{T}'(F)$, et supposons que $\omega_{\chi'_1}=\omega_{\chi'_2}$. D'après les calculs ci--dessus et la propriété (5), le cocycle $\eta_{\chi'_1}\eta_{\chi'_2}^{-1}$ annule l'image de $\iota_0$. L'homomorphisme $\iota_0$ s'insère dans une suite exacte
$$
{\rm H}^{1,0}(\Gamma_F; Z(G_{\rm SC})\xrightarrow{1-\theta} Z(G))
\xrightarrow{\iota_0}{\rm H}^{1,0}(\Gamma_F; T_{0,{\rm sc}} \xrightarrow{1-\theta} T_0) \xrightarrow{\psi} 
{\rm H}^{1,0}(\Gamma_F; T_{0,{\rm ad}} \xrightarrow{1-\theta} T_{0,{\rm ad}}). 
$$
Dire que $\eta_{\chi'_1}\eta_{\chi'_2}^{-1}$ annule l'image de $\iota_0$ revient à dire que le caractère associé à $\eta_{\chi'_1}\eta_{\chi'_2}^{-1}$ du groupe central de la suite ci--dessus est le composé de $\psi$ et d'un caractère du groupe de droite. Autrement dit (cf. la remarque 1), cela revient à dire que $\eta_{\chi'_1}\eta_{\chi'_2}^{-1}$ appartient au sous--groupe de ${\rm H}^{1,0}(W_F; \hat{T}_0 
\xrightarrow{1-\hat{\theta}} \hat{T}_{0,{\rm ad}})$ engendré par les images naturelles de 
$\hat{T}_{0,{\rm ad}}^{\Gamma_F,\circ}$ et de l'homomorphisme
$$
\hat{\psi}: {\rm H}^{1,0}(W_F; \hat{T}_{0,{\rm sc}}\xrightarrow{1-\hat{\theta}} 
\hat{T}_{0,{\rm sc}})\rightarrow {\rm H}^{1,0}(W_F; \hat{T}_0
\xrightarrow{1-\hat{\theta}} \hat{T}_{0,{\rm ad}}).
$$
On peut se débarrasser du groupe $\hat{T}_{0,{\rm ad}}^{\Gamma_F,0}$: c'est l'image naturelle de $\hat{T}_{0,{\rm sc}}^{\Gamma_F,\circ}$, donc il est contenu dans l'image de $\hat{\psi}$. Soit un couple $(\beta,t)$ définissant un élément de ${\rm H}^{1,0}(W_F; \hat{T}_{0,{\rm sc}}\xrightarrow{1-\hat{\theta}} 
\hat{T}_{0,{\rm sc}})$ tel que $\hat{\psi}(\beta,t)= \eta_{\chi'_1}\eta_{\chi'_2}^{-1}$. D'après les calculs plus haut (cf. les remarques 2 et 3), on a
$$
\eta_{\chi'_1}\eta_{\chi'_2}^{-1}= (\mu_{\chi'_1}\mu_{\chi'_2}^{-1},1).
$$
Il existe donc un $u\in \hat{T}_0$ tel que
$$
\mu_{\chi'_2}(w)\mu_{\chi'_1}(w)^{-1}=w_{T_0}(u)u^{-1}\pi(\beta(w)),\quad w\in W_F,\leqno{(6)}
$$
$$
1= (1-\hat{\theta})(u_{\rm ad})t.\leqno{(7)}
$$
Comme dans la remarque 3, on note $w_{T_0}$ l'action de $w$ sur $\hat{T}_0$, vu comme tore dual de $T_0$. En particulier, on a $\phi_{T_0}= {\rm Int}_h \circ \phi$, où le dernier $\phi$ est l'action naturelle du Frobenius sur $\hat{G}$. 

On écrit $u= z_u\pi(u_{\rm sc})$ avec $z_u\in Z(\hat{G})$ et $u_{\rm sc}\in \hat{G}_{\rm SC}$. On peut remplacer $\beta$ par $\beta'$ défini par
$$
\beta'(w)= w_{T_0}(u_{\rm sc})u_{\rm sc}^{-1}\beta(w),\quad w\in W_F,
$$
et remplacer $t$ par
$$
t'= (1-\hat{\theta})(u_{\rm sc})t.
$$
Cela nous ramène au cas où $u=z_u$. En oubliant cet épisode, on suppose simplement que $u\in Z(\hat{G})$. La relation (7) équivaut alors à $t\in \hat{Z}_{\rm sc}=Z(\hat{G}_{\rm SC})$. Puisque $\mu_{\chi'_1}$ et $\mu_{\chi'_2}$ sont à valeurs dans $\hat{T}_0^{\hat{\theta},\circ}$, la relation (6) entra\^{\i}ne que $\beta$ prend ses valeurs dans $\hat{Z}_{\rm sc}\hat{T}^{\hat{\theta}}_{0,{\rm sc}}$. On a vu dans la remarque 3 que $\hat{T}_{0,{\rm sc}}^{\hat{\theta}}=(1-\phi_{T_0})(\hat{T}_{0,{\rm sc}}^{\hat{\theta}})$. On peut donc écrire $\beta(\phi)= y(1-\phi_{T_0})(u')$ avec $y\in \hat{Z}_{\rm sc}$ et $u'\in \hat{T}_{0,{\rm sc}}^{\hat{\theta}}$. La relation de cocycle dit que $(1-\hat{\theta})(y)=\phi_{T_0}(t)t^{-1}$. Alors $y$ vérifie l'hypothèse du lemme de \ref{séparation des données}. En conséquence, on peut écrire $y= \eta(x,\bs{h}_{\rm sc})(1-\phi_{T_0})(u'')$ avec $x \in \mathfrak{Z}(\tilde{s}_{\rm sc},\bs{h}_{\rm sc})$ et $u''\in \hat{Z}_{\rm sc}$. En posant $u_0= u^{-1}\pi(u'u'')$, on a $u_0\in Z(\hat{G})\hat{T}_0^{\hat{\theta},\circ}$ et la relation (6) devient
$$
\mu_{\chi'_2}(w)\mu_{\chi'_1}(w)^{-1}= (1-\phi_{T_0})(u_0)\pi(\eta(x,\bs{h}_{\rm sc})),\quad w\in W_F.
$$
On se rappelle que $\phi_{T_0}= {\rm Int}_h\circ \phi$. Mais alors l'élément $u_0\pi(x)$ réalise un automorphisme de la donnée $\bs{T}'$ qui envoie le caractère $\chi'_1$ sur $\chi'_2$. Cela prouve le point (ii) de la proposition.
\end{proof}

\begin{marema4}
{\rm \`A l'aide des mêmes calculs, on pourrait prouver simultanément cette proposition et celle de \ref{séparation des données}. Mais il nous a semblé plus clair de séparer les preuves. }
\end{marema4}

\subsection{Action de $\ES{C}$ sur les transferts sphériques.}Soit $\bs{T}'=(T',\ES{T}',\tilde{s})$ un élément de $\mathfrak{E}_{\rm t-nr}$. On a vu (\ref{l'application tau}) qu'à un caractère affine unitaire et non ramifié $(\chi',\tilde{\chi}')$ de $\wt{T}'(F)$ est associé un triplet elliptique essentiel non ramifié $\tau'= 
\tau_{\chi'} \in E_{\rm ell}^{\rm nr}$, bien défini à conjugaison près. Ce triplet $\tau'$ se relève en un triplet $\bs{\tau}'=\bs{\tau}_{\tilde{\chi}'}\in \ES{E}_{\rm ell}^{\rm nr}$, lui aussi bien défini à conjugaison près, en imposant la condition $\Theta_{\bs{\tau}'}(\bs{1}_{\wt{K}})= \tilde{\chi}'(\bs{1}_{\wt{K}'})$, où $(K'\!,\wt{K}')$ est le sous--espace hyperspécial de $\wt{T}'(F)$ associé à $(K,\wt{K})$. D'après le lemme 4 de \ref{action du groupe C} et la proposition de \ref{action de C sur les transferts usuels}, les distributions $\Theta_{\bs{\tau}'}$ et $\bs{\rm T}_{\bs{T}'}(\wt{\chi}')$ sont des vecteurs propres pour l'action (par conjugaison) de $\ES{C}$ relativement à des caractères $\omega_{\tau'}$ et $\omega_{\chi'}$ de ce groupe.

\begin{mapropo}
Soient $\bs{T}'=(T',\ES{T}',\tilde{s})\in \mathfrak{E}_{\rm t-nr}$ et $\chi'$ un caractère unitaire et non ramifié de $\wt{T}'(F)$. Posons
$\tau'= \tau_{\chi'}\;(\in E_{\rm ell}^{\rm nr}/{\rm conj.})$. On a l'égalité
$$
\omega_{\tau'}= \omega_{\chi'}.
$$
\end{mapropo}

\begin{proof}Reprenons les constructions et les notations de la preuve de la proposition de \ref{action de C sur les transferts usuels}. Dans cette preuve, on a défini un homomorphisme $\ES{C}\rightarrow {\rm H}^{1,0}(\Gamma_F; Z(G_{\rm SC})\buildrel 1-\theta\over{\longrightarrow}Z(\hat{G}))$, 
que l'on note ici $\jmath_Z$. On a donc $\jmath_Z(\bs{c})= b_{\bs{c}}$, $\bs{c}\in \ES{C}$. On a aussi un tore maximal $\hat{T}_0$ de $\hat{G}$, muni de l'action galoisienne non ramifiée telle que $\phi_{T_0}= {\rm Int}_h \circ \phi$, où $\phi$ est l'action naturelle du Frobenius sur $\hat{G}$. Ainsi, $\hat{T}_0^{\hat{\theta},\circ}$ muni de son action galoisienne, s'identifie à $\hat{T}'$ muni de son action galoisienne. Fixons un élément $h_{0,{\rm sc}}\in \hat{G}_{\rm SC}^{\hat{\theta}}$ qui normalise $\hat{T}$ et dont l'action sur ce tore co\"{\i}ncide avec ${\rm Int}_h$ (par exemple l'élément $h_{\rm sc}^0$ de la remarque 3 de \ref{action de C sur les transferts usuels}). Notons 
$\bar{\beta}_0$ le cocycle non ramifié de $W_F$ dans $\hat{T}_0$ tel que $\bar{\beta}_0(\phi)= b'h\pi(h_{0,{\rm sc}})^{-1}$, où $b'=\mu_{\chi'}(\phi)$ est la valeur en $\phi$ d'un cocycle non ramifié $\mu_{\chi'}: W_F \rightarrow \hat{T}'$ associé à $\chi'$. D'après \ref{l'application tau}, on peut supposer que $b'$ appartient à $Z(\hat{G})\cap \hat{T}'$. Le couple $(\bar{\beta}_0,s_{\rm ad}^{-1})$ détermine un élément de ${\rm H}^{1,0}(W_F;\hat{T}_0\xrightarrow{1-\hat{\theta}}\hat{T}_{0,{\rm ad}})$, que l'on note $\bar{\eta}_0$. Cet élément n'est autre que l'inverse de l'élément noté $\eta_{\chi'}$ dans la preuve de la proposition de \ref{action de C sur les transferts usuels}. D'autre part, notons $\jmath_0: \ES{C} \rightarrow {\rm H}^1(W_F; T_{0,{\rm sc}}\xrightarrow{1-\theta} \hat{T}_0)$ le composé de $\jmath_Z$ et de l'homomorphisme
$$
\iota_0: {\rm H}^{1,0}(\Gamma_F; Z(G_{\rm SC})\xrightarrow{1-\theta}Z(\hat{G}))
\rightarrow 
{\rm H}^1(W_F; T_{0,{\rm sc}}\xrightarrow{1-\theta} \hat{T}_0).
$$
Alors le caractère $\omega_{\chi'}$ est donné par
$$\omega_{\chi'}(\bs{c})= \langle \jmath_0(\bs{c}), \bar{\eta}_0\rangle^{-1},\quad \bs{c}\in \ES{C}.
\leqno{(1)}
$$ 

\'Ecrivons $\tau'=(T,\lambda,\tilde{r})$. Rappelons comment ce triplet a été construit. On a choisi un élément $y\in \hat{G}_{\rm SC}$ tel que ${\rm Int}_{y^{-1}}(\bs{h})= \underline{h}\phi$ pour un élément $\underline{h}\in \hat{T}$, et $\lambda$ est le caractère unitaire et non ramifié de $T(F)$ associé au cocycle non ramifié $W_F\rightarrow \hat{T}$ dont la valeur en $\phi$ est $b'\underline{h}$. L'automorphisme ${\rm Int}_y$ induit une application bijective de $R^{\wt{G},\omega}(\lambda)\rightarrow \wt{\bs{S}}(\varphi'\!,a)$, où $\varphi': W_F\rightarrow {^LG}$ est le paramètre tempéré et non ramifié déduit de $\mu_{\chi'}$ via l'isomorphisme ${^LT'}\simeq \ES{T}'\;(\subset {^LG})$. L'élément $\tilde{r}'$ est celui correspondant à l'image de $\tilde{s}\in \wt{S}_{\varphi'\!,a}$ dans le $S$--groupe tordu $\wt{\bs{S}}(\varphi'\!,a)= \wt{S}_{\varphi'\!,a}/S_{\varphi'}^\circ Z(\hat{G})^{\Gamma_F}$. L'automorphisme ${\rm Int}_{\tilde{r}}$ conserve $T$ et agit sur ce tore par $w\theta$ pour un certain élément $w \in W^{\Gamma_F}$. On a des complexes en dualité
$$
T_{\rm sc}\xrightarrow{1-w\theta} T,\quad \hat{T} \xrightarrow{1- \hat{\theta}w^{-1} }\hat{T}_{\rm ad}.
$$
Remarquons que $1-w\theta$ se restreint en $1-\theta$ sur $Z(G)$. On a donc un homomorphisme
$$
\iota:{\rm H}^{1,0}(\Gamma_F; Z(G_{\rm SC})\xrightarrow{1-\theta} Z(G))
\rightarrow {\rm H}^{1,0}(\Gamma_F; T_{\rm sc} \xrightarrow{1-w\theta} T)
$$
qui, composé avec $\jmath_Z$, donne un homomorphisme
$$
\jmath: \ES{C}\rightarrow {\rm H}^{1,0}(\Gamma_F; T_{\rm sc} \xrightarrow{1-w\theta} T). 
$$
D'autre part, l'automorphisme ${\rm Int}_{y^{-1}}$ de $\hat{G}$ envoie le tore $\hat{U}'= S_{\varphi'}^0$  sur $\hat{U}=\hat{T}^{\Gamma_F,\circ}$, et $\wt{S}_{\varphi',a}$ sur $\wt{S}_{\psi,a}$, où $\psi: W_F\rightarrow {^LG}$ est le paramètre tempéré non ramifié donné par $\psi(\phi)=b'\underline{h}\rtimes \phi$. Les éléments de $\wt{S}_{\psi,a}$ sont de la forme $x\hat{\theta}$ où $x$ est un élément de $N_{\hat{G}}(\hat{T})$ qui opère sur $\hat{T}$ par un élément de $W^{\Gamma_F}$. En particulier, 
puisque ${\rm Int}_{y^{-1}}(\tilde{s})= y^{-1}s\theta(y)\hat{\theta}$ appartient à $\wt{S}_{\psi,a}$, l'élément $n= y^{-1}s\theta(y)$ appartient à $N_{\hat{G}}(\hat{T})$, et d'après la remarque 2 de 
\ref{séries principales nr}, les deux termes $\hat{\theta}w^{-1}$ et ${\rm Int}_n\circ \hat{\theta}$ ont la même action sur $\hat{T}$. L'automorphisme ${\rm Int}_n$ de $\hat{T}$ commute à l'action galoisienne, et on sait que le groupe $W^{\Gamma_F}=W^\phi$ se relève dans la composante neutre $\hat{G}_{\rm SC}^{\phi,\circ}$ du sous--groupe $\hat{G}_{\rm SC}^\phi\subset \hat{G}_{\rm SC}$ des éléments qui sont fixés par $\phi$ (ou encore par $\Gamma_F$). On peut donc décomposer $n$ en $n=\xi\pi(u)$ où $\xi\in \hat{T}$ et $u\in \hat{G}_{\rm SC}^{\phi,\circ}$. Notons $\mu_\lambda: W_F\rightarrow \hat{T}$ le cocycle non ramifié dont la valeur en $\phi$ est $b'\underline{h}$. Il résulte de la description ci--dessus que $(1- {\rm Int}_n\circ \hat{\theta})(\mu_{\lambda})$ est le cocycle non ramifié dont la valeur en $\phi$ est $a(\phi)^{-1}\xi \phi(\xi)^{-1}$. Soit $\xi_{\rm ad}\in \hat{T}_{\rm ad}$ la projection de $\xi$ sur $\hat{G}_{\rm AD}$. Le couple $(\mu_\lambda, \xi_{\rm ad}^{-1})$ détermine donc un élément $\bar{\eta}\in {\rm H}^{1,0}(W_F; \hat{T}\xrightarrow{1-{\rm Int}_n\circ \hat{\theta}} \hat{T}_{\rm ad})$. 

Soit $\bs{\tau}'\in \ES{E}_{\rm ell}^{\rm nr}$ un relèvement de $\tau'$. 

\begin{monlem}
Pour $\bs{c}\in \ES{C}$, on a l'égalité ${^{\bs{c}}(\Theta_{\bs{\tau}'}})= \langle \jmath(\bs{c}), \bar{\eta} \rangle^{-1}\Theta_{\bs{\tau}'} $
\end{monlem} 

\begin{proof}
Rappelons que l'homomorphisme $\jmath_Z$ se factorise en un isomorphisme de $\overline{\ES{C}}=\ES{\ES{C}}/\pi_{\ES{C}}(G_{\rm SC}(F))$ sur ${\rm H}^{1,0}(\Gamma_F; Z(G_{\rm SC})\xrightarrow{1-\theta} Z(G))$. Par suite, l'égalité de l'énoncé est claire si $\bs{c}\in \pi(G_{\rm SC}(F))$. Comme le groupe $G_{\rm AD}(F)$ est produit de $T_{\rm ad}(F)$ et de l'image naturelle de $G_{\rm SC}(F)$ dans $G_{\rm AD}(F)$, il suffit de prouver l'égalité de l'énoncé pour $\bs{c}= \bs{q}(z,g)$ avec $g_{\rm ad}\in T_{\rm ad}(F)$, donc $g\in T(\overline{F})$. Soit donc $(z,g)\in C$ tel que $g_{\rm ad}\in T_{\rm ad}(F)$. Posons $\bs{c}=\bs{q}(z,g)\in \ES{C}$ et $t= zg^{-1} {\rm Int}_w\circ \theta(g)\in T(\overline{F})$. Alors $t\in T(F)$ et l'on a (\ref{action du groupe C}, lemme 3)
$$
{^{\bs{c}}(\Theta_{\bs{\tau}'})}= \lambda(t) \Theta_{\bs{\tau}'}.\leqno{(2)}
$$
Pour prouver le lemme, il faut donc montrer l'égalité
$$
\lambda(t)= \langle \jmath(\bs{c}), \bar{\eta} \rangle^{-1}.\leqno{(3)}
$$ 
On a $\lambda(t)= \langle t, \mu_\lambda \rangle$, où l'accouplement est celui entre $T(F)$ et ${\rm H}^1(W_F; \hat{T})$. Rappelons la définition de $\jmath(\bs{c})$. On écrit $g= z_g \pi(g_{\rm sc})$ avec $z_g\in Z(G)$ et $g_{\rm sc}\in G_{\rm SC}$, et l'on note $\sigma \mapsto \alpha(\sigma)$ le cocycle galoisien à valeurs dans $Z(G_{\rm SC})$ défini par $\alpha(\sigma)=g_{\rm sc}\sigma(g_{\rm sc})^{-1}$. Alors $\jmath(\bs{c})$ est l'élément de ${\rm H}^{1,0}(\Gamma_F; T_{\rm sc}\xrightarrow{1-\theta} T)$ défini par le couple $(\alpha, z^{-1}(1-\theta)(z_g))$. Puisque $g_{\rm sc}$ appartient à $T_{\rm sc}$, on vérifie que ce couple est cohomologue à $(1,t^{-1})$. On a des homomorphismes
$$
{\rm H}^0(\Gamma_F; T)\rightarrow {\rm H}^{1,0}(\Gamma_F; T_{\rm sc} \xrightarrow{1-w\theta} T)
$$
et
$$
{\rm H}^1(W_F,\hat{T})\rightarrow {\rm H}^{1,0}(W_F; \hat{T} \xrightarrow{1-{\rm Int}_n\circ \hat{\theta}} \hat{T}_{\rm ad})
$$
qui, d'après \cite[A.3.13]{KS1}, sont en dualité (le signe dispara\^{\i}t d'après une correction ultérieure \cite{KS2}). On en déduit l'égalité (3).
\end{proof}  

Pour démontrer la proposition, il reste à prouver l'égalité
$$
\langle \jmath_0(\bs{c}), \bar{\eta}_0\rangle = \langle \jmath (\bs{c}), \bar{\eta}\rangle, \quad \bs{c}\in \ES{C}. \leqno{(4)}
$$ 
On a le diagramme commutatif exact suivant de complexes de tores
$$
\xymatrix{
Z(G_{\rm SC}) \ar[d]_{{\rm diag}_{-}}\ar[r]^{1-\theta}&  Z(G)\ar[d]^{{\rm diag}_{-}}\\
T_{0,{\rm sc}}\times T_{\rm sc}\ar[d]\ar[r]^{(1-\theta)\times (1-w\theta)} & T_0\times T \ar[d]\\
(T_{0,{\rm sc}}\times T_{\rm sc})/{\rm diag}_{-}(Z(G_{\rm SC})) \ar[r]& (T_0\times T)/{\rm diag}_{-}(Z(G))
}
$$
où ${\rm diag}_{-}$ est le plongement antidiagonal, les flèches verticales du bas sont les projections naturelles, et la flèche horizontale du bas est celle déduite de $(1-\theta)\times (1-w\theta)$ par passage aux quotients. On en déduit une suite exacte de groupes de cohomologie
$$
{\rm H}^{1,0}(\Gamma_F; Z(G_{\rm SC})\xrightarrow{1-\theta}Z(G))\rightarrow 
{\rm H}^{1,0}(\Gamma_F; T_{0,{\rm sc}}\times T_{\rm sc}\xrightarrow{(1-\theta)\times (1-w\theta)}T_0\times T)\rightarrow\cdots 
$$
$$
\cdots \rightarrow {\rm H}^{1,0}(\Gamma_F; (T_{0,{\rm sc}}\times T_{\rm sc})/{\rm diag}_{-}(Z(G_{\rm SC}))\xrightarrow{(1-\theta)\times (1-w\theta)} (T_0\times T)/{\rm diag}_{-}(Z(G)).
$$
Le couple $(\bar{\eta}_0,\bar{\eta})$ définit une forme linéaire sur le groupe central, et (4) signifie qu'elle annule l'image du premier groupe, ou encore qu'elle se factorise en une forme linéaire sur le dernier groupe. Posons $X=(T_0\times T)/{\rm diag}_{-}(Z(G))$, $Y= (T_{0,{\rm sc}}\times T_{\rm sc})/{\rm diag}_{-}(Z(G_{\rm SC}))$, et notons $\hat{X} \xrightarrow{f} \hat{Y}$ le complexe dual du  complexe de tores $Y\xrightarrow{(1-\theta)\times (1-w\theta)} X$. On a dualement un homomorphisme
$$
{\rm H}^{1,0}(W_F; \hat{X} \xrightarrow{f} \hat{Y})\rightarrow {\rm H}^{1,0}(W_F;\hat{T}_0\xrightarrow{1-\hat{\theta}} \hat{T}_{0,{\rm ad}})\times {\rm H}^{1,0}(W_F; \hat{T}\xrightarrow{1-{\rm Int}_n\circ \hat{\theta}} \hat{T}_{\rm ad}),\leqno{(5)}
$$
et il suffit de prouver que
\begin{enumerate}
\item[(6)]le couple $(\hat{\eta}_0,\hat{\eta})$ appartient à l'image de l'homomorphisme (5).
\end{enumerate}
Commen\c{c}ons par décrire les groupes $\hat{X}$ et $\hat{Y}$. Le second est facile à identifier: on a $\hat{Y}= (\hat{T}_{0,{\rm sc}}\times \hat{T})/ {\rm diag}(Z(\hat{G}_{\rm SC}))$, où ${\rm diag}$ est le plongement diagonal. Le premier est plus compliqué. Algébriquement, $\hat{X}$ est le groupe des $(t_0,t,t_{\rm sc})\in \hat{T}_0\times \hat{T}\times \hat{T}_{\rm sc}$ tels que $t_0t^{-1}= \pi(t_{\rm sc})$ (algébriquement, on a $\hat{T}_0=\hat{T}$). La structure galoisienne est bizarre (cf. \cite[2.2]{Stab I}), mais elle est évidemment non ramifiée et cela nous suffira. L'homomorphisme $f$ est défini de la fa\c{c}on suivante. Soit $(t_0,t,t_{\rm sc})\in \hat{X}$. On écrit $t=z_t \pi(t^*_{\rm sc})$ avec $z_t\in Z(\hat{G})$ et $t^*_{\rm sc}\in \hat{T}_{\rm sc}$. Alors on a $t_0= z_t \pi(t^*_{\rm sc}t_{\rm sc})$ et $f(t_0,t,t_{\rm sc})$ est l'image de $((1-\theta)(t^*_{\rm sc}t_{\rm sc}),(1-{\rm Int}_n\circ \hat{\theta})(t^*_{\rm sc}))$ dans $\hat{Y}$. En se remémorant les définitions, on voit que, pour prouver (6), il suffit de trouver un élément $(t_0,t,t_{\rm sc})\in \hat{X}$ et un couple $(d_{0,{\rm sc}},d)\in \hat{T}_{0,{\rm sc}}\times \hat{T}$ vérifiant les conditions suivantes: 
\begin{enumerate}
\item[(7)]$t_0= b'h\pi(h_{0,{\rm sc}})^{-1}$, $t=b'\underline{h}$;
\item[(8)] l'image de $d_{0,{\rm sc}}$ dans $\hat{T}_{0,{\rm ad}}$ est $s_{\rm ad}^{-1}$ et l'image de $d$ dans 
$\hat{T}_{\rm ad}$ est $\xi_{\rm ad}^{-1}$;
\item[(9)] il existe $\zeta\in Z(\hat{G}_{\rm SC})$ tel que $(\phi_{T_0}(d_{0,{\rm sc}})d_{0,{\rm sc}}^{-1}, \phi_T(d)d^{-1})= (\zeta,\zeta)f(t_0,t,t_{\rm sc})$.
\end{enumerate}
On écrit $h=z_h\pi(h_{\rm sc})$ et $s=z_s\pi(s_{\rm sc})$ avec $z_h,\,z_s\in Z(\hat{G})$ et $h_{\rm sc},\, s_{\rm sc}\in \hat{G}_{\rm SC}$. Soit $a_{\rm sc}$ l'élément de $Z(\hat{G}_{\rm SC})$ défini par $$
a_{\rm sc}= s_{\rm sc}\hat{\theta}(h_{\rm sc})\phi(s_{\rm sc})^{-1}h_{\rm sc}^{-1}.
$$
On pose $\underline{h}_{\rm sc}= y^{-1}h_{\rm sc}\phi(y)$. On définit $t_0$ et $t$ par (7), et on pose 
$t_{\rm sc}= h_{\rm sc}h_{0,{\rm sc}}^{-1}\underline{h}_{\rm sc}$. On vérifie que $\pi(t_{\rm sc})= t_0t^{-1}$, donc $t_{\rm sc}$ appartient bien à $\hat{T}_{\rm sc}$. On pose $d_{0,{\rm sc}}= s_{\rm sc}^{-1}$ et $n_{\rm sc}= y^{-1}s_{\rm sc}\hat{\theta}(y)$. L'image de $n_{\rm sc}$ dans $\hat{G}_{\rm AD}$ est la même que celle de $n=\xi \pi(u)$. Donc $n_{\rm sc} = \xi_{\rm sc}u$, où $\xi_{\rm sc}$ est un élément de $\hat{T}_{\rm sc}$ qui se projette sur $\xi_{\rm ad}$ (dans $\hat{G}_{\rm AD}$). On pose $d=\pi(\xi_{\rm sc})^{-1}$. Alors (8) est vérifié. On calcule $f(t_0,t,t_{\rm sc})$ en posant $t'_{\rm sc}= \underline{h}_{\rm sc}$. On a $t'_{\rm sc}t_{\rm sc}= h_{\rm sc}h_{0,{\rm sc}}^{-1}$, donc
$$
f(t_0,t,t_{\rm sc})= ((1-\hat{\theta})(h_{\rm sc}h_{0,{\rm sc}}^{-1}),(1-{\rm Int}_n\circ \hat{\theta})(\underline{h}_{\rm sc})).
$$
On se rappelle que $h_{0,{\rm sc}}$ est fixé par $\hat{\theta}$. Par conséquent $(1-\hat{\theta})(h_{\rm sc}h_{0,{\rm sc}}^{-1})=(1-\hat{\theta})(h_{\rm sc})$, et ce terme vaut (calcul) $a_{\rm sc}^{-1}\phi_{T_0}(d_{0,{\rm sc}})d_{0,{\rm sc}}^{-1}$. On a
$$
(1-{\rm Int}_n\circ \hat{\theta})(\underline{h}_{\rm sc})=\underline{h}_{\rm sc} n_{\rm sc} \hat{\theta}(\underline{h}_{\rm sc})^{-1} n_{\rm sc}^{-1}.
$$
En conjuguant par $y^{-1}$ l'égalité $s_{\rm sc}\hat{\theta}h_{\rm sc}\phi=a_{\rm sc}h_{\rm sc}\phi s_{\rm sc}\hat{\theta}$, on obtient l'égalité
$$n_{\rm sc}\hat{\theta}\underline{h}_{\rm sc}\phi= a_{\rm sc} \underline{h}_{\rm sc}\phi n_{\rm sc}\hat{\theta}.
$$
On voit alors que
\begin{eqnarray*}
\underline{h}_{\rm sc}n_{\rm sc}\hat{\theta}(\underline{h}_{\rm sc})^{-1}n_{\rm sc}^{-1}&=&a_{\rm sc}^{-1}\underline{h}_{\rm sc}n_{\rm sc}\phi(n_{\rm sc})^{-1}\underline{h}_{\rm sc}^{-1}\\
&=& a_{\rm sc}^{-1}\underline{h}_{\rm sc}\xi_{\rm sc}\phi(\xi_{\rm sc})^{-1}\underline{h}_{\rm sc}^{-1}\\
&=&a_{\rm sc}^{-1}\xi_{\rm sc}\phi(\xi_{\rm sc})^{-1}\\
&=& a_{\rm sc}^{-1}\phi_T(d)d^{-1}.
\end{eqnarray*}
Mais alors la condition (9) est vérifiée avec $\zeta =a_{\rm sc}$. Cela achève la démonstration de la proposition. 
\end{proof}

%


\end{document}